%% file: Thesis.tex
\documentclass[11pt,a4paper,oneside]{amsbook} 
\pdfoutput=1

\numberwithin{section}{chapter}
\numberwithin{equation}{chapter}
\numberwithin{figure}{chapter}
\numberwithin{table}{chapter}

\usepackage[top=1in, bottom=1.25in, left=1.5in, right=1in]{geometry}
\usepackage{microtype} 
\usepackage{amsmath}
\usepackage{amssymb}
\usepackage{mathrsfs} 
\usepackage{graphicx} 
\usepackage{caption}
\usepackage{enumerate} 
\usepackage{color}
\usepackage{tikz} 
\usetikzlibrary{calc} 
\usepackage{rotating} 
\usepackage{multirow} 
\usepackage[section]{placeins} 

\usepackage[colorlinks=false, pdfborder={0 0 0}]{hyperref}

\newtheorem{theorem}{Theorem}[chapter]
\newtheorem{proposition}[theorem]{Proposition}
\newtheorem{lemma}[theorem]{Lemma}
\newtheorem{corollary}[theorem]{Corollary}

\newtheorem{observation}[theorem]{Observation}
\newtheorem*{definition}{Definition}
\newtheorem{conjecture}{Conjecture}
\newtheorem*{thm_nonumber}{Theorem}
\newtheorem*{obs_nonumber}{Observation}
\newtheorem*{prop_nonumber}{Proposition}
\newtheorem*{lemma_nonumber}{Lemma}

\newcommand{\bbF}{\mathbb{F}}
\newcommand{\bbL}{\mathbb{L}}

\newcommand{\R}{\mathbb{R}}
\newcommand{\bbS}{\mathbb{S}}
\newcommand{\Z}{\mathbb{Z}}
\newcommand{\N}{\mathbb{N}}

\newcommand{\cD}{\mathscr{D}}
\newcommand{\cE}{\mathscr{E}}
\newcommand{\cF}{\mathscr{F}_{\lambda}}
\newcommand{\cG}{\mathscr{G}}
\newcommand{\cH}{\mathscr{H}}
\newcommand{\cI}{\mathscr{I}}
\newcommand{\cO}{\mathcal{O}}
\newcommand{\cP}{\mathscr{P}}
\newcommand{\cQ}{\mathscr{Q}}
\newcommand{\cR}{\mathcal{R}}

\newcommand{\cT}{\mathscr{T}}
\newcommand{\cX}{\mathscr{X}}

\newcommand{\bfe}{\mathbf{e}}

\newcommand{\bfL}{\mathbf{L}}

\newcommand{\bfv}{\mathbf{v}}
\newcommand{\bfw}{\mathbf{w}}
\newcommand{\bfx}{\mathbf{x}}
\newcommand{\bfy}{\mathbf{y}}

\newcommand{\bA}{\bar{A}}
\newcommand{\F}{F_{\lambda}}
\newcommand{\Hl}{H_{\lambda}}
\newcommand{\Le}{\mathbb{L}^e}
\newcommand{\id}{\mathrm{id}}
\newcommand{\cIe}{\mathscr{I}^e}
\newcommand{\Ie}{I^e}
\newcommand{\Xe}{X^e}
\newcommand{\tZ}{\tilde{z}}
\newcommand{\trho}{\tilde{\rho}}
\newcommand{\brho}{\bar{\rho}}
\newcommand{\lZ}{(\lambda\Z)^2}
\newcommand{\phil}{\varphi^{\lambda}}

\newcommand{\lcm}{\mathrm{lcm}}
\newcommand{\Ot}{{\mathcal O}_{\tau}}

\def\defn#1{\textbf{#1}}
\def\Fix#1{\mathrm{Fix}\,#1}
\def\mod#1{\hskip 15pt \left(\mathrm{mod}\;\,#1\right)}
\def\hl#1{#1}

\def\fl#1{\lfloor #1\rfloor}
\def\Bfl#1{\left\lfloor #1 \right\rfloor}
\def\ceil#1{\lceil #1\rceil}
\def\Bceil#1{\left\lceil #1\right\rceil}

\title{Near-integrable behaviour\\in a family of discretised rotations}
\author{Heather Reeve-Black}
\address{School of Mathematical Sciences, Queen Mary, 
University of London, London E1 4NS, UK}
\email{h.reeve-black@qmul.ac.uk}
\date{\today}

\begin{document}

\mainmatter
\include{FrontMatter}

\include{Introduction}
\include{Preliminaries}

\include{IntegrableLimit}

\include{PerturbedDynamics}

\include{ApeirogonLimit}

\include{PerturbedDynamicsAtInfinity}

\include{Conclusion}

\appendix

\include{Appendix}

\backmatter
\bibliographystyle{alpha}
\bibliography{Bibliography}

\end{document}

%% file: FrontMatter.tex
\begin{titlepage}
\setcounter{page}{-2}

\raggedleft
\includegraphics[scale=0.27]{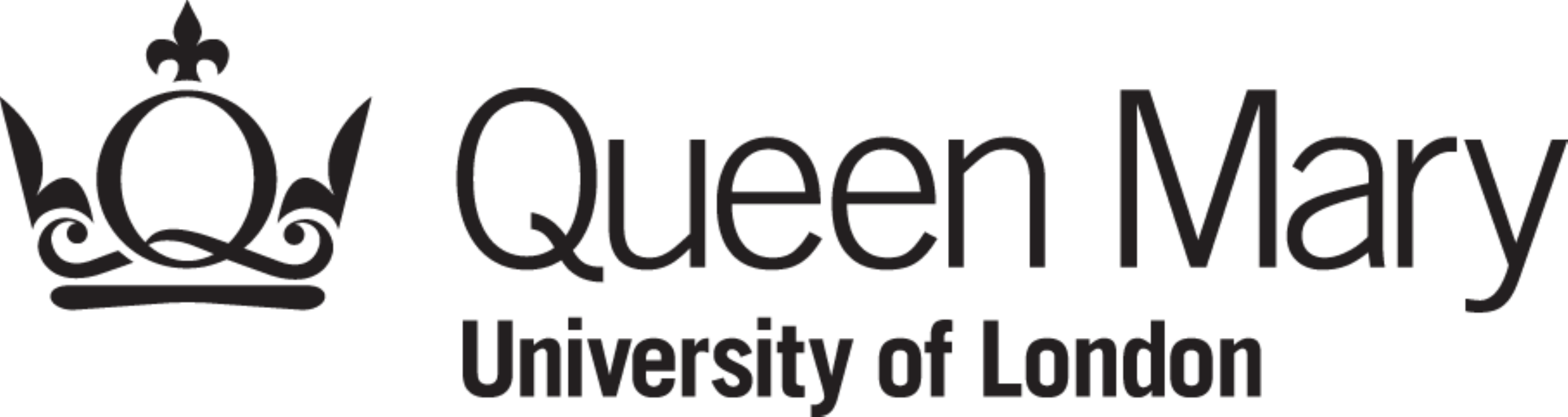}

\begin{center}

~ \\[4.5cm]

\textbf{\Huge Near-integrable behaviour\\[0.5cm]in a family of discretised rotations}\\[2cm]

{\Huge Heather Reeve-Black} \\[0.5cm]

\textsc{\large School of Mathematical Sciences\\[0.1cm] Queen Mary University of London}\\[6cm]

{\Large Submitted in partial fulfilment of the requirements\\[0.15cm]of the Degree of Doctor of Philosophy.}\\[0.3cm]

{\Large May 2014}

\end{center}
\end{titlepage}


\newpage 
\thispagestyle{empty}
~ \\[1cm]
\begin{center}
\textbf{\Large Declaration of originality}\\[1cm]
\end{center}

I, Heather Reeve-Black, confirm that the research included within this thesis is my own work or that where it has been carried out in collaboration with, or supported by others, that this is duly acknowledged below and my contribution indicated. Previously published material is also acknowledged below.

I attest that I have exercised reasonable care to ensure that the work is original, and does not to the best of my knowledge break any UK law, infringe any third party's copyright or other Intellectual Property Right, or contain any confidential material.

I accept that the College has the right to use plagiarism detection software to check the electronic version of the thesis.

I confirm that this thesis has not been previously submitted for the award of a degree by this or any other university.

The copyright of this thesis rests with the author and no quotation from it or information derived from it may be published without the prior written consent of the author.
\vspace{1cm}

\qquad Heather Reeve-Black \\

\qquad 19.05.2014
\vspace{2.5cm}

\noindent
Details of collaboration and publications: The work in this thesis was supervised by F. Vivaldi (throughout) and J.A.G. Roberts (chapter \ref{chap:Apeirogon}). The majority of chapters \ref{chap:IntegrableLimit} and \ref{chap:PerturbedDynamics} have been published in
{\it `Near-Integrable Behaviour in a Family of Discretized Rotations'}, 
H. Reeve-Black \& F. Vivaldi, Nonlinearity {\bf 26}, 2013.

\newpage 
\thispagestyle{empty}
~ \\[1cm]

\newpage 
\thispagestyle{plain} 
~ \\[1cm]
\begin{center}
\begin{minipage}{14cm}
\textsc{Abstract.} We consider a one-parameter family of invertible maps of a two-dimensional lattice, 
obtained by applying round-off to planar rotations.
All orbits of these maps are conjectured to be periodic.
We let the angle of rotation approach $\pi/2$, and show that the limit of 
vanishing discretisation is described by an integrable piecewise-affine 
Hamiltonian flow, whereby the plane foliates into families of invariant 
polygons with an increasing number of sides.

\hskip 15pt
Considered as perturbations of the flow, the lattice maps assume a different character, 
described in terms of strip maps: a variant of those found in outer billiards of polygons.
Furthermore, the flow is nonlinear (unlike the original rotation), 
and a suitably chosen Poincar\'{e} return map satisfies a twist condition.

\hskip 15pt
The round-off perturbation introduces KAM-type phenomena:
we identify the unperturbed curves which survive the perturbation, 
and show that they form a set of positive density in the phase space. 
We prove this considering symmetric orbits, under a condition that
allows us to obtain explicit values for densities.

\hskip 15pt
Finally, we show that the motion at infinity is a dichotomy:
there is one regime in which the nonlinearity tends to zero, leaving only the perturbation,
and a second where the nonlinearity dominates. 
In the domains where the nonlinearity remains, numerical evidence
suggests that the distribution of the periods of orbits is consistent with that of random dynamics, 
whereas in the absence of nonlinearity, the fluctuations result in intricate discrete resonant structures.
\end{minipage}
\end{center}


\newpage
\tableofcontents


\newpage \thispagestyle{plain}
~ \\[1cm]
\begin{center}
\textbf{\Large Acknowledgements}\\[1cm]
\end{center}

Thank-you first and foremost to Franco Vivaldi for supervising this project---I hope you've enjoyed it as much as I have.
Before I arrived at Queen Mary, I was advised that ``if you don't like Franco, then you don't like life'',
and fortunately nothing in these last three years has caused me to question my will to live.

Thank-you also to Oliver Jenkinson for his supervision during my first year at Queen Mary,
on an entirely different but equally interesting project---it's a shame I didn't find any exciting dominant measures!
Thanks to Shaun Bullett and Christian Beck, my second supervisors, 
who oversaw the bureaucracy of my progression with interest and encouragement.
Thanks to everyone at Queen Mary's School of Mathematical Sciences for giving me such an enjoyable place to work,
albeit one which was regularly without water or heating; and for buying my cakes.

Thank-you to John Roberts for hosting me in Sydney in late 2012: an incredible opportunity,
despite the cockroaches.

Thank-you to Jeroen Lamb and Stefano Luzzatto, whose lectures at Imperial College inspired me to get into the
field of dynamical systems and ergodic theory.
Thanks in particular to Jeroen for all his help and advice whilst applying for a PhD studentship.

Thank-you to all the other students I've met along the way, both at Queen Mary and elsewhere. 
Thanks to the come-dine-with-me crew for their fabulous cooking.
In particular, thanks to Julia Slipantschuk for her companionship and her tolerance of my taste in film.
Thanks to Georg Ostrovski for my first citation: the cheque's in the post.

Thank-you to the EPSRC for funding my studies, and to the Eileen Eliza Colyer Prize 
and the Queen Mary Postgraduate Research Fund for their contributions to my jet-setting lifestyle.

Finally, thank-you to the bank of mum and dad for putting a roof over my head for the duration of my studies.
Fortunately for you I didn't stay at Imperial College, otherwise you would have had to buy a house in Kensington.
Unfortunately for you, this is just the beginning: now I'm heading out into the world unemployed and overqualified.

%% file: Introduction.tex
\chapter{Introduction} \label{chap:Introduction}

In this thesis we study the family of maps given by
\begin{equation} \label{def:F}
F:\, \Z^2 \to \Z^2 \hskip 20pt
(x,y)\,\mapsto\,(\fl{\lambda x} - y, \,x)
\hskip 20pt |\lambda|<2.
\end{equation}
For each value of the real parameter $\lambda$, the function $F$ is
an invertible map on the lattice of integer points in the plane.
Despite its simplicity, this model displays a rich landscape 
of mathematical phenomena, connecting discrete dynamics and arithmetic. 

The family (\ref{def:F}) first arose as a model of elliptic motion subject to round-off \cite{Vivaldi94b}.
If we remove the floor function in equation \eqref{def:F},
we obtain the one-parameter family of linear maps of the plane
\begin{equation} \label{def:A}
A:\, \R^2 \to \R^2 \hskip 20pt
(x,y)\,\mapsto\,(\lambda x - y, \,x)
\hskip 20pt \lambda = 2\cos(2\pi\nu),
\end{equation}
which are linearly conjugate to rotation by the angle $2\pi\nu$.
The invariant curves of $A$ are ellipses, given by level sets of the functions
\begin{equation} \label{def:Q_lambda}
 \cQ_{\lambda}(x,y)=x^2-\lambda x y +y^2,
\end{equation}
and all orbits are either periodic or quasi-periodic,
according to whether the rotation number $\nu$ is rational or irrational, respectively.
The map $F$ is a discretisation of $A$, 
obtained by composing it with the piecewise-constant function
\begin{equation} \label{eq:R}
 R: \, \R^2 \to \Z^2 
 \hskip 40pt
 R(x,y) = (\fl{x}, \fl{y}),
\end{equation}
where $\fl{\cdot}$ is the floor function---the largest integer not exceeding its argument. 
The floor function models the effect of round-off, pushing 
the image point to the nearest integer point on 
the left\footnote{The choice of round-off scheme is discussed further in section \ref{sec:Round-off}.}. 
Thus the model $F$ is an example of a Hamiltonian (i.e., symplectic) map subject to uniform, invertible
round-off, of the style introduced by Rannou \cite{Rannou}. 

In this context, we think of \eqref{def:F} as a perturbed Hamiltonian system, 
so a natural property to consider is its stability \cite{Vivaldi94b,LowensteinHatjispyrosVivaldi,LowensteinVivaldi98,KouptsovLowensteinVivaldi02,Vivaldi06}. 
Since $F$ is invertible, boundedness of orbits is equivalent to periodicity.
\begin{conjecture}[\cite{Vivaldi06}] \label{conj:Periodicity}
For all real $\lambda$ with $|\lambda|<2$, all orbits of $F$ are periodic\footnote{A
general conjecture on the boundedness of discretised Hamiltonian rotations was first 
formulated in \cite{Blank94}.}.
\end{conjecture}

This is where the arithmetic flavour of the family \eqref{def:F} becomes apparent.
A closely related conjecture has been stated in the field of number theory:
the map $F$ appears (with a slightly different round-off scheme) in the 
guise of an integer sequence, as part of a problem concerning 
shift radix systems \cite{AkiyamaBrunottePethoThuswaldner,AkiyamaBrunottePethoSteiner06}.
Conjecture \ref{conj:Periodicity} holds trivially for the integer parameter values $\lambda=0,\pm1$, 
where the map $F$ is of finite order.
Beyond this, the boundedness of all round-off orbits has been proved for only 
\emph{eight} values of $\lambda$, which correspond to the rational 
values of the rotation number $\nu$ for which $\lambda$ is a quadratic irrational:
\begin{equation}\label{eq:Lambdas}
\lambda=\frac{\pm1\pm\sqrt{5}}{2},\quad \pm\sqrt{2},\quad \pm\sqrt{3}.
\end{equation}
(The denominator of $\nu$ is $5,\,10,\,8$, and $12$, respectively.)
The case $\lambda=(1-\sqrt{5})/2$ was established in 
\cite{LowensteinHatjispyrosVivaldi}, with computer assistance.
Similar techniques were used to extend the result to the other parameter 
values, but only for a set of initial conditions having full density
\cite{KouptsovLowensteinVivaldi02}.
The conjecture for the eight parameters \eqref{eq:Lambdas} was settled 
in \cite{AkiyamaBrunottePethoSteiner08} with an analytical proof.
More recently, Akiyama and Peth\H{o} \cite{AkiyamaPetho} proved that 
(\ref{def:F}) has infinitely many periodic orbits for any parameter value.
We shall not make any further progress on conjecture \ref{conj:Periodicity} in this work.

The feature of the parameter values \eqref{eq:Lambdas} which enabled the 
resolution of conjecture \ref{conj:Periodicity} in these cases, 
is that the map $F$ admits a dense and uniform embedding in 
a two-dimensional torus, where the round-off map extends continuously 
to a piecewise isometry (which has \emph{zero entropy} and is not ergodic).
The natural density on the lattice $\Z^2$ is carried into the Lebesgue measure,
namely the Haar measure on the torus.
For any other rational value of $\nu$, the parameter $\lambda$ is an algebraic number of higher degree, and
there is a similar embedding in a higher-dimensional torus \cite{LowensteinVivaldi00,BruinLambert}; 
these systems are still unexplored, even in the cubic case.

\begin{figure}[t]
        \centering
        \begin{minipage}{7cm}
          \centering
	  \includegraphics[scale=0.35]{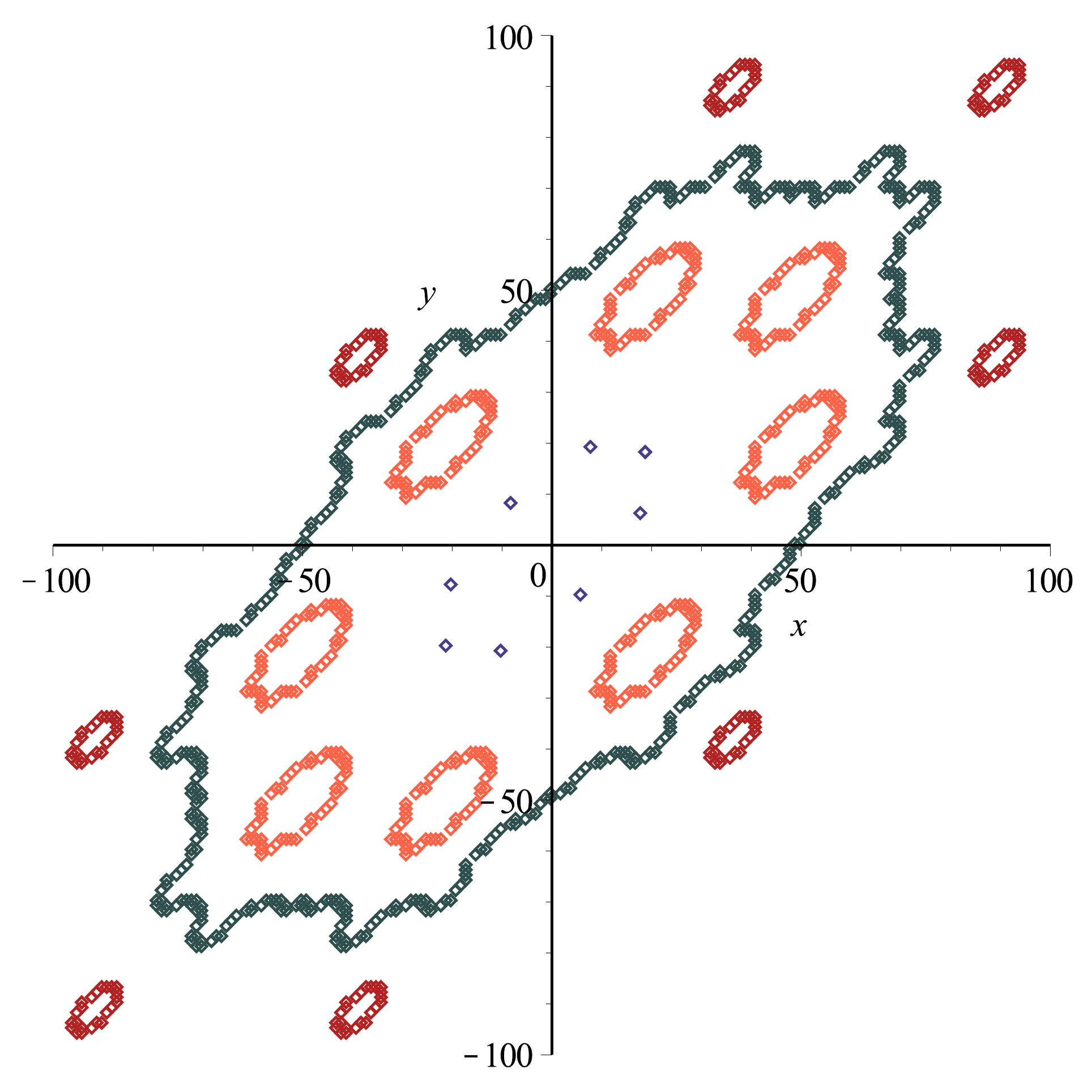} \\
	  (a) $\; \lambda=\sqrt{2}$, $\nu=1/8$ \\
        \end{minipage}
        \quad
        \begin{minipage}{7cm}
	  \centering
	  \includegraphics[scale=0.35]{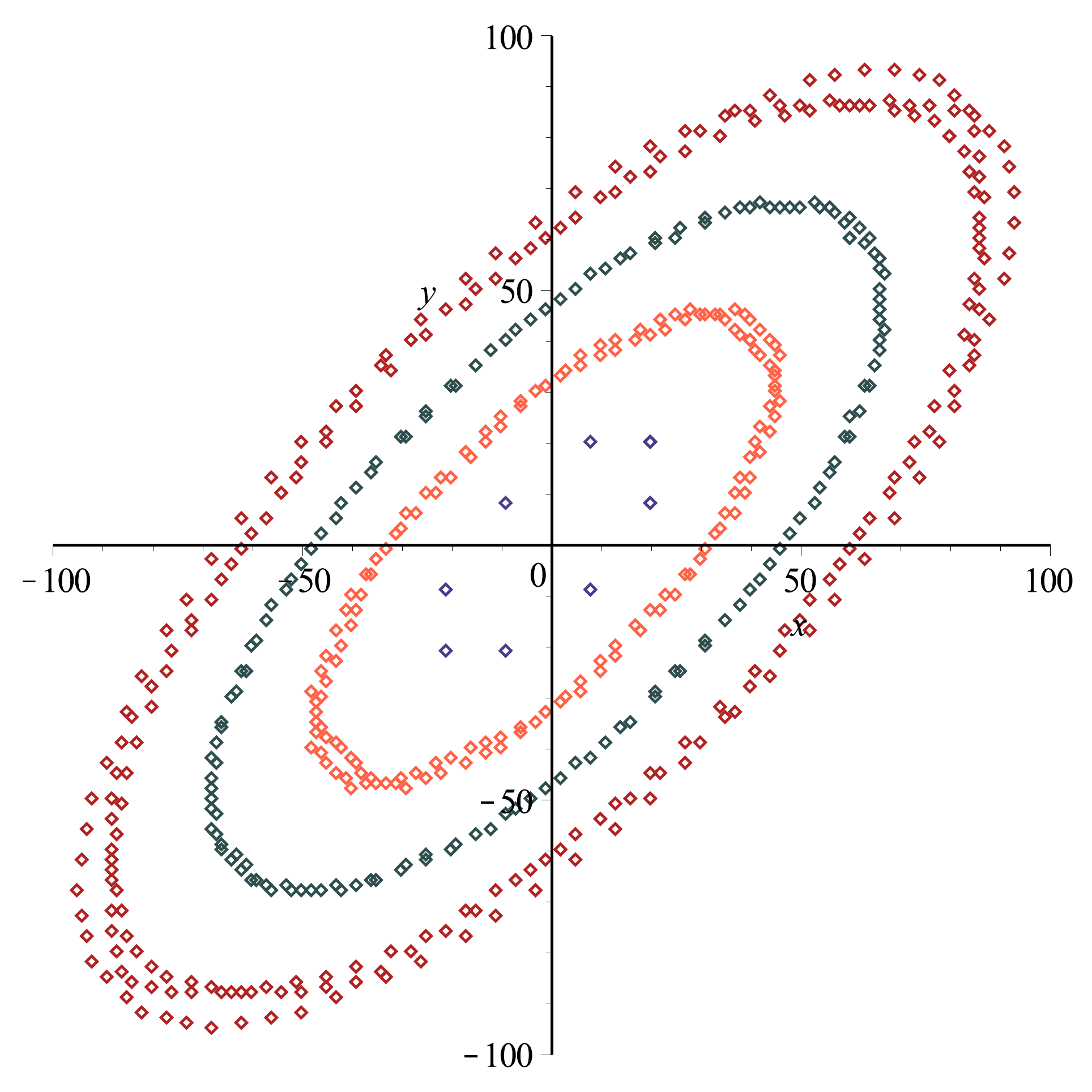} \\
	  (b) $\; \lambda=10/7$, $\nu\approx1/8$ \\
	\end{minipage}
        \caption{\hl{A selection of orbits of $F$ when the parameter $\lambda$ is (a) an algebraic integer and 
        (b) a rational number. All orbits are periodic, and the period of the orbits shown ranges from $8$ (both cases) to $235$ 
        (rational case) and $511$ (algebraic case).}}
        \label{fig:FOrbits}
\end{figure}


Irrational values of $\nu$ bring about a different dynamics, and a different 
theory. The simplest cases correspond to rational values of $\lambda$: 
in particular, to rational numbers whose denominator is the power of a prime $p$. 
In this case the map $F$ admits a dense and uniform embedding in the ring 
$\Z_p$ of $p$-adic integers \cite{BosioVivaldi}. 
The embedded system extends continuously to the composition of a full 
shift and an isometry (which has \emph{positive entropy}), 
and the natural density on $\Z^2$ is now carried into the Haar measure on $\Z_p$.
This construct was used to prove a central limit theorem for the 
departure of the round-off orbits from the unperturbed ones \cite{VivaldiVladimirov}. 
This phenomenon injects a probabilistic element in the determination of the 
period of the lattice orbits, highlighting the nature of the difficulties 
that surround conjecture \ref{conj:Periodicity}. 

\medskip

In this work we explore a new parameter regime, and the obvious next step is to
consider the \emph{approach} to a rational rotation number.
We choose the easiest such case---the approach to one of the cases \eqref{eq:Lambdas} seems
excessively complicated---and consider the limit $\lambda\to0$, 
corresponding to the rotation number $\nu\to 1/4$.
This is one of five limits (the other limits being $\lambda\to\pm1,\pm2$) 
where the dynamics at the limit is trivial because there is no round-off.

What we find is a new natural embedding of $F$, this time into the plane, 
and a new dynamical mechanism, namely a discrete-space version of \emph{linked strip maps}:
maps originally introduced in the study of \emph{outer billiards} or \emph{dual billiards} of polygons 
(for background, see \cite[Section III]{Tabachnikov}).
This construction was later generalised by Schwartz \cite{Schwartz11}.

We rescale the lattice $\Z^2$ by a factor of $\lambda$---to obtain the map $\F$ of equation 
\eqref{def:F_lambda}---then embed it in $\R^2$ (see figure \ref{fig:PolygonalOrbits}).
Now the parameter $\lambda$ controls not only the rotation number $\nu$, but also the lattice spacing.
The limiting behaviour is described by a piecewise-affine Hamiltonian.
The invariant curves of this Hamiltonian are polygons, 
and the fourth iterates of $F$ move parallel to the edge vectors of these polygons.
The role of the strip map is to aggregate this locally uniform behaviour 
into a sequence of translations: one for each edge.
The perturbation occurs near the vertices.
In this much, the map $F$ bears a strong resemblance to the strip map construction of outer billiards.
The difference in our case is that the number of sides of the invariant polygons 
increases with the distance from the origin; 
near the origin they are squares, while at infinity they approach circles.
Hence our version of the strip map is composed of an ever increasing number of components,
and results in a perturbation of increasing complexity---a feature which
cannot be achieved in outer billiards of polygons without changing the shape of the billiard.

\begin{figure}[t]
        \centering
        \begin{minipage}{7cm}
          \centering
	  \includegraphics[scale=0.35]{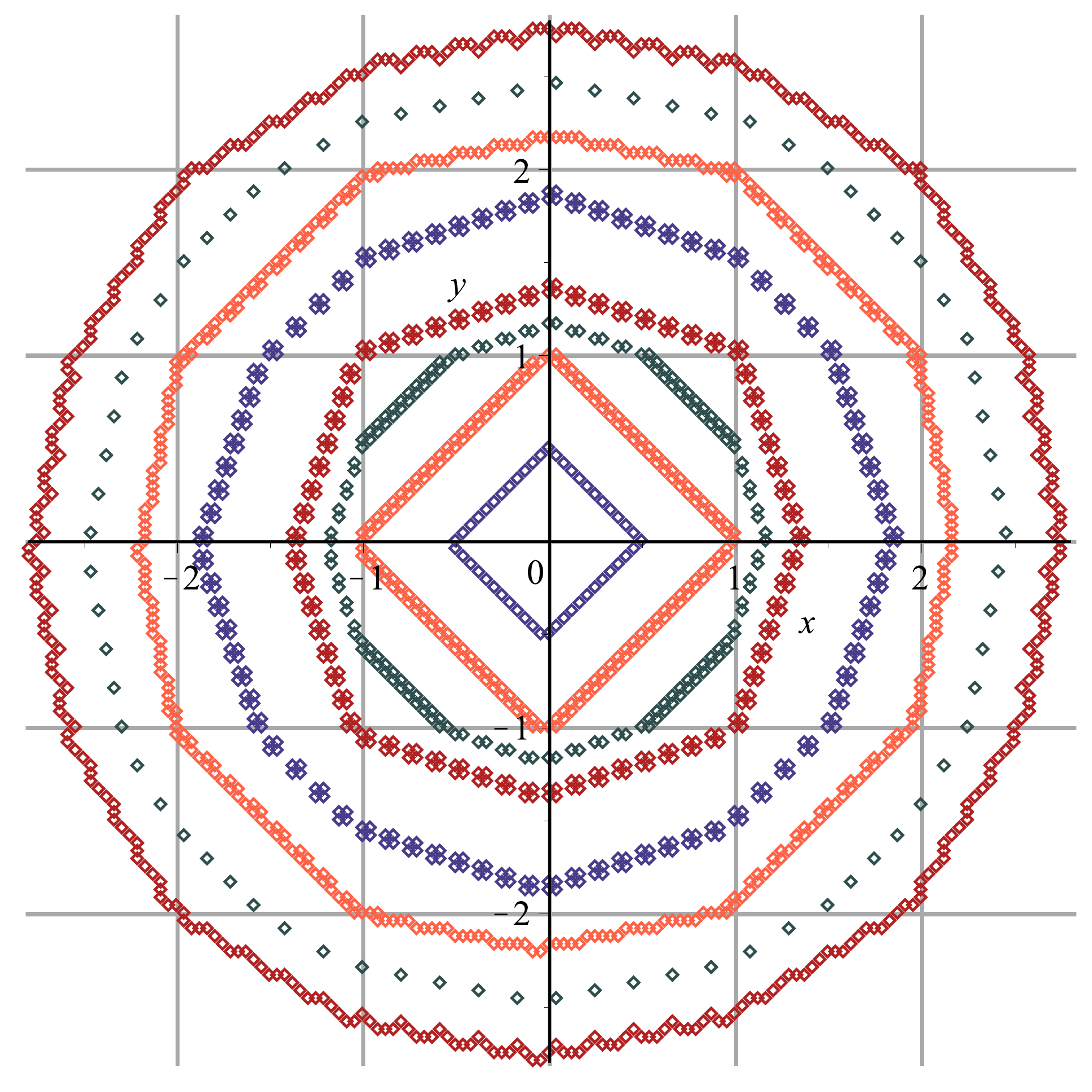} \\
	  (a) $\; \lambda=1/24$ \\
        \end{minipage}
        \quad
        \begin{minipage}{7cm}
	  \centering
	  \includegraphics[scale=0.35]{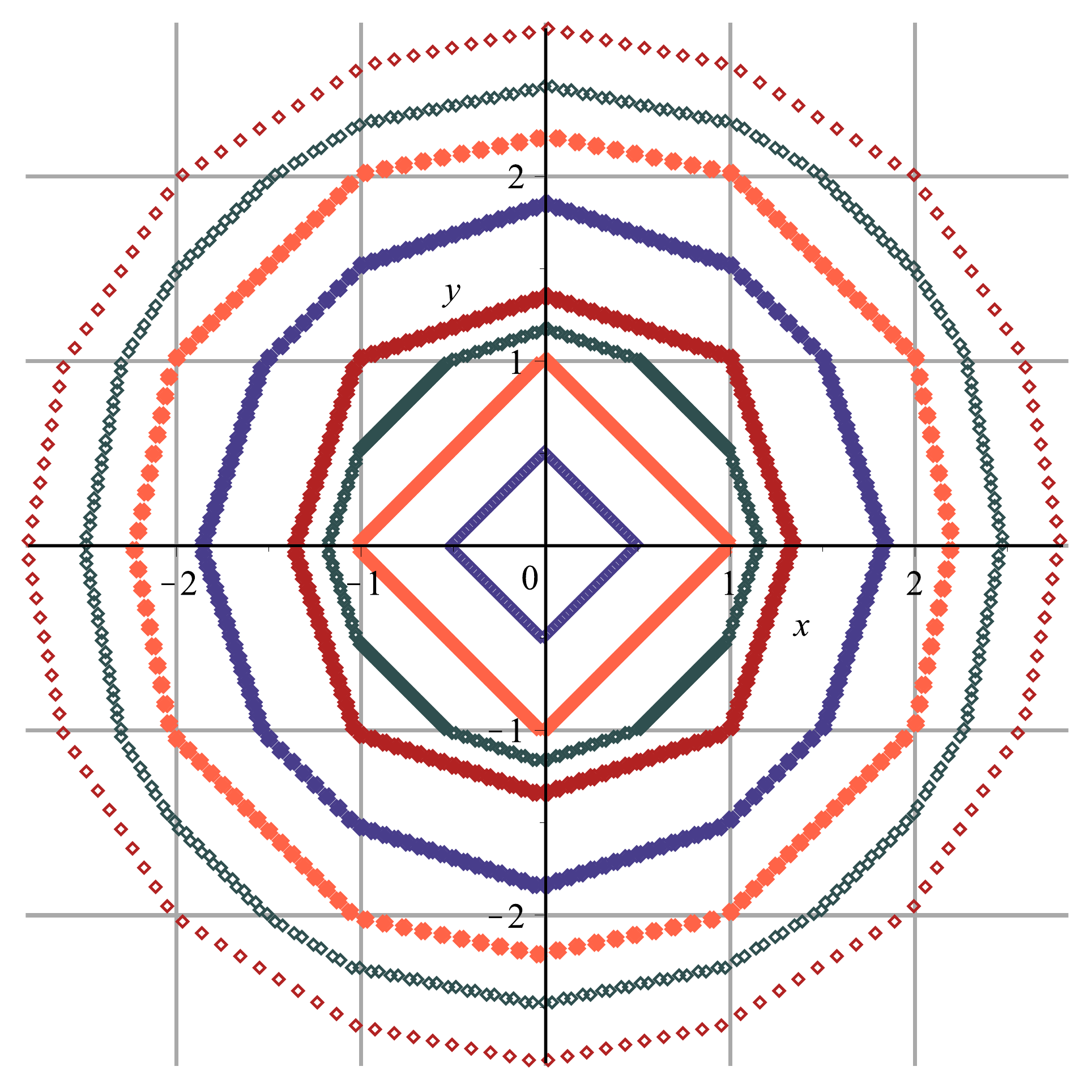} \\
	  (b) $\; \lambda=1/48$ \\
	\end{minipage}
        \caption{\hl{A selection of periodic orbits of the rescaled map $F_{\lambda}$, for two small values of 
        the parameter $\lambda$. The lattice spacing is such that each unit distance (illustrated by the grey lines) 
        contains $1/\lambda$ lattice points.
        Here $\nu\approx 1/4$ but there are no orbits which have period $4$: instead the periods of orbits 
        cluster around integer multiples of a longer recurrence time} (see figure \ref{fig:PeriodFunction}, page \pageref{fig:PeriodFunction}).}
        \label{fig:PolygonalOrbits}
\end{figure}

In this regime, we are led to consider the map $F$ as a perturbation of an
integrable Hamiltonian system, but the integrable system is no longer a rotation, and the 
perturbation is no longer caused by round-off. 
Thus the limit $\lambda\to 0$ is singular. 
Furthermore, the integrable system is nonlinear, i.e., its time-advance 
map satisfies a twist condition. The parameter $\lambda$ 
acts as a perturbation parameter, and a discrete version of near-integrable 
Hamiltonian dynamics emerges on the lattice when the perturbation is switched on.

If we were considering near-integrable Hamiltonian dynamics on the continuum, 
then we would be in the realm of KAM theory 
(for background, see \cite[section 6.3]{ArrowsmithPlace}),
according to which a positive fraction of invariant curves, identified
by their rotation number, will survive a sufficiently small smooth perturbation.
In this scenario, the complement of the KAM curves consists of a hierarchical arrangement of 
island chains and thin stochastic layers, and
the KAM curves disconnect the space, thereby ensuring the stability of the irregular orbits.
However, the map $F$ is defined on a discrete space, and the perturbation is discontinuous, 
so no such general theory applies.

Before we describe our findings in more detail, we set the scene
by outlining previous work on space discretisation, 
which consists of a patchwork of loosely connected phenomena, 
arising in a variety of different contexts.
In particular, we highlight the occurrence of near-integrable phenomena.

\vfill

\section{Near-integrability in a discrete phase space} \label{sec:discrete}

There are various approaches to space discretisation, 
which fall into two broad categories.
\begin{enumerate}[(i)]
 \item \emph{Invariant structures}\\ 
 This category comprises maps which preserve some finite or countable (and arithmetically interesting) subset of the phase space.
 This includes the restriction of algebraic maps with algebraic parameters to discrete rings or fields, and piecewise-isometries involving rational rotations. In these cases, the dynamics of the original map remain unchanged, but we consider discrete subsets of parameter values and discrete subsets of the phase space. 
 \item \emph{Round-off}\\ 
 Here we consider maps which are formed from the composition of a map of a continuum with a \emph{round-off function}, 
 which forces the map to preserve a given finite or countable set (typically a lattice). 
 In this case, the original dynamics are subject to a discontinuous perturbation, 
 so that the relationship between the dynamics of the discrete system and those of the original system is often unclear.
\end{enumerate}

We are interested in the range of dynamical behaviours that can be observed in a discrete phase space, and how these can be described. 
This question manifests itself slightly differently for the two types of discretisation. 
In the case of invariant structures, we are typically concerned with which dynamical features \emph{remain} after discretisation: 
what mark does the behaviour of the original system leave on that of its discrete counterpart? 
In the case of round-off, we are interested in those features which are \emph{created} by discretisation: 
how does the behaviour of a dynamical system change when we force it onto a lattice?

In both cases, we can ask: are the features of the original map recovered in the fine-discretisation limit?

In smooth Hamiltonian dynamics, near-integrable behaviour is characterised by the presence of invariant KAM curves, 
on which the motion is quasi-periodic, separated by periodic island chains and stochastic layers.
Reproducing the structures of KAM theory in a discrete space is problematic.
On a lattice, quasi-periodic orbits do not exist. 
Surrogate KAM surfaces must thus be identified, and their evolution must be 
tracked, as the perturbation parameter is varied.
Furthermore, these orbits need not disconnect 
the space, so their relevance to stability must be re-assessed.

We introduce the various types of discrete system below. 
In each case, we describe examples which will be relevant in what follows, 
and discuss the possibility of near-integrable phenomena.

\subsection*{Restriction to discrete rings and fields}

If an algebraic map preserves a subset of the phase space which is a discrete ring or field, 
then we can study its dynamics when restricted to this subset. 
In the finite case, this is equivalent to studying (subsets of) the periodic orbits of a map.

The dominant example in this class is the family of hyperbolic toral automorphisms (or \emph{cat maps}): 
chaotic (Anosov) Hamiltonian maps, whose set of periodic orbits is precisely the set of rational points, 
and which preserve lattices of rational points with any given denominator.
The special arithmetic properties of these maps enable a complete classification of the periodic orbits on such a rational lattice, 
and there is a wealth of literature on this topic, 
including \cite{HannayBerry,PercivalVivaldi,Keating91a,DysonFalk,EspostiIsola,BehrendsFiedler}.
The arithmetic of the denominator limits the number of allowed periods on any given lattice, 
and the resulting period distribution function is singular.

Rational restrictions of maps with milder statistical properties have also been considered. 
The Casati-Prosen triangle maps \cite{CasatiProsen,HorvatEspostiIsola} are 
a family of zero entropy maps of the torus rooted in quantum chaos,
which are conjectured to be mixing for irrational parameters,
and preserve rational lattices for rational parameters.
In this case, the maps have time-reversal symmetry, and
the distribution of periods on such lattices is conjectured to converge 
to a smooth distribution in the fine-discretisation limit \cite{NeumarkerRobertsVivaldi}.
We discuss this result further in section \ref{sec:time-reversal}.

\subsection*{Reduction to finite fields}

For an algebraic system it is natural to replace the coordinate field with a finite field, 
for instance the field $\bbF_p$ of integers modulo a prime $p$.
The resulting map has a finite phase space, so all its orbits are (eventually) periodic. 

The reduction process dispenses with the topology of the original map, but preserves algebraic properties, 
such as symmetries or the presence of an integral. 
Consequently, there is no near-integrable regime, and one witnesses a discontinuous 
transition from integrable to non-integrable behaviour. 
This transition manifests itself probabilistically via a (conjectured) abrupt 
change in the asymptotic (large field) distribution of the periods of the orbits, 
which can be used as a tool to detect integrability
\cite{RobertsVivaldi03,RobertsJogiaVivaldi,JogiaRobertsVivaldi}.
Similarly the reduction to finite fields can be used to detect 
time-reversal symmetry \cite{RobertsVivaldi05,RobertsVivaldi09} (see also section \ref{sec:time-reversal}).

\subsection*{Piecewise-isometries}

Piecewise-isometries are a generalisation of interval exchange transformations to higher dimensions, 
in which the phase space (typically the plane or the torus) is partitioned into a finite number of sets, 
called \emph{atoms}, and the map acts as a different isometry on each atom. (For background, see \cite{Goetz}.) 
It has been shown that all piecewise-isometries have zero entropy \cite{GutkinHaydn,Buzzi}. 
Furthermore, for piecewise-isometries involving rational parameters, 
the dynamics are discrete in the sense that the phase space features a 
countable hierarchy of (eventually) periodic polygons, which move rigidly under the dynamics 
(see \cite{AdlerKitchensTresser} and references therein). 
In particular, this class of piecewise-isometries include the much-studied piecewise-rotations 
of the torus with rational rotation number (see, for example, \cite{Kahng02,KouptsovLowensteinVivaldi02,Kahng04,KouptsovLowensteinVivaldi04,GoetzPoggiaspalla}).
These are the systems in which the discretised rotation $F$ was embedded 
in order to settle conjecture \ref{conj:Periodicity} for the parameter values \eqref{eq:Lambdas}.
 
An example of a family of piecewise-isometries in unbounded phase space is the family of dual billiard maps on polygons \cite[Section III]{Tabachnikov}. \hl{When the polygon has rational coordinates, the dynamics are discrete and all orbits are periodic.} 
As we have already mentioned, the dynamical mechanism which underlies the behaviour of $F$ in the limit $\lambda\to 0$
has much in common with that of outer billiards of polygons.
A near-integrable regime of a kind exists for these maps in the form of \emph{quasirational} polygons, 
for which all orbits remain bounded thanks to the existence of bounding invariants called 
\emph{necklaces}\footnote{For smooth billiard tables, 
the outer billiards map is a twist map admitting KAM curves, 
which ensure the boundedness of orbits (see \cite[Section I]{Tabachnikov}).} 
\cite{VivaldiShaidenko,Kolodziej,GutkinSimanyi}---however, 
in the quasirational case the dynamics are no longer discrete.
\hl{Only recently has an unbounded outer billiard orbit been exhibited} \cite{Schwartz07}.

\subsection*{Round-off}

One typically thinks of round-off in the context of computer arithmetic,
where real numbers are represented with finite precision.
In this context, it is the relationship between computer-generated orbits and 
the true orbits of a dynamical system which is of principal interest.
This issue can be tackled to an extent by \emph{shadowing}, 
whereby a perturbed orbit (in this case a discretised orbit) of a chaotic map
is guaranteed to be close to an orbit of the unperturbed system
(see \cite[Section 18.1]{KatokHasselblatt}, 
or \cite{HammelYorkeGrebogi,GrebogiHammelYorkeSauer} for results specific to round-off).
However, shadowing tells us nothing about whether the behaviour of perturbed orbits is typical,
or what happens to orbits over long timescales,
where round-off typically introduces irreversible behaviour
\cite{BinderJensen,BeckRoepstorff,GoraBoyarsky}.
In rare cases, round-off fluctuations act like small-amplitude noise, 
and give rise to Gaussian transport; 
more commonly, the propagation of round-off error must be described as a deterministic 
(as opposed to probabilistic) phenomenon.

A rigorous analysis of round-off in floating-point arithmetic is very difficult:
the set of representable numbers is neither uniform nor arithmetically closed,
hence calculations are performed in a modified arithmetic which is not even associative.
To put the study of round-off on a solid footing, 
it is preferable to consider calculations in fixed-point (i.e., integer) arithmetic,
which is closed under ring operations (discounting overflow).
Several authors have used explicit fixed-point approximations of real Hamiltonian maps 
in numerical experiments, which have the advantage that iteration can be performed exactly,
and that invertibility can be retained
(see Rannou et. al.~\cite{Rannou,Karney,EarnTremaine}).

Both Blank \cite{Blank89, Blank94} and Vladimirov \cite{Vladimirov} have presented
theoretical frameworks in which to study the statistical behaviour of round-off.
Blank considers the properties of ensembles of round-off maps with varying discretisation length,
whereas Vladimirov equips a discrete phase space with a measure which can be used to quantify
the deviation of exact and numerical trajectories.

In this work, we are interested in round-off as a dynamical phenomenon in its own right.
Like Rannou, we consider a uniform, invertible discretisation of a Hamiltonian map.
In fact, we consider a discretisation of a rotation---the prototypical integrable Hamiltonian map.
Applying such arithmetically well-behaved round-off to simple linear systems like the rotation
leads to dynamical phenomena which are born of discontinuity.
The model $F$, as introduced in the previous section, 
has been studied by several authors from various points of view.
We study a parameter regime in which the behaviour of $F$ can be described as near-integrable.
The only other example of near-integrability in round-off dynamics is a numerical
study of a perturbed twist map \cite{ZhangVivaldi}: 
a simpler model which we will return to in the conclusion.

\section{Main results \& outline of the thesis}

Chapter \ref{chap:Preliminaries} provides the reader with some technical background.
We discuss round-off, and justify the choice of round-off scheme employed in the model $F$.
Then we discuss time-reversal symmetry: the map $F$ is \emph{reversible} with respect to
reflection in the line $x=y$, and \emph{symmetric} orbits will play a key role in our analysis.

In chapter \ref{chap:IntegrableLimit} we consider the limit $\lambda\to 0$, 
which we refer to as the \emph{integrable limit}.
We describe the orbits closest to the origin, which have a particularly simple form,
and motivate a rescaling of the lattice $\Z^2$ by a factor of $\lambda$.
We introduce a piecewise-affine Hamiltonian function $\cP$
(equation \eqref{eq:Hamiltonian}), whose invariant curves are polygons,
representing the limiting foliation of the plane for the rescaled system (see figure \ref{fig:polygon_classes}).
The set of invariant polygons is partitioned by \emph{critical polygons}, 
which contain $\Z^2$ points, into infinitely many \emph{polygon classes}, 
which can be characterised arithmetically in terms of sums of squares.
Each polygon class is assigned a symbolic coding, which describes its 
path relative to the lattice $\Z^2$.

\begin{thm_nonumber}[Theorem \ref{thm:Polygons}, page \pageref{thm:Polygons}]
The level sets of $\cP$ are convex polygons.
The polygon $\cP(z)=\alpha$ is critical if and only if $\alpha\in\cE$,
where $\cE$ is the set of natural numbers which can be written as the sum of two squares:
\begin{displaymath}
 \cE = \{0,1,2,4,5,8,9,10,13,16,17,\dots \} .
\end{displaymath}
\end{thm_nonumber}

To match Hamiltonian flow and lattice map, we exploit the fact that, for small 
$\lambda$, the composite map $F^4$ is close to the identity. After scaling, it is 
possible to identify the action of $F^4$ with a time-advance map of the flow
(in the spirit of Takens' theorem \cite[section 6.2.2]{ArrowsmithPlace}).
This time-advance map assumes the role of the unperturbed dynamics.
The two actions agree along the sides of the polygons, but differ in vanishingly
small regions near the vertices. This discrepancy provides the perturbation mechanism.

The period function of $F$ displays a non-trivial clustering
of the periods around integer multiples of a basic \hl{recurrence time} (see
figure \ref{fig:PeriodFunction}),
and all orbits recur to a small neighbourhood of the symmetry axis $x=y$.
In section \ref{sec:Recurrence}, we define a Poincar\'{e} return map $\Phi$ to reflect this behaviour, 
and show that the \emph{return orbits}---the partial orbits \hl{iterated up to their recurrence time}---shadow 
the integrable orbits.

\FloatBarrier

\begin{figure}[t]
  \centering
  \includegraphics[scale=0.35]{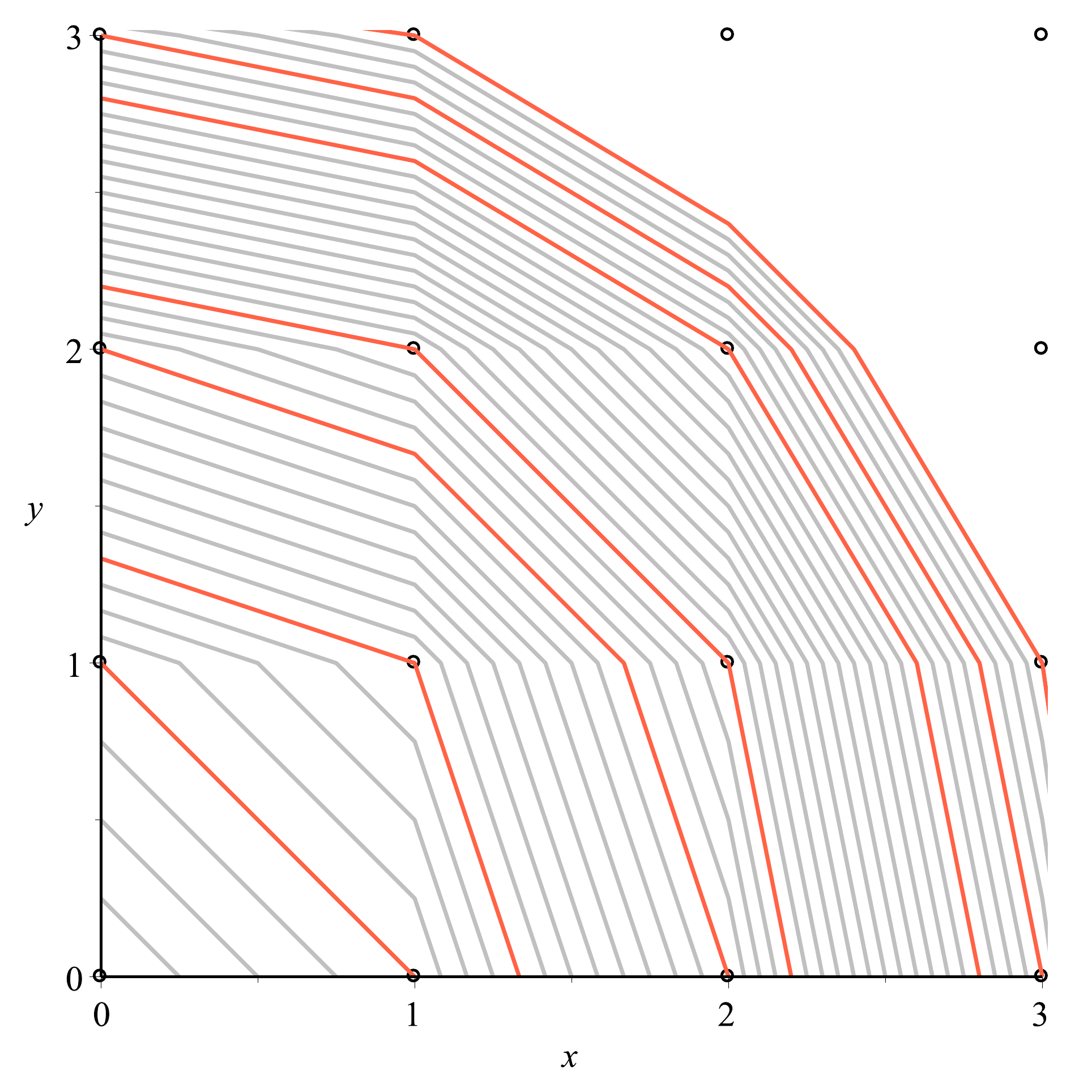}
  \caption{A selection of polygons $\cP(x,y)=\alpha$, for values of $\alpha$ in the interval $[0,10]$.
  The critical polygons are shown in red: the polygon classes are the annuli bounded between pairs of adjacent critical polygons.
  All polygons are symmetric under reflection in the coordinate axes, and in the line $x=y$.}
  \label{fig:polygon_classes}
\end{figure}

\begin{figure}[t]
        \centering
        \includegraphics[scale=0.35]{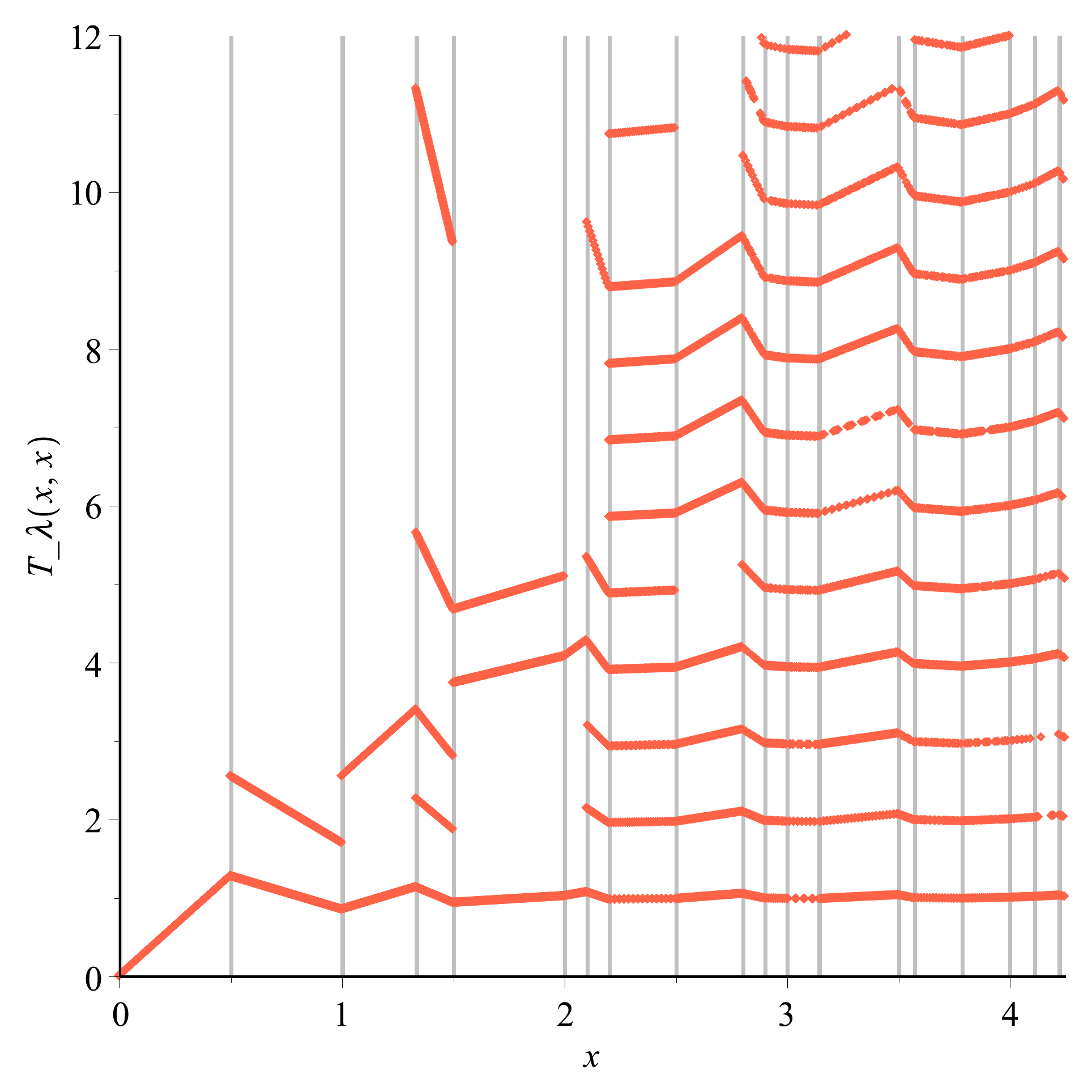}
        \caption{The normalised period function $T_\lambda(z)$ (see section \ref{sec:Recurrence}) 
        for points $z=(x,x)$, and $\lambda=1/5000$.
         The vertical lines mark the location of critical polygons, which 
         pass through lattice points.}
        \label{fig:PeriodFunction}
\end{figure}

\vfill

\FloatBarrier

\begin{thm_nonumber}[Theorem \ref{thm:Hausdorff}, page \pageref{thm:Hausdorff}]
For any $w\in\R^2$, let $\Pi(w)$ be the orbit of $w$ under the Hamiltonian flow, 
and let $\cO(w,\lambda)$ be the return orbit of the lattice point in $\lZ$ associated with $w$. 
Then
\begin{displaymath}
 \lim_{\lambda\to 0} d_H 
   \left(\Pi(w),\cO(w,\lambda)\right)=0,
\end{displaymath}
where $d_H$ is the Hausdorff distance on $\R^2$.
\end{thm_nonumber}

In section \ref{sec:IntegrableReturnMap}, 
we calculate the period of the Hamiltonian flow as a function of the value of the Hamiltonian.
This leads us to conclude that the unperturbed return map is nonlinear, and that
the nonlinearity is piecewise-affine on the polygon classes.

In section \ref{sec:LimitsPm1}, we show briefly that a similar 
construction applies in the limits $\lambda\to\pm1$, 
where the rotation number $\nu$ approaches $1/6$ and $1/3$, respectively.

\medskip

In chapter \ref{chap:PerturbedDynamics} we consider the behaviour of the perturbed return map $\Phi$.
The lowest branch of the period function of $F$ comprises the \emph{minimal orbits:} 
\hl{the fixed points of the return map, 
for which the effects of the perturbation cancel out}.
These are the orbits of the integrable system that survive the perturbation, 
and we treat them as analogues of KAM tori.
The other orbits mimic the divided phase space structure of a near-integrable 
area-preserving map, in embryonic form near the origin, and with increasing 
complexity at larger amplitudes.

The main result of this chapter is that for infinitely many polygon classes, 
the symmetric minimal orbits occupy a positive density of the phase space. 
The restriction to infinitely many classes---as opposed to all classes---stems 
from a coprimality condition we impose in order to achieve convergence
of the density.

Matching the orbits of the perturbed return map to the polygon classes of the 
integrable flow is a delicate procedure, 
requiring the exclusion of certain anomalous domains, and establishing 
that the size of these domains is negligible in the limit. 
We do this in section \ref{sec:RegularDomains}.
Then, in section \ref{sec:MainTheorems}, 
we state the chapter's main theorems.
The first theorem states that, within each polygon class, the return map 
commutes with translations by the elements of a two-dimensional lattice, 
which is independent of $\lambda$ up to scale, provided that $\lambda$ is sufficiently small
(see figure \ref{fig:lattice_Le}, page \pageref{fig:lattice_Le}).
\hl{To avoid excessive notational overhead, the below statement of the theorem has been somewhat simplified.}

\begin{thm_nonumber}[Theorem \ref{thm:Phi_equivariance}, page \pageref{thm:Phi_equivariance}]
Associated with each polygon class, indexed by $e\in\cE$, there is an integer lattice $\Le\subset\Z^2$, 
\hl{such that over a suitable domain}, and for all sufficiently small $\lambda$,
the return map $\Phi$ is equivariant under the group of 
translations generated by $\lambda\Le$:
\begin{displaymath}
 \forall l\in\Le:
 \hskip 10pt 
 \Phi(z + \lambda l) = \Phi(z) + \lambda l .
\end{displaymath}
\end{thm_nonumber}

The second theorem states that, if the symbolic coding of a polygonal class
satisfies certain coprimality conditions, then the density 
of symmetric minimal orbits among all orbits becomes 
independent of $\lambda$, provided that $\lambda$ is small enough.
This density is a positive rational number, which is computed explicitly.
As the number of sides of the polygons increases to infinity, the density 
tends to zero.
An immediate corollary of theorem \ref{thm:minimal_densities} is the existence of a positive
lower bound for the density of minimal orbits---symmetric or otherwise.

\begin{thm_nonumber}[Theorem \ref{thm:minimal_densities}, page \pageref{thm:minimal_densities}]
There is an infinite sequence of polygon classes, indexed by $e\in\cE$, 
such that within each polygon class, and for all sufficiently small $\lambda$, 
the number of symmetric fixed points of $\Phi$ modulo $\lambda\Le$ is non-zero and independent of $\lambda$. 
Thus the asymptotic density of symmetric fixed points within these polygon classes converges and is positive.
\end{thm_nonumber}

The analysis of the return map requires tracking the return orbits, and this is done
through repeated applications of a \emph{strip map}, an acceleration device which 
exploits local integrability. This is a variant of a construct introduced for outer 
billiards of polygons (see \cite[chapter 7]{Schwartz}, and references therein), 
although in our case the strip map has an increasing number of components, 
providing a dynamics of increasing complexity.
We introduce the strip map in section \ref{sec:StripMap}, and establish
some of its properties. 
There is a symbolic coding associated with the strip map; its cylinder sets
in the return domain form the congruence classes of the local lattice structure.
\hl{This fact gives a `non-Archimedean' character to the dynamics.
We prove the correspondence between the symbolic coding and the return map}
in section \ref{sec:lattice}, and this result leads to the conclusion of the proof of the main theorems.

\medskip

In chapter \ref{chap:Apeirogon} we explore the behaviour of the unperturbed return map at infinity.
A change of coordinates shows that the unperturbed return map is a linear twist map on the cylinder.
We study the asymptotics of the period function of the integrable flow, 
and find that it undergoes damped oscillations as the distance from the origin increases.
A suitable scaling uncovers a limiting functional form
which has a singularity in its derivative. 
This leads to a discontinuity in the asymptotic behaviour of the twist map. 
Typically the twist converges to zero, 
and hence the unperturbed return map converges (non-uniformly) to the identity. 
However, for the polygon classes which correspond to perfect squares 
(recall that the polygon classes are classified by the sums of squares), 
the twist converges to a non-zero value.
\hl{Again we state a somewhat simplified version of the result.}

\begin{figure}[t]
        \centering
        \includegraphics[scale=0.17]{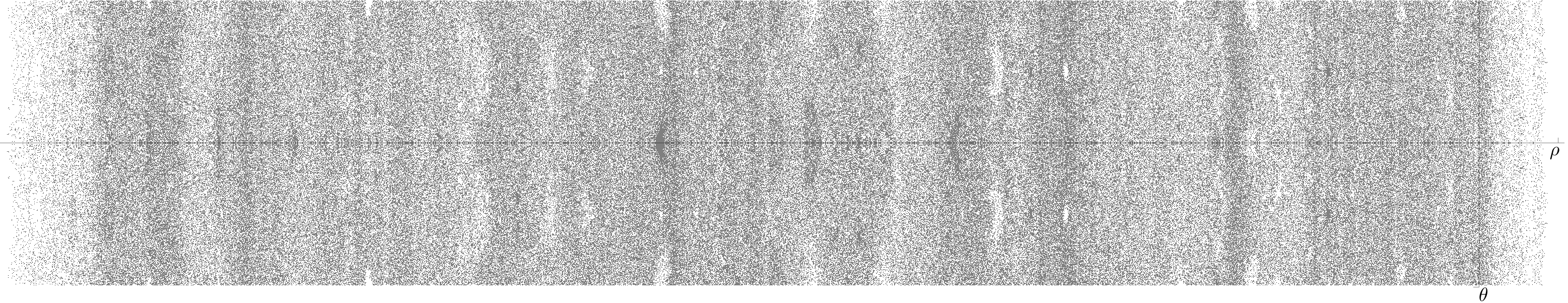} \\
        (a) $e=40000=200^2$ \\
        ~ \\
        \includegraphics[scale=0.17]{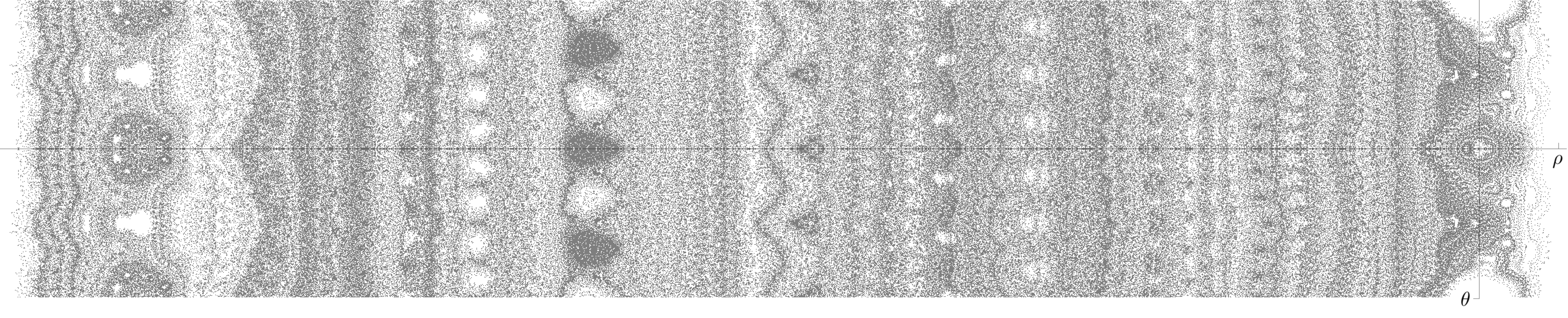} \\
        (b) $e=40309\approx200.8^2$
        \caption{\hl{Two pixel plots showing a large number of symmetric orbits of the return map $\Phi$ 
        in the cylindrical coordinates $(\theta,\rho)\in\bbS^1 \times \R$, 
        where the $\rho$-axis is the symmetry axis, and the $\theta$-axis is a fixed line of the twist dynamics.
        The resolution is such that the width of the cylinder (the $\theta$ direction) consists of approximately $280$ lattice sites.
        In both cases, the orbits plotted occupy almost half of the region of phase space pictured.
        The stark contrast between the two plots is caused by the difference in the twist $K(e)$:
        in plot (a) $K(e)\approx 4$, whereas in plot (b) $K(e)\approx -0.1$.
        The values of $\lambda$ used are (a) $\lambda\approx 7\times 10^{-9}$ and (b) $\lambda\approx 4\times 10^{-8}$.
        In figure (b), the primary resonance at the origin is clearly visible, whereas a period $2$
        resonance, which occurs at $\rho=1/2K(e)$, is seen to the left of the plot.}}
        \label{fig:resonance}
\end{figure}


\begin{thm_nonumber}[Theorem \ref{thm:Omega_e}, page \pageref{thm:Omega_e} \& Proposition \ref{prop:Tprime_asymptotics}, page \pageref{prop:Tprime_asymptotics}]
Associated with each polygon class, indexed by $e\in\cE$, there is a change of coordinates
which conjugates the unperturbed return map to the linear twist map $\Omega^e$, given by
\begin{displaymath}
 \Omega^e : \bbS^1 \times \R \to \bbS^1 \times \R
    \hskip 40pt 
 \Omega^e(\theta,\rho) = \left( \theta + K(e)\rho , \rho \right),
\end{displaymath}
where $K(e)$ is the twist.
Furthermore, as $e\to\infty$, the limiting behaviour of $K(e)$ is singular:
\begin{displaymath}
 K(e) \to \left\{ \begin{array}{ll} 4 \quad & \sqrt{e}\in\N \\ 0 \quad & \mbox{otherwise.} \end{array}\right.
\end{displaymath}
\end{thm_nonumber}

Finally, in chapter \ref{chap:PerturbedAsymptotics}, we study the perturbed dynamics at infinity.
The contents of this chapter are based on extensive numerical experiments, tracking large orbits of
$F$ in integer arithmetic. 
We study the phase portrait of the perturbed return map.
In the cases where the twist of the unperturbed return map converges to zero, 
the form of the perturbation is laid bare, and we find delicate discrete resonance structures (see figure \ref{fig:resonance}(b)). 
However, in the cases where the twist persists, the phase portrait is featureless and uniform 
over length scales comparable with the strip's width (see figure \ref{fig:resonance}(a)). 
This local uniformity allows us to compute the period distribution 
function within these polygon classes numerically, and show that it is
consistent with the period statistics of a random reversible map.

\begin{obs_nonumber}[Observation \ref{obs:De}, page \pageref{obs:De}]
As $e\to\infty$ and $K(e)\to 4$, i.e., on the subsequence of perfect squares, 
the distribution of periods among orbits in each polygon class converges to a limiting distribution.
This limiting distribution corresponds to a random reversible map in a discrete phase space of diverging cardinality.
\end{obs_nonumber}

\medskip

We finish with some concluding remarks and open questions.

%% file: Preliminaries.tex
\chapter{Preliminaries} \label{chap:Preliminaries}

In this chapter we provide the reader with background material
and key results on various topics which will be referred to throughout what follows.

\section{Round-off} \label{sec:Round-off}

When we speak of round-off as a method of discretisation, 
we refer to the scenario in which a map $T$ on some set $X$ is replaced by 
a perturbed map $F$ on some finite or countable subset $L\subset X$, 
which is not invariant under $T$.
The map $F$ is the composition of $T$ with a \defn{round-off function} $R$,
which associates each point in $X$ with some (nearby) point in $L$.
The choice of $L$ and $R$ are referred to as the \defn{round-off scheme}.

In chapter \ref{chap:Introduction}, we mentioned the frameworks for round-off introduced by
Blank \cite{Blank94} and Vladimirov \cite{Vladimirov}:
we discuss their respective approaches briefly here.

The most general model is given by Blank,
who introduces the notion of an \emph{$\epsilon$-discretisation} $X_{\epsilon}$ of a compact set $X$,
which is simply an ordered collection of points in which neighbouring points are
separated by a distance of at most $\epsilon$. 
The corresponding round-off function associates each point in $X$ with 
the closest point in $X_{\epsilon}$ 
(or, in case there are several such points, the point which is smallest with respect to the ordering of $X$).

Vladimirov considers discretisations of linear maps of $\R^n$
on the integer lattice $\Z^n$. 
The space is discretised via the introduction of so-called \emph{cells} $\Omega$,
which tile the space under translation by elements of $\Z^2$:
 $$ \Omega + \Z^n = \R^n, \qquad \forall z\in\Z^n\setminus\{0\}: \quad \Omega \cap (\Omega + z) = \emptyset. $$
Then the round-off function is called a \emph{quantizer}, 
and may be any map which commutes with translations by 
elements of the lattice $\Z^n$, and whose associated cells are Jordan measurable.

In our case we follow the conventions of \cite{AdlerKitchensTresser,KouptsovLowensteinVivaldi02},
borrowed from maps on the torus.
We consider a discretisation of a planar map onto a two-dimensional lattice $\bbL$ given by
 $$ \bbL = C \Z^2, $$
where $C$ is a $2\times 2$ non-singular matrix.
A \defn{fundamental domain} $\Omega$ of $\bbL$ is a set which 
tiles the plane under translation by elements of $\bbL$,
as per the above definition of a cell,
and we restrict our attention to fundamental domains whose
closure is a parallelogram.
Then the round-off function $R$ associates each point $z\in\R^2$ 
with the unique lattice point $l\in\bbL$ such that $z\in l+\Omega$, i.e.,
 $$ R: \R^2 \to \bbL \hskip 40pt R(z) = (z-\Omega)\cap\bbL. $$
 
When modelling round-off as performed in fixed-point arithmetic,
it is typical to use a uniform square lattice of the form\footnote{We use $\N$ to denote the set of positive integers.}
 $$ \bbL = \left(\frac{\Z}{N}\right)^2 \hskip 40pt N\in\N, $$
where $1/N$ is the \defn{lattice spacing},
and the fundamental domain corresponds to \defn{nearest-neighbour rounding}:
 $$ \Omega = \frac{1}{N}\;\left[-\frac{1}{2},\frac{1}{2}\right)^2. $$

\medskip

In the case of the discretised rotation $F$ of (\ref{def:F}),
the unperturbed dynamics are given by the elliptic motion $A$ of equation (\ref{def:A}),
the lattice $\bbL$ is simply the integer lattice $\Z^2$,
and the fundamental domain is the unit square $\Omega=[0,1)^2$.

\begin{figure}[t]
        \centering
        \begin{minipage}{7cm}
          \centering
	  \includegraphics[scale=0.35]{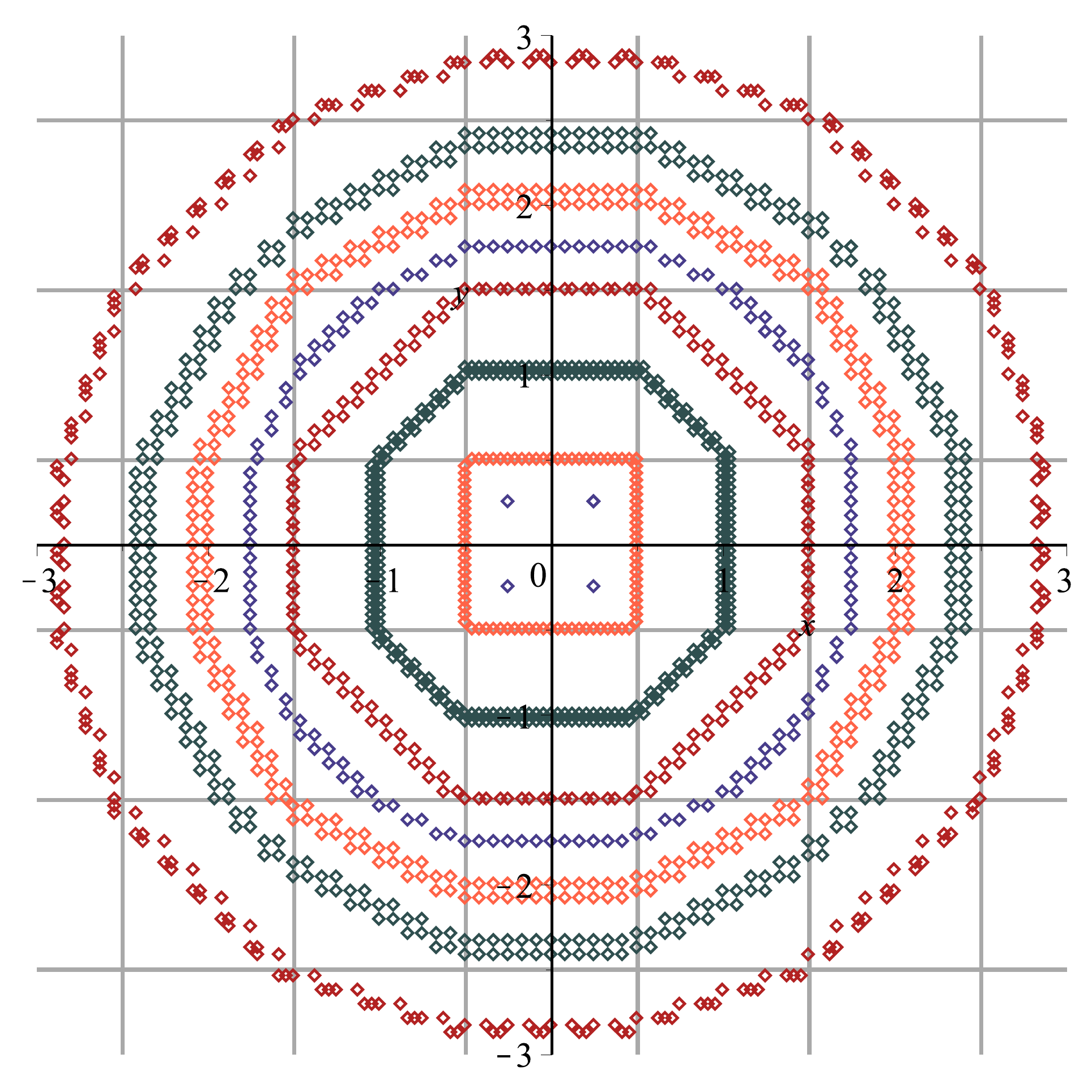} \\
	  (a) $\; \lambda=1/24$ \\
        \end{minipage}
        \quad
        \begin{minipage}{7cm}
	  \centering
	  \includegraphics[scale=0.35]{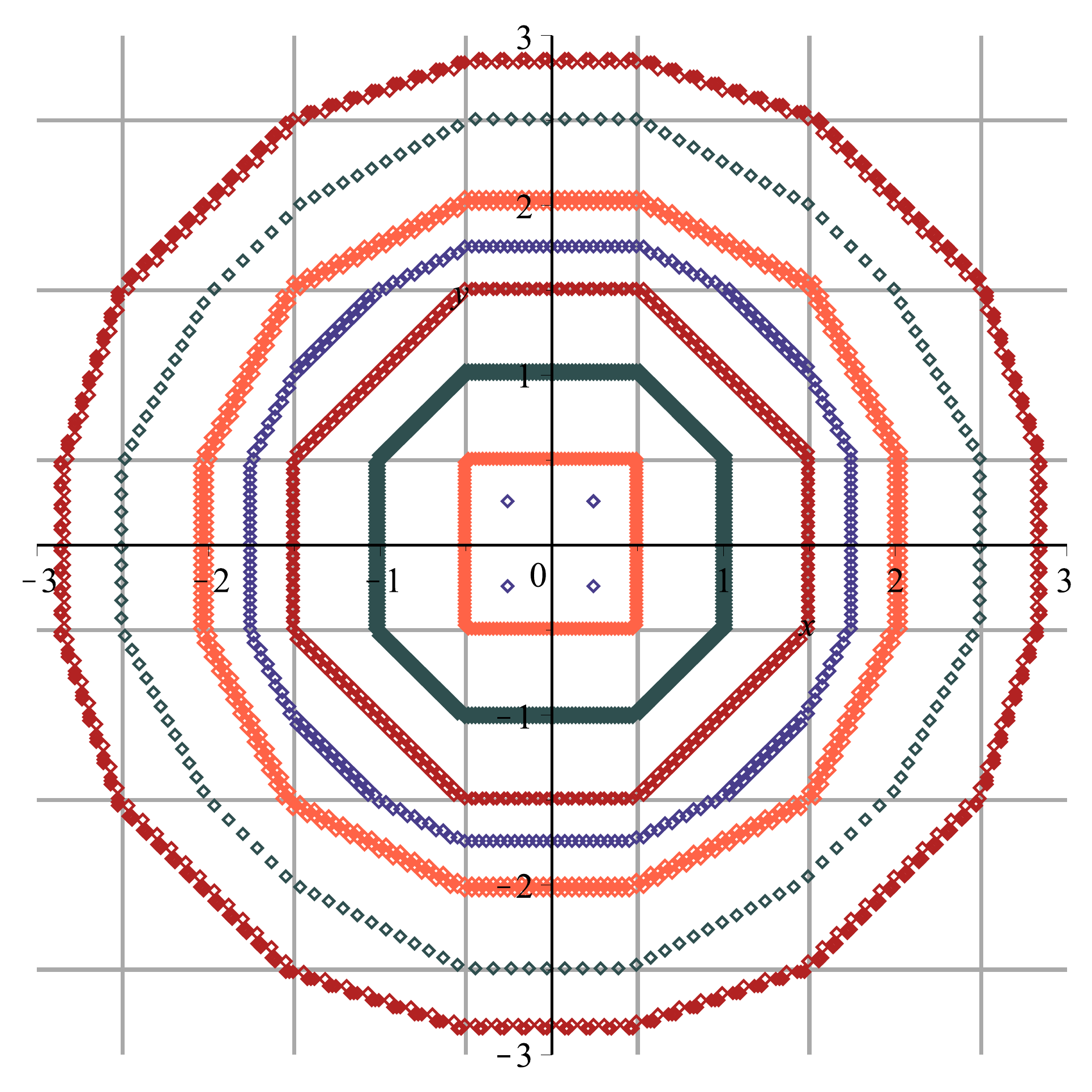} \\
	  (b) $\; \lambda=1/48$ \\
	\end{minipage}
    \caption{\hl{A selection of periodic orbits of the discretised rotation formed from the composition of 
	the elliptic motion} (\ref{def:A}) \hl{with the nearest-neighbour round-off function,
	for two small values of $\lambda$. The lattice spacing is such that each unit distance contains $1/\lambda$ lattice points.
	The grey lines are the lines $x = n+1/2$, $ y = n+1/2$ for $n\in\Z$.
	Here $\nu\approx 1/4$, and the orbits closest to the origin are periodic with period $4$.}}
    \label{fig:NearestInteger}
\end{figure}


The lattice spacing of $F$ is fixed at one, 
but there is no loss of generality: 
the linearity of the underlying dynamics $A$ ensures that 
the discretisation $F_{\alpha}$ with lattice spacing $\alpha>0$ 
(defined using the lattice $\bbL=(\alpha\Z)^2$ and $\Omega=[0,\alpha)^2$)
is conjugate to $F$ via a rescaling:
 $$ F_{\alpha}(z) = \alpha F(z/\alpha) \hskip 40pt z\in(\alpha\Z)^2. $$
We will make use of this fact in the next chapter, 
where we rescale the lattice $\Z^2$.

With regard to the choice of fundamental domain, 
the specific form of $A$ means that the round-off function only affects the $x$-coordinate,
so that $F$ is invertible (and reversible---see next section) irrespective of the choice of $\Omega$.
Furthermore, we may take $\Omega$ to be symmetric under reflection in the line $y=x$ without loss of generality.
If this is the case, then it is a straightforward exercise to show that the inverse of 
$F$ is the discretisation of the inverse of $A$:
 $$ F^{-1} = R \circ A^{-1}. $$

The choice of $\Omega=[0,1)^2$ corresponds to rounding down,
which is arithmetically nicer due to its consistency with modular arithmetic.
This choice also maximises the asymmetry of $F$ under reflection in the line $y=-x$,
i.e., the asymmetry in the direction perpendicular to the symmetry of $F$.
However, the character of the results described in this thesis is 
not heavily dependent on the choice of round-off scheme:
we compare the orbits seen in figure \ref{fig:PolygonalOrbits}
to those in figure \ref{fig:NearestInteger}, 
which were calculated using a nearest-neighbour round-off scheme (i.e., $\Omega=[-1/2,1/2)^2$).

\section{Time-reversal symmetry} \label{sec:time-reversal}

For a detailed background on the subject of time-reversal symmetry, 
we refer the reader to the surveys \cite{QuispelRoberts} and \cite{LambRoberts}, and references therein. 

In broad terms, time-reversal symmetry is a property of an invertible dynamical system, whereby reversing the direction of time 
maps valid trajectories into other valid trajectories. For example, consider the position $x(t)$ of a particle of unit mass moving 
in a conservative force field $f=-\nabla V$, where $V$ is the potential. The equation governing the motion is
\begin{equation} \label{eq:conservativeForce}
 \ddot{x} = f(x),
\end{equation}
where a dot denotes a derivative with respect to time. The equation (\ref{eq:conservativeForce}) is invariant under the transformation
$t\mapsto -t$, and if $x(t)$ is a solution, then $x(-t)$ is also a solution, with the same initial position but opposing initial velocity.

It is standard practice to write systems such as (\ref{eq:conservativeForce}) in the Hamiltonian formalism, where the motion of the 
particle is described by its position $q(t)=x(t)$ and momentum $p(t)=\dot{x}(t)$. The Hamiltonian $H(q,p)$ is given by 
 $$ H(q,p) = \frac{p^2}{2} + V(q), $$
and the governing equations are Hamilton's equations:
\begin{equation} \label{eq:HamiltonsEquations}
 \frac{dq}{dt} = \frac{\partial H}{\partial p} = p \hskip 40pt \frac{dp}{dt} = -\frac{\partial H}{\partial q} = f(q).
\end{equation}
In this setting, the time-reversal symmetry property of the system (\ref{eq:conservativeForce}) hinges upon the fact that the Hamiltonian is even in the momentum coordinate:
\begin{equation*} 
 H(q,p)= H(q,-p),
\end{equation*}
so that the system of equations (\ref{eq:HamiltonsEquations}) are invariant under the transformation
\begin{equation} \label{eq:HamitonianSymmetry}
 (t, q, p) \mapsto (-t, q, -p).
\end{equation}
Thus, in classical mechanics, time-reversal symmetry describes the invariance of a system under the reversal of the time-direction combined with a reflection in phase space, which changes the sign of the momentum coordinate.

In section \ref{sec:Hamiltonian}, we introduce an abstract, piecewise-affine Hamiltonian of the plane with such a classical time-reversal symmetry. The Hamiltonian $\cP$ (see equation (\ref{eq:Hamiltonian})), which represents the limiting dynamics of the discretised rotation $F$, is even in both coordinates, so that orbits of the corresponding flow are reversed by reflections in both axes.

However, the invariance of a system under a transformation such as (\ref{eq:HamitonianSymmetry}) is not a good definition for time-reversal symmetry, since it is not a coordinate independent property. This leads us to the definition of \defn{reversibility}, originally proposed (in a more restricted form\footnote{Devaney required that the phase space have even dimension, and that the involution $G$ fix a subspace with half the dimension of the phase space.}) by Devaney \cite{Devaney}. In this definition, the reflection in the momentum coordinate is replaced by an arbitrary \defn{involution} $G$ of the phase space, i.e., a map $G$ whose second iterate is the identity: $G^2 = \id$. The definition of reversibility can be applied to any flow (not necessarily Hamiltonian) or any map, and we shall be primarily interested in the latter.

\begin{definition}
A map $F$ is \defn{reversible} if it can be expressed as the product of two involutions:
\begin{equation} \label{eq:ReversibilityInvolutions}
 F= H\circ G \hskip 40pt H= F\circ G \hskip 40pt G^2 = H^2 = \id.
\end{equation}
The involutions $G$ and $H$ are called \defn{reversing symmetries}.
\end{definition}

A reversible map $F$ is necessarily invertible, and an equivalent definition of reversibility is to require that $F$ is conjugate to its inverse via an involution $G$:
\begin{equation*} 
 F^{-1} = G \circ F \circ G \hskip 40pt G^2 = \id.
\end{equation*}
The definition of reversibility is consistent with the idea that the involution $G$ reverses the direction of time, since if $x^{\prime}=F(x)$, then
 $$ G(x^{\prime}) = F^{-1} ( G(x)). $$

We note that reversible maps arise naturally from reversible flows, since any Poincar\'{e} return map or time-advance map of a reversible flow yields a reversible map. However, the definition (\ref{eq:ReversibilityInvolutions}) is purely algebraic, and there is no requirement that $F$ have any smoothness properties. We refer the reader to \cite{QuispelRoberts} for a wealth of examples of reversible flows and maps, both in physics and dynamical systems theory.

The lattice map $F$ of equation (\ref{def:F}) is reversible, and its reversing symmetries are given by
\begin{equation} \label{def:GH}
G(x,y) = (y,x)  \hskip 40pt H(x,y) = (\lfloor \lambda y\rfloor - x,  y).
\end{equation}

\subsection*{Symmetric orbits}

Let $O$ denote some (forwards and backwards) orbit of a reversible map $F$. The orbit $O$ is called \defn{symmetric} with respect to the reversing symmetry $G$ if it is setwise invariant under $G$:
 $$ G(O) = O. $$
An orbit which is not symmetric is called \defn{asymmetric}.

It is typical that the symmetric orbits of a dynamical system exhibit different behaviour to the asymmetric orbits. In \cite{QuispelRoberts}, symmetric orbits are associated with the type of universal behaviour displayed by conservative (symplectic) systems, whereas asymmetric orbits are associated with the type of universal behaviour displayed by dissipative systems. In our work we also differentiate between symmetric and asymmetric orbits, with the latter ultimately dominating the phase space (see section \ref{chap:PerturbedAsymptotics}).

In principle, it is not clear how to locate or identify symmetric orbits, other than by an exhaustive search of the phase space. To this end, we introduce the \defn{fixed space} $\Fix{G}$ of a reversing symmetry $G$:
 $$ \Fix{G} = \{ z \, : \; G(z)=z \}. $$
The symmetric orbits of a reversible map are structured by these fixed spaces, as we see in the following folklore theorem.

\begin{theorem}\cite[Theorem 4.2]{LambRoberts} \label{thm:SymmetricOrbits}
 Let $F$ be a reversible map in the sense of (\ref{eq:ReversibilityInvolutions}). Then the orbits of $F$ satisfy the following properties:
 \begin{enumerate}[(i)]
 \item An orbit is symmetric if and only if it intersects the set $\Fix{G}\cup\Fix{H}$.
 \item A symmetric orbit intersects the set $\Fix{G} \cup \Fix{H}$ exactly once if and only if it is either aperiodic or a fixed point.
 \item A symmetric periodic orbit which is not a fixed point intersects the set $\Fix{G} \cup \Fix{H}$ exactly twice; it has even period $2p$ if and only if it intersects one of the sets $\Fix{G} \cap F^p(\Fix{G})$ or $\Fix{H} \cap F^p(\Fix{H})$, and has odd period $2p+1$ if and only if it intersects the set $\Fix{G} \cap F^p(\Fix{H})$.
 \end{enumerate}
\end{theorem}

In dynamical systems of the plane, it is often the case that the fixed space of a reversing symmetry is a line, in which case we refer to a \defn{symmetry line}. In our case, the fixed space of the involution $G$ given in (\ref{def:GH}) is a symmetry line:
\begin{equation} \label{eq:FixG}
 \Fix{G} = \{ (x,y)\in\R^2 \, : \; x=y \},
\end{equation}
whereas the fixed set of the involution $H$ is given by the collection of line segments
\begin{equation} \label{eq:FixH}
 \Fix{H} = \{ (x,y)\in\R^2 \, : \; 2x = \lfloor \lambda y\rfloor \}.
\end{equation}
(We think of $G$ and $H$ as acting on the plane, although they will also be applied to lattice subsets thereof.)

\subsection*{Reversible and equivariant dynamics}

We note at this point that reversing symmetries do not just come in pairs, but in typically infinite sequences: if $G$ is a reversing symmetry of $F$, then $F^n\circ G$ is also a reversing symmetry for all $n\in\Z$. The set of $G$ and its iterates under the motion form a \defn{family} of reversing symmetries, and we call two reversing symmetries $G$ and $G^{\prime}$ \defn{independent} if $G^{\prime}$ is not a member of the family generated by $G$. A system which has just one family of reversing symmetries is called \defn{purely reversible}.

Furthermore, the composition of two reversing symmetries is a \defn{symmetry}, i.e., a map $S$ on the phase space that maps valid trajectories onto valid trajectories. In the case of a map $F$, this means that $F$ commutes with $S$:
 $$ S \circ F = F \circ S, $$
or, if the symmetry $S$ is invertible, that $F$ is \defn{equivariant} under $S$:
 $$ F = S^{-1} \circ F \circ S. $$
If we compose two reversing symmetries from the same family, then we obtain a trivial symmetry, i.e., an iterate of $F$. If, however, a dynamical system has a pair of independent reversing symmetries, then their composition is a non-trivial symmetry.

The group generated by the reversing symmetries of a dynamical system is called the \defn{reversing symmetry group}, which contains the \defn{symmetry group} as a normal subgroup. Thus the study of reversible dynamics can be approached as an extension to that of equivariant dynamics.

The lattice map $F$ is purely reversible---its reversing symmetry group consists of the family 
generated by the reflection $G$ of (\ref{def:GH}).

\subsection*{Time-reversal symmetry in discrete spaces} 

Finally we summarise a series of papers \cite{RobertsVivaldi05,RobertsVivaldi09,NeumarkerRobertsVivaldi} 
concerning universal behaviour among reversible maps with a finite phase space.

If an invertible map $F$ has a finite phase space, then all orbits of $F$ are periodic, and the period $T(z)$ of 
some point $z$ in the phase space is given by
\begin{equation*} 
 T(z) = \min\{k \, : \; F^k(z)=z \}.
\end{equation*}
Furthermore, if the phase space consists of $N$ points, then the \defn{period distribution function} $\cD$ of $F$ is given by
\begin{equation} \label{def:pdf}
 \cD(x) = \frac{1}{N} \, \# \{ z \, : \; T(z)\leq \kappa x \},
\end{equation}
where $\kappa$ is some scaling parameter, so that $\cD(x)$ is the fraction of points in the phase space whose period under $F$ is less than or equal to $\kappa x$.

The period distribution function $\cD$ is a non-decreasing step function, with $\cD(0)=0$ and $\cD(x)\rightarrow 1$ as $x\rightarrow\infty$. The possible values of $\cD$ are restricted to the set
 $$ \left\{ \frac{k}{N} \, : \; 0\leq k \leq N \right\}, $$
and its discontinuities lie in the set
 $$ \left\{ \frac{k}{\kappa} \, : \; k \in \N \right\}. $$
 
In \cite{RobertsVivaldi09}, it was shown that for a suitable choice of the scaling parameter $\kappa$, 
the expected period distribution function of a random reversible map on $N$ points 
converges to a universal limiting distribution as $N\rightarrow\infty$. 
This universal distribution function is given by a gamma (or Erlang) distribution
\begin{equation} \label{def:R(x)}
 \cR(x) = 1 - e^{-x}(1+x).
\end{equation}
The scaling parameter $\kappa$ depends on the fraction of the phase space occupied by the fixed spaces of the reversing symmetries of the map.

\begin{theorem} \cite[Theorem A]{RobertsVivaldi09} \label{thm:GammaDistribution}
 Let $(G,H)$ be a pair of random involutions of a set $\Omega$ with $N$ points, and let
 \begin{equation*} 
  g=\#\Fix{G} \hskip 40pt h=\#\Fix{H} \hskip 40pt \kappa = \frac{2N}{g+h}.
 \end{equation*}
 Let $\cD_N(x)$ be the expectation value of the fraction of $\Omega$ occupied by periodic orbits of $H\circ G$ with period less than $\kappa x$, computed with respect to the uniform probability. If, with increasing $N$, $g$ and $h$ satisfy the conditions
\begin{equation} \label{eq:g_h_conds}
 \lim_{N\rightarrow\infty} g(N) + h(N) = \infty \hskip 40pt \lim_{N\rightarrow\infty} \frac{g(N)+h(N)}{N} = 0,
\end{equation}
 then for all $x\geq 0$, we have the limit 
 $$ \cD_N(x) \rightarrow \cR(x), $$
where $\cR(x)$ is the universal distribution (\ref{def:R(x)}). Moreover, almost all points in $\Omega$ belong to symmetric periodic orbits.
\end{theorem}

Note that since Theorem \ref{thm:GammaDistribution} treats the composition of two random involutions, 
the resulting reversible map will be purely reversible with full probability as $N\rightarrow\infty$. 
Thus we expect (\ref{def:R(x)}) to be the limiting distribution for suitably `random' (in particular, non-integrable) 
purely reversible maps\footnote{Theorem 
\ref{thm:GammaDistribution} has recently been extended---see \cite{Hutz}.}.

Experimental evidence suggests that $\cR(x)$ is indeed the limiting distribution for a number of planar algebraic maps. 
The distribution (\ref{def:R(x)}) was first identified in \cite{RobertsVivaldi05} as the signature of time-reversal symmetry 
in non-integrable planar polynomial automorphisms (i.e., polynomial maps with a polynomial inverse), 
of which the area-preserving H\`{e}non map is an example. 
The maps were reduced to permutations of finite fields of increasing size. 
For suitably chosen parameter values, this reduction preserves the invertibility, symmetry and non-integrability 
of the original map, but also introduces modular multiplication---the ingredient which provides the `randomness'.

More recently, in \cite{NeumarkerRobertsVivaldi}, the distribution (\ref{def:R(x)}) has been observed for the 
Casati-Prosen family of maps---a two parameter family of reversible maps of the torus, 
which have zero entropy but are conjectured to be mixing. 
For rational parameter values, these maps preserve rational lattices, 
and thus can be restricted directly to finite sets of increasing size: 
the distribution $\cR(x)$ is conjectured to be the limiting distribution for a set of rational parameter values with full measure.

In chapter \ref{chap:PerturbedAsymptotics}, \hl{we consider the period distribution function of the perturbed return map $\Phi$ 
on each of the polygon classes}---indexed
by the sums of two squares. 
We provide numerical evidence that $\cR(x)$ is the limiting distribution 
along the subsequence of polygon classes which correspond to perfect squares, 
where the nonlinearity of the return map persists at infinity. 
In our case, the finite structure arises naturally from the symmetry of the system, 
since on each polygon class, the return map is equivariant with respect to a group of lattice translations. 
As the number of equivalence classes modulo this sequence of lattices diverges, 
the period distribution function converges to $\cR(x)$.

%% file: IntegrableLimit.tex
\chapter{The integrable limit}\label{chap:IntegrableLimit}

In this chapter we introduce the behaviour of the discretised rotation $F$ in the limit $\lambda\to 0$,
which we call the \defn{integrable limit}.
For ease of exposition, we assume that $\lambda>0$.
We embed the phase space $\Z^2$ into the plane,
to obtain a rescaled lattice map $\F$,
and show that the limiting dynamics are described by an integrable, piecewise-affine Hamiltonian system.
This Hamiltonian system is accompanied by a natural
partition of the plane into a countable sequence of polygonal annuli,
which classify the integrable orbits according to a certain symbolic coding.

Then we define a return map of $\F$, whose domain is a thin strip $X$ aligned along the symmetry axis.
We show that the orbits of $\F$, up to the time of first return to $X$, 
shadow the orbits of the integrable Hamiltonian system.
Furthermore, we show that the integrable dynamics are nonlinear: 
the return map corresponding to the Hamiltonian flow satisfies a twist condition.

Finally, we briefly describe the dynamics of $F$ in the limits $\lambda\to\pm 1$: 
the other parameter regimes in which the dynamics at the limit is an exact rotation. 
Preliminary observations suggest that we can expect similar behaviour to the $\lambda\to 0$ case.

Much of the work in this chapter has been published in \cite{ReeveBlackVivaldi}.

\section{The rescaled lattice map} \label{sec:Rescaling}

We make some elementary observations about the behaviour 
of orbits of the discretised rotation $F$ in the limit $\lambda\to 0$.
Recall that when $\lambda$ is small, the map $F$ is the discretisation of a
rotation whose angle is close to $\pi/2$.
We find that no orbits are periodic with period four 
(excluding the origin---a fixed point of the dynamics which we ignore in this discussion).
\hl{Instead the orbits of minimal period are the fixed points of a secondary recurrence:
in particular, the fourth iterates of $F$ induce a perturbed rotation (on polygons, rather than circles),
whose angle approaches zero as $\lambda\to 0$}.
Thus all orbits of $F$ visit all four quadrants of the plane,
and for the rest of this section we restrict our attention to those orbits which begin in the first quadrant.

In the following proposition, we describe the orbits closest to the origin.
We see that the orbit of any point, for sufficiently small $\lambda$,
is symmetric and coincides with a level set of $|x|+|y|$ everywhere except in the third quadrant,
where the orbit slips by one lattice point (see figure \ref{fig:squares}).
Each orbit \hl{is a fixed point of the secondary recurrence, 
and preserves a convex polygon---an approximate square.
We refer to the recurrence time of the orbits,
during which these invariant polygons are populated, 
as one} \defn{revolution} \hl{about the origin}.

\begin{figure}[t]
        \centering
        \includegraphics[scale=0.35]{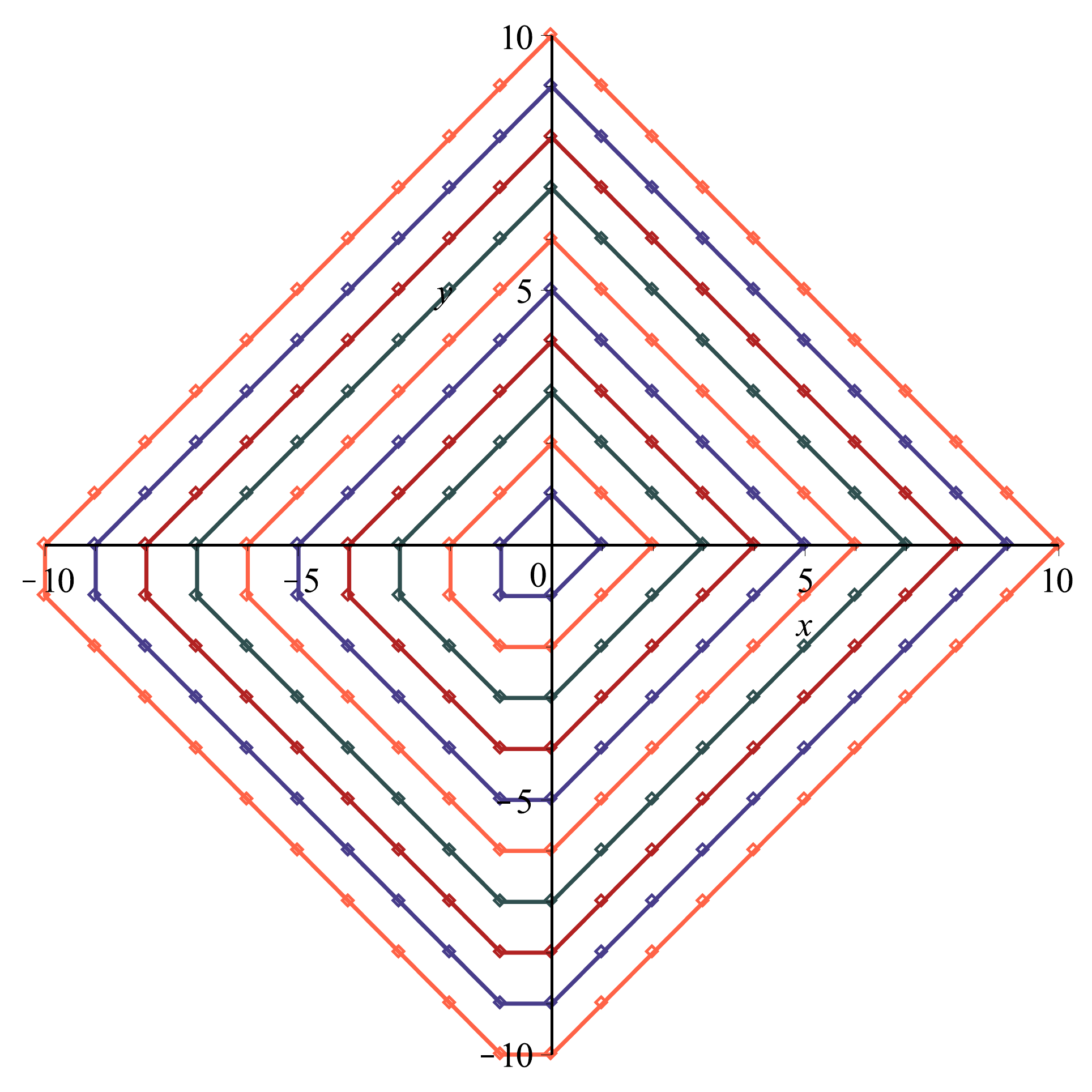}
        \caption{Periodic orbits of $F$ in the region $|x|+|y|\leq 1/\lambda$.
        \hl{Each orbit lies on a convex polygon which is close to a square.}}
        \label{fig:squares}
\end{figure}

\begin{proposition} 
Let $\lambda>0$. For all $z=(x,y)\in\Z^2$ with $x,y\geq0$ satisfying
\begin{equation} \label{eq:lambda_bound}
 x+y < \frac{1}{\lambda},
\end{equation}
the orbit of $z$ under $F$ is symmetric and periodic with minimal period
 $$ 4(x+y)+1. $$
\label{prop:square_orbits}
\end{proposition} 

\begin{proof}
We begin by considering the fourth iterates of $F$.
For $\lambda>0$ and $(x,y)$ satisfying $0\leq x \leq 1/\lambda-1$, $1\leq y \leq 1/\lambda$, we have
 \begin{align}
  F(x,y) &= (\fl{\lambda x} -y,x) = (-y,x), \nonumber \\
  F^2(x,y) &= (\fl{-\lambda y}-x,-y) = (-(x+1),-y), \nonumber \\
  F^3(x,y) &= (\fl{-\lambda(x+1)}+y,-(x+1)) = (y-1,-(x+1)), \nonumber \\
  F^4(x,y) &= (\fl{\lambda(y-1)}+x+1,y-1) = (x+1,y-1). \label{eq:F4}
 \end{align}
Thus every fourth iterate of $F$ translates such points by the vector $(1,-1)$.

Now let $z=(x,y)$ with $x,y\geq0$ satisfying (\ref{eq:lambda_bound}).
(We assume that $z\neq(0,0)$.)
We show that the orbit of $z$ intersects both of the fixed sets $\Fix{G}$ and $\Fix{H}$,
and thus is symmetric and periodic by theorem \ref{thm:SymmetricOrbits} part (iii)
(see page \pageref{thm:SymmetricOrbits}).
Recall that the line $\Fix{G}$ is the set of points with $x=y$, whereas the set $\Fix{H}$ includes the line segment
 $$ \{ (0,y) \, : \; \fl{\lambda y} =0 \} \subset \Fix{H} $$
(see equations (\ref{eq:FixG}) and (\ref{eq:FixH})).

Using (\ref{eq:F4}), we have that the orbit of $z$ intersects $\Fix{H}$ at the point
 $$ F^{-4x}(z) = z - x(1,-1) = (0, x+y)\in\Fix{H}. $$
To show that the orbit intersects $\Fix{G}$, there are two cases to consider.
If the difference between $x$ and $y$ is even, i.e., if
 $$ x-y = 2m \hskip 40pt m\in\Z, $$
then we have
 $$ F^{-4m}(z) = z - m(1,-1) = (x-m, y+m)\in\Fix{G}. $$
By theorem \ref{thm:SymmetricOrbits} part (iii), the orbit of $z$ is symmetric and periodic.
Furthermore, we have
 $$ F^{-4m}(z) \in \Fix{G} \cap F^{4(x-m)}(\Fix{H}), $$
so that the period of $z$ is given by
 $$ 2(4x-4m)+1 = 2(4x-2(x-y))+1 = 4(x+y)+1 $$
as required.
 
If the difference between $x$ and $y$ is odd, so that
 $$ x-y = 2m-1 \hskip 40pt m\in\Z, $$
then applying $F^{-4m}$ gives
 $$ F^{-4m}(z) = z - m(1,-1) = (x-m, y+m)=(x-m,x-m+1). $$
If we now apply $F^2$, to move the orbit into the third quadrant, then we have
\begin{align*}
 F^{-4m+1}(z) &= (-(x-m+1),x-m), \\
 F^{-4m+2}(z) &= (-(x-m+1), -(x-m+1)) \in\Fix{G}.
\end{align*}
Again the orbit of $z$ is symmetric and periodic, and since
 $$ F^{-4m+2}(z) \in \Fix{G} \cap F^{4(x-m)+2}(\Fix{H}), $$
the period of $z$ is given by
 $$ 2(4x-4m+2)+1 = 2(4x-2(x-y))+1 = 4(x+y)+1. $$
This completes the proof.
\end{proof}

However, this is just the beginning of the story, since for all $\lambda>0$ there are
points in $\Z^2$ which do not satisfy (\ref{eq:lambda_bound}).
Further from the origin, we find that not all orbits are symmetric;
that orbits trace a sequence of different polygonal shapes;
and \hl{that orbits may make more than one revolution about the origin} 
(period multiplication---see figure \ref{fig:PeriodFunction}).

If we restrict our attention to the symmetric orbits, 
in particular the orbits which intersect the positive half of the symmetry line $\Fix{G}$,
we have the following description of a collection of orbits which, 
like those in proposition \ref{prop:square_orbits},
\hl{make just one revolution about the origin}.
We refer to such orbits as \defn{minimal orbits}.
We defer the proof of proposition \ref{prop:octagon_orbits} to appendix \ref{chap:Appendix}.

\begin{proposition} \label{prop:octagon_orbits}
For all $\lambda>0$ and $x\in\N$ in the range
 $$ \frac{1}{2\lambda}+2 \leq x\leq \frac{1}{\lambda}-1, $$
the orbit of $z=(x,x)$ under $F$ is symmetric and minimal if and only if
 $$ 2x + \Bceil{\frac{1}{\lambda}} - 2\left\fl{ \frac{1}{\lambda}\right} \equiv 2 \mod{3}. $$
\end{proposition}

The novel element in this proposition is the appearance of congruences---a
feature which will be developed further in chapter \ref{chap:PerturbedDynamics}.

\medskip

Both propositions \ref{prop:square_orbits} and \ref{prop:octagon_orbits} suggest 
that the analysis of the limit $\lambda\to 0$ requires some scaling.
For $\lambda>0$\footnote{The rescaled lattice map with $\lambda<0$ is related to the $\lambda>0$ case
via $F_{-\lambda}=R_x \circ F_{\lambda} \circ R_y$, where $R_x$ and $R_y$ are reflections
in the $x$ and $y$ axes, respectively.}, we normalise the natural length scale $1/\lambda$ by introducing 
the \defn{scaled lattice map} $\F$, which is conjugate to $F$,
and acts on points $z=\lambda(x,y)$ of the scaled lattice $\lZ$:
\begin{equation}\label{def:F_lambda}
\F: \lZ \to \lZ 
 \hskip 40pt
 \F(z)=\lambda F(z/\lambda)
  \hskip 40pt
 \lambda>0.
\end{equation}
The discretisation length of $F_\lambda$ is $\lambda$.
Then we define the \defn{discrete vector field}, which measures the 
deviation of $\F^4$ from the identity:
\begin{equation}\label{def:v}
 \bfv: \; \lZ \to \lZ
\hskip 40pt
\bfv(z) = \F^4(z)-z.
\end{equation}

To capture the main features of $\bfv$ on the scaled lattice, 
we introduce an \defn{auxiliary vector field} $\bfw$ on the plane, 
given by
\begin{equation}\label{def:w}
 \bfw: \; \R^2 \to \Z^2
\hskip 40pt
\bfw(x,y)=(2\fl{ y }+1,-(2\fl{ x }+1)).
\end{equation}
The field $\bfw$ is constant on every translated unit square 
(called a \defn{box})
\begin{equation}\label{def:B_mn}
 B_{m,n} = \{ (x,y)\in\R^2 \, : \; \fl{x} =m, \; \fl{y} = n \}, \quad m,n\in\Z
\end{equation}
and we denote the value of $\bfw$ on $B_{m,n}$ as
\begin{equation} \label{def:w_mn}
 \bfw_{m,n}=(2n+1,-(2m+1)). 
\end{equation}

The following proposition, whose proof we postpone to below (section \ref{sec:Recurrence}, page \pageref{proof:mu_1}), 
states that if we ignore a set of points of zero density, 
then the \hl{vector fields} $\bfv$ and $\bfw$ are parallel.

\begin{proposition} \label{prop:mu_1}
 For $r>0$, we define the set
\begin{equation}\label{eq:A}
 A(r,\lambda) = \{ z\in\lZ \, : \; \| z \|_{\infty} < r \},
\end{equation}
(with $\| (u,v) \|_{\infty} = \max (|u|,|v|)$), 
and the ratio
 \begin{displaymath}
  \mu_1(r,\lambda) = \frac{ \# \{z\in A(r,\lambda) \, : \; \bfv(z) = \lambda\bfw(z) \} }{\# A(r,\lambda)} .
 \end{displaymath}
Then we have
 \begin{displaymath}
  \lim_{\lambda\to 0} \mu_1(r,\lambda) = 1 .
 \end{displaymath}
\end{proposition}

The integrable limit of the system \eqref{def:F} is the asymptotic regime 
that results from replacing $\bfv$ by $\lambda\bfw$. 
The points where the two vector fields differ have
the property that $\lambda x$ or $\lambda y$ is close to an integer. 
The perturbation of the integrable orbits will take place in these small domains.

\section{The integrable Hamiltonian}\label{sec:Hamiltonian}

We define the real function 
\begin{equation}\label{def:P}
P:\R\to \R \hskip 40pt P(x)=\fl{x}^2+(2\fl{x}+1)\{x\},
\end{equation}
where $\{x\}$ denotes the fractional part of $x$.
The function $P$ is piecewise-affine, and coincides with the function 
$x\mapsto x^2$ on the integers, thus:
\begin{equation}\label{eq:SqrtP}
 P(\fl{ x }) = \fl{ x }^2 \hskip 40pt \fl{ \sqrt{P(x)} } = \fl{ x }.
\end{equation}
Using the second statement in (\ref{eq:SqrtP}), we can invert $P$ up to sign by defining
\begin{equation}\label{def:Pinv}
 P^{-1}: \R_{\geq 0} \to \R_{\geq 0} 
\hskip 40pt 
x \mapsto \frac{x + \fl{ \sqrt{x} } (1+\fl{ \sqrt{x} 
})}{2\fl{ \sqrt{x} }+1} ,
\end{equation}
so that $(P^{-1}\circ P)(x) = |x|$.

We define the following Hamiltonian
\begin{equation}\label{eq:Hamiltonian}
 \cP: \; \R^2 \; \to \R
 \hskip 40pt
 \cP(x,y) = P(x)+P(y).
\end{equation}
The function $\cP$ is continuous and piecewise-affine. 
It is differentiable in $\R^2\setminus \Delta$, where $\Delta$ is
the set of orthogonal lines given by
\begin{equation}\label{eq:Delta}
\Delta=\{(x,y)\in\R^2 \, : \; (x-\fl{ x})(y-\fl{ y})=0\}.
\end{equation}
The set $\Delta$ is the boundary of the boxes $B_{m,n}$, defined in (\ref{def:B_mn}).
The associated Hamiltonian vector field, defined for all points 
$(x,y)\in\R^2\setminus \Delta$, is equal to the vector field 
$\bfw$ given in (\ref{def:w}):
\begin{equation}\label{eq:HamiltonianVectorField}
 \left( \frac{ \partial\cP(x,y)}{\partial y}, 
   -\frac{\partial\cP(x,y)}{\partial x} \right) = \bfw(x,y)
 \hskip 40pt 
(x,y)\in \R^2\setminus \Delta.
\end{equation}

We say that the function
 $$ \gamma : \R \to \R^2 $$
is a \defn{flow curve} of the Hamiltonian $\cP$ if it satisfies
 $$ \frac{d\gamma(t)}{dt} = \bfw(\gamma(t)) \hskip 40pt t\in\R. $$
Then the flow $\varphi$ associated with $\cP$ is the family of \defn{time-advance maps} $\varphi^t$ 
satisfying\footnote{We claim without proof that $\varphi$ is well-defined everywhere except the origin.}
\begin{equation*} 
 \varphi^t: \R^2 \to \R^2 \hskip 40pt \varphi^t(\gamma(s)) = \gamma(s+t)
\end{equation*}
for any flow curve $\gamma$ and all $s,t\in\R$.

Proposition \ref{prop:mu_1} states that the vector fields $\bfv$ and $\lambda\bfw$ agree almost everywhere in the limit $\lambda\to 0$.
In turn, the piecewise-constant form of $\bfw$ ensures that $\phil$ is equal to $\F^4$ almost everywhere.

\begin{corollary} \label{corollary:mu_1}
 Let $r>0$ and $A(r,\lambda)$ be as in equation (\ref{eq:A}).
 Then
 \begin{displaymath}
  \lim_{\lambda\to 0} \left( \frac{ \# \{z\in A(r,\lambda) \, : \; \phil(z) = \F^4(z) \} }{\# A(r,\lambda)} \right) = 1.
 \end{displaymath}
\end{corollary}

\medskip

For a point $z\in\R^2$, we write $\Pi(z)$ for the orbit of $z$ under the flow $\varphi$, i.e., the level set of 
$\cP$ passing through $z$:
\begin{equation} \label{def:Pi(z)}
 \Pi(z) = \{ w\in\R^2 \, : \; \cP(w) = \cP(z) \}.
\end{equation}                                 
Below (theorem \ref{thm:Polygons}) we shall see that these sets are polygons,
whose vertices belong to $\Delta$.
The \defn{value} of a polygon $\Pi(z)$ is the real number $\cP(z)$, and if 
$\Pi(z)$ contains a lattice point, then we speak of a \defn{critical polygon}.
The critical polygons act as separatrices, and form a distinguished subset of the plane:
\begin{equation*}
\Gamma=\bigcup_{z\in\Z^2}\Pi(z).
\end{equation*}
All topological information concerning the Hamiltonian $\cP$ is
encoded in the partition of the plane generated by $\Gamma\cup\Delta$.

To characterise $\cP$ arithmetically, we consider the Hamiltonian 
\begin{equation*} 
 \cQ(x,y) = x^2 + y^2,
\end{equation*}
which is equal to the member $\cQ_{0}$ of the family of functions (\ref{def:Q_lambda}), 
and represents the unperturbed rotations (no round-off) in the limit $\lambda\to 0$. 
Its level sets are circles, and the circles containing lattice points will be
called \defn{critical circles}. 
By construction, the functions $\cP$ and $\cQ$ coincide over $\Z^2$, and 
hence the value of every critical polygon belongs to $\cQ(\Z^2)$, the set 
of non-negative integers which are representable as the sum of two squares. 
We denote this set by $\cE$.

A classical result, due to Fermat and Euler, states that
a natural number $n$ is a sum of two squares if and only if
any prime congruent to 3 modulo 4 which divides $n$ occurs 
with an even exponent in the prime factorisation of $n$  
\cite[theorem 366]{HardyWright}. 
We refer to $\cE$ as the set of \defn{critical numbers},
and use the notation
\label{def:cE}
\begin{displaymath}
 \cE = \{e_i \, : \; i\geq 0\} = \{0,1,2,4,5,8,9,10,13,16,17,\dots \} .
\end{displaymath}
There is an associated family of \defn{critical intervals}, given by
\begin{equation}\label{def:Ie}
 \cI^{e_i} = (e_i,e_{i+1}).
\end{equation}

Let us define
$$
\cE(x)=\#\{e\in\cE \, : \; e\leq x\}.
$$
The following result, due to Landau and Ramanujan, gives the asymptotic 
behaviour of $\cE(x)$ (see, e.g., \cite{MoreeKazaran})
\begin{equation}\label{eq:LandauRamanujan}
\lim_{x\to\infty}\frac{\sqrt{\ln x}}{x}\,\cE(x)=K,
\end{equation}
where $K$ is the Landau-Ramanujan constant
$$
K=\frac{1}{\sqrt{2}}\prod_{p\,\,\mathrm{prime}\atop {p\equiv 3\;\mathrm{mod}\; 4}}
\left(1-\frac{1}{p^2}\right)^{-1/2}\,=\,0.764\ldots .
$$
Furthermore, let $r(n)$ be the number of representations of the integer $n$
as a sum of two squares. To compute $r(n)$, we first factor $n$ as follows
$$
n=2^a\prod_i p_i^{b_i}\prod_j q_j^{c_j},
$$
where the $p_i$ and $q_j$ are primes congruent to 1 and 3 modulo 4, respectively.
(Each product is equal to 1 if there are no prime divisors of the corresponding type.)
Then we have \cite[theorem 278]{HardyWright}
\begin{equation}\label{eq:r}
r(n)=4\prod_i(b_i+1)\prod_j\left(\frac{1+(-1)^{c_j}}{2}\right).
\end{equation}
Note that this product is zero whenever $n$ is not a critical number, 
i.e., $r(n)=0$ if $n\notin\cE$.

We now have the following characterisation of the invariant curves of the Hamiltonian $\cP$.
\begin{theorem}\label{thm:Polygons}
The level sets $\Pi(z)$ of $\cP$ are convex polygons, invariant under 
the dihedral group $D_4$, generated by the two orientation-reversing involutions
\begin{equation}\label{eq:Dihedral}
 G:\quad (x,y) \mapsto (y,x) 
\hskip 40pt
 G':\quad (x,y) \mapsto (x,-y).
\end{equation}
The polygon $\Pi(z)$ is critical if and only if
$
 \cP(z)\in\cE.
$
The number of sides of $\Pi(z)$ is equal to
\begin{equation}\label{eq:NumberOfSides}
4(2\left\fl{\sqrt{\cP(z)}\right}+1)-r(\cP(z)),
\end{equation}
where the function $r$ is given in (\ref{eq:r}).
For every $e\in\cE$, the critical polygon with value $e$ intersects 
one and only one critical circle, namely that with the same value.
The intersection consists of $r(e)$ lattice points, and the polygon lies
inside the circle.
\end{theorem}

\begin{proof}
The symmetry properties follow from the fact that the Hamiltonian $\cP$ is 
invariant under the interchange of its arguments, and the function $P$ is even:
\begin{align*} 
P(-x) &= \fl{ -x }^2 + \{-x\}(2\fl{ -x }+1) \\
 &= \left\{ \begin{array}{ll} (-\fl{ x }-1)^2 - (1-\{x\})(2\fl{ x }+1) & x\notin\Z\\ 
(-\fl{ x })^2 & x\in\Z
\end{array} \right. \\
 &= \fl{ x }^2 + \{x\}(2\fl{ x }+1) = P(x).
\end{align*}

The vector field (\ref{eq:HamiltonianVectorField}) is piecewise-constant, 
and equal to $\bfw_{m,n}$ in the box $B_{m,n}$ (cf.~equations
(\ref{def:B_mn}) and (\ref{def:w_mn})). 
Hence a level set $\Pi(z)$ is a union of line segments. 
Since the Hamiltonian $\cP$ is continuous, $\Pi(z)$ is connected.
Thus $\Pi(z)$ is a polygonal curve. 
It is easy to verify that no three segments can have an end-point in common 
(considering end-points in the first octant will suffice). 
Equally, segments cannot intersect
inside boxes, because they are parallel there. But a non self-intersecting 
symmetric polygonal curve must be a polygon. 

Next we prove convexity. Due to dihedral symmetry, if $\Pi(z)$ is convex 
within the open first octant $0<y<x$, then it is piecewise-convex. 
Thus we suppose that $\Pi(z)$ has an edge in the box $B_{m,n}$, where 
$0< n\leq m$. The adjacent edge in the direction of the flow must be 
in one of the boxes
\begin{displaymath}
B_{m,n-1},\quad B_{m+1,n-1},\quad B_{m+1,n}.
\end{displaymath}
Using (\ref{def:w_mn}) one verifies that the three determinants
\begin{displaymath}
\det(\bfw_{m,n},\bfw_{m,n-1})
\qquad
\det(\bfw_{m,n},\bfw_{m+1,n-1})
\qquad
\det(\bfw_{m,n},\bfw_{m+1,n})
\end{displaymath}
are negative. This means that, in each case, at the boundary between adjacent boxes, 
the integral curve turns clockwise. So $\Pi(z)$ is piecewise-convex.
It remains to prove that convexity is preserved across the boundaries of the 
first octant, which belong to the fixed sets $\Fix{G}$ (the line $x=y$)
and $\Fix{G}'$ (the line $y=0$) of the involutions (\ref{eq:Dihedral}). 
Indeed, $\Pi(z)$ is either orthogonal to $\mathrm{Fix}\,G$ (in which case
convexity is clearly preserved), or has a vertex $(m,m)$ on it; in the latter case, 
the relevant determinant is $\det(\bfw_{m-1,m},\bfw_{m,m-1})=-8m<0$.
The preservation of convexity across $\Fix{G}'$ is proved similarly, and thus
$\Pi(z)$ is convex.

The statement on the criticality of $\cP(z)$ follows from 
the fact that, on $\Z^2$, we have $\cP=\cQ$.

Consider now the edges of $\Pi(z)$. The intersections of $\Pi(z)$ 
with the $x$-axis have abscissas $\pm P^{-1} (\cP(z))$. 
Using (\ref{eq:SqrtP}) we have that there are 
$2\fl{\sqrt{\cP(z)}}+1$ integer points between them, hence as 
many lines orthogonal to the $x$-axis with integer abscissa. 
The same holds for the $y$-axis. If $\Pi(z)$ is non-critical, 
it follows that $\Pi(z)$ intersects $\Delta$ in exactly 
$4(2\left\fl{\sqrt{\cP(z)}\right}+1)$ points, 
each line being intersected twice. 
Because the vector field changes across each line, 
the polygon has $4(2\left\fl{\sqrt{\cP(z)}\right}+1)$ vertices. 
If the polygon is critical, then we have $\cP(z)=e\in\cE$. At each 
of the $r(e)$ vertices that belong to $\Z^2$, two lines in $\Delta$ 
intersect, resulting in one fewer vertex. So $r(e)$ vertices must be removed 
from the count.

Next we deal with intersections of critical curves.
Let us consider two arbitrary critical curves
$$
\cP(x,y)=e
\qquad
\cQ(x,y)=e+f
\qquad
e, e+f\in\cE.
$$
This system of equations yields
$$
\{x\}^2 + \{y\}^2 -\{x\} -\{y\}=f,
$$
which is a circle with centre at $(1/2,1/2)$ and radius $\rho$, where
\begin{equation}\label{eq:rho}
\rho^2=f+\frac{1}{2}.
\end{equation}
Since we must have $0\leq\{x\},\{y\}<1$, we find $\rho^2\leq 1/2$, and
since $f$ is an integer, we obtain $\{x\}=\{y\}=f=0$. So critical
polygons and circles intersect only if they have the same value, 
and their intersection consists of lattice points. Then the 
number of these lattice points is necessarily equal to $r(e)$.

Finally, since the critical curve $\cP(x,y)=e$ is a convex polygon,
whose only intersections with the critical circle $\cQ(x,y)=e$
occur at vertices, we have that critical polygons lie inside critical circles.
\end{proof}

From this theorem it follows that the set $\Gamma$ of critical polygons
partitions the plane into concentric domains, which we call \defn{polygon classes}. 
Each domain contains a single critical circle, and has no lattice points in its interior.
The values of all the polygons in a class is a critical interval of the form (\ref{def:Ie}), 
and we associate the critical number $e\in\cE$ with the polygon class $\cP^{-1}(\cIe)$.
There is a dual arrangement for critical circles. 
Because counting critical polygons is the same as counting critical circles,
the number of critical polygons (or, equivalently, of polygon classes) contained 
in a circle of radius $\sqrt{x}$ is equal to $\cE(x)$, with asymptotic 
formula (\ref{eq:LandauRamanujan}).
From equation (\ref{eq:rho}), one can show that the total variation 
$\Delta\cQ(\alpha)$ of $\cQ$ along the polygon $\cP(z)=\alpha$ satisfies the bound
$$
\Delta\cQ(\alpha)\leq \frac{1}{2}
$$
which is strict (e.g., for $\alpha=1$). 

\subsection*{Symbolic coding of polygon classes} 

In theorem \ref{thm:Polygons} we classified the invariant curves of 
the Hamiltonian $\cP$ in terms of critical numbers. We found that the
set $\Gamma$ of critical polygons partitions the plane into concentric annular 
domains---the polygon classes. 
In this section we define a symbolic coding on the set of classes, 
which specifies the common itinerary of all orbits in a class,
taken with respect to the lattice $\Z^2$. 

Suppose that the polygon $\Pi(z)$ is non-critical. Then all vertices of 
$\Pi(z)$ belong to $\Delta\setminus \Z^2$, where $\Delta$ was 
defined in (\ref{eq:Delta}). Let $\xi$ be a vertex.
Then $\xi$ has one integer and one non-integer coordinate, and we let $u$
be the value of the non-integer coordinate. 
We say that the vertex $\xi$ is of \defn{type} $v$ if $\fl{ |u| } =v$.
Then we write $v_j$ for the type of the $j$th vertex, where the vertices of $\Pi(z)$ 
are enumerated according to their position in the plane, 
starting from the positive half of symmetry line $\Fix{G}$ and proceeding clockwise.

\begin{figure}[b]
        \centering
        \includegraphics[scale=0.9]{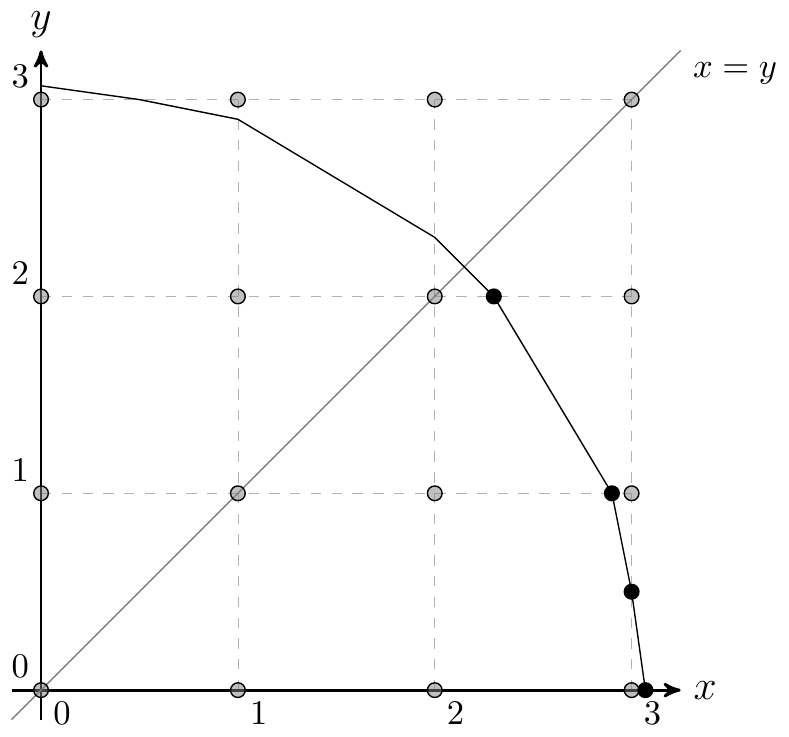}
        \caption{A polygon with $\cP(z)$ in the interval $(9,10)$ and its vertices in the first octant.}
        \label{fig:V(9)}
\end{figure}

The sequence of vertex types $v_j$ reflects the eight-fold symmetry of $\Pi(z)$. 
Hence if the $k$th vertex lies on the $x$-axis, then there are $2k-1$ vertices 
belonging to each quarter-turn, and the vertex types satisfy
\begin{equation} \label{eq:v_symmetry}
 v_j = v_{2k-j} = v_{(2k-1)i+j}, \hskip 20pt 
1\leq j \leq k, \quad 0\leq i \leq 3. 
\end{equation}
Thus it suffices to consider the vertices in the first octant, and 
the \defn{vertex list} of $\Pi(z)$ is the sequence of vertex types
\label{def:VertexList}
\begin{displaymath}
V=(v_1,\dots,v_{k}).
\end{displaymath}
We note that the vertex list can be decomposed into two disjoint subsequences; 
those entries belonging to a vertex with integer $x$-coordinate and those 
belonging to a vertex with integer $y$-coordinate. 
These subsequences are non-decreasing and non-increasing, respectively.

From theorem \ref{thm:Polygons}, it follows that for 
every $e\in\cE$, the set of polygons $\Pi(z)$ with
$
 \cP(z)\in \cIe
$
have the same vertex list.
Let $k$ be the number of entries in the vertex list. Since the polygon
$\Pi(z)$ is non-critical, equation (\ref{eq:NumberOfSides}) 
gives us that 
$4( 2\fl{ \sqrt{e} } +1) = 4(2k-1)$, and hence
\begin{equation*}
k=\#V=\fl{ \sqrt{e}} +1.
\end{equation*}
Any two polygons with the same vertex list have not only the same number of edges, 
but intersect the same collection of boxes, and have the same collection of tangent 
vectors. The critical polygons which intersect the lattice $\Z^2$, where 
the vertex list is multiply defined, form the boundaries between classes. 
The symbolic coding of these polygons is ambiguous, but this item will not
be required in our analysis.

Thus the vertex list is a function on classes, hence on $\cE$. 
For example, the polygon class identified with the interval $\cI^9=(9,10)$ 
(see figure \ref{fig:V(9)}) has vertex list
\begin{displaymath}
V(9)=( 2, 2, 0, 3).
\end{displaymath}
For each class, there are two vertex types which we can calculate explicitly: 
the first and the last. If $\alpha\in\cIe$, and the polygon $\cP(z)=\alpha$ intersects 
the symmetry line $\Fix{G}$ at some point $(x,x)$ in the first quadrant, 
then $v_1 = \fl{x}$. 
By the definition (\ref{eq:Hamiltonian}) of the Hamiltonian $\cP$, $x$ satisfies
 $$ \cP(z) = 2 P(x) = \alpha. $$
Thus inverting $P$ and using (\ref{eq:SqrtP}), it is straightforward to show that
the first vertex type is given by
\begin{equation}\label{def:v1}
 v_1 = \fl{P^{-1}(\alpha/2)} = \fl{\sqrt{e/2}} \hskip 40pt \alpha\in\cIe.
\end{equation}
Similarly the last vertex type, corresponding to the vertex on the $x$-axis, is given by
\begin{equation} \label{def:vk}
 v_k = \fl{P^{-1}(\alpha)} = \fl{\sqrt{e}} \hskip 40pt \alpha\in\cIe. 
\end{equation}

\begin{table}[!h]
        \centering
                \begin{tabular}{|l|l|}
                \hline
                e & V(e) \\
                \hline
                9 & $ (2,2,0,3) $\\
                10 & $(2,1,3,3)$ \\
                18 &  $(3,3,1,4,4)$ \\
                29 & $(3,4,2,5,5,5)$ \\
                49 & $(4,5,3,6,6,6,0,7)$ \\
                52 & $(5,4,6,6,6,1,7,7)$ \\
                \hline
                \end{tabular}
        \caption{A table showing the vertex list $V(e)$ for a selection of 
 critical numbers $e$. Notice that the first entry in the vertex list is 
 always $\fl{\sqrt{e/2}}$, the last is $\fl{\sqrt{e}}$, and the
 number of entries in the list is $k=\fl{\sqrt{e}}+1$.}
\label{table:V(e)}
\end{table}

\section{Recurrence and return map} \label{sec:Recurrence} \label{SEC:RECURRENCE}

We have already seen that the lattice map $F$ is reversible
with respect to the reflection $G$ of equation (\ref{def:GH}).
The scaled map $F_\lambda$ has the same property, and all orbits of 
$F_\lambda$ return repeatedly to a neighbourhood of the symmetry line 
$\Fix{G}$, i.e., the line $x=y$.

From equation (\ref{def:A}), the rotation number $\nu$ has the asymptotic form
\begin{displaymath}
 \nu = \frac{1}{2\pi}\arccos\left(\frac{\lambda}{2}\right) 
= \frac{1}{4}- \frac{\lambda}{4\pi} + O(\lambda^3) \hskip 40pt \lambda\to 0.
\end{displaymath}
The integer $t=4$ is the \defn{zeroth-order recurrence time} of orbits under 
$\F$, that is, the number of iterations needed for a point to return 
to an $O(\lambda)$-neighbourhood of its starting point. It turns out 
(see below---lemma \ref{lemma:Lambda}) 
that the field $\bfv(z)$ (equation (\ref{def:v})) is non-zero 
for all non-zero points $z$, so no orbit has period four.
Accordingly, for the limit $\lambda\to 0$, we define the \defn{first-order recurrence time} 
$t^*$ of the rotation to be the next time of closest approach:
\begin{equation}\label{eq:tstar}
 t^*(\lambda) = \min \left\{k\in\N \, : \; d_H(k\nu,\N) \leq d_H(4\nu,\N), \; k>4\right\}
              = \frac{\pi}{\lambda} + O(1),
\end{equation}
where $d_H$ is the Hausdorff distance, and the expression $d_H(x,A)$, with $x\in\R$, 
is to be understood as the Hausdorff distance between the sets $\{x\}$ and $A$. 

The integer $t^*$ provides a natural recurrence timescale for $\F$. 
Let $T(z)$ be the minimal period of the point $z\in\Z^2$ under $F$, 
so that $T(z/\lambda)$ is the corresponding function for 
points $z\in\lZ$ under $\F$.
(In accordance with the \hl{periodicity conjecture}, page \pageref{conj:Periodicity}, we assume that this 
function is well-defined.) 
Since, as $\lambda\to 0$, the recurrence time $t^*$ diverges,
the periods of the orbits will cluster around integer multiples 
of $t^*$, giving rise to branches of the period function (figure \ref{fig:PeriodFunction}). 
The lowest branch corresponds to orbits which perform a single revolution
about the origin---the minimal orbits---and their period is approximately equal to $t^*$.
The period function $T$ has a normalised counterpart, given by (cf.~(\ref{eq:tstar}))
\begin{equation*}
T_{\lambda}: \lZ \to \frac{\lambda}{\pi}\,\N
 \hskip 40pt
 T_\lambda(z)=\frac{\lambda}{\pi}\,T(z/\lambda).
\end{equation*}
The values of $T_\lambda$ oscillate about the integers.

We construct a Poincar\'e return map $\Phi$ on a neighbourhood of the 
positive half of the symmetry line $\Fix{G}$. 
Let $d(z)$ be the perpendicular distance between a point $z$ and 
$\Fix{G}$:
\begin{equation*}
 d(z) = d_H(z,\Fix{G}).
\end{equation*}

We define the domain $X$ of the return map $\Phi$ to be the set of points 
$z\in(\lambda\Z_{\geq0})^2$ which are closer to $\Fix{G}$ than their 
preimages under ${F}_\lambda^4$, and at least as close as their images:
\begin{equation}\label{def:X}
 X = \{z\in(\lambda\Z_{\geq0})^2 \, : \; d(z) \leq d(\F^4(z)), \; d(z) < d(\F^{-4}(z)) \}.
\end{equation}
\hl{According to corollary} \ref{corollary:mu_1} (page \pageref{corollary:mu_1}), when $\lambda$ is small,
\hl{the fourth iterates of $\F$ typically agree with $\phil$, the time-$\lambda$ advance map of the flow.
Thus, in a neighbourhood of the symmetry line $\Fix{G}$, the map $\F^4$ is
simply a translation perpendicular to $\Fix{G}$}:
 $$ \F^4(z) = z + \lambda\bfw_{m,m}  \hskip 20pt z\in B_{m,m}, \; m\in\Z_{\geq 0}, $$
\hl{where $\bfw_{m,m}$ is the local component of the Hamiltonian vector field $\bfw$ 
in $B_{m,m}$} (see equation \eqref{def:w_mn}).
It follows that the main component of $X$ in $B_{m,m}$ is a 
thin strip of width $\lambda\|\bfw_{m,m}\|$ lying parallel to the symmetry line $\Fix{G}$
(see figure \ref{fig:lattice_Le}, page \pageref{fig:lattice_Le}).
\hl{Furthermore, it is natural to identify the sides of this strip,
which are connected by the translation $z\mapsto z+\lambda\bfw_{m,m}$,
so that locally the dynamics take place on a cylinder.}
This description breaks down when $z$ is close to a vertex, 
i.e., close to the boundary of $B_{m,m}$.
We formalise these properties below (section \ref{sec:RegularDomains}).

The transit time $\tau$ to the set $X$ is well-defined for all $z\in\lZ$: 
\begin{equation} \label{eq:tau}
 \tau : \lZ \to \N
 \hskip 40pt \tau(z)=\min \{ k\in\N \, : \; \F^k(z) \in X \}.
\end{equation}
Thus the first return map $\Phi$ is the function
\begin{equation} \label{def:Phi}
  \Phi : X \to X
 \hskip 40pt
   \Phi(z) = \F^{\tau(z)}(z).
\end{equation}
We refer to the orbit of $z\in X$ up to the return time $\tau(z)$ as the 
\defn{return orbit} of $z$:
\begin{equation*}
 \Ot(z) = \{ \F^k(z) \, : \; 0\leq k \leq \tau(z) \} \hskip 40pt z\in X.
\end{equation*}
We let $\tau_{-}$ be the transit time to $X$ under $\F^{-1}$: 
\begin{displaymath}
 \tau_{-} : \lZ \to \Z_{\geq0}
 \hskip 40pt
  \tau_{-}(z) = \min \{ k\in\Z_{\geq0} \, : \; \F^{-k}(z) \in X \},
\end{displaymath}
so that the return orbit for a general $z\in\lZ$ is given by
\begin{equation*}
 \Ot(z) = \{ \F^k(z) \, : \; -\tau_{-}(z)\leq k \leq \tau(z) \} \hskip 40pt z\in \lZ.
\end{equation*}

\medskip


To associate a return orbit with an integrable orbit, we define the rescaled 
round-off function $R_{\lambda}$, which rounds points on the plane down to 
the next lattice point:
\begin{equation} \label{eq:R_lambda}
 R_{\lambda}: \R^2 \to \lZ 
  \hskip 40pt
 R_{\lambda}(w)=\lambda R(w/\lambda),
\end{equation}
where $R$ is the integer round-off function (\ref{eq:R}).
For every point $w\in\R^2$ and every $\delta>0$, the set of points
$$
\{z\in\lZ \, : \; z= R_{\lambda}(w), \; \,0< \lambda<\delta\}
$$
that represent $w$ on the lattice as $\lambda\to0$ is countably infinite.
The corresponding set of points on $\Z^2$, before rescaling, is unbounded.

According to proposition \ref{prop:mu_1}, the points of the scaled lattice 
$\lZ$ at which the (rescaled) integrable and discrete vector fields have 
different values are rare, as a proportion of lattice points. The following 
result shows that these points are also rare within each return orbit. 

\begin{proposition} \label{prop:mu_2}
 For any $w\in\R^2$, if we define the ratio
 \begin{displaymath}
  \mu_2(w,\lambda) = \frac{\#\{z\in\Ot(R_{\lambda}(w)) \, : \; \bfv(z)=\lambda\bfw(z)\}}{\# \Ot(R_{\lambda}(w))} ,
 \end{displaymath}
 then we have
 \begin{displaymath}
  \lim_{\lambda\to 0} \mu_2(w,\lambda) = 1 .
 \end{displaymath}
\end{proposition} 

Finally we formulate a shadowing theorem, which states that
for timescales corresponding to a first return to the domain $X$, every 
integrable orbit has a scaled return orbit that shadows it. 
Furthermore, this scaled return orbit of the round-off map 
converges to the integrable orbit in the Hausdorff metric as $\lambda\to 0$,
\hl{so that up to their natural recurrence time, orbits of $\F$ render increasingly accurate
approximations of the flow trajectories}.

\begin{theorem} \label{thm:Hausdorff}
For any $w\in\R^2$, let $\Pi(w)$ be the orbit of $w$ under the flow $\varphi$, and let
$\Ot(R_{\lambda}(w))$ be the return orbit at the \hl{corresponding}
lattice point. Then
\begin{displaymath}
 \lim_{\lambda\to 0} d_H 
   \left(\Pi(w),\Ot(R_{\lambda}(w))\right)=0,
\end{displaymath}
where $d_H$ is the Hausdorff distance on $\R^2$.
\end{theorem}

This result justifies the term `integrable limit' assigned to the flow $\varphi$ generated by $\cP$.
The proofs for proposition \ref{prop:mu_2} and theorem \ref{thm:Hausdorff} can be found below.

\subsection*{Transition points}

To establish propositions \ref{prop:mu_1} and \ref{prop:mu_2},
we seek to isolate the lattice points $z\in\lZ$ where 
the discrete vector field $\bfv(z)$ deviates from the scaled auxiliary 
vector field $\lambda\bfw(z)$. We say that a point $z\in\lZ$ is a 
\defn{transition point} if $z$ and its image under $F_\lambda^4$
do not belong to the same box, namely if
\begin{displaymath}
 R(\F^4(z))\not=R(z).
\end{displaymath}
Let $\Lambda$ be the set of transition points. Then
\begin{equation} \label{eq:Lambda}
 \Lambda = \bigcup_{m,n\in\Z} \Lambda_{m,n},
\end{equation}
where
\begin{displaymath}
 \Lambda_{m,n} =  \F^{-4}(B_{m,n}\cap\lZ)\setminus B_{m,n}.
\end{displaymath}

\begin{figure}[t]
        \centering
        \includegraphics[scale=0.9]{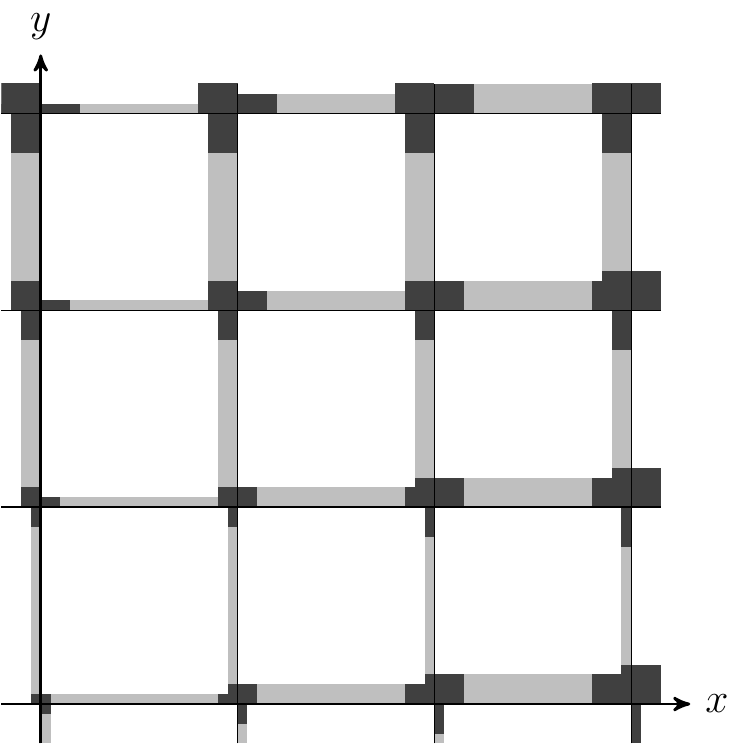}
        \caption{The structure of phase space. The boxes $B_{m,n}$, bounded by the
         set $\Delta$, include regular domains (white) where the motion is integrable,
         \hl{i.e., where $\F^4(z)=\phil(z)$}. By corollary \ref{corollary:mu_1}, page \pageref{corollary:mu_1}, 
         \hl{the lattice points in these domains have full density as $\lambda\to 0$}.
         The darker regions comprise the set $\Lambda$ of transition points, which is introduced below.
         \hl{The transition points are where the perturbation from the integrable motion occurs.}
         The darkest domains belong to the set $\Sigma\subset \Lambda$, defined in (\ref{eq:Sigma}),
         \hl{which forms a neighbourhood of the set $\Z^2$.
         Perturbed orbits which intersect the set $\Sigma$ are analogous to the critical polygons of 
         the flow, and we exclude them from our analysis.}}
        \label{fig:LambdaSigma_plot}
\end{figure}

For small $\lambda$, the set of transition points consists of thin 
strips of lattice points arranged along the lines $\Delta$ (see figure 
\ref{fig:LambdaSigma_plot}). The following key lemma states that, for sufficiently 
small $\lambda$, all points $z\not=(0,0)$ where $\bfv(z)\neq\lambda\bfw(z)$ 
are transition points. 

\begin{lemma} \label{lemma:Lambda}
Let $A(r,\lambda)$ be as in equation (\ref{eq:A}).
Then for all $r>0$ there exists $\lambda^*>0$ such that, for all $\lambda<\lambda^*$ 
and $z\in A(r,\lambda)$, we have
\begin{displaymath}
 z \notin \Lambda \cup \{(0,0)\} \quad \Rightarrow \quad \bfv(z)= \lambda\bfw(z).
\end{displaymath}
\end{lemma}

\begin{proof}
Let $r>0$ be given, and let $z=\lambda(x,y)\in A(r,\lambda)$.
We show that if $\lambda$ is sufficiently small (and $z\not=0$), 
then
\begin{displaymath}
\bfv(z)\neq \lambda\bfw(z) 
\quad \Rightarrow \quad 
 R(\F^4(z))\not=R(z).
\end{displaymath}

Since $z\in A(r,\lambda)$, we have $z\in B_{m,n}$ for some $|m|,|n|\leq\ceil{r}$, 
where $\ceil{\cdot}$ is the ceiling function, defined by the 
identity $\ceil{x}=-\fl{ -x }$. 
Through repeated applications of $\F$, we have
\begin{equation}\label{eq:Fabcd}
\begin{array}{llll}
 &\F(z)   = \lambda(-y+m,x)              &R(F_\lambda(z))=(-(a+1),m),\\
 \noalign{\vspace*{2pt}}
 &F^2_{\lambda}(z) = \lambda(-x-a-1,-y+m)         &R(F_\lambda^2(z))=(-(b+1),-(a+1)),\\
 \noalign{\vspace*{2pt}}
 &F^3_{\lambda}(z) = \lambda(y-m-b-1,-x-a-1)      &R(F_\lambda^3(z))=(c,-(b+1)), \\
 \noalign{\vspace*{2pt}}
 &F^4_{\lambda}(z) = \lambda(x+a+c+1,y-m-b-1)  \quad   &R(F_\lambda^4(z))=(d,c),
\end{array}
\end{equation}
where $m=\fl{\lambda x}$, $n=\fl{\lambda y}$, and the integers $a,b,c,d$ are given by
\begin{equation}\label{eq:abcd}
\begin{array}{rl}
  a+1 &= \ceil{\lambda (y-m)},\\
 \noalign{\vspace*{2pt}}
  b+1 &= \ceil{\lambda (x+a+1)},  \\
 \noalign{\vspace*{2pt}}
  c &= \fl{ \lambda (y-m-b-1) } \\
 \noalign{\vspace*{2pt}}
  d &= \fl{ \lambda (x+a+c+1) }.
\end{array}
\end{equation}
The integers $a$, $b$, $c$ and $d$ label the boxes in which each iterate occurs,
and also give an explicit expression for the round-off term 
$\fl{ \lambda x }$ at each step.
Thus reading from the last of these equations, the discrete vector field $\bfv$ of $z\in B_{m,n}$ is given by
\begin{equation} \label{eq:v_abcd}
\bfv(z)= F^4_{\lambda}(z) - z = 
\lambda( a+c+1, -(m+b+1)),
\end{equation}
and $z$ is a transition point whenever at least one of the equalities 
$d=m$ and $c=n$ on the final pair of box labels fails.


If the integers $m$, $a$, $b$, $c$ are sufficiently small relative to the 
number of lattice points per unit length, i.e., if
\begin{equation} \label{eq:abc_ineq}
\max (|m|,|a+1|,|m+b+1|,|a+c+1|)<1/\lambda, 
\end{equation}
then the map $F^4_{\lambda}$ moves the point $z$ at most one box in each of the $x$ and $y$ directions, 
so that the labels $a$, $b$, $c$ and $d$ satisfy
\begin{equation} \label{eq:bd_ac_sets}
b,d\in\{m-1,m,m+1\}, \hskip 20pt a,c\in\{n-1,n,n+1\}. 
\end{equation}
Similarly, (\ref{eq:abc_ineq}) dictates that the discrepancy between each of the pairs $(b,d)$, $(a,c)$ cannot be too large:
\begin{equation} \label{eq:bd_ac_ineq}
 |b-d|, |a-c| \leq 1.
\end{equation}

Letting $\lambda^* = 1/(2\ceil{r} +3)$, we obtain 
\begin{align*}
& \max (|m|,|a+1|,|m+b+1|,|a+c+1|) \\
	& \quad \leq \max (|m|+|b+1|,|a+1|+|c|) \\
	& \quad \leq \max (2|m|+|b-m|,2|n|+|a-n|+|c-n|)+1 \\
	& \quad \leq 2\ceil{r}+3 \leq 1/\lambda^*,
\end{align*}
so that (\ref{eq:bd_ac_sets}) and (\ref{eq:bd_ac_ineq}) hold for all $\lambda<\lambda^*$. 
Then the expression (\ref{eq:v_abcd}) for $\bfv$, combined with the inequality (\ref{eq:bd_ac_ineq}),
gives that $\bfv(z)= \lambda\bfw(z)$ if and only if
\begin{equation} \label{eq:v=w}
m=b \hskip 40pt n = a = c. 
\end{equation}

Suppose now that $z$ is not a transition point, so that $c=n$ and $d=m$, 
but that $\bfv(z)\neq \lambda\bfw(z)$, so that at least one of the 
equalities (\ref{eq:v=w}) fails. 
If $a\neq n$, straightforward manipulation of inequalities 
shows that the only combination of values which satisfies (\ref{eq:bd_ac_sets}) is
\begin{displaymath}
 a=n-1, \, m=0, \, b=-1, \, \lambda y = n.
\end{displaymath}
In particular, we have $b\neq m$.
Conversely if $b\neq m$, then using also the inequality (\ref{eq:bd_ac_ineq}) gives
\begin{displaymath}
 b=m-1, \, c=0, \, a=-1, \, \lambda x = m,
\end{displaymath}
so that $n=c\neq a$.
Hence, combining these, the only possibility is $m=a+1=b+1=c=0$, 
which corresponds to the unique point $z=(0,0)$.
\end{proof}

By construction, the auxiliary vector field $\bfw$ is 
equal to the Hamiltonian vector field associated 
with $\cP$ (see equation (\ref{eq:HamiltonianVectorField})). 
Since $\bfw$ is piecewise-constant, it follows that $\phil$, the time-$\lambda$ advance map of the Hamiltonian flow,
is equal to a translation by $\lambda\bfw$ everywhere except across the discontinuities of $\bfw$, 
i.e., except at transition points.

Furthermore, lemma \ref{lemma:Lambda} gives us that, for sufficiently small $\lambda$,
any $z\in A(r,\lambda)\setminus\{(0,0)\}$ \hl{which is not a transition point} satisfies $\lambda\bfw(z)=\bfv(z)$.
Hence, a simple consequence of lemma \ref{lemma:Lambda} is that $\F^4$ is equal to a
time-$\lambda$ advance of the flow everywhere except at the transition points.

\begin{corollary} \label{corollary:Lambda}
Let $A(r,\lambda)$ be as in equation (\ref{eq:A}).
Then for all $r>0$ there exists $\lambda^*>0$ such that, for all $\lambda<\lambda^*$ 
and $z\in A(r,\lambda)$, we have
\begin{displaymath}
 z \notin \Lambda \cup \{(0,0)\} \quad \Rightarrow \quad \F^4(z)= \phil(z).
\end{displaymath}
\end{corollary}

We now use lemma \ref{lemma:Lambda} to prove proposition \ref{prop:mu_1}, 
given in section \ref{sec:Hamiltonian}, page \pageref{prop:mu_1}.

\begin{proof}[Proof of proposition \ref{prop:mu_1}] \label{proof:mu_1}
From equation (\ref{eq:A}), we have that the number of lattice points 
in the set $A(r,\lambda)$ is given by 
\begin{displaymath}
  \# A(r,\lambda) = \left( 2 \Bceil{\frac{r}{\lambda}} -1\right)^2.
\end{displaymath}
By lemma \ref{lemma:Lambda}, for sufficiently small $\lambda$, 
every non-zero point $z\in A(r,\lambda)$ 
satisfying $\bfv(z)\neq \lambda\bfw(z)$ is a transition point, so has 
$z\in \Lambda_{m,n}$ for some $m,n\in\Z$ with $|m|,|n|\leq \ceil{r}+1$. 
Furthermore, every set $\Lambda_{m,n}$ is composed of two strips,
each of unit length, and width approximately equal to $\lambda(2m+1)$ 
and $\lambda(2n+1)$, respectively.
We can bound the number of lattice points in the set $\Lambda_{m,n}$ explicitly by
\begin{displaymath}
  \# \Lambda_{m,n} \geq \frac{|2m+1|+|2n+1|-c}{\lambda},
\end{displaymath}
for some positive constant $c$, independent of $m$ and $n$. (Indeed $c=3$ is sufficient -- 
cf.~the methods used in the proof of proposition \ref{prop:Xe}.)
It follows that for fixed $r>0$, as $\lambda\to 0$ we have the estimate
\begin{align*}
  \mu_1(r,\lambda) &= 1 - \frac{ \# \{z\in A(r,\lambda) \, : \; 
    \bfv(z) \neq \lambda\bfw(z) \} }{\# A(r,\lambda)} \\
 &\geq 1 - \frac{ \# \left(A(r,\lambda) \cap \Lambda \right)}{\# A(r,\lambda)} \\
 &\geq 1 - \frac{1}{\# A(r,\lambda)} \sum_{|m|,|n|\leq \ceil{r}+1} \#  \Lambda_{m,n} \\
 &\geq 1 - \left( 2 \Bceil{\frac{r}{\lambda}} -1\right)^{-2} \sum_{|m|,|n|\leq \ceil{r}+1} 
          \frac{|2m+1|+|2n+1|-c}{\lambda} \\
 &= 1 + O(\lambda).
\end{align*}
Since $\mu_1(r,\lambda)\leq 1$, the proof is complete.
\end{proof}

\medskip

To prove proposition \ref{prop:mu_2} and theorem \ref{thm:Hausdorff}
we need a second lemma, which bounds the variation in the Hamiltonian function 
$\cP$ along perturbed orbits $\Ot(R_{\lambda}(w))$ as $\lambda\to 0$, where $w\in\R^2$.
By corollary \ref{corollary:Lambda}, we know $\cP$ is invariant under $\F^4$ at all points $z\notin\Lambda$,
so that variations can only occur when the fourth iterates of $\F$ hit a transition point.
However, the number of transition points encountered by a perturbed orbit in one revolution 
is (essentially) equal to the number of vertices of the corresponding polygon,
which is independent of $\lambda$.
Furthermore, the magnitude of the perturbation from the integrable motion at such a transition point
is $O(\lambda)$ as $\lambda\to 0$. Hence we have the following result.

\begin{lemma} \label{lemma:cP_variation}
Let $w\in\R^2$ and let $z=R_{\lambda}(w)\in\lZ$ be the rounded lattice point associated with $w$. 
Then as $\lambda\to 0$:
 $$ \forall \xi\in \Ot(z): \hskip 20pt |\cP(\xi)-\cP(w)| = O(\lambda). $$
\end{lemma}

We postpone the proof of lemma \ref{lemma:cP_variation} to appendix \ref{chap:Appendix},
and proceed with the proof of proposition \ref{prop:mu_2} (page \pageref{prop:mu_2}).

\begin{proof}[Proof of proposition \ref{prop:mu_2}]
Let $w\in \R^2$ be given, and let $z=R_{\lambda}(w)$. 
For small $\lambda$, the polygons $\Pi(z)$ and $\Pi(w)$ are close, since
 $$ \|z - w\| = O(\lambda) $$
as $\lambda\to 0$.

Consider the polygons $\Pi(w)^{\pm}$ given by
 $$ \Pi(w)^{\pm} = \{ \xi \, : \; \cP(\xi) = \cP(w)\pm1 \},  $$
where the abscissae $x^{\pm}$ of the intersections of $\Pi(w)^{\pm}$
with the positive $x$-axis are given by
 $$ x^{\pm} = P^{-1}(\cP(w)\pm1). $$
Without loss of generality, we may assume that neither of these polygons 
is critical. Thus each of these integrable orbits intersects as many boxes
as it has sides. For the larger polygon $\Pi(w)^+$, the number of sides 
(see theorem \ref{thm:Polygons}, page \pageref{thm:Polygons}) is given by
 $$ 4\left( 2\Bfl{\sqrt{\cP(w)+ 1}}+1 \right). $$

By construction, the return orbit of $z$ contains exactly one 
transition point for every time the fourth iterates of $\F$ move the orbit from one of the boxes $B_{m,n}$ to another.
Furthermore, the fourth iterates of $\F$ move points parallel to the flow within each box,
so that, per revolution, there is exactly one transition point per box that the orbit intersects.
By lemma \ref{lemma:cP_variation}, the return orbit $\Ot(z)$ is bounded
between the polygons $\Pi(w)^{\pm}$ for sufficiently small $\lambda$.
Hence the number of boxes intersected in any one revolution 
around the origin cannot exceed the number of sides of $\Pi(w)^{+}$:
 $$ \# \left(\Ot(z)\,\cap\,\Lambda \right) \leq 4\left( 2\Bfl{\sqrt{\cP(w)+ 1}}+1 \right). $$

\medskip

Now we consider the total number of points in the return orbit $\Ot(z)$.
Since the perturbed orbit is bounded below by the integrable orbit $\Pi(w)^{-}$, it must contain a point $\xi$,  
close to the positive $x$-axis, with $x$-coordinate not less than $x^-$.
Similarly for the negative $x$-axis.
The return orbit moves between neighbouring points via the action of $\F^4$, 
i.e., by translations of the vector field $\bfv$. 
If $\xi=\lambda(x,y)\in\Ot(z)$, then for sufficiently small $\lambda$, equations (\ref{eq:v_abcd}) and 
(\ref{eq:bd_ac_sets}) from the proof above can be combined to give
\begin{align*}
 \| \bfv(\xi) \| 
 &\leq \lambda \sqrt{(|2\fl{\lambda y}+1| +2)^2 + (|2\fl{\lambda x}+1| +1)^2} \\
 &\leq \lambda \sqrt{(2|\fl{\lambda y}|+3)^2 + (2|\fl{\lambda x}|+2)^2} \\
 &< \lambda \sqrt{(2|\lambda y|+5)^2 + (2|\lambda x|+4)^2} \\
 &< \lambda \sqrt{2}(2x^+ +5),
\end{align*}
where $\|\xi\|_{\infty}\leq x^+$ because the orbit is bounded above by $\Pi(w)^{+}$.
Hence, for sufficiently small $\lambda$, the number of points in the orbit
is bounded below by the distance $4x^-$ divided by the maximal length of $\bfv$ along the orbit:
\begin{displaymath}
 \# \Ot(z) \geq \frac{4x^-}{\lambda \sqrt{2}(2x^+ +5)}.
\end{displaymath}

Thus, as $\lambda\to 0$, we have the estimate
 \begin{align*}
  \mu_2(w,\lambda) &= 1 - \frac{ \# \{\xi\in  \Ot(z) \, : \; 
   \bfv(\xi) \neq \lambda\bfw(\xi) \} }{\# \Ot(z)}, \\
 &\geq 1 - \frac{ \# \left(\Ot(z)\,\cap\,\Lambda \right) }{\# \Ot(z)} , \\
  &\geq 1 - \lambda \,\frac{ \sqrt{2}(2x^+ +5)\left( 2\Bfl{\sqrt{\cP(w)+ 1}}+1 \right) }{x^-} , \\
 &= 1 + O(\lambda).
 \end{align*}
Since $\mu_2(w,\lambda)\leq 1$, the proof is complete.
\end{proof}

\medskip

Finally, we can prove theorem \ref{thm:Hausdorff} of page \pageref{thm:Hausdorff}.

\begin{proof}[Proof of theorem \ref{thm:Hausdorff}]
Let $w\in \R^2$ be given, and let $z=R_{\lambda}(w)$, 
so that $\Ot(z)$ is the return orbit which shadows the integrable orbit $\Pi(w)$.
By lemma \ref{lemma:cP_variation}, the variation in $\cP$ along the orbit
of $z$ is $O(\lambda)$ as $\lambda\to 0$.
Furthermore, the derivative of $\cP$ is bounded away from zero in a neighbourhood of $\Pi(w)$,
so that points in the orbit must be close to $\Pi(w)$ in the Hausdorff metric:
 $$ \forall \xi\in \Ot(z): \hskip 20pt d_H(\xi,\Pi(w)) = O(\lambda) $$
as $\lambda\to 0$.

Neighbouring points $\xi,\xi+\bfv(\xi)$ in the return orbit $\Ot(z)$ are also 
$O(\lambda)$-close as $\lambda\to 0$, so the result follows.
\end{proof}

\section{Nonlinearity} \label{sec:IntegrableReturnMap}

So far we have seen that, in the integrable limit, orbits of the rescaled 
discretised rotation $\F$ shadow orbits of the Hamiltonian flow $\varphi$. 
In particular, in corollary \ref{corollary:mu_1}, we showed that as $\lambda\rightarrow0$, 
$\F^4$ is equal to the time-$\lambda$ advance map $\phil$ of the flow almost 
everywhere in any bounded region.

We now introduce the period $\cT$ of the flow $\varphi$:
\begin{equation} \label{def:cT(z)}
 \cT(z) :\R^2 \rightarrow \R_{\geq 0} \hskip 40pt \cT(z) = \min\{t>0 \, : \; \varphi^t(z)=z\},
\end{equation}
so that the integrable counterpart $\cF$ to the discretised rotation $\F$ is given by
\begin{equation} \label{def:cF}
 \cF : \R^2 \rightarrow \R^2  \hskip 40pt \cF(z) = \varphi^{(\lambda - \cT(z))/4}(z).
\end{equation}
In accordance with corollary \ref{corollary:mu_1}, applying $\cF^4$ is equal to a time-$\lambda$ advance of the flow:
 $$ \cF^4(z) = \varphi^{(\lambda - \cT(z))}(z) = \phil(z). $$

As we did for $\F$ in section \ref{sec:Recurrence}, 
we can define a first return map for $\cF$.
The counterpart $\cX$ to the return domain $X$ is given by the set of points in the 
plane which are closer to $\Fix{G}$ than their preimages under $\phil$, 
and at least as close as their images. In this case, the set $\cX$ takes the simple form
\begin{equation} \label{eq:cX}
 \cX = \{ \varphi^{\lambda\theta}(x,x) \, : \; x\geq 0, \; \theta\in[-1/2,1/2) \}.
\end{equation}
We have the following explicit expression for the first return map.
 
\begin{proposition} \label{prop:cPhi(z)}
Let $z\in\cX$, and let $z^{\prime}$ be the first return of $z$ to $\cX$ \hl{under $\cF$}. 
Suppose that $z=\varphi^{\lambda\theta}(x,x)$ and  $z^{\prime}=\varphi^{\lambda\theta^{\prime}}(x,x)$, 
where $\theta,\theta^{\prime}\in[-1/2,1/2)$ and $x\in\R_{\geq 0}$. 
Then $\theta^{\prime}$ is related to $\theta$ via
\begin{equation} \label{eq:theta_prime}
 \theta^{\prime} \equiv \theta+\frac{1}{4}-\frac{\cT(z)}{4\lambda} \mod{1}.
\end{equation}
\end{proposition}
 
\begin{proof}
Suppose that $t$ is the return time of $z$ to $\cX$, so that $z^{\prime}=\cF^t(z)$. 
If $z=\varphi^{\lambda\theta}(x,x)$, then by the definition (\ref{def:cF}) of $\cF$ we have
\begin{displaymath}
 \cF^k(z) = \varphi^{\lambda\theta + k(\lambda - \cT(z))/4} (x,x) \hskip 40pt k\in\Z.
\end{displaymath}
Thus, by the expression (\ref{eq:cX}) for $\cX$, the $k$th iterate of $z$ under $\cF$ lies in the set $\cX$ if and only if
\begin{displaymath}
 \theta + k \left(\frac{\lambda - \cT(z)}{4\lambda}\right) + \frac{m\cT(z)}{\lambda} \in [-1/2,1/2)
\end{displaymath}
for some $m\in\Z$. The return time $t$ is the minimal $k\in\N$ for which this inclusion holds. 
Writing $k=4l+r$ for $l\in\Z_{\geq 0}$ and $0\leq r\leq 3$, it is straightforward to see that 
$k$ is minimal when $r=1$, $m=l$, and $l$ satisfies
\begin{displaymath}
 \theta + l - \frac{ \cT(z)}{4\lambda} \in [-3/4,1/4),
\end{displaymath}
i.e., when
\begin{displaymath}
 l + \Bfl{\theta + \frac{3}{4} - \frac{\cT(z)}{4\lambda}} =0.
\end{displaymath}
Thus the return time is given by
\begin{equation} \label{eq:tau(z)}
  t = 4\Bceil{ \frac{\cT(z)}{4\lambda} - \theta - \frac{3}{4} } + 1,
\end{equation}
where we have used the relation $-\fl{x}=\ceil{-x}$.

Now if $z^{\prime}=\varphi^{\lambda\theta^{\prime}}(x,x)$, where $\theta^{\prime}\in[-1/2,1/2)$, 
then by construction, we have
 $$ \lambda\theta^{\prime} \equiv \lambda\theta + t\left(\frac{\lambda - \cT(z)}{4}\right) \mod{\cT(z)}, $$
where we write $a \equiv b ~ (\mathrm{mod} ~ c)$ for real $c$ to denote that $(a-b)\in c\,\Z$. Thus it follows from the formula (\ref{eq:tau(z)}) for the return time that
\begin{align*}
 \lambda\theta^{\prime} 
  &\equiv \lambda\theta + \lambda \Bceil{ \frac{\cT(z)}{4\lambda}-\frac{3}{4}-\theta } + \frac{\lambda - \cT(z)}{4} \mod{\cT(z)} \\
  &\equiv -\lambda \Bfl{ \theta+\frac{3}{4}-\frac{\cT(z)}{4\lambda} } + \lambda\left( \theta+\frac{3}{4}-\frac{\cT(z)}{4\lambda}\right) - \frac{\lambda}{2} \mod{\cT(z)} \\
  &= \lambda \left\{ \theta+\frac{3}{4}-\frac{\cT(z)}{4\lambda} \right\} - \frac{\lambda}{2},
\end{align*}
where $\{x\}$ denotes the fractional part of $x$. Equivalently, dividing through by $\lambda$, we can write
 $$ \theta^{\prime} \equiv \theta+\frac{1}{4}-\frac{\cT(z)}{4\lambda} \mod{1}, $$
which completes the proof.
\end{proof}

It is natural to think of the first return map as a twist map on a cylinder with coordinates $\theta$ and $\cP(z)$,
\hl{where points of the form $\varphi^{-\lambda/2}(x,x)$ and $\varphi^{\lambda/2}(x,x)$ are
identified since they both lie in the same orbit of $\cF$} (we make this construction explicit later---see section \ref{sec:cylinder_coordinates}).
To understand how this twist map behaves, we need to study how the period function $\cT(z)$ varies with $\cP(z)$.
If the period function is not constant, the return map is nonlinear.

\subsection*{The period function} 

We now produce an explicit expression for the period $\cT$ of the Hamiltonian flow. 
Recall that $\Pi(z)$ denotes the orbit of $\varphi$ passing through the point $z\in\R^2$.
In this section we adapt this notation, and write $\Pi(\alpha)$ for the polygon on which $\cP$ takes the value $\alpha$:
\begin{equation*} 
 \Pi(\alpha) = \{z\in\R^2 \, : \; \cP(z) = \alpha \} \hskip 40pt \alpha\in\R_{\geq 0}.
\end{equation*}
Similarly \hl{we} overload the notation $\cT$, and write $\cT(\alpha)$ to denote the period of the flow on the polygon $\Pi(\alpha)$.

Let $e\in\cE$ be a critical number, and let $\alpha\in \cIe$, so that $\Pi(\alpha)$ belongs to the polygon class associated with $e$. 
Recall the vertex list $V(e)$ associated with this class of polygons, whose first entry, denoted $v_1$, and last entry, denoted $v_k$, are given by equations (\ref{def:v1}) and (\ref{def:vk}), respectively. 
Then we have the following expression for the period of the flow $\varphi$.

\begin{proposition} \label{prop:T(alpha)}
 Let $\alpha\geq0$, and let $v_1$ and $v_k$ be given as in (\ref{def:v1}) and (\ref{def:vk}). 
 Then the period $\cT(\alpha)$ of the Hamiltonian flow on the polygon $\Pi(\alpha)$ is given by
\begin{equation} \label{eq:cT(alpha)}
 \frac{\cT(\alpha)}{8} = \frac{P^{-1}(\alpha/2)}{2v_1+1} -2 \sum_{n=v_1+1}^{v_k} \frac{P^{-1}(\alpha - n^2)}{4n^2-1}
\end{equation}
(where if $v_1=v_k$ the sum should be understood to be empty).
\end{proposition}
\begin{proof}
Take $\alpha\geq0$.
By the eight-fold symmetry of the level sets of $\cP$, as stated in theorem \ref{thm:Polygons}, 
it suffices to consider the intersection of the polygon $\Pi(\alpha)$ with the first octant. 
The point $(y,y)\in\Pi(\alpha)$ where the polygon intersects the positive half of the symmetry line $\Fix{G}$
satisfies 
 $$ y = P^{-1}(\alpha/2) \hskip 40pt \fl{y}=v_1, $$
whereas the point $(x,0)\in\Pi(\alpha)$ where the polygon intersects the positive $x$-axis satisfies 
 $$ x = P^{-1}(\alpha) \hskip 40pt \fl{x}=v_k. $$
We consider the time taken to flow between these two points.

To proceed, we partition the $y$-distance between the two points into a sequence of distances $d(n)$,
and corresponding flow-times $t(n)$. If $v_1=v_k$ the partition consists of a single element; 
we simply let $d(v_1) = P^{-1}(\alpha/2)$, and $t(v_1)$ is the time 
taken for the flow on $\Pi(\alpha)$ to move between the symmetry lines $x=y$ 
and $y=0$.

Suppose now that $v_1<v_k$.
Note that the $y$-coordinate of the vertex of $\Pi(\alpha)$ which lies in the first octant and has $x$-coordinate $x=n$ is given by
 $$ P^{-1}(\alpha - n^2). $$
(Such vertices do not exist if $v_1=v_k$.)
Then we let $d(v_1)$ denote the $y$-distance between the points where $\Pi(\alpha)$ 
intersects the symmetry line $x=y$ and the line $x=v_1+1$:
 $$ d(v_1) = P^{-1}(\alpha/2) - P^{-1}(\alpha - (v_1+1)^2), $$
and $d(v_k)$ denote the $y$-distance between the points where $\Pi(\alpha)$ 
intersects the line $x=v_k$ and the symmetry line $y=0$:
 $$ d(v_k) = P^{-1}(\alpha - v_k^2). $$
For any $n$ with $v_1+1\leq n \leq v_k-1$, 
$d(n)$ is the $y$-distance between the points where $\Pi(\alpha)$ 
intersects the lines $x=n$ and $x=n+1$:
 $$ d(n) = P^{-1}(\alpha - n^2) -P^{-1}(\alpha - (n+1)^2). $$
Similarly, $t(v_1)$ is the time taken for the flow on $\Pi(\alpha)$ 
to move between the symmetry line $x=y$ and $x=v_1+1$, $t(v_1)$ 
is the time taken to move between the line $x=v_k$ and the symmetry 
line $y=0$, and for $v_1+1\leq n \leq v_k-1$, $t(n)$ is the time 
taken to flow between the lines $x=n$ and $x=n+1$.

With this notation, and using the symmetry of the flow $\varphi$, 
the period $\cT(\alpha)$ satisfies
\begin{equation} \label{eq:T(r)_sum_t(n)}
 \frac{\cT(\alpha)}{8} = \sum_{n=v_1}^{v_k} t(n).
\end{equation}

For any $n\geq 0$, the auxiliary vector field has constant $y$-component 
between the lines $x=n$ and $x=n+1$, given by $-(2n+1)$ (see equation (\ref{def:w})).
Hence the times $t(n)$ and the distances $d(n)$ are related by
\begin{equation} \label{eq:t(n)}
 t(n) = \frac{d(n)}{2n+1} \hskip 40pt v_1\leq n\leq v_k.
\end{equation}

If $v_1=v_k$ then the result follows from the definition of $d(v_1)$. 
If $v_1<v_k$, substituting (\ref{eq:t(n)}) and the definition of the $d(n)$ into (\ref{eq:T(r)_sum_t(n)}) gives:
\begin{align*}
 \frac{\cT(\alpha)}{8} &= \sum_{n=v_1}^{v_k} \frac{d(n)}{2n+1} \\
 &= \frac{P^{-1}(\alpha/2)}{2v_1+1} + \sum_{n=v_1+1}^{v_k} \frac{P^{-1}(\alpha - n^2)}{2n+1}
 - \sum_{n=v_1}^{v_k-1} \frac{P^{-1}(\alpha - (n+1)^2)}{2n+1} \\
&= \frac{P^{-1}(\alpha/2)}{2v_1+1} -2 \sum_{n=v_1+1}^{v_k} \frac{P^{-1}(\alpha - n^2)}{4n^2-1},
\end{align*}
as required.
\end{proof}

Now we consider the derivative of the period function. For any integer $n$, the function $P^{-1}$, defined on $\R_{\geq0}$ in equation (\ref{def:Pinv}), satisfies $P^{-1}(n^2) = n$ and is affine on the interval $[n^2,(n+1)^2]$.
Hence $P^{-1}$ is differentiable at every $x$ which is not a perfect square, with:
 $$ \frac{dP^{-1}(x)}{dx} = \frac{1}{2\fl{\sqrt{x}}+1}, \hskip 40pt \sqrt{x} \notin\Z. $$
Letting $x=\alpha - n^2$, we have that the derivative exists whenever $\alpha$ cannot be expressed as a sum of squares.
Considered as a function of $\alpha$, $P^{-1}(\alpha - n^2)$ is constant on the intervals $\cIe$:
 $$ \frac{dP^{-1}(\alpha - n^2)}{dx} = \frac{1}{2\fl{\sqrt{e - n^2}}+1} \hskip 40pt \alpha\in\cIe. $$
Thus if $\alpha\in\cIe$ is non-critical, then $\cT(\alpha)$ is differentiable and
\begin{equation} \label{eq:Tprime(alpha)}
 \frac{\cT^{\prime}(\alpha)}{4} = \frac{1}{(2v_1+1)^2} -4 \sum_{n=v_1+1}^{v_k} \frac{1}{(4n^2-1)(2\fl{\sqrt{e - n^2}}+1)},
\end{equation}
which is a function of $e$ only. Hence the period function $\cT$ is piecewise-affine, with constant derivative on each of the intervals $\cIe$.

\section{The limits $\lambda\to \pm1$}\label{sec:LimitsPm1}

The limit $\lambda\to 0$ is not the only limit which describes an 
approach to an exact rotation (where the round-off has no effect). 
The limits $\lambda\to \pm 1$, corresponding to $\nu\to 1/6,1/3$, 
describe the approach to the cases where $F^6=\id$ and $F^3=\id$, respectively. 
It is also possible to describe these limiting dynamics with a piecewise-affine Hamiltonian, 
which we introduce briefly here\footnote{The limit $\lambda\to 1$ has also been investigated in \cite{Siu}.}. 

We analyse the limits $\lambda\to \pm 1$ by defining $\delta = \lambda \mp 1$ 
and letting $\delta\to 0$. Then the appropriate rescaling of $F$ is given by 
the map
\begin{equation*}
F_{\delta}: (\delta\Z)^2 \to (\delta\Z)^2 
 \hskip 40pt
 F_{\delta}(z)=\delta F(z/\delta)
\end{equation*}
(as for $\lambda$, we assume that $\delta>0$).
Correspondingly, the discrete vector fields $\bfv_{\pm}$ become
\begin{equation*}
 \bfv_{\pm}: \; (\delta\Z)^2 \to (\delta\Z)^2
\hskip 40pt
\bfv_{+}(z) = F_{\delta}^6(z)-z
\hskip 40pt
\bfv_{-}(z) = F_{\delta}^3(z)-z,
\end{equation*}
 and the auxiliary vector fields are given by
\begin{align*}
  \bfw_{+}(x,y)=2(\fl{y} -\fl{x-y},-(\fl{x} -\fl{y-x})), \\
  \bfw_{-}(x,y)=(\fl{y}+\fl{x+y}+1,-(\fl{x}+\fl{x+y}+1)).
\end{align*}
This time, $\bfw_{\pm}$ are constant on the collection of triangles 
\begin{align*}
 B_{m,n} &= \{ (x,y)\in\R^2 \, : \; \fl{ x } =m, \; \fl{ y } = n, \; \fl{ x \mp y }  = m \mp n\}, \\
 T_{m,n} &= \{ (x,y)\in\R^2 \, : \; \fl{ x } =m, \; \fl{ y } = n, \; \fl{ x \mp y }  = m \mp n \mp 1 \},
\end{align*}
where $m,n\in\Z$.

\begin{figure}[t]
        \centering
        \begin{minipage}{7cm}
          \centering
	  \includegraphics[scale=0.35]{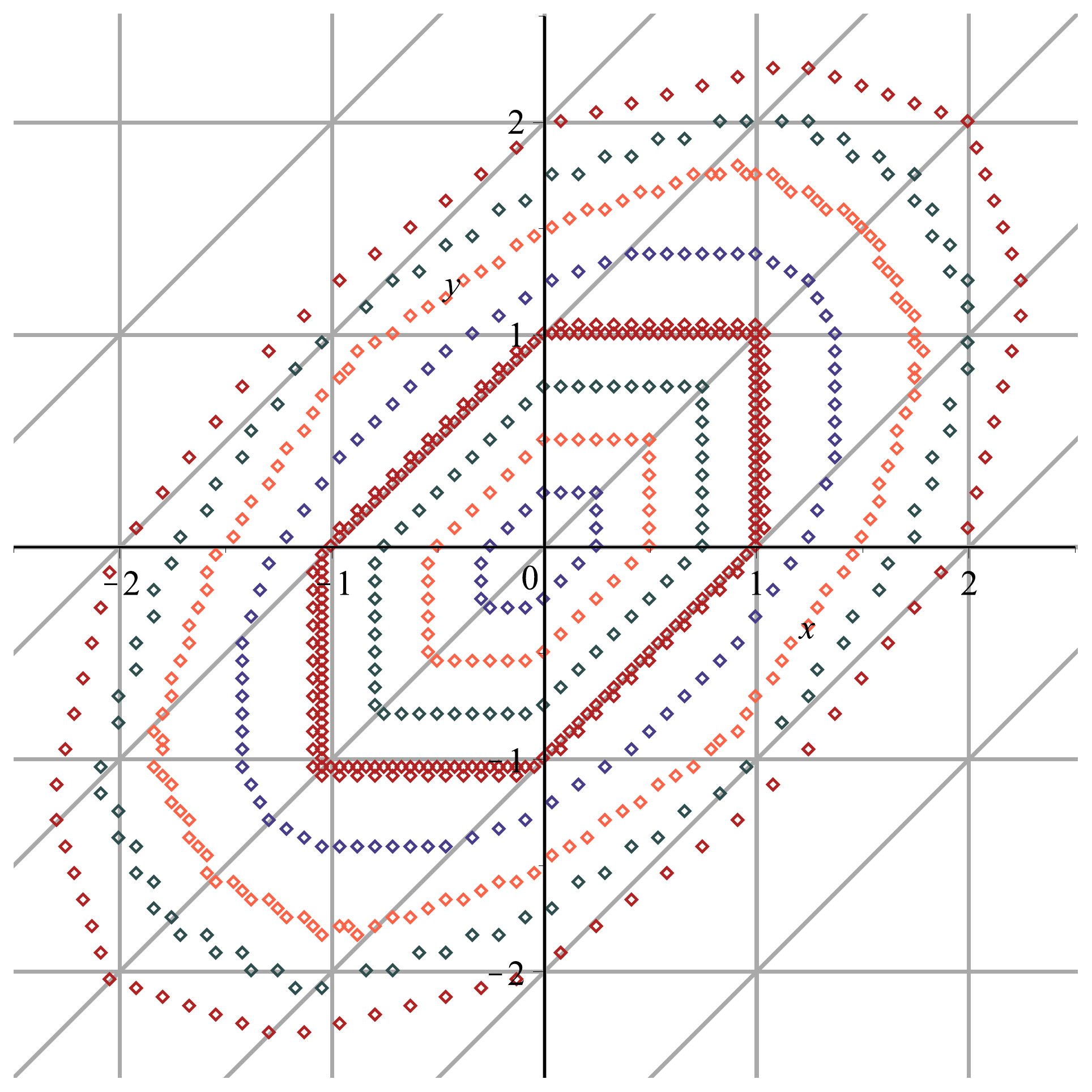} \\
	  (a) $\; \lambda=1+1/24$ \\
        \end{minipage}
        \quad
        \begin{minipage}{7cm}
	  \centering
	  \includegraphics[scale=0.35]{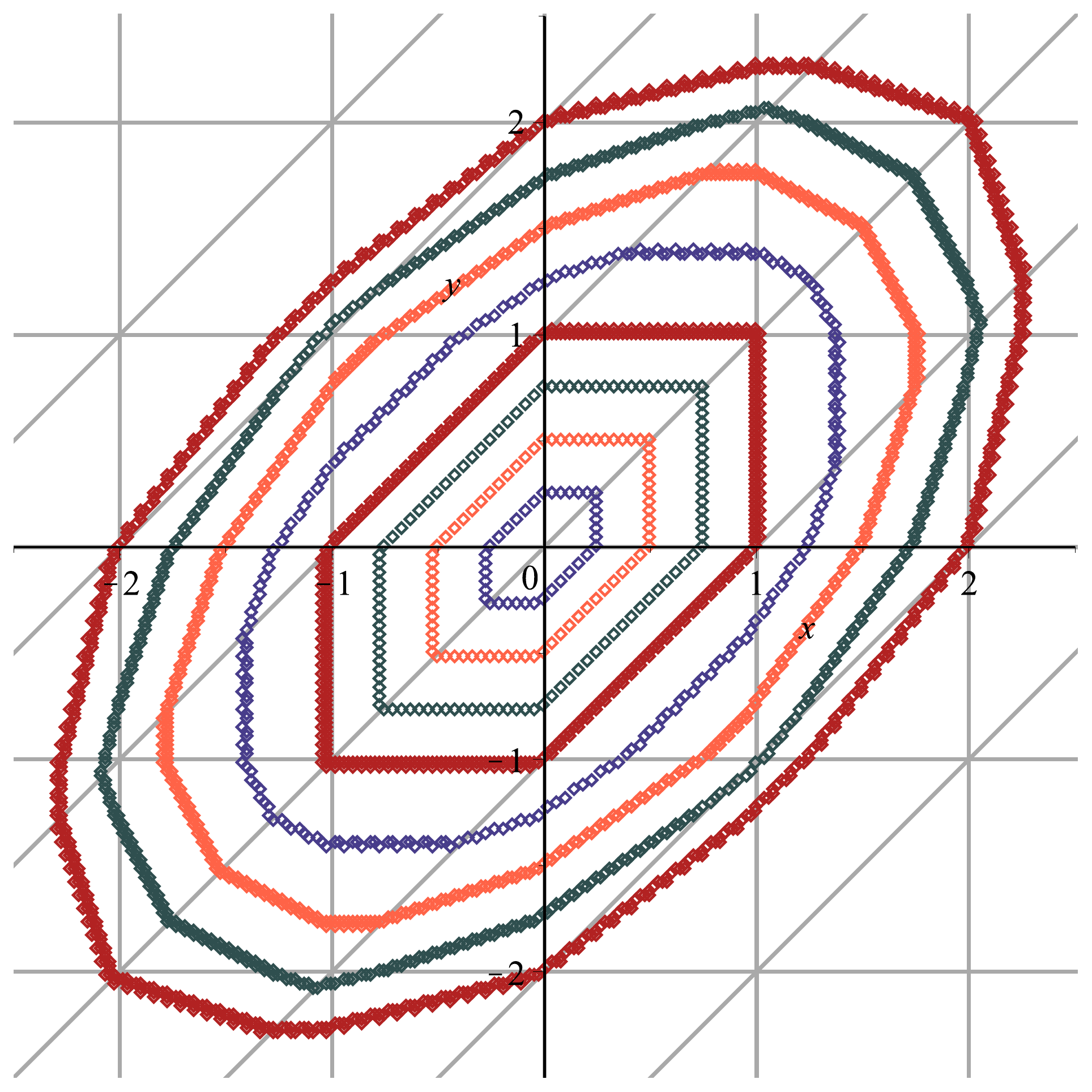} \\
	  (b) $\; \lambda=1+1/48$ \\
	\end{minipage}
        \caption{\hl{A selection of periodic orbits of the rescaled map $F_{\delta}$ in the limit $\lambda\to 1$,
        where $\nu\to 1/6$ (cf. the $\lambda\to 0$ case of figure \ref{fig:PolygonalOrbits}, page \pageref{fig:PolygonalOrbits}). 
        The grey lines show the discontinuity set $\Delta$.}}
        \label{fig:PolygonalOrbits2}
\end{figure}

As in proposition \ref{prop:mu_1}, if we ignore a subset of the lattice the lattice $(\delta\Z)^2$ of zero density, 
then the functions $\bfv$ and $\bfw$ satisfy $\bfv(z)=\delta\bfw(z)$.

Recall the piecewise-affine function $P$ defined in (\ref{def:P}). 
The Hamiltonians corresponding to the limits $\lambda\to \pm 1$ 
are given by
\begin{equation*}
 \cP: \; \R^2 \; \to \R
 \hskip 40pt
 \cP(x,y) = \left(\frac{1}{2}\right)^{(1\mp1)/2} \left( P(x)+P(x\mp y) + P(y) \right).
\end{equation*}
These Hamiltonians are again piecewise-affine and differentiable 
in $\R^2\setminus \Delta$, where $\Delta$ is the set of lines given by
\begin{equation*}
\Delta=\{(x,y)\in\R^2 \, : \; (x-\fl{ x})(y-\fl{ y})(x\mp y - \fl{ x\mp y})=0\},
\end{equation*}
the boundaries of the triangles $B_{m,n}$ and $T_{m,n}$.
One can easily verify that the Hamiltonian vector fields associated with $\cP$ 
are equal to the auxiliary vector fields $\bfw_{\pm}$ wherever they are defined.

%% file: PerturbedDynamics.tex
\chapter{The perturbed dynamics}\label{chap:PerturbedDynamics}

In chapter \ref{chap:IntegrableLimit}, the positive real line was partitioned 
into the sequence of critical intervals $\cIe$, $e\in\cE$; 
accordingly, the set $\Gamma$ of critical polygons partitioned the plane into 
the sequence of polygon classes (section \ref{sec:Hamiltonian}).
We defined a map $\Phi$, 
corresponding to the first return of $F_\lambda$ to a thin strip $X$
placed along the symmetry axis, and showed that the return orbits shadow the integrable orbits (section \ref{sec:Recurrence}).

In this chapter we partition the set $X$ into sub-domains, which play the same 
role for the perturbed orbits as the polygon classes for the integrable orbits.
Within each sub-domain, the perturbed orbits have local symmetry properties
which are $\lambda$-independent (up to scale): 
$\Phi$ commutes with translations by the elements of a two-dimensional lattice.
This lattice structure arises from the cylinder sets of a symbolic coding:
an extension of the coding of the polygon classes.

Furthermore, we identify sub-domains where the fraction of minimal orbits---the
fixed points of the return map---is also $\lambda$-independent.
Thus we show that the minimal orbits occupy a positive density of the phase space as $\lambda\to 0$.

The work in this chapter has been published in \cite{ReeveBlackVivaldi}.

\section{Regular domains} \label{sec:RegularDomains}

The return orbits $\Ot(z)$ of the perturbed dynamics 
shadow the orbits $\Pi(z)$ of the integrable Hamiltonian $\cP$, 
as we saw in theorem \ref{thm:Hausdorff} (page \pageref{thm:Hausdorff}).
Hence the polygon classes provide a natural partition of the set $X$ into the sequence of sets
 $$ \cP^{-1}(\cIe) \cap X \hskip 40pt e\in\cE. $$
However, the quantity $\cP$ is not constant along perturbed orbits. 
If we define a symbolic coding on perturbed orbits, 
it does not follow that the return orbit of some point $z\in X$ with $\cP(z)\in \cIe$ 
has the same symbolic coding as the polygon class associated with $\cIe$: 
perturbed orbits which start close to a critical polygon are likely to
wander between polygon classes.
To deal with this problem, it is necessary to replace the above sequence of sets by a sequence 
of smaller \textbf{regular domains}, and then prove that, in the limit, these 
domains still have full density in $X$.

We start by defining the \textbf{edges} of $\Ot(z)$
as the non-empty sets of the form
\begin{displaymath}
 B_{m,n}\cap \Ot(z) \hskip 40pt m,n\in\Z.
\end{displaymath}
For sufficiently small $\lambda$, consecutive edges 
of $\Ot(z)$ must lie in adjacent boxes, 
and transitions between edges occur when the orbit meets the set 
$\Lambda$, defined in equation (\ref{eq:Lambda}). 
Thus we call the set $\Ot(z)\,\cap\,\Lambda$ the set of \textbf{vertices} of $\Ot(z)$.
By analogy with the vertices of the polygons, we say that the return orbit 
$\Ot(z)$ has a vertex on $x=m$ of \textbf{type} $v$ 
if there exists a point $w\in\Ot(z)$ such that
\begin{displaymath}
 w\in B_{m,v} \cap F_{\lambda}^{-4}(B_{m-1,v}) 
  \hskip 20pt \mbox{or} \hskip 20pt w\in B_{m-1,v} \cap F_{\lambda}^{-4}(B_{m,v}).
\end{displaymath}
Similarly for a vertex on $y=n$ of type $v$.
A perturbed orbit is \textbf{critical} if it has a vertex whose type is undefined, 
i.e., if there exists $w\in\Ot(z)$ such that
\begin{displaymath}
 w\in B_{m,n} \cap F_{\lambda}^{-4}(B_{m\pm1,n\pm1})
\end{displaymath}
for some $m,n\in\Z$ (see figure \ref{fig:critical_vertex}).

\begin{figure}[t]
        \centering
        \includegraphics[scale=0.75]{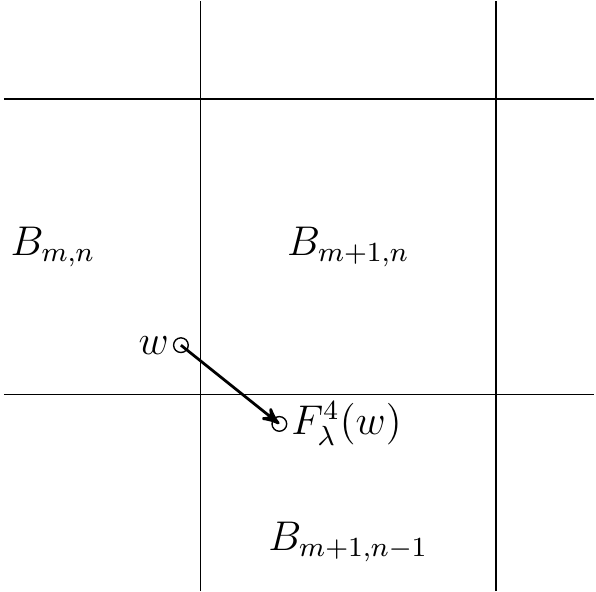}
        \caption{A critical vertex $w\in B_{m,n} \cap F_{\lambda}^{-4}(B_{m+1,n-1})$.}
        \label{fig:critical_vertex}
\end{figure}

By excluding points whose perturbed orbit is critical, 
we will construct a sequence of subsets $X^e$ of $X$ with the property that, 
for all $z\in\Xe$, the orbit $\Ot(z)$ has the same sequence of vertex types as $\Pi(z)$. 

We now give the construction of $\Xe$. 
Let the set $\Sigma\subset\Lambda$ be given by
\begin{equation}\label{eq:Sigma}
 \Sigma = \bigcup_{m,n\in\Z}\Sigma_{m,n},
\end{equation}
where
\begin{equation}\label{eq:Sigma_mn}
 \Sigma_{m,n} = \{ z\in\Lambda \, : \; \|z - (m,n)\|_{\infty} \leq \lambda( \, \|\bfw_{m,n}\|_{\infty}+2) \, \}
\end{equation}
and $\|(u,v)\|_{\infty} = \max(|u|,|v|)$. The set $\Sigma_{m,n}$ is a small domain, adjacent
to the integer point $(m,n)$ (see figure \ref{fig:LambdaSigma_plot}, page \pageref{fig:LambdaSigma_plot}).

\label{def:regular}
If $z\in X$ and $\cP(z)\in\cIe$ for some $e\in\cE$, we say that $z$ is \textbf{regular} if three properties hold:
\begin{enumerate}[(i)]
\item neither $z$ nor $\Phi(z)$ are vertices of $\Ot(z)$, i.e.,
$ \{z,\Phi(z)\}\subset X\setminus \Lambda$;
\item the return orbit $\Ot(z)$ is contained in the polygon class associated with $e$, i.e.,
$ \cP(\Ot(z)) \subset \cIe$;
\item the return orbit $\Ot(z)$ does not intersect the set $\Sigma$.
\end{enumerate}
Points which are not regular are called \textbf{irregular}. Then the set $\Xe$ is defined as
 $$ \Xe = \{ z\in X \, : \; \cP(z) \in \Ie(\lambda) \}, $$
where $\Ie(\lambda)\subset \cIe$ is the largest interval such that all points in 
$\Xe$ are regular.

In fact, we can give a more explicit expression for $\Xe$. 
Let $v_1$ be the first entry in the vertex list $V(e)$, as defined in (\ref{def:v1}). 
By construction, $\Xe$ does not intersect $\Lambda$, so that $\Xe\subset B_{v_1,v_1}\setminus\Lambda$. 
Then, for sufficiently small $\lambda$, lemma \ref{lemma:Lambda} gives us that the auxiliary vector field $\bfw$
matches the discrete vector field $\bfv$ everywhere in $\Xe$:
\begin{displaymath}
 \bfv(z) = \lambda\bfw(z) = \lambda\bfw_{v_1,v_1} \hskip 40pt z\in \Xe,
\end{displaymath}
\hl{so that the map $\F^4$ acts locally as the translation $z\mapsto z+\lambda\bfw_{v_1,v_1}$.}
Applying this to the definition (\ref{def:X}) of $X$, 
it follows that if $\lambda(x,y)\in\Xe$, then 
\begin{displaymath}
 -(2v_1+1) \leq x-y < 2v_1+1.
\end{displaymath}
Hence the set $\Xe$ is the intersection of the lattice $\lZ$ with a thin rectangle lying along the symmetry line $\Fix{G}$
(see figure \ref{fig:lattice_Le}, page \pageref{fig:lattice_Le}):
\begin{equation}\label{def:Xe}
 \Xe = \{ \lambda(x,y)\in\lZ \, : \;  -(2v_1+1) \leq x-y < 2v_1+1, \; \cP(\lambda x, \lambda y)\in\Ie(\lambda) \}.
\end{equation}
\hl{Furthermore, it is natural to identify the sides of this rectangle,
which are connected by the local vector field,
so that the dynamics take place modulo the one-dimensional module 
$\langle\lambda\bfw_{v_1,v_1}\rangle$ generated by $\lambda\bfw_{v_1,v_1}$.}

In principle, the interval $\Ie(\lambda)$ need not be uniquely defined, and may be empty. 
However, the following proposition ensures that $\Ie(\lambda)$ is well-defined for all 
sufficiently small $\lambda$, and indeed that the irregular points have zero density in $X$ 
as $\lambda\rightarrow 0$.

\begin{proposition} \label{prop:Ie}
If $e\in\cE$ and $\Ie(\lambda)$ is as above, then
 \begin{displaymath}
  \lim_{\lambda\rightarrow 0}\frac{ |\Ie(\lambda)| }{|\cIe|} = 1.
 \end{displaymath}
\end{proposition}

\begin{proof}
Let $e\in\cE$. Consider $z\in X$ such that $\cP(z)\in \cIe$, and suppose that $z$ is irregular.
We will show that $\cP(z)$ must be $O(\lambda)$-close to the boundary of $\cIe$.

If the orbit of $z$ strays between polygon classes, i.e., if condition (ii) of regularity fails, 
then we have
\begin{displaymath}
 \exists \, w\in \Ot(z): \hskip 20pt \cP(w) \notin \cIe.
\end{displaymath}
However, in lemma \ref{lemma:cP_variation} of section \ref{sec:Recurrence} (page \pageref{lemma:cP_variation}), 
we showed that the maximum variation in $\cP$ along an orbit 
$\Ot(z)$ is of order $\lambda$ as $\lambda\rightarrow 0$:
\begin{displaymath}
 \forall z\in X,\, \forall w\in \Ot(z): 
\hskip 20pt \cP(w) - \cP(z) = O(\lambda).
\end{displaymath}
Hence if $\cIe=(e,f)$, where $f$ is the successor of $e$ in the sequence $\cE$, then
\begin{equation}\label{eq:Pvariation}
 \cP(z)=\cP(w)+O(\lambda)=
\left\{\begin{array}{ll} 
 e+O(\lambda) & \quad \cP(w)\leq e \\
 f+O(\lambda) & \quad \cP(w)\geq f. \\
\end{array}\right.
\end{equation} 
In both cases, $\cP(z)$ is near the boundary of $\cIe$.

If condition (ii) holds but $z\in\Lambda$, then one of 
its coordinates must be nearly integer:
$$ 
 d_H(z,\,\Delta) = O(\lambda),
$$
where the set $\Delta$ was defined in (\ref{eq:Delta}).
However, as the domain $X$ lies in an $O(\lambda)$-neighbourhood of the 
symmetry line $\Fix{G}$, it follows that both coordinates must be nearly 
integer, and $z$ must be close to a critical polygon:
\begin{displaymath}
 d_H(z,\,\Z^2) = O(\lambda).
\end{displaymath}
Again, it follows that $\cP(z)$ lies in a $O(\lambda)$-neighbourhood of the boundary of $\cIe$.
A similar argument holds if $\Phi(z)\in\Lambda$.

Finally, if (ii) holds but (iii) fails, i.e., if there is a point $w\in\Ot(z)\cap \Sigma$,
then by construction
\begin{displaymath}
 d_H(w,\,\Z^2) = O(\lambda),
\end{displaymath}
and (\ref{eq:Pvariation}) applies as before.

Combining these observations, we have
$$
  \frac{ |\Ie(\lambda)| }{|\cIe|} = 1 - \frac{ |\cIe\setminus\Ie(\lambda)| }{|\cIe|}
 = 1 - \frac{ O(\lambda) }{|\cIe|},
$$
and the result follows.
\end{proof}

\medskip

Now we show that the sequence of sets $\Xe$ fulfil their objective, 
which was to exclude all points $z\in X$ whose perturbed orbit is critical 
in the sense defined above.

\begin{proposition} \label{prop:Xe}
If $e\in\cE$ and $z\in\Xe$, then the perturbed orbit of $z$ is not critical.
\end{proposition}
\begin{proof}
Let $e\in\cE$ and $z\in X$ with $\cP(z)\in\cIe$. 
Suppose that the return orbit of $z$ is critical, i.e., 
suppose there exists $w\in\Ot(z)$ such that
\begin{displaymath}
 w\in B_{m,n} \hskip 20pt \mbox{and} 
  \hskip 20pt F_{\lambda}^{4}(w)=w+\bfv(w)\in B_{m\pm1,n\pm1}
\end{displaymath}
for some $m,n\in\Z$. We will show that $z\notin\Xe$. 
For simplicity, we assume that $m$ and $n$ are both non-negative, so that 
by the orientation of the vector field in the first quadrant:
\begin{displaymath}
 \F^{4}(w)=w+\bfv(w)\in B_{m+1,n-1}.
\end{displaymath}

Recall the proof of lemma \ref{lemma:Lambda} (page \pageref{lemma:Lambda}), 
where we calculated the explicit form of the discrete vector field $\bfv$.
If $w=\lambda(x,y)\in B_{m,n}$, then by construction:
 $$ m=\fl{\lambda x} \hskip 40pt n=\fl{\lambda y}. $$
Furthermore, if $\F^{4}(w)\in B_{m+1,n-1}$, then the last line of equation (\ref{eq:Fabcd}) implies that
 $$ (d,c) = (m+1,n-1), $$
where the integers $c$ and $d$ are given by (\ref{eq:abcd}).
It follows that the perpendicular distance from $w$ to the lines $x=m+1$ and $y=n$, 
respectively, is bounded according to
\begin{align*}
-\lambda(a+n) = -\lambda(a+c+1) &\leq \lambda x - (m+1) < 0 \\
0 &\leq \lambda y - n < \lambda(m+b+1),
\end{align*}
where again the integers $a$ and $b$ are given by (\ref{eq:abcd}).
Combining this observation with the constraint (\ref{eq:bd_ac_sets}) on $a$ and $b$ gives
\begin{align*}
\| w-(m+1,n) \|_{\infty} &\leq \lambda\max( |a+n|, |m+b+1|) \\
 &\leq \lambda\max( 2n + |n-a|, 2m+1 + |m-b|) \\
 &\leq \lambda\max( 2n+1, 2m+2) \\
 &\leq \lambda \|\bfw_{m+1,n}\|_{\infty},
\end{align*}
where we have used the fact that $m$ and $n$ are non-negative.
Hence, by definition (\ref{eq:Sigma_mn}), $w\in\Sigma_{m+1,n}$ and $z$ is irregular, so $z\notin \Xe$. 
The cases where $m$ or $n$ are negative proceed similarly.
\end{proof}

\section{Lattice structure and orbits of minimal period} \label{sec:MainTheorems} \label{SEC:MAINTHEOREMS}

We turn our attention to the reversing symmetry group of the return map $\Phi$. 
In section \ref{sec:time-reversal} we introduced the reversing symmetry $G$ 
of the lattice map $\F$ (see equation (\ref{def:GH}), page \pageref{def:GH}). 
Since $\Phi$ is a return map of $\F$, it has an associated reversing symmetry. 
In the following proposition we describe the form of this reversing symmetry 
on the domains $\Xe$ (equation (\ref{def:Xe})).

\begin{proposition} \label{prop:Ge}
 For $e\in\cE$, let $G^e$ be the involution of $\Xe$ given by 
 \begin{equation} \label{eq:Ge}
  G^e(x,y) = \left\{ 
      \begin{array}{ll} (y,x) \quad & |x-y| < \lambda(2v_1+1) \\
		      (x,y) \quad & x-y = -\lambda(2v_1+1),
      \end{array} \right.
 \end{equation}
where $v_1 = \fl{\sqrt{e/2}}$ is the first entry of the vertex list associated with $\cIe$.
Then for all sufficiently small $\lambda$, $G^e$ is a reversing symmetry of $\Phi$ on $\Xe$ in the following sense:
\begin{equation*}
 \forall z,\Phi(z)\in\Xe:
 \hskip 20pt 
 \Phi^{-1}(z) = (G^e \circ \Phi \circ G^e)(z).
\end{equation*}
\end{proposition}

\begin{proof}
Recall that $\Xe\subset B_{v_1,v_1}\setminus\Lambda$, so that if $z\in\Xe$, then
 $$ \F^4(z) = z + \lambda\bfw(z) = z + \lambda\bfw_{v_1,v_1}. $$
Furthermore, if $z$ lies on the line $x-y=-\lambda(2v_1+1)$, 
then by the definition (\ref{def:w_mn}) of the auxiliary vector field, we have
 $$ G(z) = z + \lambda\bfw_{v_1,v_1}. $$
Combining the above, we have the following relationship between $G^e$ and $G$:
\begin{equation} \label{eq:Ge_G}
 G^e(x,y) = \left\{ 
      \begin{array}{ll} G(x,y) \quad & |x-y| < \lambda(2v_1+1), \\
		      (\F^{-4}\circ G)(x,y) \quad & x-y = -\lambda(2v_1+1).
      \end{array} \right.
\end{equation}
 
If $\tau=\tau(z)$ is the return time of $z$, then by construction
  $$ \Phi(z) = \F^{\tau}(z), $$
and the reversibility of $\F$ with respect to $G$ gives us that
\begin{equation} \label{eq:Phi_reversibility}
 (G \circ \Phi)(z) = (\F^{-\tau} \circ G)(z).
\end{equation}

Suppose now that neither $z$ nor $\Phi(z)$ lies on the line $x-y=-\lambda(2v_1+1)$.
Then combining (\ref{eq:Ge_G}) and (\ref{eq:Phi_reversibility}) gives
 $$ (G^e \circ \Phi)(z) = (\F^{-\tau} \circ G^e)(z). $$
Furthermore, this point lies in $\Xe$, so that
 $$ (G^e \circ \Phi)(z) = (\Phi^{-1} \circ G^e)(z), $$
as required.

Suppose now that $\Phi(z)$ lies on the line $x-y=-\lambda(2v_1+1)$ but $z$ does not.
In this case, combining (\ref{eq:Ge_G}) and (\ref{eq:Phi_reversibility}) gives
 $$ (G^e \circ \Phi)(z) = (\F^{-4} \circ G \circ \Phi)(z) = (\F^{-(\tau+4)} \circ G)(z) = (\Phi^{-1} \circ G^e)(z). $$
The other cases proceed similarly.
\end{proof}

\hl{Note the reversing symmetry $G^e$ of $\Phi$ is simply an adaptation of the original 
reversing symmetry $G$ of $\F$, which accounts for the natural cylindrical topology of its domain $\Xe$:}
 $$ G^e(z) = G(z) \mod{\lambda\bfw_{v_1,v_1}} \hskip 40pt z\in\Xe, $$
\hl{where we write (mod $\mathbf{a}$) for some vector $\mathbf{a}$ to 
denote congruence modulo the one-dimensional module $\langle\mathbf{a}\rangle$ 
generated by $\mathbf{a}$.}

\medskip

The return map $\Phi$ also has non-trivial symmetries. 
We define a sequence of lattices $\Le\subset\Z^2$, $e\in\cE$, independent of $\lambda$, 
such that within the domain $\Xe$, 
the return map $\Phi$ is equivariant under the group of 
translations generated by $\lambda\Le$. The construction of $\Le$ is as follows.

\begin{figure}[t]
	\centering
        \includegraphics[scale=0.85]{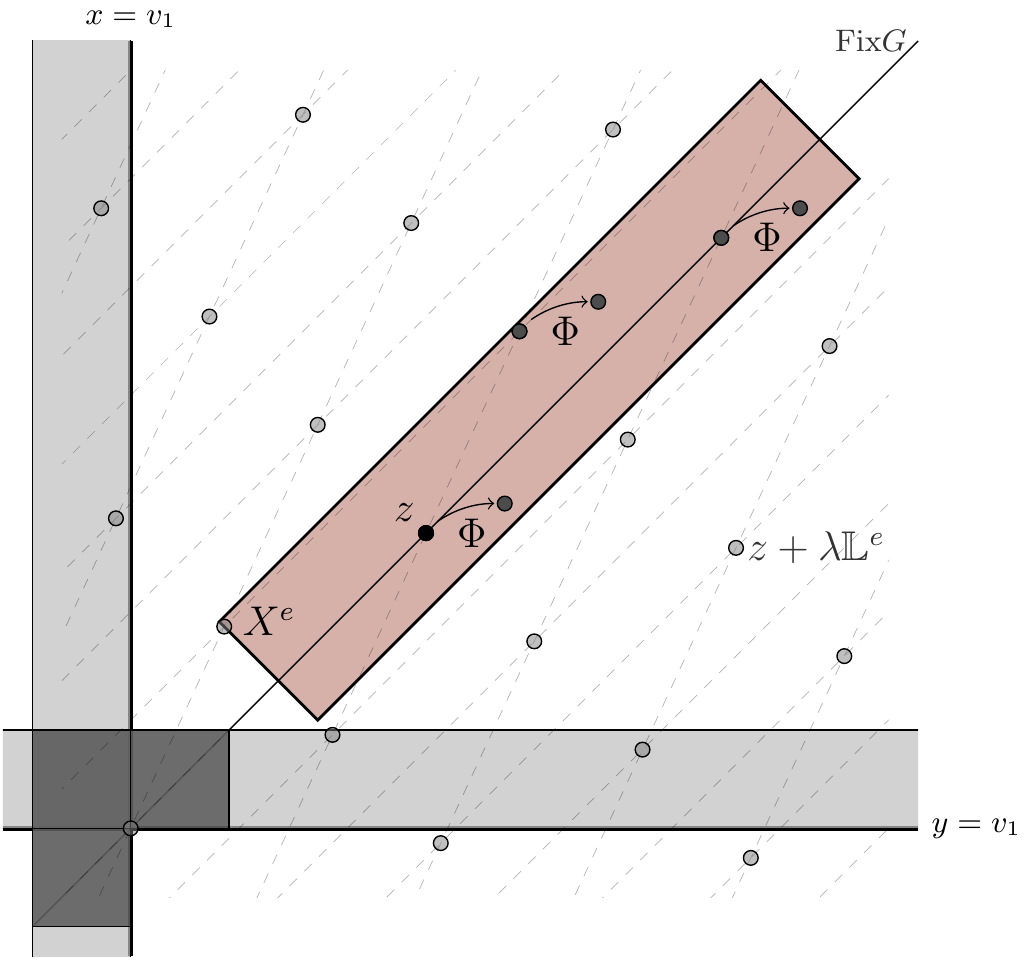}
	\caption{\hl{The set $\Xe$ (pink) is the subset of the Poincar\'{e} section $X$ whose orbits 
	under the return map $\Phi$ can be associated with the polygon class indexed by $e\in\cE$.
	The lattice $z+\lambda\Le$, for some $z\in\Xe$, is illustrated by the dashed lines:
	all points which are congruent to $z$ modulo $\lambda\Le$ exhibit the same dynamical behaviour.
	The sets $\Lambda$ and $\Sigma$ are represented in light and dark grey, respectively.}}
	\label{fig:lattice_Le}
\end{figure}

For $e\in\cE$, suppose the vertex list $V(e)=(v_1,\dots,v_k)$ 
contains $l$ distinct entries. We define the sequence 
$(\iota(j))_{1\leq j\leq l}$ such that the $\iota(j)$th entry in the vertex 
list is the $j$th distinct entry. Since all repeated entries are consecutive, 
it follows that the vertex list has the form
\begin{equation} \label{eq:v_iota}
 V(e) = (v_{\iota(1)},\dots,v_{\iota(1)},v_{\iota(2)},\dots, v_{\iota(2)}, \, 
     \dots \, ,v_{\iota(l)},\dots, v_{\iota(l)}), 
\end{equation}
with $v_{\iota(1)} = v_1$ and $v_{\iota(l)} = v_k$.
We define the vector $\bfL=\bfL(e)$ as:
\begin{equation} \label{def:L}
 \bfL = \frac{q}{2v_1+1} \; (1,1),
\end{equation}
where the natural number $q=q(e)$ is defined as follows
\begin{equation}\label{eq:q}
 q = \lcm((2v_{\iota(1)}+1)^2,(2v_{\iota(1)}+1)(2v_{\iota(2)}+1),
      \ldots ,(2v_{\iota(l-1)}+1)(2v_{\iota(l)}+1)).
\end{equation}
Here the least common multiple runs over $(2v_1+1)^2$ and all products 
of the form $(2v_j+1)(2v_{j+1}+1)$, where $v_j$ and 
$v_{j+1}$ are consecutive, distinct vertex types. 
Finally, the lattice $\Le$ is given by
\begin{equation}\label{eq:Le}
 \Le = \left\langle \bfL, \frac{1}{2}\left(\bfL-\bfw_{v_1,v_1}\right) \right\rangle,
\end{equation}
where $\langle \cdots \rangle$ denotes the $\Z$-module generated by a set 
of vectors, and the vector $\bfw_{v_1,v_1}$ given by (\ref{def:w_mn}) 
is the Hamiltonian vector field $\bfw$ in the domain $\Xe$.
We note that the vector $\bfL$ is parallel to the symmetry line $\Fix{G}$, and 
hence parallel to the domain $\Xe$, whereas the vector $\bfw_{v_1,v_1}$ is perpendicular to it.

\begin{theorem} \label{thm:Phi_equivariance}
For every $e\in\cE$, and all sufficiently small $\lambda$, 
the map $\Phi$ commutes with translations by the elements 
of $\lambda\Le$ on the domain $\Xe$:
\begin{equation} \label{eq:Phi_equivariance}
 \forall l\in\Le, \; \forall z,z+\lambda l\in\Xe:
 \hskip 20pt 
 \Phi(z + \lambda l) \equiv \Phi(z) + \lambda l \mod{\lambda\bfw_{v_1,v_1}}.
\end{equation}
\end{theorem}

There is a critical value of $\lambda$, depending on $e$, above which the statement 
of the theorem is empty, as $\Xe$ is insufficiently populated for a pair 
of points $z,z+\lambda l\in \Xe$ to exist. 
The congruence under the local (rescaled) integrable vector field 
$\lambda\bfw_{v_1,v_1}$ in equation (\ref{eq:Phi_equivariance}) 
\hl{invokes the cylindrical topology of $\Xe$, which} is necessary for 
the case that $\Phi(z) + \lambda l\notin X$.

\medskip

For certain $e\in\cE$, we can calculate the number of congruence classes of
$\lambda\Le$ corresponding to symmetric fixed points of $\Phi$, i.e., to symmetric minimal orbits.
We define the fraction of symmetric fixed points in $\Xe$:
\begin{displaymath}
 \delta(e,\lambda) = \frac{\# \{ z\in \Xe \, : \;  G(\cO(z))=\cO(z), \; \Phi(z)=z \}}{\#\Xe},
\end{displaymath}
and prove the following result on the persistence of such points in the limit 
$\lambda\rightarrow 0$.

\begin{theorem} \label{thm:minimal_densities}
Let $e\in\cE$, and let $(v_1,\dots,v_k)$ be the vertex list of the 
corresponding polygon class.
If $2v_1+1$ or $2v_k+1$ is coprime to $2v_j+1$ for all other vertex types 
$v_j$, i.e., if
\begin{equation} \label{eq:v1_k_coprimality}
\exists \; i\in\{1,k\}, \quad \forall j\in\{1,\ldots\,k\},\quad
v_j\neq v_i \,\Rightarrow \, \gcd(2v_i+1,2v_j+1) =1,
\end{equation}
then, for sufficiently small $\lambda$, the number of symmetric fixed points
of $\Phi$ in $\Xe$ modulo $\lambda\Le$ is independent of $\lambda$. 
Thus the asymptotic density of symmetric fixed points in $\Xe$ converges,
and its value is given by
\begin{equation}\label{eq:Density}
 \lim_{\lambda\rightarrow 0} \delta(e,\lambda) = \frac{1}{(2\fl{ \sqrt{e} }+1)(2\fl{ \sqrt{e/2} }+1)}.
\end{equation}
\end{theorem}

As for theorem \ref{thm:Phi_equivariance}, the smallness of $\lambda$ serves only to ensure that $\Xe$ 
is sufficiently populated for all congruence classes modulo $\lambda\Le$ to be represented. 

The condition (\ref{eq:v1_k_coprimality}) on the orbit code is clearly satisfied 
for infinitely many critical numbers $e$, e.g., those for which 
either $2\fl{\sqrt{e}}+1$ or
$2\fl{\sqrt{e/2}}+1$ is a prime number. 
The first violation occurs at $e=49$ (see table \ref{table:V(e)}, page \pageref{table:V(e)}), where $2v_1+1 = 9$ 
and $2v_k+1=15$ have a common factor.
We contrast this to the case of $e=52$, where $2v_k+1=15$ and $2v_2+1 = 9$ have a 
common factor, but $2v_1+1=11$ is prime, so the condition (\ref{eq:v1_k_coprimality}) holds.
Numerical experiments show that the density
of values of $e$ for which (\ref{eq:v1_k_coprimality}) holds decays very slowly,
reaching 1/2 for $e\approx 500,000$.

The stated condition on the orbit code is actually stronger than 
we require in the proof. This was done to simplify the formulation of the
theorem. We remark that the weaker condition is still not necessary for 
the validity of the density expression (\ref{eq:Density}). 
At the same time, there are values of $e$ for which the density of symmetric minimal 
orbits deviates from the given formula, and convergence is not guaranteed. 
Our numerical experiments show that these deviations are small, and don't seem 
connected to new dynamical phenomena. 
More significant are the fluctuations in the density of non-symmetric fixed points:
its dependence on $e$ is considerably less predictable than for symmetric orbits---see figure \ref{fig:Density}.

\begin{figure}[t]
        \centering
        \includegraphics[scale=1]{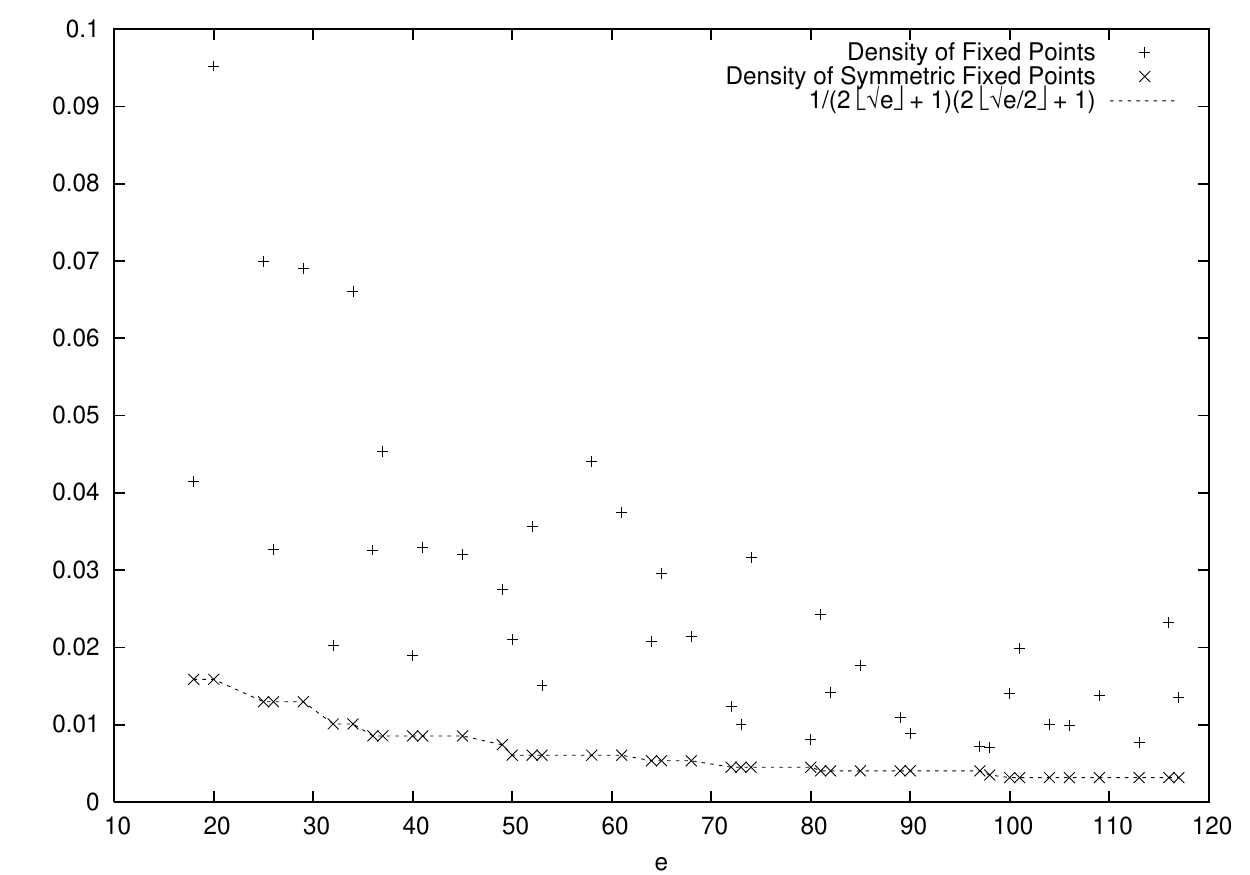}
        \caption{The density of symmetric minimal orbits, as a function 
of the critical number $e$ (calculated for suitably small values of the parameter $\lambda$). 
The solid line represents the estimate (\ref{eq:Density}).
The scattered points correspond to the density of all minimal orbits, symmetric
and non-symmetric.
}
        \label{fig:Density}
\end{figure}

The asymptotic density of symmetric fixed points in $\Xe$ provides an obvious lower bound 
for the overall density of fixed points, which we denote $\eta(e,\lambda)$:
\begin{displaymath}
 \eta(e,\lambda) = \frac{\# \{ z\in \Xe \, : \; \Phi(z)=z \}}{\# \Xe}.
\end{displaymath}

\begin{corollary} \label{cor:Density}
Let $e\in\cE$ satisfy the condition (\ref{eq:v1_k_coprimality}) of theorem \ref{thm:minimal_densities}.
Then the asymptotic density of fixed points in $\Xe$ is bounded below as follows:
 $$\liminf_{\lambda\rightarrow 0} \eta(e,\lambda) \geq \frac{1}{(2\fl{ \sqrt{e} }+1)(2\fl{ \sqrt{e/2} }+1)}. $$
\end{corollary}

Note that we do not suggest that the density $\eta(e,\lambda)$ converges as $\lambda\to0$, 
regardless of whether the condition (\ref{eq:v1_k_coprimality}) is satisfied or not.

\section{The strip map and symbolic coding of perturbed orbits}\label{sec:StripMap}

In section \ref{sec:Recurrence}, we saw that all non-zero points
where the discrete vector field $\bfv$ deviates from the scaled auxiliary 
vector field $\lambda\bfw$ lie in the set of transition points $\Lambda$ 
(lemma \ref{lemma:Lambda}, page \pageref{lemma:Lambda}).
The set $\Lambda$ consists of thin strips of lattice points aligned along the
lines $\Delta$, which form the boundaries of the boxes $B_{m,n}$.
In order to study the dynamics at these points, where the perturbations from the 
integrable limit occur, we define a transit map $\Psi$ to $\Lambda$ which we 
call the \textbf{strip map}:
\begin{equation} \label{eq:Psi}
  \Psi : \lZ \rightarrow \Lambda
 \hskip 40pt
  \Psi(z) = \F^{4t(z)}(z),
\end{equation}
where the transit time $t$ to $\Lambda$ is well-defined for all points excluding the origin:
\begin{displaymath}
 t(z)= \min \{ k\in\N  \, : \; \F^{4k}(z) \in \Lambda \} \hskip 40pt z\neq(0,0).
\end{displaymath}
(Since the origin plays no role in the present construction, to simplify 
notation we shall write $\lZ$ for $\lZ\setminus\{(0,0)\}$, where appropriate.) 
By abuse of notation, we define $\Psi^{-1}$ to be the transit map to $\Lambda$ 
under $\F^{-4}$. Note that $\Psi^{-1}$ is the inverse of $\Psi$ only on $\Lambda$.

If $z\in B_{m,n}\setminus\Lambda$ for some $m,n\in\Z$, 
then lemma \ref{lemma:Lambda} implies that $\Psi(z)$ 
satisfies \begin{equation} \label{eq:Psi_congruence2}
 \Psi(z) = z + \lambda t(z)\bfw_{m,n},
\end{equation}
where $\bfw_{m,n}$ is the value of the Hamiltonian vector field $\bfw$ in the box $B_{m,n}$. 
If $z\in\Lambda_{m,n}$, then we may have $\bfv(z)\neq \lambda\bfw(z)$, so the expression becomes
\begin{equation} \label{eq:Psi_congruence1}
 \Psi(z) = z + \bfv(z) + \lambda(t(z)-1)\bfw_{m,n}. 
\end{equation}

In the previous section, we identified the set $\Ot(z)\,\cap\,\Lambda$ as the set of 
vertices of the perturbed orbit $\Ot(z)$. Thus, within each quarter-turn, the 
strip map $\Psi$ represents transit to the next vertex. For $1\leq j\leq k$, where $k$ is the 
length of the vertex list at $z$, we say that the orbit $\Ot(z)$ \textbf{meets the $j$th vertex} 
at the point $\Psi^j(z)\in\Lambda$. 
For $z\in X$ regular, the polygon $\Pi(z)$ and the return orbit $\Ot(z)$ are non-critical,
and the number of sides of each is given by equation (\ref{eq:NumberOfSides}) of theorem \ref{thm:Polygons}. 
Thus the full set of vertices of $\Ot(z)$ is given by
 \begin{displaymath}
 \Ot(z)\,\cap\,\Lambda = \bigcup_{i=0}^3 \; \bigcup_{j=1}^{2k-1} \; \{ (\Psi^j\circ F_{\lambda}^i)(z) \}.
\end{displaymath}
Recall that the vertices of a polygon (or orbit) are numbered in the clockwise 
direction---the orientation of the integrable vector field $\bfw$. 
Hence the first $2k-1$ vertices (those lying in the first quarter-turn) are 
given by $(\Psi^j(z))_{1\leq j \leq 2k-1}$. 
The action of $F_{\lambda}$ moves points from one quadrant to the next in the 
opposing (anti-clockwise) direction, so that the vertices 
$((\Psi^j\circ F)(z))_{1\leq j \leq 2k-1}$ are the last $2k-1$ vertices. 
The following proposition is a simple consequence this arrangement.

\begin{proposition} \label{prop:regularity} 
Let $e\in\cE$ be a critical number, and let $k$ be the length of the
vertex list of the corresponding polygon class.
Then the return map $\Phi$ on $\Xe$ is related to $\Psi$ via
\begin{equation} \label{eq:Phi_Psi}
\Phi(z) \equiv (\Psi^{2k}\circ F_{\lambda})(z) \mod{\lambda\bfw_{v_1,v_1}}
\hskip 40pt z\in\Xe,
\end{equation}
where $v_1$ is the type of the first vertex and $\bfw_{v_1,v_1}$ 
is the value of the integrable vector field $\bfw$ at $z$.
\end{proposition}

\begin{proof} Let $z\in\Xe$ and let $w$ be the last vertex in $\Ot(z)$. 
By the preceding discussion, $w$ is given by
\begin{displaymath}
 w = (\Psi^{2k-1}\circ F_{\lambda})(z) \in \Lambda_{v_1,v_1}.
\end{displaymath}

As $z$ is regular, the point $\Phi(z)$ satisfies 
$\Phi(z)\in B_{v_1,v_1}\setminus \Lambda$.
Using the expression (\ref{eq:Psi_congruence2}) for $\Psi$ applied to $\Phi(z)$, we have
\begin{equation} \label{eq:PsiPhi(z)}
 (\Psi\circ\Phi)(z) \equiv \Phi(z) \mod{\lambda\bfw_{v_1,v_1}}.
\end{equation}
But the point $\Phi(z)$ is given by
 $$ \Phi(z) = \F^{4n}(w) $$
for some $1\leq n < t(w)$, where $t(w)$ is the transit time of $w$ to $\Lambda$.
Hence
\begin{equation} \label{eq:PsiPhi_w}
 (\Psi\circ\Phi)(z) = \Psi(w) = (\Psi^{2k}\circ F_{\lambda})(z).
\end{equation}
The result follows from combining (\ref{eq:PsiPhi(z)}) and (\ref{eq:PsiPhi_w}).
\end{proof}

For $z\in X$ regular, we use the vertices 
$(\Psi^j(z))_{1\leq j \leq 2k-1}$ in the first quarter-turn to define a 
sequence of natural numbers $\sigma(z)$ called the \textbf{orbit code} of $z$, 
which encapsulates how the perturbed orbit $\Ot(z)$ 
deviates from $\Pi(z)$.

Suppose the $j$th vertex of $\Pi(z)$ is a vertex of type $v_j$ lying on $y=n$,
and the orbit $\Ot(z)$ meets its corresponding vertex at $\Psi^j(z)$.
We define the pair $(x_j,y_j)$ via
\begin{equation}\label{eq:(xj,yj)}
 \Psi^j(z) = \lambda\left( \Bceil{\frac{v_j}{\lambda}} + x_j, \Bceil{\frac{n}{\lambda}} + y_j \right),
\end{equation}
where $x_j\geq 0$, and $|y_j|$, which is (essentially) the number of lattice points 
between $\Psi^j(z)$ and the line $y=n$, is small relative to $1/\lambda$. 
Using similar arguments to those in the proof of proposition \ref{prop:Xe} 
(i.e., by bounding the perpendicular distance from $\Psi^j(z)$ to the line $y=n$),
one can show that $y_j$ satisfies
\begin{displaymath}
 -(2v_j+1) \leq y_j <0 \hskip 20pt \mbox{or} \hskip 20pt 0\leq y_j < 2v_j+1,
\end{displaymath}
depending whether the integrable vector field is oriented in the positive or 
negative $y$-direction.
Hence the component of $\Lambda$ containing $\Psi^j(z)$ is a strip which is
$(2v_j+1)$ lattice points wide, 
and the possible values of $y_j$ form a complete set of residues modulo $2v_j+1$.
The $j$th element $\sigma_j$ of the orbit 
code $\sigma(z)$ is defined to be the unique residue in the set
$\{0,1,\dots, 2v_j\}$ which is congruent to $y_j$:
\begin{equation} \label{eq:sigma_j}
 \sigma_j \equiv y_j \mod{2v_j+1}.
\end{equation}
We call $y$ the \textbf{integer coordinate} of the vertex and $x$ the 
\textbf{non-integer coordinate}.
Similarly, if the $j$th vertex lies on $x=m$, then the $j$th element $\sigma_j$ of 
the orbit code is defined to be the residue congruent to $x_j$ modulo $2v_j+1$. 
In this case $x$ is the integer coordinate and $y$ is the non-integer coordinate.

For all vertices in the first quadrant, the fact that orbits progress clockwise 
under the action of $\F^4$ means that $y_j$ will be non-negative 
wherever $y$ is the integer coordinate, and $x_j$ will be negative 
wherever $x$ is the integer coordinate:
\begin{equation} \label{eq:xj<0,yj>0}
 -(2v_j+1) \leq x_j <0 \hskip 20pt \mbox{or} \hskip 20pt 0\leq y_j < 2v_j+1. 
\end{equation}
Thus the value of $\sigma_j$ is given explicitly by
\begin{equation} \label{eq:sigma_j_2}
 \sigma_j = x_j + 2v_j+1 \hskip 20pt \mbox{or} \hskip 20pt \sigma_j = y_j, 
\end{equation}
respectively.

In addition to the values $\sigma_{j}$ for $1\leq j \leq 2k-1$, we define $\sigma_{-1}$, 
which corresponds to the last vertex \textit{before} the symmetry line, i.e., to the point $\Psi^{-1}(z)$. 
Thus the orbit code of $z$ is a sequence $\sigma(z)=(\sigma_{-1},\sigma_1,\dots,\sigma_{2k-1})$, such that
\begin{align*}
 & 0\leq \sigma_{-1} < 2v_1+1, \\
 & 0\leq \sigma_j < 2v_j+1, \hskip 20pt 1\leq j \leq 2k-1,
\end{align*}
where the $v_j$ are the vertex types.

In the next proposition we consider how a perturbed orbit behaves at
its vertices. We find that 
the regularity of $z$ ensures that the discrete vector field $\bfv$ 
matches the Hamiltonian vector field $\lambda\bfw$ in the integer coordinate 
at $\Psi^j(z)$. The possible discrepancy in the non-integer coordinate is 
determined by the value of $\sigma_j$.
\begin{proposition} \label{thm:epsilon_j}
Let $e\in\cE$ be a critical number and let $k$ be the length of the
vertex list of the corresponding polygon class. 
For any $z\in\Xe$ and any $j\in\{-1,1,2,\ldots,2k-2\}$, 
let $m,n$ be such that $\Psi^j(z)\in\Lambda_{m,n}$.
Then the transit between the $j$th and $(j+1)$st vertices satisfies
\begin{equation}\label{eq:Psi^jp1}
 \Psi^{j+1}(z) = \Psi^j(z) + \lambda \left( t\bfw_{m,n} + \epsilon_j(\sigma_j)\bfe \right), 
\end{equation}
where $\epsilon_j$ is a function of the $j$th entry $\sigma_j$ of the orbit code $\sigma(z)$, 
$\bfe$ is the unit vector in the direction of the non-integer coordinate of the $j$th vertex,
and $t=t(\Psi^j(z))$ is the transit time.
\end{proposition}

\begin{proof}
It suffices to show that
\begin{displaymath}
\bfv(\Psi^j(z)) = \lambda \left( \bfw_{m,n} + \epsilon_j(\sigma_j) \bfe \right), 
\end{displaymath}
after which the result follows from equation (\ref{eq:Psi_congruence1}).

If $z\in X$ is regular, then the perturbed orbit $\Ot(z)$ is not critical. 
Thus for any vertex $w$, which, by construction, satisfies
\begin{displaymath}
 w\in\Lambda_{m,n}
 \qquad
 F_{\lambda}^4(w) = w+\bfv(w)\in B_{m,n} 
\end{displaymath}
for some $m,n\in\Z$, we must have either $w\in B_{m,n\pm1}$ or $w\in B_{m\pm1,n}$. 
For definiteness we suppose that $w\in B_{m,n+1}$, 
so that the vertex $w$ lies on $y=n+1$ and is of type $m$. 
The cases where $w\in B_{m,n-1}$ or $w\in B_{m\pm1, n}$ are similar.

Now the proof proceeds very much as that of proposition \ref{prop:Xe}. 
The perturbed vector field $\bfv(w)$ is given by equation (\ref{eq:v_abcd}), 
with $a,b,c,d$ as in (\ref{eq:abcd}). (Note that $R(w)=(m,n+1)$,
so the formula (\ref{eq:abcd}) must be modified accordingly.)
In this case, $R(F_{\lambda}^4(w))=(m,n)$ implies that
\begin{displaymath}
 c= n \hskip 20pt \mbox{and} \hskip 20pt d =m,
\end{displaymath}
and according to (\ref{eq:bd_ac_sets}), the remaining integers $a$ and $b$ satisfy
\begin{displaymath}
 a\in \{n,n+1\}, \hskip 40pt b=m.
\end{displaymath}
Thus we have
\begin{align*}
 \bfv(w) &= \lambda(n+a+1,-(2m+1))  \\ 
 &= \lambda (\bfw_{m,n} + (a-n)\bfe), 
\end{align*}
where $\bfe=(1,0)$ is the unit vector in the $x$-direction, the non-integer 
coordinate direction of the vertex.

If $w=\Psi^j(z)$, then $v_j=m$, and the coefficient of the difference between 
$\bfv(w)$ and $\lambda\bfw_{m,n}$ in the $x$-direction is given by
\begin{align*}
 \epsilon_j &= a-n \\
&= \ceil{\lambda(y-m)} -(n+1) \\
&= \left\{ \begin{array}{ll} 1 \; \; & \lambda y-(n+1) >\lambda m \\
            0 \; \; & \mbox{otherwise}.
           \end{array} \right.
\end{align*}
As in equation (\ref{eq:(xj,yj)}), we write
\begin{displaymath}
 y = \Bceil{\frac{n+1}{\lambda}} + y_j,
\end{displaymath}
where by (\ref{eq:xj<0,yj>0}), $y_j$ satisfies $0\leq y_j<2m+1$.
Then, by (\ref{eq:sigma_j_2}), we have $y_j=\sigma_j$, and 
the function $\epsilon_j$ is given by
\begin{displaymath}
 \epsilon_j(\sigma_j) = \left\{ \begin{array}{ll} 1 \; \; & \sigma_j > m \\
            0 \; \; & \mbox{otherwise,}
           \end{array} \right.
\end{displaymath}
which completes the proof.
\end{proof}

Note that the function $\epsilon_j$ depends on $j$ via $m$.
In what follows we shall write $\epsilon_j$, omitting the argument.

We think of an orbit as moving according to the integrable vector 
field at all points except the vertices, where there is a mismatch between
integrable and non-integrable dynamics, and points are given a small 
`kick' in the non-integer coordinate direction.

In particular, in the situation described in the proof of proposition \ref{thm:epsilon_j}, 
the strip containing $\Psi^j(z)$, 
which lies on the boundary between the boxes $B_{m,n+1}$ and $B_{m,n}$,
can be decomposed into two sub-strips: $0\leq y_j\leq m$ and $m<y_j<2m+1$.
In the sub-strip with $0\leq y_j\leq m$, which lies closest to $B_{m,n}$,
we have
 $$ \bfv(\Psi^j(z)) = \lambda\bfw_{m,n}, $$
whereas in the sub-strip with $m<y_j<2m+1$, which lies further into $B_{m,n+1}$,
we have
 $$ \bfv(\Psi^j(z)) = \frac{\lambda}{2}(\bfw_{m,n} + \bfw_{m,n+1}). $$
An analogous behaviour occurs in other situations.

\section{Proofs for section \ref{sec:MainTheorems}} \label{sec:lattice}

We prove theorems \ref{thm:Phi_equivariance} \& \ref{thm:minimal_densities} 
via several lemmas. 
The first and most significant step is to show that the orbit codes $\sigma(z)$ of points $z\in \Xe$ are in 
one-to-one correspondence with the set $\Z^2/\,\Le$ of equivalence classes modulo $\Le$. 
We do this by constructing a sequence of nested lattices whose congruence classes are the 
cylinder sets of the orbit code.

\subsection*{Orbit codes and lattice structure}

We define recursively a finite integer sequence $(q_j)$, 
$j=1,\ldots,2k-1$, as follows:
\begin{align}
   q_1&=(2v_1+1)^2\nonumber \\
   q_j&=\begin{cases}
         q_{j-1}&\mbox{if}\,\, v_j=v_{j-1}\\
         \lcm((2v_j+1)(2v_{j-1}+1),q_{j-1}) &\mbox{if}\,\,\, v_j\neq v_{j-1}
        \end{cases}
      &j>1. \label{eq:q_j}
\end{align}
Then we let
\begin{equation}\label{eq:p_j}
p_j=q_j/(2v_j+1)	\hskip 40pt	 j=1,\ldots,2k-1.
\end{equation}
By construction, $p_j$ is also an integer.
After defining the associated sequence of vectors 
\begin{equation*} 
 \bfL_j = \frac{q_j}{2v_1+1} \; (1,1), 
\end{equation*}
we let the lattices $\Le_j$ be the $\Z$-modules with basis
\begin{equation} \label{eq:lattice_j}
 \Le_j =\left\langle \bfL_j, 
   \frac{1}{2}\left(\bfL_j -\bfw_{v_1,v_1}\right) \right\rangle. 
\end{equation}
By construction
\begin{displaymath}
 \Le_{2k-1} \subseteq \Le_{2k-2} \subseteq \dots \subseteq \Le_1 \subset \Z^2.
\end{displaymath}

We claim that for all $1\leq j \leq 2k-1$, the closed form expression for $q_j$ is given by
\begin{equation}
 q_j = \lcm((2v_{\iota(1)}+1)^2,(2v_{\iota(1)}+1)(2v_{\iota(2)}+1),
  \dots,(2v_{\iota(i-1)}+1)(2v_{\iota(i)}+1)), \label{eq:q_j_closed_form}
\end{equation}
where $i$ is the number of distinct entries in the list $(v_1,v_2,\dots,v_j)$. 
That the lowest common multiple (\ref{eq:q_j_closed_form}) runs over all products 
$(2v_j+1)(2v_{j+1}+1)$ of consecutive, distinct vertex types follows from the form 
(\ref{eq:v_iota}) of the vertex list and the symmetry (\ref{eq:v_symmetry}) of the vertex types. 
Furthermore, since all distinct vertex types occur within the first $k$ vertex types, 
the expression (\ref{eq:q_j_closed_form}) implies that the sequence $(q_j)$ is eventually 
stationary:
\begin{equation} \label{eq:q_j=q}
 q_j = q, \quad \Le_j = \Le \hskip 40pt k\leq j \leq 2k-1, 
\end{equation}
where $q$ and $\Le$ are given by equations (\ref{eq:q}) and (\ref{eq:Le}).

Any $z,\tZ \in \Xe$ which are congruent modulo $\lambda\Le_j$ are related by
\begin{equation} \label{eq:z_tilde}
 \tZ = z + \frac{\lambda}{2}\left( (2a+b) \bfL_j -b\bfw_{v_1,v_1}\right),
\end{equation}
where $a,b\in\Z$ are the coordinates of $\tZ-z$ relative to the module basis. 
For a given $z$, $\tZ$ is determined uniquely by the coefficient $2a+b$, 
because if $z=\lambda(x,y)$, then
\begin{align*}
 x\geq y \; \; &\Rightarrow \; \; b\in\{0,1\}, \\
 x<y \; \; &\Rightarrow \; \; b\in\{-1,0\}.
\end{align*}
The point $z$ itself corresponds to $a=b=0$.

For given $e$, the following result details the role of the $\Le_j$ as cylinder sets of the orbit code. 
Applying the result for $j=2k-1$, along with the observation (\ref{eq:q_j=q}), implies that two points 
share the same orbit code if and only if they are congruent modulo $\lambda\Le$.

\begin{lemma} \label{thm:sigma_lattices}
Let $e$ be a critical number, let $k$ be the length of the vertex list of the 
corresponding polygon class, and let $p_j$ and $\Le_j$ be as above.
For any $1\leq j \leq 2k-1$ and all $z,\tZ\in \Xe$, the following three statements are equivalent:
\begin{enumerate}[(i)]
 \item the orbit codes of $z$ and $\tZ$ match up to the $j$th entry,
 \item $z$ and $\tZ$ are congruent modulo $\lambda\Le_j$,
 \item the points $\Psi^j(z)$ and $\Psi^j(\tZ)$ are congruent modulo 
$\lambda p_j\bfe$, where $\bfe$ is the unit vector in the 
direction of the non-integer coordinate of the $j$th vertex.
\end{enumerate}
\end{lemma}

\begin{proof} For $e\in\cE$, let $z,\tilde z\in\Xe$ and let the orbit codes
of $z$ and $\tZ$ be denoted by
$(\sigma_{-1},\sigma_1,\dots,\sigma_{2k-1})$, and
$(\tilde\sigma_{-1},\tilde\sigma_1,\dots,\tilde\sigma_{2k-1})$, respectively.
We think of $z$ as fixed and $\tZ$ as identified by the coordinates of $\tZ-z$ in the relevant module.
We will make extensive use of proposition \ref{thm:epsilon_j}, page \pageref{thm:epsilon_j},
which describes the behaviour of $\Psi$ as a function of the orbit code.

We proceed by induction on $j$, with two induction hypotheses. 
Firstly we suppose that $(i)$ is equivalent to $(ii)$, so that for any 
$1\leq j\leq 2k-1$: 
\begin{equation}
(\sigma_{-1},\sigma_1,\dots,\sigma_j)=
(\tilde\sigma_{-1},\tilde\sigma_1,\dots,\tilde\sigma_j)
\quad\Leftrightarrow\quad
 \tZ \equiv z \mod{\lambda\Le_j}. \tag{H1}\label{eq:H1}
\end{equation}
Secondly, we suppose that $(ii)$ is equivalent to $(iii)$. 
In particular:
\begin{equation} \tag{H2}\label{eq:H2}
\tZ = z + \frac{\lambda}{2}\left( (2a+b) \bfL_j -b\bfw_{v_1,v_1}\right)
\quad\Leftrightarrow\quad
\Psi^j(\tZ) = \Psi^j(z) + \lambda(2a +b) p_j \bfe,
\end{equation}
where $\bfe$ is the unit vector in the direction of the non-integer coordinate of that vertex.

We begin with the base case $j=1$. 
Suppose that the first vertex of a polygon in class $e$ lies on $y=v_1$, 
so that $y$ is its integer coordinate (if $x$ is the integer coordinate, 
then the analysis is identical). 
By symmetry, the previous vertex lies on $x=v_1$ and its integer coordinate is $x$.
Using the property of $\Psi$ given in equation (\ref{eq:Psi_congruence2})
applied to $z\in B_{v_1,v_1}\setminus\Lambda$, we have
\begin{equation} \label{eq:Psi_congr4}
  \Psi(z) \equiv z \mod{\lambda\bfw_{v_1,v_1}}.
\end{equation}
Furthermore, by proposition \ref{thm:epsilon_j}:
\begin{equation} \label{eq:Psi_congruence3}
 \Psi^{-1}(z) + \lambda\epsilon_{-1}\bfy \equiv  z \mod{\lambda\bfw_{v_1,v_1}},
\end{equation}
where $\bfy=(0,1)$ is the non-integer coordinate vector for the $(-1)$th vertex. 
Thus if
 $$ z=\lambda\left( \Bceil{\frac{v_1}{\lambda}}+x,\Bceil{\frac{v_1}{\lambda}}+y \right), $$ 
then by the definition (\ref{eq:sigma_j}) of the orbit code,
the $x$- and $y$-components of equations (\ref{eq:Psi_congruence3}) and (\ref{eq:Psi_congr4}), 
respectively, give us that the first two entries in the orbit code $\sigma(z)$ satisfy
\begin{align*}
 x \equiv \sigma_{-1} \mod{2v_1+1}, \\
 y \equiv \sigma_1 \mod{2v_1+1}.
\end{align*}
It follows that $z,\tZ\in \Xe$ share the partial code $(\sigma_{-1},\sigma_1)$ if and only if
\begin{displaymath}
 \tZ \equiv z \mod{(\lambda(2v_1+1)\Z)^2}.
\end{displaymath}
The lattice $\Le_1$ is given by (cf.~(\ref{eq:lattice_j}))
\begin{displaymath}
\Le_1 = \left\langle \bfL_1, \frac{1}{2} \left(\bfL_1 -\bfw_{v_1,v_1}\right) \right\rangle, 
\end{displaymath}
where $\bfL_1=p_1(1,1)$, $p_1=2v_1+1$ and
\begin{displaymath}
 \frac{1}{2} \left(\bfL_1 -\bfw_{v_1,v_1}\right) = \frac{1}{2} (p_1-p_1,p_1+p_1) = p_1\bfy.
\end{displaymath}
Thus $\Le_1=((2v_1+1)\Z)^2$ and the first hypothesis holds.

Now let $z,\tZ\in \Xe$ satisfy (\ref{eq:z_tilde}) with $j=1$.
If $\Psi(z) = F_{\lambda}^{4t}(z) = z +\lambda t \bfw_{v_1,v_1}$, where $t\in\N$ is the 
transit time to $\Lambda$, then the identities
\begin{align*}
\tZ + \lambda(t+ a +b)\bfw_{v_1,v_1}
&= \Psi(z) + \frac{\lambda}{2} (2a+b) \left(\bfL_1 +\bfw_{v_1,v_1}\right)  \\
&= \Psi(z) + \lambda(2a +b)p_1 \bfx,
\end{align*}
where $\bfx=(1,0)$ is the non-integer coordinate vector for the first vertex,
show that $\tZ$ has transit time $t+a+b$ ,
and therefore $\Psi(\tZ) = \Psi(z) + \lambda(2a +b)p_1 \bfx$, as required
(see figure \ref{fig:Psi_diagram}).
This completes the basis for induction.

\begin{figure}[t]
        \centering
        \includegraphics[scale=0.85]{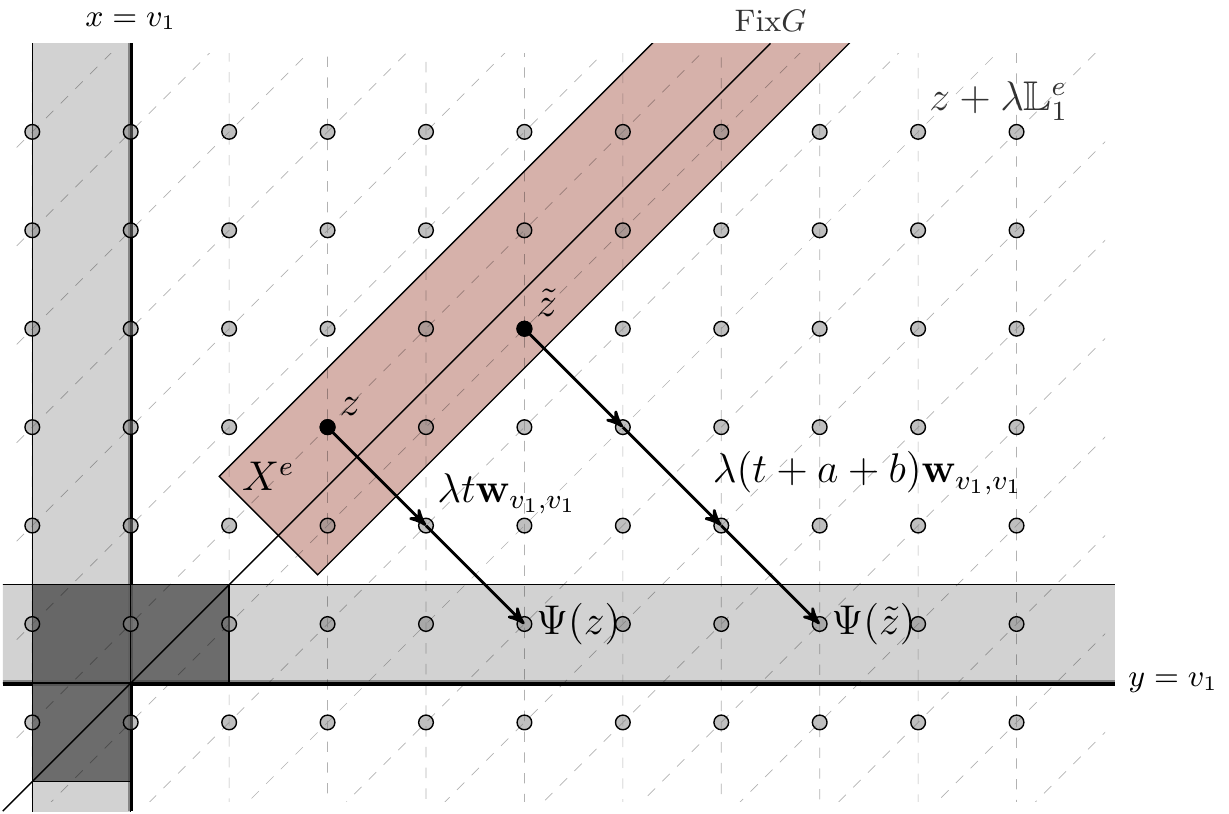}
        \caption{The points $z,\tZ\in\Xe$ and $\Psi(z),\Psi(\tZ)\in\Lambda$. 
        The set $\Xe$ is shown in red, whereas the sets $\Lambda$ and $\Sigma$ are represented in light and dark grey, respectively.}
        \label{fig:Psi_diagram}
\end{figure}

Now we proceed with the inductive step. 
Let (\ref{eq:H1}) and (\ref{eq:H2}) hold for some $j\geq 1$. 
Then $z$ and $\tZ$ are related as in equation (\ref{eq:z_tilde}), for some $a$, $b$. 
We think of $\tilde{\sigma}_{j+1}$, the $(j+1)$th entry of the orbit code of $\tZ$,
as a function of $(a,b)$. We suppose that the $j$th vertex lies 
on $y=n$ for some $n\in\Z$ (again the case in which the vertex lies on $x=m$ is identical). 
Let the pair $(x_j,y_j)$ be defined from $\Psi^j(z)$ via equation (\ref{eq:(xj,yj)}). 
Similarly, $\Psi^j(\tZ)$ defines the pair $(\tilde{x}_j,\tilde{y}_j)$.

By (\ref{eq:H2}), $\Psi^j(\tZ)$ satisfies
\begin{displaymath}
 \Psi^j(\tZ) = \Psi^j(z) + \lambda(2a +b) p_j \bfx.
\end{displaymath}
Combining this expression with proposition \ref{thm:epsilon_j} 
applied to $\Psi^j(\tZ)\in\Lambda_{v_j,n-1}$, we obtain
\begin{align}
 \Psi^{j+1}(\tZ) &= \Psi^j(\tZ) + \lambda\epsilon_j \bfx + \lambda\tilde t \bfw_{v_j,n-1} \nonumber \\
 &= \Psi^j(z) + \lambda(2a +b) p_j\bfx + \lambda\epsilon_j \bfx + \lambda\tilde t  \bfw_{v_j,n-1}, \label{eq:Psi^jp1tilde}
\end{align}
where $\tilde t$ is the transit time of $\tilde z$ to $\Lambda$.

There are now two cases to consider. \\

\noindent
\textit{Case 1: $v_j=v_{j+1}$.}

In this case the $j$th and $(j+1)$th vertices lie on parallel lines, 
which we take to be $y=n$ and $y=n-1$, so $\Psi^{j+1}(z)$ is given by
\begin{align*}
 \Psi^{j+1}(z) = \lambda\left( \Bceil{\frac{v_j}{\lambda}} + x_{j+1},\Bceil{\frac{n-1}{\lambda}} + y_{j+1} \right),
\end{align*}
and similarly for $\Psi^{j+1}(\tZ)$. 
According to the definitions (\ref{eq:q_j}), (\ref{eq:p_j}) and (\ref{eq:lattice_j}), 
we have $p_j=p_{j+1}$ and $\Le_j=\Le_{j+1}$. 
Thus, to show that (\ref{eq:H1}) continues to hold, we need to show that $\tilde{\sigma}_{j+1}=\sigma_{j+1}$ for all $(a,b)$. 
Similarly we need to show that the vector $\Psi^{j+1}(\tZ)-\Psi^{j+1}(z)$ 
is equal to the vector $\Psi^{j}(\tZ)-\Psi^{j}(z)$ of hypothesis (\ref{eq:H2}).

Because $y$ is the integer coordinate of both the $j$th and $(j+1)$th vertices,
the transit time is the same for the orbits of $z$ and $\tilde z$. 
Therefore proposition \ref{thm:epsilon_j}  
and equation (\ref{eq:Psi^jp1tilde}) with $\tilde t =t$ give us that
\begin{align*}
 \Psi^{j+1}(\tZ) &= \Psi^j(z) + \lambda(2a +b) p_j\bfx 
     + \lambda\epsilon_j \bfx + \lambda t \bfw_{v_j,n-1}, \\
 &= \Psi^{j+1}(z) + \lambda(2a +b) p_j\bfx,
\end{align*}
and the second hypothesis (\ref{eq:H2}) remains satisfied. 
Furthermore, $\Psi^{j+1}(\tZ)$ and $\Psi^{j+1}(z)$ have the same integer 
($y$) coordinate. It follows that, by the definition (\ref{eq:sigma_j}) 
of the orbit code, $\tilde \sigma_{j+1}=\sigma_{j+1}$ and (\ref{eq:H1}) 
is also satisfied.

By the $y$-component of (\ref{eq:Psi^jp1tilde}), the value of ${\sigma}_{j+1}$ 
is determined explicitly by the congruence
\begin{equation} \label{eq:sigma_j+1_case1}
 \Bceil{\frac{n-1}{\lambda}} + {\sigma}_{j+1} \equiv
     \Bceil{\frac{n}{\lambda}} + \sigma_j \mod{2v_j+1}.  
\end{equation}
This congruence shows that if $v_j=v_{j+1}$, 
then there is a permutation $\pi$ of the set $\{0,1,\dots,2v_j\}$, 
dependent on $\lambda$ but not on $z$, such that all orbit codes which have $j$th entry $\sigma_j$
will have $(j+1)$th entry $\pi(\sigma_j)$.

\medskip

\noindent
\textit{Case 2: $v_j\neq v_{j+1}$.}

In this case the $j$th and $(j+1)$th vertices lie on perpendicular lines.
We take these to be the lines $y=n$ and $x=v_j+1$, respectively, so that 
$v_{j+1}=n-1$ and $\Psi^{j+1}(z)$ is given by
\begin{align*}
 \Psi^{j+1}(z) = \lambda\left( \Bceil{\frac{v_j+1}{\lambda}} 
    + x_{j+1}, \Bceil{\frac{n-1}{\lambda}} + y_{j+1}\right).
\end{align*}
(If $x$ is the integer co-ordinate, then the analysis is identical.)
We shall demonstrate the form of $\Le_{j+1}$ by identifying those pairs 
$(a,b)$ for which $\tilde{\sigma}_{j+1}=\sigma_{j+1}$.

Taking the $x$-coordinate of equation (\ref{eq:Psi^jp1tilde}), and recalling the explicit form 
(\ref{eq:sigma_j_2}) of the orbit code, we see that $\tilde{\sigma}_{j+1}$ is determined by
\begin{equation} \label{eq:t,2a+b_eqn}
 \Bceil{\frac{v_j}{\lambda}} + x_j + (2a+b)p_j +\epsilon_j +\tilde t(2v_{j+1}+1) 
     = \Bceil{\frac{v_j+1}{\lambda}} +\tilde{\sigma}_{j+1} - (2v_{j+1}+1). 
\end{equation}
We think of this as an integer equation of the form $A(2a+b)+B\tilde t=C$, which has solutions 
$2a+b\in\Z$ and $\tilde t\in\N$ for some given value of $\tilde{\sigma}_{j+1}$ if and only if
\begin{displaymath}
 C = \Bceil{\frac{v_j+1}{\lambda}} +\tilde{\sigma}_{j+1} 
   - (2v_{j+1}+1) - \Bceil{\frac{v_j}{\lambda}} - x_j - \epsilon_j
\end{displaymath}
is sufficiently large and $C\equiv 0 \; (\mathrm{mod} \; \gcd(A,B))$, i.e., if $\lambda$ is sufficiently small and $\tilde{\sigma}_{j+1}$ satisfies the congruence
\begin{equation}
   \tilde{\sigma}_{j+1} \equiv \Bceil{\frac{v_j}{\lambda}} 
 + x_j+\epsilon_j-\Bceil{\frac{v_j+1}{\lambda}} \mod{\gcd( p_j, 2v_{j+1}+1)}. \label{eq:sigma_j+1_case2}
\end{equation}
To find the lattice $\Le_{j+1}$, we need to solve this equation in the case $\tilde{\sigma}_{j+1}=\sigma_{j+1}$.

By assumption, the point $z$, given by the module coordinates $a=b=0$, corresponds to the solution 
$2a+b=0$, $\tilde t=t$, for some transit time $t\in\N$. 
Hence the general solution of (\ref{eq:t,2a+b_eqn}) is given by
\begin{align}
 \tilde t&= t -  s\; \frac{p_j}{\gcd( p_j, 2v_{j+1}+1)}, \label{eq:GeneralSolution_t}\\
 2a+b&= s \; \frac{2v_{j+1}+1}{\gcd( p_j, 2v_{j+1}+1)}, \label{eq:GeneralSolution_2a+b}
\end{align}
for $s\in\Z$. The second of these equations implies that $s$ must have the same parity 
as $2a+b$, so we can write $s=2\tilde{a}+b$, where $\tilde{a}\in\Z$ and 
$b\in\{0,\pm 1\}$ for an appropriate choice of sign.
Substituting this expression into equation (\ref{eq:z_tilde}), the points $\tZ$ 
for which $\tilde{\sigma}_{j+1}=\sigma_{j+1}$ are given by
\begin{align*}
 \tZ &= z + \frac{\lambda}{2}\left( s \; \frac{2v_{j+1}+1}{\gcd( p_j, 2v_{j+1}+1)} \bfL_j -b\bfw_{v_1,v_1}\right) \\
 &= z  + \frac{\lambda}{2}\left( (2\tilde{a}+b) \bfL_{j+1} -b\bfw_{v_1,v_1}\right).
\end{align*}
The last equality is justified by the identities
\begin{align*}
 \frac{2v_{j+1}+1}{\gcd( p_j, 2v_{j+1}+1)} \; \bfL_j 
&= \frac{(2v_{j+1}+1)(2v_j+1)}{\gcd( (2v_j+1)p_j, (2v_{j+1}+1)(2v_j+1))} 
   \; \frac{q_j}{2v_1+1} \; (1,1) \\
&= \frac{q_{j+1}}{2v_1+1} \; (1,1) = \bfL_{j+1},
\end{align*}
where we have used the relationship $\lcm(a,b) = a b/\gcd(a,b)$.
Therefore the first hypothesis (\ref{eq:H1}) remains satisfied.

Substituting the general solution (\ref{eq:GeneralSolution_t}) and (\ref{eq:GeneralSolution_2a+b})  
into equation (\ref{eq:Psi^jp1tilde}) 
and using proposition \ref{thm:epsilon_j}, we find
\begin{align*}
\Psi^{j+1}(\tZ) &= \Psi^j(z) + \lambda(2a +b) p_j\bfx + 
      \lambda\epsilon_j \bfx + \lambda\tilde t \bfw_{v_j,v_{j+1}} \\
&= \Psi^{j+1}(z) + \frac{\lambda s p_j}{\gcd( p_j, 2v_{j+1}+1)} \left( (2v_{j+1}+1)\bfx - \bfw_{v_j,v_{j+1}}\right) \\
&= \Psi^{j+1}(z) + \frac{\lambda(2\tilde{a}+b)p_j}{\gcd( p_j, 2v_{j+1}+1)} (2v_j+1)\bfy \\
&= \Psi^{j+1}(z) + \lambda(2\tilde{a}+b)p_{j+1}\bfy.
\end{align*}
Hypothesis (\ref{eq:H2}) remains satisfied, completing the induction.

Thus hypotheses (\ref{eq:H1}) and (\ref{eq:H2}) hold for all $1\leq j\leq 2k-1$,
and the equivalence of the three statements follows.
\end{proof}

Lemma \ref{thm:sigma_lattices} gives two characterisations
of the cylinder sets of the lattices $\Le_j$.
Firstly, a cylinder set $z+\lambda\Le_j$ can be identified by the partial orbit code 
$(\sigma_{-1},\sigma_1,\dots,\sigma_j)$ 
(by the equivalence of statements (i) and (ii)).
Secondly, it can be identified by a pair of the form $(\sigma_j,\gamma_j)$,
where $\sigma_j$ is the $j$th entry in the orbit code (i.e., a residue modulo $2v_j+1$),
and $\gamma_j$ is a residue modulo $p_j=q_j/(2v_j+1)$ 
(by the equivalence of statements (ii) and (iii)---see equation (\ref{eq:H2})).

In particular, since $\Le\subset\Le_j$ for all $1\leq j \leq 2k-1$,
we can identify the cylinder sets of $\Le$ by pairs of the form $(\sigma_j,\gamma_j)$,
where $\gamma_j$ is a residue modulo $q/(2v_j+1)$.
For any $j$, there is a one-to-one correspondence between 
the set of all orbit codes and the set of pairs $(\sigma_j,\gamma_j)$.
We will use this alternative construction below.

\begin{corollary} \label{cor:sigma_lattice}
Let $e$ be a critical number, let $k$ be the length of the vertex list of the 
corresponding polygon class, and let $j$ be in the range $1\leq j \leq 2k-1$.
Then two points $z$ and $\tZ$ in $\Xe$ have the same orbit code if and only if
the points $\Psi^j(z)$ and $\Psi^j(\tZ)$ are congruent modulo $\lambda q/(2v_j+1)\,\bfe$, 
where $\bfe$ is the unit vector in the direction of the non-integer coordinate 
of the $j$th vertex.
\end{corollary}



\medskip

We have shown the equivalence between orbit codes
and congruence classes of $\Le$.
To complete the proof of theorem \ref{thm:Phi_equivariance}, 
we show that the orbit code $\sigma(z)$ 
determines uniquely the behaviour of $z$ under the return map $\Phi$.

\begin{proof}[Proof of theorem \ref{thm:Phi_equivariance}]

Consider two points $z,z+\lambda l\in\Xe$ for some $l\in\Le$ given by
\begin{displaymath}
 l= \frac{1}{2}\left( (2a+b) \bfL -b\bfw_{v_1,v_1}\right).
\end{displaymath}
According to lemma \ref{thm:sigma_lattices} applied to $j=2k-1$, 
these two points have the same orbit code and reach the $(2k-1)$th vertex 
at the points $\Psi^{2k-1}(z),\Psi^{2k-1}(z+\lambda l)\in\Lambda_{v_1,-(v_1+1)}$, 
which are congruent modulo 
$\lambda p_{2k-1}\bfe$, where $\bfe$ is the unit vector in the 
non-integer direction. In particular, $\Psi^{2k-1}(z)$ and $\Psi^{2k-1}(z+\lambda l)$ are related via
\begin{align}
\Psi^{2k-1}(z+\lambda l) &= \Psi^{2k-1}(z) + \lambda(2a+b)p_{2k-1}\bfe, \nonumber \\
&= \Psi^{2k-1}(z) + \lambda(2a+b)\frac{q}{(2v_1+1)} \; \bfe, \label{eq:vertex_2k-1}
\end{align}
where we have replaced $q_{2k-1}$ by $q$ using equation (\ref{eq:q_j=q}),
and $v_{2k-1}$ by $v_1$ using the symmetry (\ref{eq:v_symmetry}) of the vertex types. 
We will show that the points where they reach the last vertex are related by a similar equation:
\begin{equation}
(\Psi^{2k-1}\circ F_{\lambda})(z+\lambda l) = (\Psi^{2k-1}\circ F_{\lambda})(z) 
   + \lambda(2a+b)\frac{q}{(2v_1+1)} \; \bfe^{\perp}, \label{eq:vertex_8k-5}
\end{equation}
where the unit vector $\bfe^{\perp}$, the non-integer direction of the last vertex, is perpendicular to $\bfe$.

The last vertex of the return orbit $\Ot(z)$ lies in the set $\Lambda_{v_1,v_1}$, 
so must be close to the image of the $(2k-1)$th vertex under $F_{\lambda}$ (see figure \ref{fig:Phi_Psi}). 
If the $(2k-1)$th vertex lies on the line $x=v_1+1$, it is a simple exercise to show that these 
two points are in fact equal, i.e., that $(\Psi^{2k-1}\circ F_{\lambda})(z) = (F_{\lambda}\circ \Psi^{2k-1})(z)$ 
for any $z\in\Xe$. We consider the less obvious case in which the $(2k-1)$th 
vertex lies on the line $y=-v_1$ and the non-integer direction is $\bfx=(1,0)$. 
By the orientation of the vector field in the fourth quadrant, the orbit of 
the point $z$ reaches this vertex at the point $\Psi^{2k-1}(z)$ given by:
\begin{equation}
 \Psi^{2k-1}(z) = \lambda\left( \Bceil{\frac{v_1}{\lambda}} + x_{2k-1},
 \Bceil{\frac{-v_1}{\lambda}} + y_{2k-1}\right), \label{eq:Psi_2k-1}
\end{equation}
where $x_{2k-1}\geq 0$ and $0\leq y_{2k-1}< 2v_1+1$.

\begin{figure}[ht]
        \centering
        \includegraphics[scale=0.7]{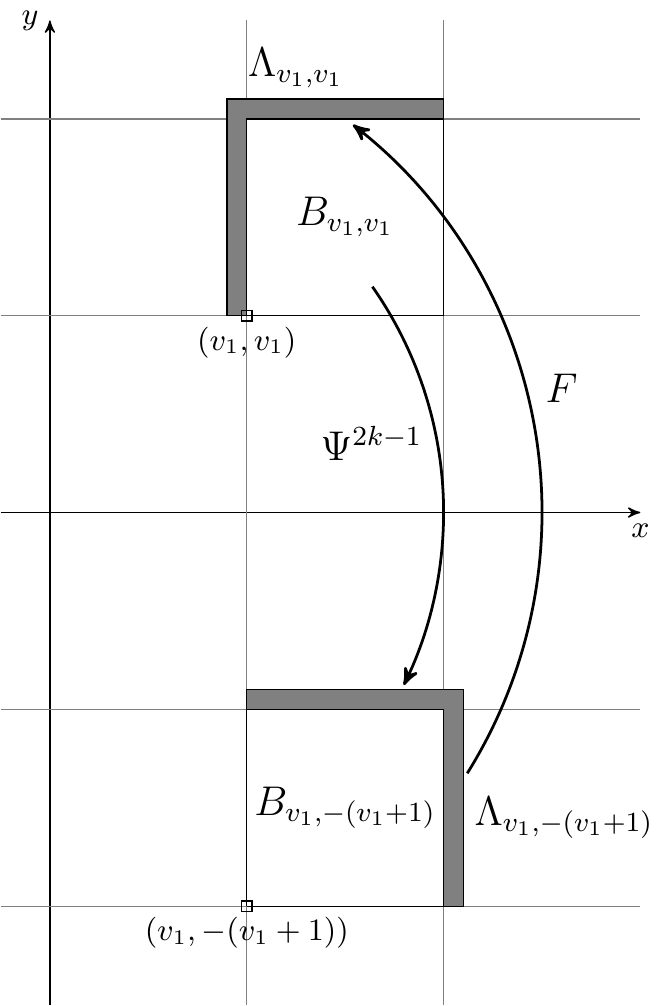}
        \caption{The $(2k-1)$th and last vertices, joined by the action of $F_{\lambda}$. }
        \label{fig:Phi_Psi}
\end{figure}

Applying $F_{\lambda}$ to (\ref{eq:vertex_2k-1}) and substituting the expression (\ref{eq:Psi_2k-1}), we get:
\begin{align*}
 (F_{\lambda}\circ \Psi^{2k-1})(z+\lambda l) 
  &=(F_{\lambda}\circ \Psi^{2k-1})(z) +\lambda(2a+b)\frac{(2v_k+1) \;p}{(2v_1+1)} \; \bfy \\
  & \hskip -20pt =\lambda\left( v_1-\Bceil{\frac{-v_1}{\lambda}} - y_{2k-1}, 
    \Bceil{\frac{v_1}{\lambda}} + x_{2k-1} \right) +\lambda(2a+b)\frac{(2v_k+1) \;p}{(2v_1+1)} \; \bfy,
\end{align*}
where the non-integer direction of the last vertex is $\bfy=(0,1)$.
By equation (\ref{eq:xj<0,yj>0}), if the first component of this point satisfies
$$ 
-(2v_1+1) \leq \left( v_1-\Bceil{\frac{-v_1}{\lambda}} 
    - y_{2k-1} \right) - \Bceil{\frac{v_1}{\lambda}} < 0,
$$
then this point is the last vertex of the orbit:
\begin{equation} \label{eq:Phi_equivariance_case1}
 (\Psi^{2k-1}\circ F_{\lambda})(z+\lambda l) = (F_{\lambda}\circ \Psi^{2k-1})(z+\lambda l). 
\end{equation}
If the above inequality is not satisfied, then it must be the upper bound that fails.
In this case $(F_{\lambda}\circ \Psi^{2k-1})(z+l)\in B_{v_1,v_1}$, 
and we apply $F_{\lambda}^{-4}$ to find:
\begin{align}
 (\Psi^{2k-1}\circ F_{\lambda})(z+\lambda l) &= (F_{\lambda}^{-3}\circ \Psi^{2k-1})(z+\lambda l) \nonumber \\
&= (F_{\lambda}\circ \Psi^{2k-1})(z+\lambda l) - \bfv((F_{\lambda}^{-3}\circ \Psi^{2k-1})(z+\lambda l)) \nonumber \\
&= (F_{\lambda}\circ \Psi^{2k-1})(z+\lambda l) -\lambda\bfw_{v_1,v_1}- \lambda\epsilon \, \bfy, \label{eq:Phi_equivariance_case2}
\end{align}
where the error term $\epsilon$ is independent of $(a,b)$ by proposition \ref{thm:epsilon_j}.
In both cases (\ref{eq:Phi_equivariance_case1}) and (\ref{eq:Phi_equivariance_case2}) the relationship (\ref{eq:vertex_8k-5}) follows.

Using (\ref{eq:vertex_8k-5}), the property (\ref{eq:Psi_congruence1}) of $\Psi$, 
and the expression (\ref{eq:Phi_Psi}) for $\Phi$, we obtain
\begin{align*}
 \Phi(z+\lambda l) &\equiv (\Psi^{2k-1}\circ F_{\lambda})(z+\lambda l) + \bfv((\Psi^{2k-1}\circ F_{\lambda})(z+\lambda l)) 
                 \mod{\lambda\bfw_{v_1,v_1}} \\
&\equiv (\Psi^{2k-1}\circ F_{\lambda})(z) + \frac{\lambda(2a+b)q}{(2v_1+1)} \; 
        \bfy + \bfv((\Psi^{2k-1}\circ F_{\lambda})(z)) \mod{\lambda \bfw_{v_1,v_1}} \\
&\equiv \Phi(z) + \frac{\lambda(2a+b)q}{(2v_1+1)} \; 
             \bfy \mod{\lambda\bfw_{v_1,v_1}} \\
&\equiv \Phi(z) + \frac{\lambda}{2}\,\frac{(2a+b)q}{(2v_1+1)} \; 
    \left((\bfy+\bfx)+(\bfy-\bfx)\right) \mod{\lambda \bfw_{v_1,v_1}} \\
&\equiv \Phi(z) + \frac{\lambda}{2}(2a+b)\left( \bfL - \frac{q}{(2v_1+1)^2} \; 
         \bfw_{v_1,v_1} \right) \mod{\lambda\bfw_{v_1,v_1}} \\
&\equiv \Phi(z) + \frac{\lambda}{2}\left((2a+b)\bfL - b \bfw_{v_1,v_1}\right) \mod{\lambda \bfw_{v_1,v_1}} \\
&\equiv \Phi(z) + \lambda l \mod{\lambda \bfw_{v_1,v_1}},
\end{align*}
where we have also used the fact that $(\bfy+\bfx)=(1,1)$, 
$(\bfy-\bfx)=(-1,1)$, and that $q/(2v_1+1)^2$ is odd.
This completes the proof of theorem \ref{thm:Phi_equivariance}.
\end{proof}

\medskip
\subsection*{The density of fixed points}

The set $\Theta^e$ of possible orbit codes is a subset of the product space
\begin{equation*} 
 \Theta^e \subseteq \{ 0,1,\dots,2v_1 \} \times \prod_{j=1}^{2k-1} \{ 0,1,\dots,2v_j \}.
\end{equation*}
The set of congruence classes modulo $\Le$ is given by the quotient space $\Z^2/\,\Le$,
so that the total number of possible orbit codes is given by
\begin{equation} \label{eq:theta_e}
\# \Theta^e = \#\left(\Z^2/\,\Le\right) = -\frac{1}{2} \; \det\left(\bfL, \bfw_{v_1,v_1}\right) = q(e). 
\end{equation}
We note that although the lattice $\Le$ is independent of $\lambda$, the set of orbit codes $\Theta^e$ is not. 

In the next lemma, we identify the orbit codes which correspond to symmetric 
fixed points of $\Phi$. 
Subsequently, in lemmas \ref{lemma:sigma_j_I} and \ref{lemma:sigma_j_II}, we identify values of $e$ 
for which the number of codes in $\Theta^e$ which satisfy the conditions of lemma 
\ref{lemma:minimal_codes} is independent of $\lambda$.
The proof of theorem \ref{thm:minimal_densities} will then follow.

\begin{lemma} \label{lemma:minimal_codes}
For any $e\in\cE$ with vertex list $(v_1,\dots,v_k)$, $z\in \Xe$ and sufficiently small $\lambda$, 
the point $z$ is a symmetric fixed point of $\Phi$ if and only if its orbit code 
$\sigma(z)=(\sigma_{-1},\sigma_1,\dots,\sigma_{2k-1})$ satisfies:
\begin{center}
 \textit{(i)} \; $\sigma_{-1}=\sigma_1$, \hspace{2cm} \textit{(ii)} \; $2\sigma_k \equiv v_k \mod{2v_k+1}$.
\end{center}
\end{lemma}

\begin{proof}
Recall the reversing symmetries $G$ and $H$ of $F$, introduced in equation (\ref{def:GH}),
and their fixed spaces (\ref{eq:FixG}) and (\ref{eq:FixH}).
We have already noted that $G$ is also a reversing symmetry of $\F$,
and the rescaled version of $H$ is given by
 $$ \Hl(x,y) = (\lambda\fl{y}-x,y) \hskip 40pt \Fix{\Hl} = \{ (x,y)\in\R^2 \, : \; 2x=\lambda\fl{y} \}. $$
Recall also the reversing symmetry $G^e$ of $\Phi$ on $\Xe$ (equation (\ref{eq:Ge})).

As in the proof of proposition \ref{prop:square_orbits}, we show that orbits are symmetric
and periodic using theorem \ref{thm:SymmetricOrbits}, page \pageref{thm:SymmetricOrbits}.

Take $e\in\cE$ and a point $z\in\Xe$. Suppose that $z$ is a symmetric fixed point of $\Phi$. 
We must have $G^e(z)=z$, so the alternative expression (\ref{eq:Ge_G}) for $G^e$
gives that either $z\in\Fix{G}$, or
 $$ F^4(z) = G(z) \hskip 20pt \Leftrightarrow \hskip 20pt F^2(z)\in\Fix{G}. $$
Thus the return orbit $\Ot(z)$ intersects $\Fix{G}$.

If $z$ is non-zero, the orbit of $z$ intersects the set $\Fix{G}\cup\Fix{\Hl}$ at exactly two points, 
and as $z$ is a fixed point of $\Phi$, these two points must occur within a single revolution. 
Hence we have:
 $$ \#\left( \Ot(z) \cap \left(\Fix{G}\cup\Fix{\Hl}\right)\right)=2. $$
If the return orbit intersects $\Fix{G}$ twice, then these points must occur in the first and third quadrants, so that
$$ 
 z \in \Fix{G} \cap F^{-2}(\Fix{G}). 
$$
By theorem \ref{thm:SymmetricOrbits} part (iii), this implies that $z$ is periodic with period four, 
and we have already observed that there are no points with minimal period four. 
Thus the return orbit $\Ot(z)$ must intersect $\Fix{\Hl}$.

Points in $\Fix{\Hl}$ lie on disjoint vertical line segments of length one, in an 
$O(\lambda)$-neighbourhood of the $y$-axis. Recall that the polygon $\Pi(z)$ intersects 
the axes at vertices of type $v_k=\fl{\sqrt{e}}$, and hence intersects the $y$-axis in the boxes:
 $$ B_{0,v_k} \hskip 20pt \mbox{and} \hskip 20pt  B_{0,-(v_k+1)}. $$
If $v_k$ is even, it follows that the relevant segment $H^+$ of $\Fix{\Hl}$ is given by:
$$ 
 H^+ = \{ \lambda(x,y)\in(\lambda\Z)^2 \, : \; 2x =\fl{ \lambda y } = v_k \}, 
$$
which lies in the positive half-plane.
Similarly if $v_k$ is odd, the relevant segment $H^-$ of $\Fix{\Hl}$ is given by:
$$ 
 H^- = \{ \lambda(x,y)\in(\lambda\Z)^2 \, : \; 2x =\fl{ \lambda y } = -(v_k+1) \}, 
$$
which lies in the negative half-plane.

The proof now proceeds in two parts.  \\

\noindent
\textit{(i) $\Ot(z)$ intersects $\Fix{G}$ if and only if $\sigma_{-1}=\sigma_1$.}

If $z=\lambda(x,y)$, then the property $\sigma_{-1}=\sigma_1$ is satisfied if and only if:
\begin{displaymath}
 x \equiv y \mod{2v_1+1},
\end{displaymath}
i.e., if and only if $z\in\Fix{G^e}$. 
We have already seen that $z\in\Fix{G^e}$ implies that $\Ot(z)$ intersects $\Fix{G}$.
Since the only points in $\Ot(z)$ which are close to $\Fix{G}$ are $z$ and $\F^2(z)$,
the converse also holds. \\

\noindent
\textit{(ii) $\Ot(z)$ intersects $\Fix{\Hl}$ if and only if}
$2\sigma_k \equiv v_k$ (mod $2v_k+1$).

Instead of considering the sets $H^+$ and $H^-$ directly, we consider 
their images under $G$ and $F_{\lambda}$, respectively, which lie in a neighbourhood of the $x$-axis:
\begin{align}
 G(H^+) &= \{ \lambda(x,y)\in(\lambda\Z)^2 \, : \; 2y=\fl{\lambda x} = v_k,\nonumber \\ 
 F_{\lambda}(H^-) &= \{ \lambda(x,y)\in(\lambda\Z)^2 \, : \; 2y =\fl{-\lambda(x+1)}
   = -(v_k+1) \}. \label{eq:F(FixH)}
\end{align}
In (\ref{eq:F(FixH)}), we assume that $\lambda(v_k+1)/2<1$, 
so that $F_{\lambda}(w) = \lambda(-1-y,x)$ for all $w=\lambda(x,y)\in H^-$. 
The orbit $\Ot(z)$ intersects $\Fix{\Hl}$ if and only if it intersects the relevant 
one of these sets, according to the parity of $v_k$.

The polygon $\Pi(z)$ intersects the $x$-axis at the $k$th vertex, where $k$ is the length of the vertex list $V(e)$. 
The return orbit $\Ot(z)$ reaches the $k$th vertex at the point $\Psi^k(z)$, given in the notation 
of (\ref{eq:(xj,yj)}) by
\begin{displaymath}
 \Psi^k(z) = \lambda\left( \Bceil{\frac{v_k}{\lambda}} + x_k, y_k\right),
\end{displaymath}
where, by (\ref{eq:sigma_j_2}), $y_k=\sigma_k$ is non-negative. 
Hence if $v_k$ is even, $\Ot(z)$ intersects $\Fix{\Hl}$ if and only if:
 $$ \Psi^k(z) \in G(H^+) \hskip 20pt \Leftrightarrow \hskip 20pt \sigma_k = v_k/2. $$
If $v_k$ is odd, then $\Ot(z)$ intersects $\Fix{\Hl}$ if and only if:
\begin{align*}
 \F^4(\Psi^k(z)) \in F_{\lambda}(H^-) \hskip 20pt \Leftrightarrow \hskip 20pt 
\sigma_k &= -(v_k+1)/2 + (2v_k+1) \\
&= (3v_k+1)/2.
\end{align*}
The congruence $2\sigma_k \equiv v_k$  (mod $2v_k+1$) 
covers both of these cases, which completes the proof.
\end{proof}

\medskip

For all $e\in\cE$ and sufficiently small $\lambda$, 
the set $\Xe$---see equation (\ref{def:Xe})---is non-empty and 
contains at least one element from every congruence class 
modulo $\lambda\Le$. We now seek to identify the number of congruence classes 
whose orbit code satisfies the conditions of lemma \ref{lemma:minimal_codes}.

As discussed in the proof of lemma \ref{lemma:minimal_codes}, the points $z\in\Xe$ whose orbit code 
$\sigma(z) = (\sigma_{-1},\sigma_1,\dots,\sigma_{2k-1})$ satisfies $\sigma_{-1} = \sigma_1$ are precisely 
those which lie in $\Fix{G^e}$, i.e., with
 $$ x \equiv y \mod{2v_1+1}. $$
All such points lie on one of two lines, parallel to the first generator $\bfL$ of the lattice $\Le$. 
Furthermore, all points on one line are congruent to those on the other, 
as they are connected by the second generator $\lambda (\bfL-\bfw_{v_1,v_1})/2$.
Hence the number of points in $\Fix{G^e}$ modulo $\lambda\Le$ is
\begin{equation*} 
 \frac{\#\left(\Z^2/\,\Le\right)}{2v_1+1} = \frac{q}{2v_1+1}, 
\end{equation*}
where we have used the expression (\ref{eq:theta_e}) for $\#\left(\Z^2/\,\Le\right)$.

It remains to determine what fraction of the points in $\Fix{G^e}$ satisfy
the second condition of lemma \ref{lemma:minimal_codes}. 
We do this by identifying values of $e$ for which all possible values of 
$\sigma_k$ occur with equal frequency, independently of $\lambda$.

If $k=1$, i.e., if a polygon class has just one vertex in the first octant ($e=0$), 
then for any given $\sigma^*\in\{0,1,\dots,2v_1\}$, 
the points $z=\lambda(x,y)\in\Fix{G^e}$ with $\sigma_1 = \sigma^*$ satisfy 
\begin{displaymath}
 x \equiv y \equiv \sigma^* \mod{2v_1+1}.
\end{displaymath}
Such points form a fraction 
\begin{displaymath}
 \frac{1}{2v_1+1}
\end{displaymath}
of all points in $\Fix{G^e}$ modulo $\lambda\Le$. 
Hence all possible values of $\sigma_1$ occur with equal frequency.
(In fact this case is trivial, since $v_1=0$ and all orbits are symmetric 
fixed points of $\Phi$---see proposition \ref{prop:square_orbits}.)
More generally if $v_k=v_1$, i.e., if all vertices of the polygon class have the same type ($e=0,2,8$),
then all possible values of $\sigma_k$ occur with equal frequency in $\Fix{G^e}$ modulo $\lambda\Le$.
This follows from the fact that, whenever $v_j=v_{j+1}$, 
the map $\sigma_j \mapsto \sigma_{j+1}$ is a permutation of the set $\{0,1,\dots,2v_j\}$, 
as we saw in case 1 of the proof of lemma \ref{thm:sigma_lattices}.

The following lemma deals with the case that a polygon class has two or more distinct vertex types.
We consider cylinder sets of the form $z+\lambda\Le_j$,
i.e., sets of points whose orbit codes match up to the $j$th entry,
where $j$ is the penultimate distinct vertex type.
We show that under a certain congruence condition on the vertex list,
all possible values of $\sigma_k$ occur with equal frequency within any such cylinder set.

\begin{lemma} \label{lemma:sigma_j_I}
Let $e\in\cE$. Suppose that the vertex list $(v_1,v_2,\dots,v_k)$ of the 
associated polygon class has at least two distinct entries and let $j=\iota(l)-1$,
where $(\iota(i))_{i=1}^l$ is the sequence of distinct vertex types defined in (\ref{eq:v_iota}). 
Furthermore, suppose that the vertex types satisfy
\begin{equation}
 \gcd(2v_k+1,p_j) =1. \label{eq:coprimality_I}
\end{equation}
Then for every $z\in\Xe$, all $\sigma^*\in\{0,1,\dots,2v_k\}$,
and all sufficiently small $\lambda$, the number of points in the set $(z + \lambda\Le_j)$ modulo $\lambda\Le$ 
whose orbit code has $k$th entry $\sigma^*$ is
\begin{equation*} 
 \frac{1}{2v_k+1} \, \# \left(\Le_j/\,\Le\right). 
\end{equation*}
\end{lemma}

\begin{proof}
Suppose that $e\in\cE$, that the vertex list $V(e)$ has at least two distinct entries,
and that the coprimality condition (\ref{eq:coprimality_I}) holds. 
Let $z\in\Xe$ have orbit code $\sigma(z) = (\sigma_{-1},\sigma_1,\dots,\sigma_{2k-1})$ 
and let the pair $(x_j,y_j)$ be defined as in equation (\ref{eq:(xj,yj)}), where 
$j=\iota(l)-1$.

Since $z+\lambda\Le_j$ is a cylinder set in the sense of lemma \ref{thm:sigma_lattices}, 
the orbit codes of all points $\tZ\in z + \lambda\Le_j$ match that of $z$ up to the $j$th entry.
We have to show that, among these orbit codes, all possible values of $\sigma_k$ occur with equal frequency.

Since $j=\iota(l)-1$, we have
 $$ v_{j+1} = v_{\iota(l)} = v_k. $$
Let $\tZ\in z + \lambda\Le_j$ and let the $(j+1)$th entry 
of the orbit code of $\tZ$ be $\tilde{\sigma}_{j+1}$. 
We show first that all possible values of $\tilde{\sigma}_{j+1}$ occur with equal frequency.

By construction $v_j\neq v_{j+1}$, so the possible values of $\tilde{\sigma}_{j+1}$ are determined by 
case 2 of the proof of lemma \ref{thm:sigma_lattices}. 
In the course of the proof, we saw that the occurrence of points $\tZ$ with some fixed value of  
$\tilde{\sigma}_{j+1}$ correspond to solutions of an integer equation, given in the case where $y$ is the 
non-integer coordinate of the $j$th vertex by equation (\ref{eq:t,2a+b_eqn}). 
(A similar equation holds when $x$ is the non-integer coordinate.) 
Each solution $(2a+b,\tilde{t})\in \Z\times\N$ determines the module coordinates $(a,b)$ of 
$\tZ-z$ in $\lambda\Le_j$ and the transit time $\tilde{t}$ of $\tZ$ from the $j$th vertex to the $(j+1)$th.

Solutions of (\ref{eq:t,2a+b_eqn}) occur for all values of $\tilde{\sigma}_{j+1}$ 
satisfying the congruence (\ref{eq:sigma_j+1_case2}), and the condition that $\lambda$ 
be sufficiently small ensures that 
at least one such solution is realised by a point $\tZ\in \Xe$. 
By construction, each distinct value of $\tilde{\sigma}_{j+1}$ 
which has a solution defines a unique point in $z + \lambda\Le_j$ modulo $\lambda\Le_{j+1}$.

However, due to the coprimality condition (\ref{eq:coprimality_I}), the modulus of the congruence 
(\ref{eq:sigma_j+1_case2}) is unity. Hence solutions occur for all possible values of $
\tilde{\sigma}_{j+1}$, and each corresponds to a unique congruence class of 
$z + \lambda\Le_j$ modulo $\lambda\Le_{j+1}$. Furthermore, by (\ref{eq:q_j=q}), the lattices 
$\Le_{j+1}$ and $\Le$ are equal, hence all possible values of $\tilde{\sigma}_{j+1}$ 
occur exactly once in $z + \lambda\Le_j$ modulo $\lambda\Le$.

If $j+1=\iota(l)=k$ then this completes the proof. If $\iota(l)<k$, take $i$ in the range 
$\iota(l) \leq i <k$. By the definition of $\iota(l)$ as the index of the last distinct vertex type, 
we have $v_{i}=v_{i+1}=v_k$. As discussed above, the map $\sigma_i \mapsto \sigma_{i+1}$ 
is a permutation of the set $\{0,1,\dots,2v_i\}$ whenever $v_i = v_{i+1}$. 
Hence the equal frequency of the possible values of $\tilde{\sigma}_{i}$ implies that of $\tilde{\sigma}_{i+1}$ 
and the result follows.
\end{proof}

\medskip

In the previous section (equation (\ref{eq:sigma_j})), we defined 
the $j$th entry $\sigma_j$ of the orbit code $\sigma(z)$ as the residue modulo $2v_j+1$ 
of the integer coordinate of $\Ot(z)$ at the $j$th vertex.
At the conclusion of the proof of lemma \ref{thm:sigma_lattices},
we remarked that we can also define the sequence $\gamma(z)$, whose $j$th entry
$\gamma_j$ is the residue of the non-integer coordinate modulo $q/(2v_j+1)$ 
(see corollary \ref{cor:sigma_lattice}, page \pageref{cor:sigma_lattice}).
The residue $\gamma_j$ is an alternative encoding of the other entries in the orbit code,
so that for any $z,\tZ\in\Xe$ and any $j$:
\begin{displaymath} \label{eq:(sigma_j,gamma_j)}
 \tZ \equiv z \mod{\lambda\Le} \hskip 20pt \Leftrightarrow  \hskip 20pt (\sigma_j,\gamma_j) = (\tilde{\sigma}_j,\tilde{\gamma}_j).
\end{displaymath}

In the following lemma, we use $\gamma(z)$ to identify polygon classes where, 
among points in $\Fix{G^e}$ and for all $j$, 
all possible values $\sigma_j\in\{0,1,\dots,2v_j\}$ occur with equal 
frequency modulo $\lambda\Le$, independently of $\lambda$.

\begin{lemma} \label{lemma:sigma_j_II}
Let $e\in\cE$. Suppose that the vertex list $(v_1,v_2,\dots,v_k)$ of the 
associated polygon class is such that $2v_1+1$ is coprime to $2v_j+1$ for all other vertex types $v_j$:
\begin{equation} \label{eq:v1_coprimality}
 \gcd(2v_1+1,2v_j+1) =1 \hskip 40pt 2\leq j \leq k, \; v_j\neq v_1. 
\end{equation}
Then for sufficiently small $\lambda$, for all $j$ in $1\leq j \leq 2k-1$,
and all $\sigma^*\in\{0,1,\dots,2v_j\}$, the number $n_j$ of points $z\in \Fix{G^e}$ 
modulo $\lambda\Le$ whose orbit code has $\sigma_j = \sigma^*$ is given by:
\begin{equation}\label{eq:n_j}
n_j=\frac{\#\left(\Z^2/\,\Le\right)}{(2v_1+1)(2v_j+1)}.
\end{equation}
\end{lemma}
\begin{proof}
We use induction on $j$. Consider points $z\in\Fix{G^e}$ whose orbit code 
has $j$th value $\sigma_j$, for some arbitrary 
$\sigma_j\in\{0,1,\dots,2v_j\}$ and $j\in\{1,\dots,2k-1\}$. 
Let the sequence $\gamma(z)$ be denoted $(\gamma_{-1},\gamma_1,\dots,\gamma_{2k-1})$. 
Our induction hypotheses are that: \\
(i) the number of such points is given by $n_j$ (equation (\ref{eq:n_j})), 
where $q(e)=\#\left(\Z^2/\,\Le\right)$ and the coprimality condition (\ref{eq:v1_coprimality}) 
ensures that $n_j$ is a natural number; \\
(ii) for each residue $r \in \{0,1,\dots, n_j-1\}$ modulo $n_j$, there is a unique congruence class modulo $\lambda\Le$ whose value of $\gamma_j$ satisfies
\begin{displaymath}
 \gamma_j \equiv r \mod{n_j}.
\end{displaymath}

The base case is $j=1$. The points $z=\lambda(x,y)\in\Fix{G^e}$ with some fixed value of $\sigma_1$ satisfy:
\begin{displaymath}
 x \equiv y \equiv \sigma_1 \mod{2v_1+1}.
\end{displaymath}
Such points are congruent modulo $(\lambda(2v_1+1)\Z)^2$, hence the number of such points modulo $\lambda\Le$ is
\begin{displaymath}
 \frac{q}{(2v_1+1)^2} = n_1.
\end{displaymath}
By lemma \ref{thm:sigma_lattices}, if $z$ is one such point, then any other point $\tZ$ reaches the first vertex at
\begin{displaymath}
 \Psi(\tZ) = \Psi(z) + \lambda s p_1 \bfe
\end{displaymath}
for some $s\in\Z$, where $p_1=2v_1+1$ and $\bfe$ is the unit vector in the non-integer coordinate direction. 
Then by the construction of $\gamma(z)$, if $\gamma(\tZ)=(\tilde{\gamma}_{-1},\tilde{\gamma}_1,\dots,\tilde{\gamma}_{2k-1})$, 
the value of $\tilde{\gamma}_1$ is related to $\gamma_1$ by
\begin{displaymath}
\tilde{\gamma}_1 \equiv \gamma_1 + s (2v_1+1) \mod{(2v_1+1) \, n_1},
\end{displaymath}
and $z$ and $\tZ$ are congruent modulo $\lambda\Le$ if and only if $\tilde{\gamma}_1 = \gamma_1$. 
Thus the $n_1$ distinct points modulo $\lambda\Le$ correspond to distinct values of $s$ modulo $n_1$. Furthermore, if we consider the value of $\tilde{\gamma}_1$ modulo $n_1$, we have:
\begin{displaymath}
 \tilde{\gamma}_1 \equiv \gamma_1 + s (2v_1+1) \mod{n_1}.
\end{displaymath}
Now $2v_1+1$ is coprime to the modulus, as by the coprimality condition (\ref{eq:v1_coprimality}) and the construction (\ref{eq:q}) of $q$, two is the highest power of $(2v_1+1)$ that divides $q$. It follows that each distinct value of $\tilde{\gamma}_1$ is distinct modulo $n_1$. This completes the base case.

To proceed with the inductive step, we suppose that the above hypotheses hold for some $j\in\{1,\dots,k-1\}$. 
In the proof of lemma \ref{thm:sigma_lattices} we used equation (\ref{eq:Psi^jp1}) to
describe the behaviour of points as they move from one vertex to the next in two cases. 
The first case occurs when $v_j = v_{j+1}$, so that the $j$th and $(j+1)$th vertices lie on parallel lines, and $n_j = n_{j+1}$. 
In this case, the value of $\sigma_{j+1}$ is determined uniquely by the value of $\sigma_j$. 
In particular, we saw that if the $j$th vertex lies on $y=n$ and the $(j+1)$th vertex lies on $y=n-1$, 
then $\sigma_{j+1}$ and $\sigma_j$ are related by equation (\ref{eq:sigma_j+1_case1}).

We can use the same methods, considering this time the non-integer component of equation 
(\ref{eq:Psi^jp1}), to show that $\gamma_{j+1}$ is determined by the 
pair $(\sigma_j,\gamma_j)$ via:
\begin{displaymath}
 \gamma_{j+1} \equiv \gamma_{j} +\epsilon_j +(2n-1)t \mod{(2v_1+1) \, n_j},
\end{displaymath}
where $\epsilon_j=\epsilon_j(\sigma_j)$, and $t = t(\sigma_j)$ is the transit time between vertices. 
The one-to-one relationship between $\sigma_j$ and $\sigma_{j+1}$, ensures that there are $n_{j+1}=n_j$ 
points in $\Fix{G^e}$ that achieve any given value of $\sigma_{j+1}$ at the $(j+1)$th vertex. 
Similarly, for any given value of $\sigma_j$, the above congruence
establishes a one-to-one relationship between $\gamma_j$ and $\gamma_{j+1}$ modulo $(2v_1+1)n_j$.
Because this bijection is a translation, it also holds modulo $n_j$.
In other words, there is a skew-product map of residue classes modulo $n_j$: 
$(\sigma_j,\gamma_j)\mapsto(\sigma_{j+1},\gamma_{j+1})$.
This completes the inductive step for the first case.

In the second case, where $v_j \neq v_{j+1}$, the $j$th and $(j+1)$th vertices lie on perpendicular lines. 
Again referring to the proof of lemma \ref{thm:sigma_lattices}, taking equation (\ref{eq:t,2a+b_eqn}) 
modulo $2v_{j+1}+1$ gives the following expression for $\sigma_{j+1}$ in terms of the pair $(\sigma_j,\gamma_j)$:
\begin{displaymath}
\Bceil{\frac{v_j+1}{\lambda}} +\sigma_{j+1} 
 \equiv \Bceil{\frac{v_j}{\lambda}} + \gamma_j +\epsilon_j \mod{2v_{j+1}+1}.
\end{displaymath}
Here we were able to replace $x_j$ with $\gamma_j$ as, by the construction (\ref{eq:q}) of $q$, 
$2v_{j+1}+1$ is a divisor of the modulus $q/(2v_j+1)=(2v_1+1)n_j$ which defines $\gamma_j$. 
If the coprimality condition (\ref{eq:v1_coprimality}) holds, then $2v_{j+1}+1$ also divides $n_j$. 
Hence for any given pair $(\sigma_j,\sigma_{j+1})$, there are $n_j/(2v_{j+1}+1)$ values of 
$\gamma_j$ modulo $n_j$ for which the following congruence is satisfied:
\begin{equation}\label{eq:gamma_j_II}
 \gamma_j \equiv \Bceil{\frac{v_j+1}{\lambda}} +\sigma_{j+1} 
 - \Bceil{\frac{v_j}{\lambda}} -\epsilon_j + s(2v_{j+1}+1) \mod{n_j}
\end{equation}
where $s\in\Z$. The total number of points with any given value of $\sigma_{j+1}$ is thus:
\begin{displaymath}
 (2v_j+1) \times \frac{n_j}{2v_{j+1}+1} = n_{j+1},
\end{displaymath}
which completes the inductive step for hypothesis (i).

Taking the second component of equation (\ref{eq:Psi^jp1}) modulo $n_{j+1}$ gives an expression for 
$\gamma_{j+1}$ in terms of the pair $(\sigma_j,\gamma_j)$:
\begin{equation} \label{eq:gamma_jp1}
 \gamma_{j+1} 
     \equiv \Bceil{\frac{n}{\lambda}} + \sigma_j + (2v_j+1)t - \Bceil{\frac{n-1}{\lambda}} \mod{n_{j+1}},
\end{equation}
where $t=t(\sigma_j,\gamma_j)$. For a given pair $(\sigma_j,\sigma_{j+1})$, 
$t$ is given by equation (\ref{eq:t,2a+b_eqn}). 
Hence taking equation (\ref{eq:t,2a+b_eqn}) modulo $n_j$, a multiple of $2v_{j+1}+1$, 
and using the expression (\ref{eq:gamma_j_II}) for $\gamma_j$, 
it follows that the values of $t$ satisfy
\begin{align*}
t+1 &\equiv \frac{ \ceil{(v_j+1)/\lambda} +\sigma_{j+1} 
  - \ceil{v_j/\lambda} - \gamma_j -\epsilon_j}{2v_{j+1}+1} \mod{\frac{n_j}{2v_{j+1}+1}} \\
 &\equiv -s \mod{n_j/(2v_{j+1}+1)},
\end{align*}
where $s\in\Z$. Thus $t$ takes all values modulo $n_j/(2v_{j+1}+1) = n_{j+1}/(2v_j+1)$. 
Applying this to equation (\ref{eq:gamma_jp1}) and letting $\sigma_j$ vary across the range $\sigma_j\in\{0,1,\dots,2v_j\}$, 
we see that $\gamma_{j+1}$ achieves a complete set of residue classes modulo $n_{j+1}$, as required. 
This completes the inductive step for hypothesis (ii) and the result follows from hypothesis (i) for $j=k$.
\end{proof}

\medskip

Finally, we can give the proof of theorem \ref{thm:minimal_densities} (page \pageref{thm:minimal_densities})
on the density of minimal orbits.

\begin{proof}[Proof of theorem \ref{thm:minimal_densities}]

Let $e\in\cE$ be given and let $\sigma^*$ be the unique element of the set $\{0,1,\dots,2v_k\}$ that satisfies
\begin{displaymath}
2\sigma^* \equiv v_k \mod{2v_k+1}.
\end{displaymath}
By lemma \ref{lemma:minimal_codes}, $z\in\Xe$ is a symmetric fixed point of
$\Phi$ if and only if its orbit code satisfies $\sigma_{-1}=\sigma_1$,
i.e., if $z\in\Fix{G^e}$, and $\sigma_k=\sigma^*$.

We will show that the number of points in $\Fix{G^e}$ modulo $\lambda\Le$
whose orbit code has $\sigma_k=\sigma^*$ is given by
 $$ \frac{q}{(2v_1+1)(2v_k+1)}, $$
where $q$, given by (\ref{eq:q}), is the total number of points modulo $\lambda\Le$.

For $e=0,2,8$, all elements of the vertex list are the same. This case is dealt
with by the discussion preceding lemma \ref{lemma:sigma_j_I}.
Thus we assume that the vertex list contains at least two distinct elements.

Suppose first that $2v_1+1$ is coprime to $2v_j+1$ for all $v_j\neq v_1$,
so that the condition (\ref{eq:v1_coprimality}) for lemma \ref{lemma:sigma_j_II}  holds.
Applying the lemma for $j=k$, we have that for sufficiently small $\lambda$, 
the number of points in $\Fix{G^e}$ modulo $\lambda\Le$ whose orbit code has $k$th entry $\sigma^*$ is
given by 
 $$ n_j = \frac{q}{(2v_1+1)(2v_k+1)}, $$
as required.

Suppose now that $2v_k+1$ is coprime to $2v_j+1$ for all $v_j\neq v_k$.
Let $j=\iota(l)-1$ be the penultimate distinct vertex type in the vertex list.
Note that $v_j = v_{\iota(l-1)}$ and $q_j=q_{\iota(l-1)}$, where
$q_j$ is given in closed form by (\ref{eq:q_j_closed_form}),
and that $v_i\neq v_k$ for all $1\leq i \leq j$.
It follows that $2v_k+1$ is coprime to $q_{\iota(l-1)}$.
Similarly $2v_k+1$ is coprime to $p_{\iota(l-1)}=q_{\iota(l-1)}/(2v_j+1)$,
and the condition (\ref{eq:coprimality_I}) of lemma \ref{lemma:sigma_j_I} holds.
Applying the lemma, we have that in every cylinder set of $\lambda\Le_j$, 
the number of points modulo $\lambda\Le$ whose orbit code has $k$th entry $\sigma^*$ is given by
\begin{displaymath}
 \frac{1}{2v_k+1} \#\left(\Le_j/\Le\right) = \frac{1}{2v_k+1}\,\frac{q}{q_j}.
\end{displaymath}
The set $\Fix{G^e}$ is the union of $q_j/(2v_1+1)$ such cylinder sets.
Hence, as before, the number of symmetric fixed points in $\Xe$ modulo $\lambda\Le$ is
\begin{displaymath}
\frac{1}{2v_k+1}\,\frac{q}{q_j}\times \frac{q_j}{2v_1+1} = \frac{q}{(2v_1+1)(2v_k+1)}.
\end{displaymath}
This number is independent of $\lambda$, which completes the proof of the first statement.

We have shown that for sufficiently small $\lambda$, and if (\ref{eq:v1_k_coprimality}) 
holds, then the fraction of symmetric fixed points of $\Phi$ in each fundamental domain of $\lambda\Le$ is
\begin{displaymath}
\frac{1}{(2v_1+1)(2v_k+1)} = \frac{1}{(2\fl{\sqrt{e/2}}+1)(2\fl{\sqrt{e}}+1)},
\end{displaymath}
where we have used equations (\ref{def:v1}) and (\ref{def:vk}) for $v_1$ and $v_k$. 
It remains to show that the density $\delta(e,\lambda)$ of symmetric fixed points in $\Xe$ converges 
to this fraction as $\lambda\rightarrow 0$.

By equation (\ref{def:Xe}), the domain $\Xe$ is a subset of the lattice $(\lambda\Z)^2$ bounded 
by a rectangle lying parallel to the symmetry line $\Fix{G}$.
Similarly, a fundamental domain of the lattice $\lambda\Le$ is a subset of $(\lambda\Z)^2$ bounded 
by a parallelogram of the form
\begin{displaymath}
\{ \alpha \bfL + \frac{\beta}{2}(\bfL-\bfw_{v_1,v_1}) \, : \; \alpha,\beta\in[0,\lambda) \},
\end{displaymath}
where the generator $\bfL$ is also parallel to the symmetry line. 
These parallelograms tile the plane under translation by the elements of $\lambda\Le$.

The width of $\Xe$ (taken in the direction perpendicular to $\Fix{G}$) is $\lambda\|\bfw_{v_1,v_1}\|$---exactly twice 
that of the above parallelogram (see figure \ref{fig:lattice_Le}, page \pageref{fig:lattice_Le}). 
The number of parallelograms which fit lengthwise into $\Xe$, however, goes to infinity as $\lambda$ goes to zero.
If $\Ie(\lambda)=(\alpha_1,\alpha_2)\subset\cIe$, then the length $d$ of $\Xe$ parallel to $\Fix{G}$ is given by
\begin{align*}
 d &= \sqrt{2}\left(P^{-1}(\alpha_2/2)-P^{-1}(\alpha_1/2)\right) \\
  &=\frac{1}{\sqrt{2}}\left(\frac{|\Ie|}{2v_1+1}\right) \\
  &=\frac{1}{\sqrt{2}}\left(\frac{|\cIe|}{2v_1+1}\right) + O(\lambda)
\end{align*}
as $\lambda\to 0$, where we have used the expression (\ref{def:Pinv}) for $P^{-1}$, 
and proposition \ref{prop:Ie} for the length of $\Ie(\lambda)$.
Thus, the number of parallelograms which can be contained in the rectangle bounding $\Xe$ is at least
\begin{displaymath}
2 \left(\Bfl{ \frac{d}{\lambda \|\bfL\|} } -1 \right) - 8,
\end{displaymath}
where $\fl{d/\lambda \|\bfL\|}$ is the number of times that the vector 
$\bfL$ fits lengthways into the rectangle, 
we subtract $1$ for the slope of the parallelogram, and we subtract $8$ for the parallelograms which intersect 
the boundary. Each parallelogram contains a complete fundamental domain of $\lambda\Le$, and their 
contribution to $\delta(e,\lambda)$ dominates in the limit $\lambda\rightarrow 0$.

Explicitly, the number of points in $\Xe$ scales like
\begin{align*}
 \# \Xe &= \frac{2(2v_1+1)}{\lambda}\left(P^{-1}(\alpha_2/2)-P^{-1}(\alpha_1/2)\right) + O(1) \\
 &= \frac{|\Ie|}{\lambda} + O(1) \\
 &= \frac{|\cIe|}{\lambda} + O(1)
\end{align*}
as $\lambda\to 0$, whereas the length of $\bfL$ is given by (\ref{def:L}) as:
 $$ \|\bfL\| = \frac{\sqrt{2}q}{2v_1+1}, $$
and $\#\left(\Z^2/\,\Le\right)=q(e)$.
Hence the density $\delta(e,\lambda)$ satisfies
\begin{align*}
\delta(e,\lambda) &= \frac{\#\left(\Z^2/\,\Le\right)}{\# \Xe}
\left(  \frac{2\Bfl{ d/\lambda \|\bfL\| } - 10}{(2v_1+1)(2v_k+1)} + O(1) \right) \\
&= \left(\frac{\lambda q}{|\cIe|} + O(\lambda^2) \right)
\left(  \frac{2 d/\lambda \|\bfL\|}{(2v_1+1)(2v_k+1)} + O(1) \right) \\
&= \left(\frac{\lambda q}{|\cIe|} + O(\lambda^2) \right)
\left(  \frac{ |\cIe|/\lambda q(e)}{(2v_1+1)(2v_k+1)} + O(1) \right) \\
 & = \frac{1}{(2v_1+1)(2v_k+1)} + O(\lambda)
\end{align*}
as $\lambda\rightarrow 0$.
\end{proof}

%% file: ApeirogonLimit.tex
\chapter{The limit $e\to\infty$} \label{chap:Apeirogon}

In chapter \ref{chap:IntegrableLimit} we introduced the piecewise-affine Hamiltonian $\cP$, 
which describes the behaviour of the rescaled discretised rotation $\F$ in the integrable limit ($\lambda\rightarrow 0$). 
We saw that orbits of the flow $\varphi$ associated with $\cP$ are convex polygons 
(theorem \ref{thm:Polygons}, page \pageref{thm:Polygons}), 
which have a natural classification indexed by the set of critical numbers $\cE$.

In this chapter we consider the behaviour of the Hamiltonian system at infinity, 
where the index $e\in\cE$ diverges. 
In this limit the number of discontinuities experienced by an orbit 
(i.e., the number of vertices of the polygons) is unbounded, 
whilst the magnitude of these discontinuities becomes arbitrarily small.

We find that the limiting behaviour is a dichotomy. 
Typically the limiting flow is linear, like the underlying rigid rotation: 
however, for a certain subsequence of values of $e$ the nonlinearity persists. 
We focus on the return map of the flow introduced in section \ref{sec:IntegrableReturnMap}: 
the integrable counterpart to the return map $\Phi$. 

\section{A change of coordinates} \label{sec:cylinder_coordinates}

In section \ref{sec:IntegrableReturnMap}, we remarked that it is natural 
to think of the first return map of the flow as a twist map on a cylinder.
To formalise this description, we need to make a change of coordinates.

Recall the domain $\cX$ of the integrable return map, introduced in equation (\ref{eq:cX}), page \pageref{eq:cX}.
By analogy with the perturbed case (cf. section \ref{sec:RegularDomains}, page \pageref{def:regular}), 
we call a point $z\in\cX$ with $\cP(z)\in\cIe$ \defn{regular} with respect to the flow if
\begin{equation*} 
 \phil(z) = z + \lambda\bfw(z) = z + \lambda\bfw_{v_1,v_1}.
\end{equation*}
(Recall that regular points $z\in\Xe$ satisfy $\F^4=z + \lambda\bfw_{v_1,v_1}$, among other things.)
Then we can define the sequence of sets:
\begin{equation} \label{def:cXe}
 \cX^e = \{ z\in\cX \, : \; \cP(z) \in \Ie \},
\end{equation}
where $\Ie\subset \cIe$ is the largest interval such that all points in $\cX^e$ are regular.
As in the perturbed case, it is straightforward to show that the union 
of the $\cX^e$ have full density in $\cX$.
Note that the domain $\Xe$ is a subset of $\cX^e$, since
 $$ z\in\Xe \quad \Rightarrow \quad z\notin\Lambda \quad \Rightarrow \quad \phil(z) = z + \lambda\bfw_{v_1,v_1}. $$

To compare the actions of the unperturbed return map for varying values of $e$, we define the two-parameter family of maps:
 $$ \eta^e(\lambda): \cX^e \rightarrow \bbS^1 \times \R \hskip 40pt e\in\cE, \; \lambda>0. $$
The map $\eta^e(\lambda)$ is a change of coordinates $z=(x,y)\mapsto (\theta,\rho)$, with
\begin{equation} \label{def:rho_theta}
 \theta(z) = \frac{1}{\lambda} \, \frac{x-y}{2(2v_1+1)} \hskip 20pt \rho(z) = \frac{1}{\lambda} \, \frac{x+y-2x_0}{2(2v_1+1)},
\end{equation}
where $v_1 = \fl{\sqrt{e/2}}$ and $z_0=(x_0,x_0)\in\cX^e$ is some fixed point of the return map lying on $\Fix{G}$. This change of coordinates is the composition of several elements: a rotation through an angle $\pi/4$, which maps the symmetry line $\Fix{G}$ onto the $\rho$-axis; a rescaling of the plane by a factor of $1/\lambda\sqrt{2}(2v_1+1)$, which normalises the range of the coordinate $\theta$; and a translation, which ensures that the preimage $z_0$ of the origin $(\theta,\rho)=(0,0)$ is a fixed point. Such a fixed point is guaranteed to exist for sufficiently small $\lambda$, as by proposition \ref{prop:cPhi(z)}, we may take any $z_0\in\Fix{G}$ satisfying:
\begin{equation} \label{eq:z_0}
 \frac{1}{4} - \frac{\cT(z_0)}{4\lambda} \equiv 0 \mod{ 1}.
\end{equation}
In what follows, we omit the $\lambda$ dependence and simply write $\eta^e$.

For $z,\varphi^{\lambda t}(z)\in\cX^e$, the flow acts as:
 $$ \varphi^{\lambda t}(z) = z + \lambda t\bfw_{v_1,v_1}, $$
where $\bfw_{v_1,v_1}$ is perpendicular to $\Fix{G}$,
so that the coordinate $\theta$ is parallel to the direction of flow, and $\rho$ is perpendicular to it:
\begin{equation*} 
 \varphi^{\lambda t}(\theta, \rho) = (\theta + t, \rho).
\end{equation*}
(We use the symbol $\varphi$ to denote the flow in both coordinate spaces.)
Thus if we identify the interval $[-1/2,1/2)$ with the unit circle $\bbS^1$, it is straightforward to see that the coordinate $\theta$ plays the same role as in the expression (\ref{eq:cX}) for $\cX$.

In the following proposition, we show that under this change of coordinates, the unperturbed return map acts as a linear twist map.

\begin{theorem} \label{thm:Omega_e}
For $e\in\cE$, let $\Omega^e$ be the map
\begin{equation} \label{def:Omega^e}
 \Omega^e : \bbS^1 \times \R \rightarrow \bbS^1 \times \R
    \hskip 40pt 
 \Omega^e(\theta,\rho) = \left( \theta - \frac{1}{2} (2v_1+1)^2\rho\cT^{\prime}(e) , \rho \right).
\end{equation}
Furthermore, let $z\in\cX^e$, and let $z^{\prime}$ be the first return of $z$ to $\cX^e$ \hl{under $\cF$}.
Then for any sufficiently small $\lambda>0$:
 $$ \eta^e(z)=(\theta,\rho) \hskip 20pt \Rightarrow  \hskip 20pt \eta^e (z^{\prime}) = \Omega^e(\theta,\rho). $$
\end{theorem}

Here the derivative $\cT^{\prime}$ of the period function is undefined on the critical numbers, 
and constant on the sequence of sets $\cIe$ (see equation (\ref{eq:Tprime(alpha)})). 
Hence we write $\cT^{\prime}(e)$ to denote the value of $\cT^{\prime}$ on $\cIe$:
 $$ \cT^{\prime}(e) = \lim_{\alpha\rightarrow e^+} \cT^{\prime}(\alpha). $$

\begin{proof}
For $e\in\cE$, pick $z\in\cX^e$ and let $(\theta,\rho)=\eta^e(z)$. 
Then we have $z=\varphi^{\lambda\theta}(u,u)$ and $z^{\prime}=\varphi^{\lambda\theta^{\prime}}(u,u)$ 
for some $u\geq 0$, where $\theta^{\prime}$ is given by equation (\ref{eq:theta_prime}) of proposition \ref{prop:cPhi(z)}.
In the new coordinates, this is equivalent to
 $$ \eta^e (z^{\prime}) = \left( \theta + \frac{1}{4} - \frac{\cT(z)}{4\lambda} , \rho \right), $$
where we have used the fact that the coordinate $\theta$ is periodic.
 
Let $z=(x,y)$ and $z_0=(x_0,y_0)$. 
Since the Hamiltonian function $\cP$ is affine on $\cX^e$, expanding $\cP$ about $z_0$ gives
\begin{align}
 \cP(z) &= \cP(z_0) + (x+y-2x_0)(2v_1+1) \nonumber \\
 &= \cP(z_0) + 2\lambda(2v_1+1)^2\rho. \label{eq:cP(z)_rho_expansion}
\end{align}
Similarly, since $\cT$ is affine on the interval $\cIe$:
 $$ \cT(z) = \cT(z_0) + 2\lambda(2v_1+1)^2\rho\cT^{\prime}(e). $$
As $z_0$ is a fixed point, it must satisfy (\ref{eq:z_0}).
Thus we obtain
\begin{align*}
 \eta^e (z^{\prime}) &= \left( \theta + \frac{1}{4} - \frac{\cT(z_0) + 2\lambda(2v_1+1)^2\rho\cT^{\prime}(e)}{4\lambda} , \rho \right) \\
 &= \left( \theta - \frac{1}{2} (2v_1+1)^2\rho\cT^{\prime}(e) , \rho \right),
\end{align*}
which completes the proof.
\end{proof}

\begin{figure}[t]
  \centering
  \includegraphics[scale=1]{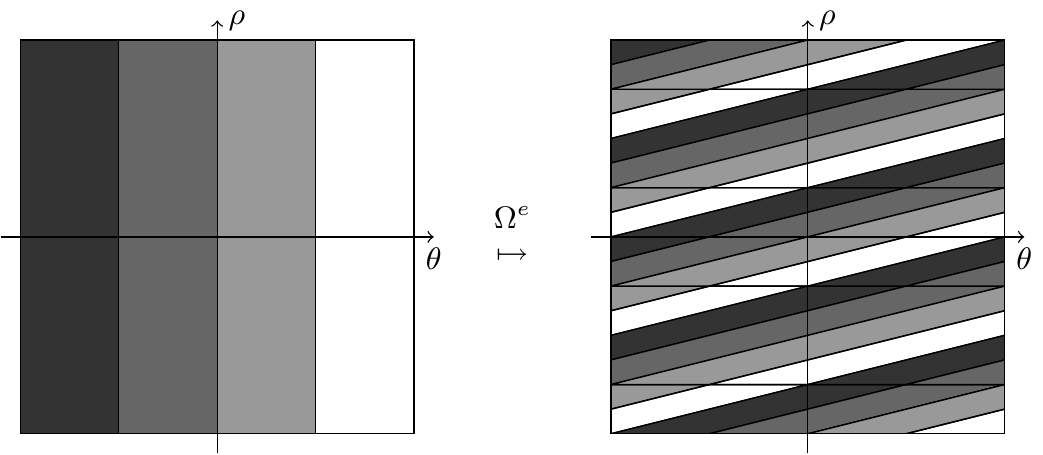} 
  \caption{A schematic representation of the action of $\Omega^e$ on the space $\bbS^1\times\R$}
\end{figure}

Note that the action of the conjugate map $\Omega^e$ is independent of the parameter $\lambda$, 
which appears only in the domain $\eta^e(\cX^e)$ in which the conjugate map is valid. 
Using the definition (\ref{def:cXe}) of the domain $\cX^e$ and the above expansion 
(\ref{eq:cP(z)_rho_expansion}) of $\cP$ about the fixed point $z_0$, 
we see that this domain is given by
\begin{displaymath}
 \eta^e(\cX^e) = \{ (\theta,\rho)\in\bbS^1 \times \R \, : \; \cP(z_0) + 2\lambda(2v_1+1)^2\rho \in \Ie \}.
\end{displaymath}
The range of values of $\rho$ in the domain grows like
\begin{displaymath}
 \frac{|\Ie|}{2\lambda(2v_1+1)^2} = \frac{|\cIe|}{2\lambda(2v_1+1)^2} + O(1)
\end{displaymath}
as $\lambda\rightarrow 0$, and in the integrable limit, we think of the conjugate map as valid on the whole of $\bbS^1\times\R$.

The map $\Omega^e$ is a linear twist map, and we \hl{denote the twist by}
\begin{equation} \label{def:K(e)}
 K(e) = -\frac{1}{2} (2v_1+1)^2 \cT^{\prime}(e) \hskip 40pt e\in\cE.
\end{equation}
Each $\Omega^e$ is reversible, and can be written as the composition of the involutions $\cG$ and $\cH^e$, where
\begin{equation*}
 \cG(\theta,\rho) = (-\theta,\rho) \hskip 40pt \cH^e(\theta,\rho) = \left( -\theta + K(e)\rho , \rho \right).
\end{equation*}
The fixed spaces of these involutions are given by 
\begin{align}
 \Fix{\cG} &= \{ (\theta, \rho) \in \bbS^1\times\R  \, : \; \theta\in\{-1/2,0\} \}, \nonumber \\ 
 \Fix{\cH^e} &= \{ (\theta, \rho) \in \bbS^1\times\R \, : \; \theta=\frac{1}{2} \, K(e)\rho \}. \label{eq:Fix(cH)} 
\end{align}
(Note the connection between the involution $\cG$ and the reversing symmetry $G^e$ of the perturbed return map $\Phi$, which was introduced in section \ref{sec:MainTheorems}.)

The map $\Omega^e$ is also reversible with respect to the reflection $(\theta,\rho)\mapsto(\theta,-\rho)$,
and equivariant under the group \hl{generated by the translation}
\begin{equation} \label{eq:rho_bar}
 \rho \mapsto \rho + \bar{\rho} \hskip 40pt \bar{\rho} = \frac{1}{K(e)} = \frac{-2}{(2v_1+1)^2\cT^{\prime}(e)}.
\end{equation}
Each circle $\rho=$ const. is invariant under $\Omega^e$, 
and motion restricted to this circle is a rotation with rotation number $\rho/\bar{\rho}$ (mod $1$).

\section{The limiting dynamics} \label{sec:limiting_dynamics} \label{SEC:LIMITING_DYNAMICS}

Now we are in a position to study the dynamics of the sequence of maps $\Omega^e$ in the limit $e\rightarrow\infty$. 
As in section \ref{sec:IntegrableReturnMap}, where we studied the nonlinearity of the flow, 
we turn our attention to the behaviour of the period function $\cT$.

\subsection*{The period function at infinity}

We wish to study the behaviour of the period $\cT(\alpha)$ of the Hamiltonian flow 
$\varphi$ in the limit $\alpha\rightarrow \infty$. 
As one would expect, the period of the piecewise-affine flow converges to the period $\pi$ of the underlying rotation. 
However, the period undergoes damped oscillations and, after a suitable rescaling to restore these oscillations, 
we find that the divergence of the period from its asymptotic value converges to a limiting functional form. 
We give this limiting form in the following theorem.

\begin{theorem} \label{thm:T_asymptotics}
Let $b\in[0,1)$, and let $\alpha=\alpha(v_k,b)$ be given by:
\begin{equation} \label{def:b}
 \alpha = (v_k+b)^2 \hskip 40pt v_k\in\N,
\end{equation}
so that $b$ is the fractional part of $\sqrt{\alpha}$ and $v_k$ is the integer part.
Then as $v_k\rightarrow \infty$:
\begin{equation} \label{eq:T_asymptotics}
 v_k^{3/2} \left(\frac{\cT(\alpha)-\pi}{4}\right) \rightarrow \frac{1}{3}(2b+1)^{3/2} - \sqrt{2b} - \epsilon(b),
\end{equation}
where $\epsilon$ is a bounded function.
\end{theorem}

In what follows, we will give an expression for the function $\epsilon$ and bound its range explicitly,
but for now we think of it as an error term which is small but non-vanishing as $\alpha\rightarrow \infty$. 
The first two terms are sufficient to give an accurate qualitative description of the function (see figure \ref{fig:cT(alpha)}).
Furthermore, we claim without proof that the convergence in theorem \ref{thm:T_asymptotics} is uniform in $b$.

\begin{figure}[!h]
        \centering
        \includegraphics[scale=0.4]{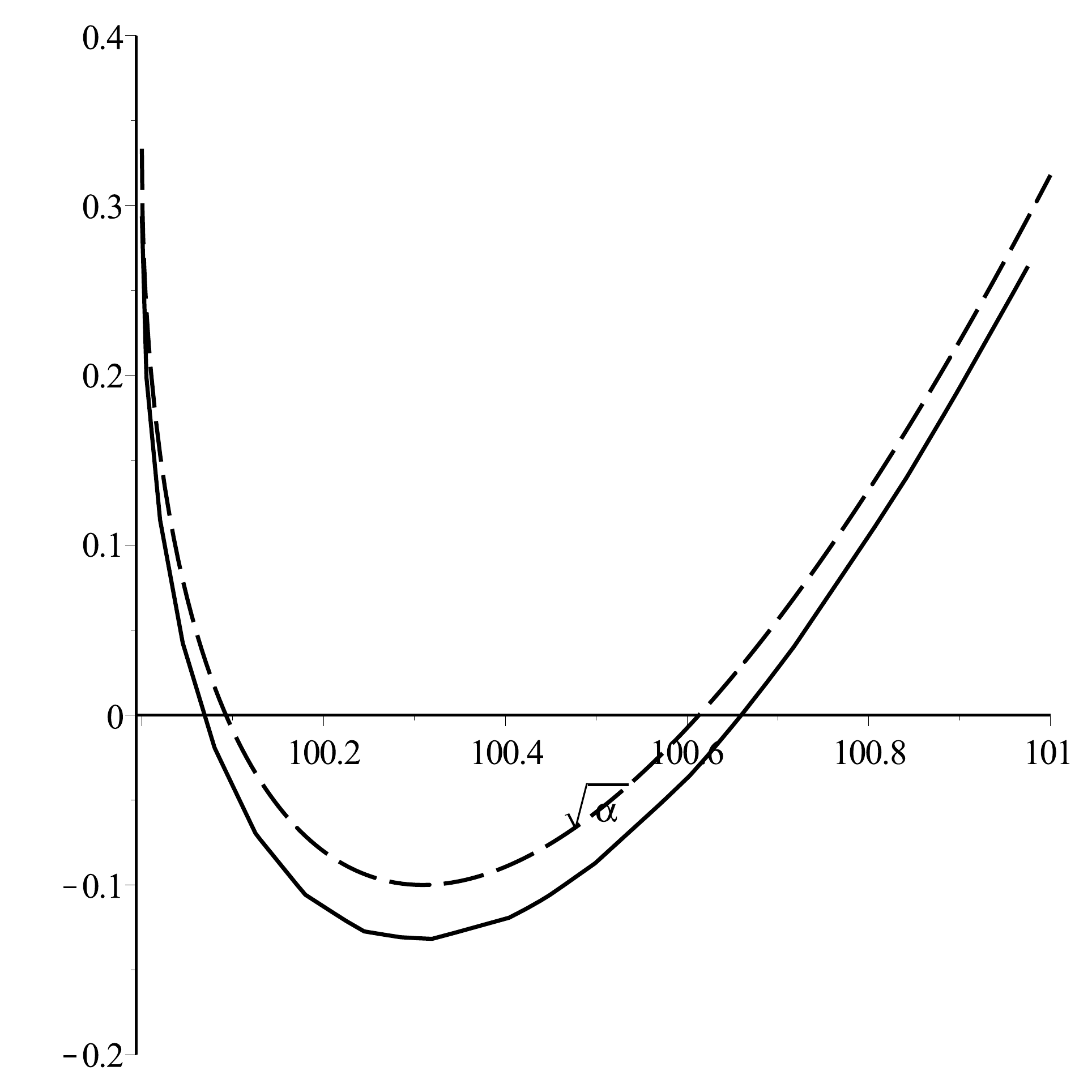} 
        \caption{This figure plots the function $v_k^{3/2} \left(\cT(\alpha)-\pi\right)/4$ (solid line) against $\sqrt{\alpha}$ for $\sqrt{\alpha}\in[100,101)$, i.e., for $v_k=100$ and $b\in[0,1)$. The dotted line shows the function $\frac{1}{3}(2b+1)^{3/2} - \sqrt{2b}$.}
        \label{fig:cT(alpha)}
\end{figure}

We prove this theorem via a number of lemmas, the proofs of which we postpone until the next section.
To begin, we consider the formula for the period function $\cT(\alpha)$, 
as given in proposition \ref{prop:T(alpha)} (page \pageref{prop:T(alpha)}):
\begin{equation} \label{eq:cT(alpha)_II}
 \frac{\cT(\alpha)}{8} = \frac{P^{-1}(\alpha/2)}{2v_1+1} -2 \sum_{n=v_1+1}^{v_k} \frac{P^{-1}(\alpha - n^2)}{4n^2-1}.
\end{equation}
The function $P^{-1}$, defined in (\ref{def:Pinv}), admits the alternative expression:
\begin{equation} \label{eq:Pinv2}
 P^{-1}(x) = \sqrt{x} - \frac{\{\sqrt{x}\}(1-\{\sqrt{x}\})}{2\fl{ \sqrt{x} } + 1},
\end{equation}
where $\{ x \} = x-\fl{x}$ represents the fractional part of a real number $x$. 
For large argument, $P^{-1}$ is well approximated by a square-root. 
We use this fact to approximate the summand in (\ref{eq:cT(alpha)_II}).

\begin{lemma} \label{lemma:f_sum} \label{LEMMA:F_SUM}
As $\alpha\rightarrow\infty$, we have:
 \begin{equation} \label{eq:sqrt_estimate}
 \sum_{n=v_1+1}^{v_k} \left( \frac{P^{-1}(\alpha - n^2)}{4n^2-1} - \frac{\sqrt{\alpha - n^2}}{4n^2} \right) = O\left(\frac{1}{\alpha}\right),
\end{equation}
where $v_1=\fl{\sqrt{\alpha/2}}$ and $v_k=\fl{\sqrt{\alpha}}$.
\end{lemma}

Then we approximate the sum in equation (\ref{eq:sqrt_estimate}) with an integral. 
To do this, we note that for any function $f$ which is integrable on the interval $[v_1,v_k]$, we can re-write the sum over $f$ as:
\begin{equation} \label{eq:sum_int_expansion}
 \sum_{n=v_1+1}^{v_k} f(n) = \int_{v_1+1/2}^{v_k-1/2} f(x) \,dx + f(v_k) + \sum_{n=v_1+1}^{v_k-1} \int_{n-1/2}^{n+1/2} f(n) - f(x) \, dx.
\end{equation}
All but one of the terms in the sum are approximated by an integral, with the sum over integrals constituting the error in this approximation.
The remaining term---$f(v_k)$---cannot be approximated in this way since the interval on which $f$ is integrable does not allow.

We apply this formula to $f(x) = x^{-2}\sqrt{\alpha-x^2}$. Recall from (\ref{def:b}) that $b$ denotes the fractional part of $\sqrt{\alpha}$. We write $a$ for the fractional part of $\sqrt{\alpha/2}$:
\begin{equation} \label{def:a}
 \frac{\alpha}{2} = (v_1+a)^2 \hskip 40pt a\in[0,1),
\end{equation}
which gives the following expression for the behaviour of the integral in (\ref{eq:sum_int_expansion}).

\begin{lemma} \label{lemma:f_integral} \label{LEMMA:F_INTEGRAL}
As $\alpha\rightarrow\infty$, we have:
\begin{equation*} 
\int_{v_1+1/2}^{v_k-1/2} \frac{\sqrt{\alpha - x^2}}{x^2} \, dx
= 1 - \frac{\pi}{4} + \frac{2a-1}{\sqrt{2\alpha}} - \frac{1}{3}\frac{(2b+1)^{3/2}}{\alpha^{3/4}} + O\left(\frac{1}{\alpha}\right),
\end{equation*}
where $a$ and $b$ denote the fractional parts of $\sqrt{\alpha/2}$ and $\sqrt{\alpha}$, respectively.
\end{lemma}

Finally we define the function $\epsilon(b)$ as the rescaled limit of the error term in the expression (\ref{eq:sum_int_expansion}).
We defer the proof of lemma \ref{lemma:epsilon_bounds} to appendix \ref{chap:Appendix}.

\begin{lemma} \label{lemma:epsilon_bounds}
For $b\in[0,1)$ and $v_k\in\N$, let $v_1=\fl{(v_k+b)/\sqrt{2}}$.
Then the following limit exists
\begin{equation} \label{eq:epsilon(b)}
 \epsilon(b) = \lim_{v_k\rightarrow\infty} \left( v_k^{3/2} \; \sum_{n=v_1+1}^{v_k-1} 
      \int_{n-1/2}^{n+1/2} \frac{\sqrt{(v_k+b)^2 - n^2}}{n^2} - \frac{\sqrt{(v_k+b)^2 - x^2}}{x^2} \, dx \right),
\end{equation}
and satisfies
\begin{equation*} 
 \frac{1}{36} \, \frac{1}{\sqrt{3(b+1)}} \leq \epsilon(b) \leq \frac{1}{12} \, \frac{1}{\sqrt{b+1}} \, \frac{2b+3}{2b+2}.
\end{equation*}
\end{lemma}

Using these three lemmas, we proceed with the proof of theorem \ref{thm:T_asymptotics}.

\begin{proof}[Proof of Theorem \ref{thm:T_asymptotics}]
The period $\cT(\alpha)$ of the flow on $\Pi(\alpha)$ is given by equation (\ref{eq:cT(alpha)_II}): 
combining this with lemma \ref{lemma:f_sum}, we have that $\cT(\alpha)$ satisfies
\begin{equation} \label{eq:cT(alpha)_asymptotics}
 \frac{\cT(\alpha)}{4} = \frac{2P^{-1}(\alpha/2)}{2v_1+1} - \sum_{n=v_1+1}^{v_k} \frac{\sqrt{\alpha - n^2}}{n^2} + O\left(\frac{1}{\alpha}\right)
\end{equation}
as $\alpha\rightarrow\infty$.

To evaluate the first term of the above, we recall the definition (\ref{def:a}) of $v_1$ and $a$ as the integer and fractional parts of $\sqrt{\alpha/2}$, respectively, and apply the formula (\ref{eq:Pinv2}) for $P^{-1}$ to obtain:
\begin{align*}
\frac{P^{-1}(\alpha/2)}{2v_1+1} &=  \frac{\sqrt{\alpha/2}}{2v_1+1} + O\left(\frac{1}{\alpha}\right) \\
&=  \frac{v_1+a}{2v_1+1} + O\left(\frac{1}{\alpha}\right) \\
&=  \frac{1}{2} + \frac{a-1/2}{\sqrt{2\alpha}} + O\left(\frac{1}{\alpha}\right).
\end{align*}
For the sum, we apply the formula (\ref{eq:sum_int_expansion}) to $f(x) = x^{-2}\sqrt{\alpha-x^2}$, 
then use lemma \ref{lemma:f_integral} to get
\begin{align*}
 \sum_{n=v_1+1}^{v_k} \frac{\sqrt{\alpha - n^2}}{n^2} 
  &= 1 - \frac{\pi}{4} + \frac{2a-1}{\sqrt{2\alpha}} - \frac{1}{3}\frac{(2b+1)^{3/2}}{\alpha^{3/4}}  + \frac{\sqrt{\alpha - v_k^2}}{v_k^2} \\
 &\hskip 40pt + \sum_{n=v_1+1}^{v_k-1} \int_{n-1/2}^{n+1/2} \frac{\sqrt{\alpha - n^2}}{n^2} - \frac{\sqrt{\alpha - x^2}}{x^2} \, dx + O\left(\frac{1}{\alpha}\right).
\end{align*}
Using the definition (\ref{def:b}) of $b$, we observe that the $f(v_k)$ term behaves like:
\begin{align*}
 \frac{\sqrt{\alpha - v_k^2}}{v_k^2} &= \frac{\sqrt{2v_k b + b^2}}{v_k^2} \\
 &= \frac{\sqrt{2b}}{v_k^{3/2}} + O\left( \frac{1}{\alpha^{5/4}}\right) . 
\end{align*}

Substituting the above three expressions into (\ref{eq:cT(alpha)_asymptotics}), 
we find that the terms involving $a$ cancel, giving
\begin{align*}
 \frac{\cT(\alpha)-\pi}{4} &=  \frac{1}{3}\frac{(2b+1)^{3/2}}{\alpha^{3/4}} - \frac{\sqrt{2b}}{v_k^{3/2}} - \sum_{n=v_1+1}^{v_k-1} \int_{n-1/2}^{n+1/2} \frac{\sqrt{\alpha - n^2}}{n^2} - \frac{\sqrt{\alpha - x^2}}{x^2} \, dx + O\left(\frac{1}{\alpha}\right).
\end{align*}
We can replace the $\alpha^{3/4}$ in the denominator of the first term by $v_k^{3/2}$, since
 $$ \alpha^{-3/4} = (v_k+b)^{-3/2} = v_k^{-3/2} \left( 1 + O\left(\frac{1}{\sqrt{\alpha}}\right) \right). $$
Then multiplying by $v_k^{3/2}$, taking the limit, and noting the definition (\ref{eq:epsilon(b)}) of $\epsilon(b)$ gives the formula (\ref{eq:T_asymptotics}), as required.
The boundedness of $\epsilon(b)$ follows from lemma \ref{lemma:epsilon_bounds}.
\end{proof}

The key feature of the limiting distribution (\ref{eq:T_asymptotics}) is the singularity in its derivative at $b=0$, 
since it is the derivative of the period function which determines the behaviour of the integrable return map.

\subsection*{The integrable return map at infinity}

Recall the map $\Omega^e$ of theorem \ref{thm:Omega_e}: 
a linear twist map whose twist $K(e)$ (equation (\ref{def:K(e)})) 
determines the nonlinearity of the integrable flow $\varphi$ in $\cX^e$.
In this section we show that the twist is singular in the limit $e\rightarrow\infty$, 
and thus that there are two distinct regimes of asymptotic behaviour.

\begin{proposition} \label{prop:Tprime_asymptotics}
Let $b\in[0,1)$ and let $\alpha=\alpha(v_k,b)$ be given as in (\ref{def:b}).
Then as $v_k\rightarrow \infty$:
\begin{equation} \label{eq:Tprime_asymptotics}
 \frac{1}{2} (2v_1+1)^2 \cT^{\prime}(\alpha(v_k,b)) \rightarrow -4\delta_0(b),
\end{equation}
where $v_1=\fl{(v_k+b)/\sqrt{2}}$, and $\delta_0$ is the indicator function at zero:
 $$ \delta_0(x) = \left\{ \begin{array}{ll} 1 \quad & x=0, \\ 0 \quad & \mathrm{otherwise.}\end{array} \right. $$
\end{proposition}

The convergence (\ref{eq:Tprime_asymptotics}) is not uniform: 
if $b$ is close to zero, the convergence can be made arbitrarily slow.
For a plot of $(2v_1+1)^2 \cT^{\prime}(\alpha)/2$, see figure \ref{fig:rho_bar}(a), page \pageref{fig:rho_bar}.

The function $\cT^{\prime}$ is piecewise-constant on the sequence of intervals $\cIe$, $e\in\cE$.
To observe the convergence (\ref{eq:Tprime_asymptotics}) in the sequence of twists $K(e)$ for a given value of $b\in[0,1)$,
we define the subsequence of critical numbers $e(v_k,b)$ satisfying
\begin{equation} \label{def:e(vk,b)}
 \alpha(v_k,b)=(v_k+b)^2\in \cI^{e(v_k,b)} \hskip 40pt v_k\in\N.
\end{equation}
Then as $v_k\to\infty$, $K(e(v_k,b))\to 4\delta_0(b)$, as above.
The two regimes of behaviour for the sequence of return maps $\Omega^e$ follow directly.

\begin{corollary} \label{corollary:Omega_asymptotics}
Let $b\in[0,1)$ and let $e(v_k,b)$ be as above.
If $b=0$, then the sequence $e(v_k,b)$ is simply the sequence of squares,
and as $v_k\rightarrow\infty$, the sequence of functions $\Omega^{e(v_k,b)}$ converges pointwise to a limiting function $\Omega^{\infty}$ given by
\begin{equation*}
 \Omega^{\infty}: \bbS^1\times\R \rightarrow \bbS^1\times\R
 \hskip 40pt 
 \Omega^{\infty}(\theta,\rho) = \left( \theta + 4\rho, \rho \right).
\end{equation*}
If $b>0$, then the sequence $\Omega^{e(v_k,b)}$ converges pointwise to the identity.
\end{corollary}

The (reversing) symmetries of $\Omega^{e(v_k,b)}$ also converge as $v_k\rightarrow\infty$, 
leading, for example, to an asymptotic version of $\cH^e$. 
Here we note in particular the translation invariance (\ref{eq:rho_bar}) of the sequence $\Omega^{e(v_k,b)}$, \hl{whose magnitude satisfies}
\begin{equation} \label{eq:rho_bar_asymptotics}
 |\bar{\rho}| \rightarrow \left\{ \begin{array}{ll} 1/4 & \quad b=0 \\ \infty & \quad b>0 \end{array} \right.
\end{equation}
as $v_k\rightarrow\infty$.

\section{Proofs for section \ref{sec:limiting_dynamics}}

In the following proofs, we make extensive use of Taylor's Theorem (see, for example, \cite[Theorem 4.82]{Burkill}).

\begin{proof}[Proof of lemma \ref{lemma:f_sum}]
For any $n$ in the range $v_1+1\leq n\leq v_k$, the alternative form (\ref{eq:Pinv2}) of the function $P^{-1}$ gives us that
\begin{align*}
 \left| \, \frac{P^{-1}(\alpha-n^2)}{4n^2-1} -\frac{\sqrt{\alpha-n^2}}{4n^2} \, \right|
	 &= \left| \, \frac{\sqrt{\alpha - n^2}}{4n^2(4n^2-1)}  -
	 \frac{\{\sqrt{\alpha - n^2}\}(1-\{\sqrt{\alpha - n^2}\})}{(4n^2-1)(2\fl{ \sqrt{\alpha - n^2} } + 1)} \, \right|\\
	 &\leq \frac{\sqrt{\alpha - n^2}}{4n^2(4n^2-1)} + \frac{1/4}{(4n^2-1)(2\fl{ \sqrt{\alpha - n^2} } + 1)},
\end{align*}
where the inequality follows from the triangle inequality, and from the observation that
 $$ 0 \leq x(1-x) \leq 1/4 \hskip 40pt x\in[0,1]. $$
Furthermore, since $n\geq 1$, we have $4n^2-1 \geq 3n^2$, and hence
\begin{equation} \label{eq:f_sum_inequalityI}
 \left| \, \frac{P^{-1}(\alpha-n^2)}{4n^2-1} -\frac{\sqrt{\alpha-n^2}}{4n^2} \, \right|
	 \leq \frac{1}{12n^2} \left( \frac{\sqrt{\alpha - n^2}}{n^2} + \frac{1}{2\fl{ \sqrt{\alpha - n^2} } + 1} \right).
\end{equation}

Now we have two cases. For $n=v_k$, the square root $\sqrt{\alpha - n^2}$ satisfies
 $$ 0 \leq \sqrt{\alpha - v_k^2} < \sqrt{2v_k+1}, $$
so we observe from the inequality (\ref{eq:f_sum_inequalityI}) that
 $$ \left| \, \frac{P^{-1}(\alpha-v_k^2)}{4v_k^2-1} -\frac{\sqrt{\alpha-v_k^2}}{4v_k^2} \, \right|
	 < \frac{1}{12v_k^2} \left( \frac{\sqrt{2v_k+1}}{v_k^2} + 1 \right) = O\left(\frac{1}{\alpha}\right). $$
For $n<v_k$, the square root satisfies
 $$ 0 < \sqrt{\alpha - n^2} < 2\fl{\sqrt{\alpha - n^2}} + 1, $$
in which case the inequality (\ref{eq:f_sum_inequalityI}) gives
  $$ \left| \, \frac{P^{-1}(\alpha-n^2)}{4n^2-1} -\frac{\sqrt{\alpha-n^2}}{4n^2} \, \right|
	 <  \frac{1}{12n^2\sqrt{\alpha - n^2}} \left( \frac{\alpha - n^2}{n^2} + 1 \right) < \frac{1}{6n^2\sqrt{\alpha - n^2}}, $$
where the last inequality uses the fact that $n\geq v_1+1>\sqrt{\alpha/2}$.

Summing over $n$, we obtain
\begin{align} \label{eq:f_sum}
  \left| \, \sum_{n=v_1+1}^{v_k} \left( \frac{P^{-1}(\alpha-n^2)}{4n^2-1} -\frac{\sqrt{\alpha-n^2}}{4n^2} \right) \, \right|
  &\leq  \sum_{n=v_1+1}^{v_k} \,  \left| \, \frac{P^{-1}(\alpha-n^2)}{4n^2-1} -\frac{\sqrt{\alpha-n^2}}{4n^2} \, \right| \nonumber \\
	 &< \sum_{n=v_1+1}^{v_k-1} \, \left( \frac{1}{6n^2\sqrt{\alpha - n^2}} \right) + O\left(\frac{1}{\alpha}\right).
\end{align}

We can approximate the sum on the right hand side of (\ref{eq:f_sum}) with an integral:
\begin{align}
 \sum_{n=v_1+1}^{v_k-1} \frac{1}{n^2\sqrt{\alpha - n^2}} 
 &= \int_{v_1+1/2}^{v_k-1/2} \frac{1}{x^2\sqrt{\alpha - x^2}} \left(1 + O\left(\frac{1}{\alpha^{1/4}}\right) \right) \, dx \nonumber \\
 &= \frac{1}{\alpha} \, \Big[ \tan{\theta} \Big]_{\theta_2}^{\theta_1} \left(1 + O\left(\frac{1}{\alpha^{1/4}}\right) \right) \nonumber \\
 &= \frac{1}{\alpha} \left(1 + O\left(\frac{1}{\alpha^{1/4}}\right) \right), \label{eq:Tprime_bound2}
\end{align}
where we have used the substitution $x=\sqrt{\alpha}\cos{\theta}$, and $\tan{\theta_1}$, $\tan{\theta_2}$ are given by
\begin{align*}
 \tan{\theta_1} &= 1 + O\left(\frac{1}{\sqrt{\alpha}}\right) \\
 \tan{\theta_2} &= O\left(\frac{1}{\alpha^{1/4}}\right)
\end{align*}
(see equations (\ref{eq:tan(theta1)}) and (\ref{eq:tan(theta2)}) of the next proof).
Thus, combining (\ref{eq:Tprime_bound2}) and (\ref{eq:f_sum}), we have
 $$ \sum_{n=v_1+1}^{v_k} \left( \frac{P^{-1}(\alpha-n^2)}{4n^2-1} -\frac{\sqrt{\alpha-n^2}}{4n^2} \right)
  = O\left(\frac{1}{\alpha}\right). $$
\end{proof}

\medskip

\begin{proof}[Proof of lemma \ref{lemma:f_integral}]
We treat the integral using the substitution $x=\sqrt{\alpha}\cos{\theta}$, which gives:
\begin{align} \label{eq:tan_integral}
\int_{v_1+1/2}^{v_k-1/2} \frac{\sqrt{\alpha-x^2}}{x^2} \, dx
&= \int_{\theta_2}^{\theta_1} \tan^2\theta  \, d\theta \nonumber \\
&= \Big[ \tan\theta -\theta \, \Big]_{\theta_2}^{\theta_1},
\end{align}
where the limits $\theta_1$ and $\theta_2$ satisfy:
\begin{align}
 \cos{\theta_1} &= \frac{v_1+1/2}{\sqrt{\alpha}} = \frac{1}{\sqrt{2}}\left(  1 - \frac{a-1/2}{\sqrt{\alpha/2}} \right), \label{eq:cos1}\\
 \cos{\theta_2} &= \frac{v_k-1/2}{\sqrt{\alpha}} = 1 - \frac{b+1/2}{\sqrt{\alpha}}.  \label{eq:cos2}
\end{align}

Using Taylor's theorem, we have that
 $$ \cos^{-1}\left(\frac{1-x}{\sqrt{2}}\right) =  \frac{\pi}{4} + x + O(x^2) $$
as $x\rightarrow0$. Thus if we let $x=(a-1/2)/\sqrt{\alpha/2}$, then (\ref{eq:cos1}) gives that
\begin{equation} \label{eq:theta1}
 \theta_1 = \frac{\pi}{4} + \frac{a-1/2}{\sqrt{\alpha/2}} + O\left(\frac{1}{\alpha}\right)
\end{equation}
as $\alpha\rightarrow\infty$.

Next we have
 $$ \frac{\cos^{-1}(1-x)}{\sqrt{2x}} = 1 + \frac{x}{12} + O(x^2) $$
as $x\rightarrow0$. Applying this to (\ref{eq:cos2}) with $x=(b+1/2)/\sqrt{\alpha}$ gives
\begin{equation} \label{eq:theta2}
 \theta_2 = \frac{\sqrt{2b+1}}{\alpha^{1/4}}\left( 1 + \frac{1}{24} \frac{2b+1}{\sqrt{\alpha}} \right)
 + O\left(\frac{1}{\alpha^{5/4}}\right)
\end{equation}
as $\alpha\rightarrow\infty$.

Similarly, we find:
\begin{align}
 \tan{\theta_1} &= 1 + \frac{2a-1}{\sqrt{\alpha/2}}
      + O\left(\frac{1}{\alpha}\right), \label{eq:tan(theta1)}\\
 \tan{\theta_2} &= \frac{\sqrt{2b+1}}{\alpha^{1/4}}\left( 1 + \frac{3}{8} \frac{2b+1}{\sqrt{\alpha}} \right) 
      + O\left(\frac{1}{\alpha^{5/4}}\right). \label{eq:tan(theta2)}
\end{align}

Substituting the expressions (\ref{eq:theta1}), (\ref{eq:theta2}), (\ref{eq:tan(theta1)}) and (\ref{eq:tan(theta2)}) into the integral (\ref{eq:tan_integral}) gives the required result.
\end{proof}

\medskip

\begin{proof}[Proof of proposition \ref{prop:Tprime_asymptotics}]
From the formula (\ref{eq:Tprime(alpha)}) for the derivative of the period function, if $\alpha=(v_k+b)^2$ then
\begin{equation} \label{eq:Tprime(alpha)_II}
 \frac{1}{2} (2v_1+1)^2 \cT^{\prime}(\alpha) = 2 -8(2v_1+1)^2 \sum_{n=v_1+1}^{v_k} \frac{1}{(4n^2-1)(2\fl{\sqrt{\alpha - n^2}}+1)}.
\end{equation}
(Recall that if $\alpha\in\cIe$ for some $e\in\cE$, the floor function in the denominator of the summand satisfies $\fl{\sqrt{\alpha - n^2}} = \fl{\sqrt{e - n^2}}$.)

We consider the summand in (\ref{eq:Tprime(alpha)_II}). For $n$ in the range $v_1+1\leq n\leq v_k-1$, we have
\begin{equation*}
 \frac{1}{(4n^2-1)(2\fl{\sqrt{\alpha - n^2}}+1)} =  \frac{1}{8n^2\sqrt{\alpha - n^2}} \, \left( 1 - \frac{1}{4n^2} \right)^{-1} \left( 1 - \frac{\{\sqrt{\alpha - n^2}\}-1/2}{\sqrt{\alpha - n^2}} \right)^{-1}.
\end{equation*}
Here the square root is bounded below by
 $$ \sqrt{\alpha - n^2} \geq \sqrt{v_k^2 - (v_k-1)^2} = \sqrt{2v_k -1} > \alpha^{1/4}, $$
whereas $n^2 >\alpha/2$. Hence as $\alpha\rightarrow\infty$, we have
\begin{equation*} 
 \frac{1}{(4n^2-1)(2\fl{\sqrt{\alpha - n^2}}+1)} = \frac{1}{8n^2\sqrt{\alpha - n^2}} \left(1 + O\left(\frac{1}{\alpha^{1/4}}\right) \right).
\end{equation*}

Summing over $n$, we can use (\ref{eq:Tprime_bound2}) from the previous proof to obtain
\begin{align*}
 & 8(2v_1+1)^2 \sum_{n=v_1+1}^{v_k-1} \frac{1}{(4n^2-1)(2\fl{\sqrt{\alpha - n^2}}+1)} \\
 &= (2v_1+1)^2 \sum_{n=v_1+1}^{v_k-1} \frac{1}{n^2\sqrt{\alpha - n^2}} \left(1 + O\left(\frac{1}{\alpha^{1/4}}\right) \right) \\
 &= \frac{1}{\alpha} (2v_1+1)^2 \left(1 + O\left(\frac{1}{\alpha^{1/4}}\right) \right) \\
 &= 2 + O\left(\frac{1}{\alpha^{1/4}} \right)
\end{align*}
as $\alpha\rightarrow\infty$. Thus substituting this into (\ref{eq:Tprime(alpha)_II}), we see that only the $n=v_k$ term remains:
\begin{align} 
 \frac{1}{2} (2v_1+1)^2 \cT^{\prime}(\alpha) 
 &= \frac{-8(2v_1+1)^2}{(4v_k^2-1)\left(2\Bfl{\sqrt{\alpha-v_k^2}}+1\right)} + O\left(\frac{1}{\alpha^{1/4}} \right) \nonumber \\
 &= \frac{-4}{2\fl{\sqrt{2v_kb+b^2}}+1} \left(1 + O\left(\frac{1}{\sqrt{\alpha}}\right) \right) + O\left(\frac{1}{\alpha^{1/4}} \right). \label{eq:Tprime_final_term}
\end{align}
If $b>0$, then
\begin{equation*}
 \frac{1}{2\fl{\sqrt{2v_kb+b^2}}+1} \rightarrow 0
\end{equation*}
as $v_k\rightarrow\infty$. 
Thus the remaining term in (\ref{eq:Tprime_final_term}) goes to zero and $(2v_1+1)^2 \cT^{\prime}(\alpha)$ vanishes in the limit.
However, if $b=0$, then we have
 $$ \frac{1}{2} (2v_1+1)^2 \cT^{\prime}(\alpha) = -4 + O\left(\frac{1}{\alpha^{1/4}} \right). $$
\end{proof}

%% file: PerturbedDynamicsAtInfinity.tex
\chapter{The perturbed dynamics at infinity} \label{chap:PerturbedAsymptotics} \label{CHAP:PERTURBEDASYMPTOTICS}

Now we revert to the perturbed dynamics, and the return map $\Phi$ defined in section \ref{sec:Recurrence}. 
In chapter \ref{chap:PerturbedDynamics}, we showed that in the integrable limit $\lambda\to 0$, 
the map $\Phi$ has a natural finite structure over $\Xe$ for all $e\in\cE$: 
it is equivariant under the group of translations generated by the lattice $\lambda\Le$ 
(theorem \ref{thm:Phi_equivariance}, page \pageref{thm:Phi_equivariance}). 
Furthermore, in the previous chapter, we saw that under a suitable change of coordinates, 
the unperturbed motion corresponds to a linear twist map $\Omega^e$ on the cylinder 
(theorem \ref{thm:Omega_e}, page \pageref{thm:Omega_e}). 
The behaviour of the twist of $\Omega^e$ was shown to be singular in the limit $e\to\infty$.

We begin this chapter by reconciling these two features of $\Phi$---the lattice structure and the underlying twist map---and 
giving a qualitative description of the dynamics in the limit $e\to\infty$. 
We find that, under the aforementioned change of coordinates, 
the dynamics are that of a sequence of discrete twist maps with vanishing discretisation length. 
In the regime where the underlying twist vanishes in the limit, 
the remaining fluctuations result in intricate resonance structures, 
reminiscent of the island chains observed in Hamiltonian perturbation theory. 
Conversely, in the regime where the twist persists, 
there are no discernible phase space features.

In section \ref{sec:pdf} we turn to the period distribution function of $\Phi$. 
The reduction of $\Phi$ modulo the sequence of lattices $\lambda\Le$ 
provides a natural finite setting in which we can compare the periods of $\Phi$ 
to those of the random reversible map discussed in theorem \ref{thm:GammaDistribution} 
(page \pageref{thm:GammaDistribution}).
However, we find that the reduction is superfluous: 
the number of congruence classes of $\Z^2/\,\Le$ grows much faster than the range of a typical orbit. 
Furthermore, in the regime where the dynamics of $\Phi$ are asymptotically uniform, 
the period distribution function of $\Phi$ is well approximated by local data, 
calculated over a vanishing subset of congruence classes.

We describe an extensive numerical experiment in which we calculate the period distribution function of $\Phi$ 
for increasing values of $e$. 
The results suggests that, as $e\to\infty$, 
the distribution of periods approaches the universal distribution $\cR(x)$ of equation (\ref{def:R(x)}),
and thus is consistent with random reversible dynamics. 
As in the case of the random reversible map, we observe that symmetric periodic orbits dominate.

Throughout this chapter we adopt an informal approach, focussing on qualitative observations and numerical evidence.

\section{Qualitative description and phase plots}

Recall the map $\Omega^e$ (equation (\ref{def:Omega^e})), 
which corresponds to the unperturbed return map under the change of coordinates 
$\eta^e(\lambda):(x,y)\mapsto(\theta,\rho)$ of equation (\ref{def:rho_theta}). 
In the integrable limit $\lambda\to 0$, the image of the return domain 
$\cX^e$ under $\eta^e$ approaches the unit cylinder $\bbS^1\times\R$.
The domain $\Xe$ of the perturbed return map $\Phi$ is a subset of $\cX^e$, 
thus we can also apply the change of coordinates $\eta^e$ to $\Xe$.
Using the definition (\ref{def:Xe}) of $\Xe$, we see that its image under $\eta^e$ 
approaches the following set\footnote{Recall that in the definition of $\eta^e$, 
the preimage of the origin is some fixed point $z_0$ of the unperturbed dynamics. 
In the discrete case we round $z_0$ onto the lattice via the function $R_{\lambda}$ of equation (\ref{eq:R_lambda}). 
The choice of the fixed point $z_0$ will have no bearing on the qualitative or statistical properties of the dynamics.}
(see figure \ref{fig:Lattice2}):
\begin{equation} \label{eq:rho_theta_lattice}
 \left\{ \frac{1}{2(2v_1+1)} \, (i,i+2j) \,: \; i,j\in\Z, \; -(2v_1+1)\leq i < 2v_1+1 \right\} \subset \bbS^1\times\R,
\end{equation}
where $v_1$ of equation (\ref{def:v1}) is the integer part of $\sqrt{e/2}$.
(The image of the point $\lambda(x,y)$ corresponds to $i=x-y$ and $j=y-x_0$.)
This set is a rotated square lattice in the unit cylinder,
\hl{where the identification of points modulo $\langle \lambda\bfw_{v_1,v_1}\rangle$ 
in $(x,y)$ coordinates is reflected in the periodicity of the angular coordinate $\theta$.
Adjacent lattice points are separated by a lattice spacing of $1/\sqrt{2}(2v_1+1)$,
so that $e\to\infty$ corresponds to the continuum limit. }
We think of the dynamics of $\Phi$ as taking place in the $(\theta,\rho)$ coordinates, 
and of $\Phi$ as a discrete version of $\Omega^e$.

\begin{figure}[h]
  \centering
  \includegraphics[scale=0.8]{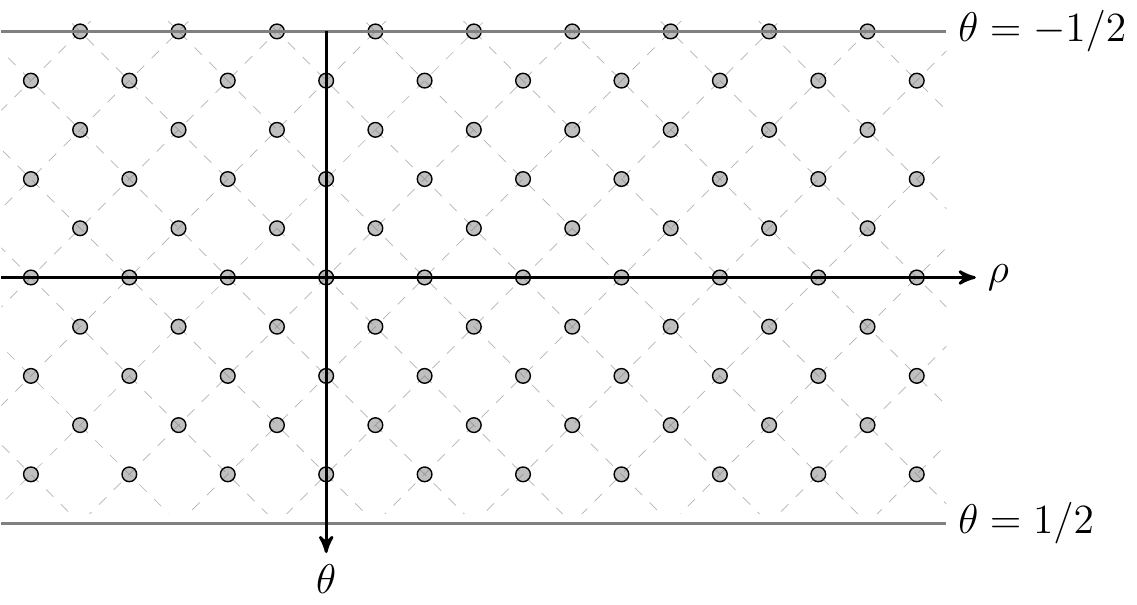} 
  \caption{The image of the lattice $\lZ$ under the change of coordinates $\eta^e(\lambda)$. 
  There are $2v_1+1$ lattice points per unit length in each of the coordinate directions.}
  \label{fig:Lattice2}
\end{figure}

The map $\Omega^e$ is a linear twist map, with characteristic length scale $\brho$ in the $\rho$-direction 
(see equation (\ref{eq:rho_bar})).
We compare this to the characteristic length scale of $\Le$: 
the lattice of section \ref{sec:MainTheorems} which characterises the symmetry properties of $\Phi$. 
In particular, $\Phi$ is invariant under translation by the vector $\lambda\bfL$ of (\ref{def:L}) which, 
in the $(\theta,\rho)$ coordinates, corresponds to the translation
\begin{equation} \label{eq:rho_tilde}
 \rho \mapsto \rho + \trho  \hskip 40pt \trho = \frac{q(e)}{(2v_1+1)^2}.
\end{equation}
A careful inspection of the definitions (\ref{eq:q}) of $q$ and (\ref{eq:Tprime(alpha)}) 
of $\cT^{\prime}(e)$ confirms that $\trho$ is an integer multiple of $\brho$, 
and hence that the group of symmetries of $\Phi$ generated by this translation form a 
subgroup of those of $\Omega^e$ generated by the translation (\ref{eq:rho_bar}).

\begin{proposition}
 For $e\in\cE$, let $\brho$ be the periodicity of $\Omega^e$ in the $\rho$-direction, as given by (\ref{eq:rho_bar}),
 and let $\trho$ be corresponding periodicity of $\Phi$ on $\eta^e(\Xe)$, as given by (\ref{eq:rho_tilde}).
 Then $\trho$ is an integer multiple of $\brho$.
\end{proposition}

\begin{proof}
Let $e\in\cE$ and $\alpha\in\cIe$, 
so that $\cT^{\prime}(e)=\cT^{\prime}(\alpha)$, and 
$\Pi(\alpha)$ belongs to the polygon class associated with $e$.

By (\ref{eq:rho_tilde}), $\trho$ is given by 
 $$ \trho = \frac{q(e)}{(2v_1+1)^2}. $$
The integer $q(e)$, given by (\ref{eq:q}), is the lowest common multiple of $(2v_1+1)^2$ and 
the sequence of factors $(2v_j+1)(2v_{j+1}+1)$, where $v_j$ and $v_{j+1}$ are consecutive
distinct vertex types of the vertex list $V(e)$.
Combining this with the definition (\ref{eq:rho_bar}) of $\brho$ 
and the formula (\ref{eq:Tprime(alpha)}) for $\cT^{\prime}(\alpha)$,
we have that the ratio $\trho/\brho$ is given by 
 $$ \frac{\trho}{\brho} = -\frac{q(e)}{2} \, \cT^{\prime}(e) 
 = -2q(e) \left( \frac{1}{(2v_1+1)^2} -4 \sum_{n=v_1+1}^{v_k} \frac{1}{(4n^2-1)(2\fl{\sqrt{e - n^2}}+1)} \right). $$
We claim that the denominator of every term in the bracketed sum divides $q(e)$.

To see that our claim holds, we note first that $(2v_1+1)^2$ divides $q(e)$ by construction.
Then, for every $n$ in the range $v_1+1\leq n \leq v_k$, we can write 
 $$ (4n^2-1)(2\fl{\sqrt{e - n^2}}+1) = (2n+1)(2\fl{\sqrt{e - n^2}}+1)(2n-1). $$
We must show that this product divides $q(e)$ for each $n$.
The first factor in this product divides $q(e)$, since for each $n$ 
there is at least one vertex $(x,v)$ of $\Pi(\alpha)$ in the first octant 
which has vertex type $n$, i.e., with $\fl{x}=n$ and $v\in\Z$.
Let $v$ be the maximal integer for which this occurs.

Now consider the vertex of $\Pi(\alpha)$ which occurs prior to $(x,v)$.
This vertex has coordinates $(n,y)$ and type $v=\fl{y}$ (see figure \ref{fig:ConsecutiveVertices}), 
where the properties (\ref{eq:SqrtP}) of $P$ give us that
 $$ \alpha = n^2 + P(y)\in\cIe \hskip 20pt \Rightarrow \hskip 20pt v = \fl{\sqrt{\alpha-n^2}} = \fl{\sqrt{e-n^2}}. $$
Consequently $n$ and $\fl{\sqrt{e-n^2}}$ are consecutive distinct vertex types,
and the product
 $$ (2n+1)(2\fl{\sqrt{e - n^2}}+1) $$
divides $q(e)$ by construction.

\begin{figure}[!h]
  \centering
  \includegraphics[scale=1]{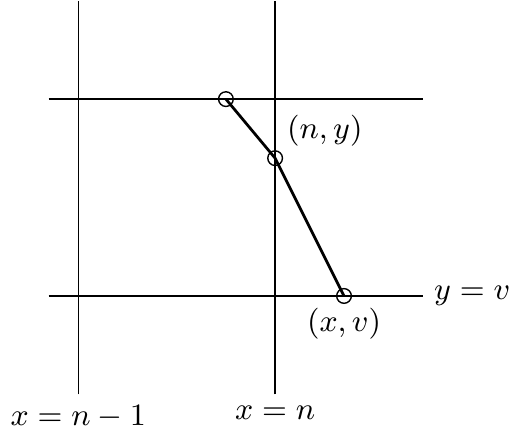} 
  \caption{Three consecutive vertices of the polygon $\Pi(\alpha)$.}
  \label{fig:ConsecutiveVertices}
\end{figure}

It remains to show that the product $(2\fl{\sqrt{e - n^2}}+1)(2n-1)$
divides $q(e)$: then the claim follows from the fact that $(2n+1)$ and $(2n-1)$
are consecutive odd numbers, and hence are coprime.
There are two cases to consider.
If $(n,y)$ is the first vertex, then $v=n-1=v_1$, and
 $$ (2\fl{\sqrt{e - n^2}}+1)(2n-1) = (2v_1+1)^2, $$
which we have already seen divides $q(e)$.
If $(n,y)$ is not the first vertex, then the vertex prior to this must be of type $n-1$.
Thus $n-1$ and $\fl{\sqrt{e-n^2}}$ are also consecutive distinct vertex types,
and the proof is complete.
\end{proof}

Thus we see that the symmetry properties of $\Phi$ are consistent with those of $\Omega^e$.
In fact, the characteristic length scale $\trho$ of the symmetry group of $\Phi$ 
is a diverging multiple of $\brho$---see table \ref{table:nu_range} (page \pageref{table:nu_range}).

Experimental observations confirm that $\Phi$ can be considered as a 
perturbation of $\Omega^e$, whose fluctuations originate from the discretisation.
These fluctuations are small relative to the width of the cylinder as $e\rightarrow\infty$, 
do not have any obvious structure, and can perturb the dynamics in both the $\rho$ and $\theta$ directions.

\medskip

Recall the sequence $e(v_k,b)$ of critical numbers defined in (\ref{def:e(vk,b)}).
By corollary \ref{corollary:Omega_asymptotics}, if $b\neq0$, then the sequence of maps $\Omega^{e(v_k,b)}$ 
converge non-uniformly to the identity as $v_k\to\infty$,
whereas if $b=0$, then $\Omega^{e(v_k,b)}$ converges to the map $(\theta,\rho)\mapsto(\theta+4\rho,\rho)$.
Accordingly, the characteristic length scale $|\brho|$ of the twist dynamics is also singular in the limit, 
\hl{with $|\brho|\to\infty$ or $|\brho|\to 1/4$}, for the $b\neq 0$ and $b=0$ cases, respectively 
(see equation \eqref{eq:rho_bar_asymptotics}).

We define the \defn{rotation number} $\nu$ of a point on the cylinder
to be its rotation number under the twist map $\Omega^e$, i.e.,
 $$ \nu(\theta,\rho) = \frac{\rho}{\brho} \mod{1} \hskip 40pt (\theta,\rho)\in\bbS^1\times\R. $$
Similarly for a point $z\in\Xe$, we write $\nu(z)$ to denote the rotation number 
of the corresponding point $\eta^e(z)$ on the cylinder.
Subsets of the cylinder of the form 
\begin{equation} \label{eq:fundamental_domain}
 \left\{ (\theta,\rho) \,: \; \nu(\theta,\rho) -m \in [-1/2,1/2) \right\}\subset \bbS^1\times\R \hskip 40pt m\in\Z 
\end{equation}
are referred to as \defn{fundamental domains} of the dynamics.
The number $N$ of lattice points of $\eta^e(\Xe)$ per fundamental domain varies like
\begin{equation} \label{eq:pts_per_fundamental_domain}
 N \sim 2(2v_1+1)^2|\brho|
\end{equation}
as $e\rightarrow\infty$.
For a given value of $e$, we expect the dynamics of $\Phi$ to be qualitatively the same
in each fundamental domain, but to vary locally according to the rotation number.
Consequently, in order to observe the global behaviour of $\Phi$, 
we need to sample whole fundamental domains.
However, this is made difficult by the divergence of $\brho$ for $b\neq 0$.

To better understand the divergence of $\brho$, we consider the limiting form 
(\ref{eq:T_asymptotics}) of the period function $\cT(\alpha)$.
By differentiating the limit and neglecting the term $\epsilon(b)$, and for large $v_k$,
we expect the piecewise-constant function $\cT^{\prime}(\alpha)$ to behave approximately as
 $$ \frac{1}{2}(2v_1+1)^2 \cT^{\prime}(\alpha(v_k,b)) \approx \frac{2}{\sqrt{v_k}} \left(\sqrt{2b+1}-1/\sqrt{2b}\right) 
    \hskip 40pt b\neq 0. $$
By (\ref{eq:rho_bar}), this leads us to expect that 
 $$ \brho(\alpha(v_k,b)) \approx \frac{\sqrt{v_k}}{2} \left(1/\sqrt{2b}-\sqrt{2b+1}\right)^{-1} \hskip 40pt b\neq 0. $$
Figure \ref{fig:rho_bar} shows that this rough analysis is valid, 
although in both cases the relationship between the two functions is far from uniform.

\begin{figure}[!h]
        \centering
        \begin{minipage}{7cm}
          \centering
	  \includegraphics[scale=0.35]{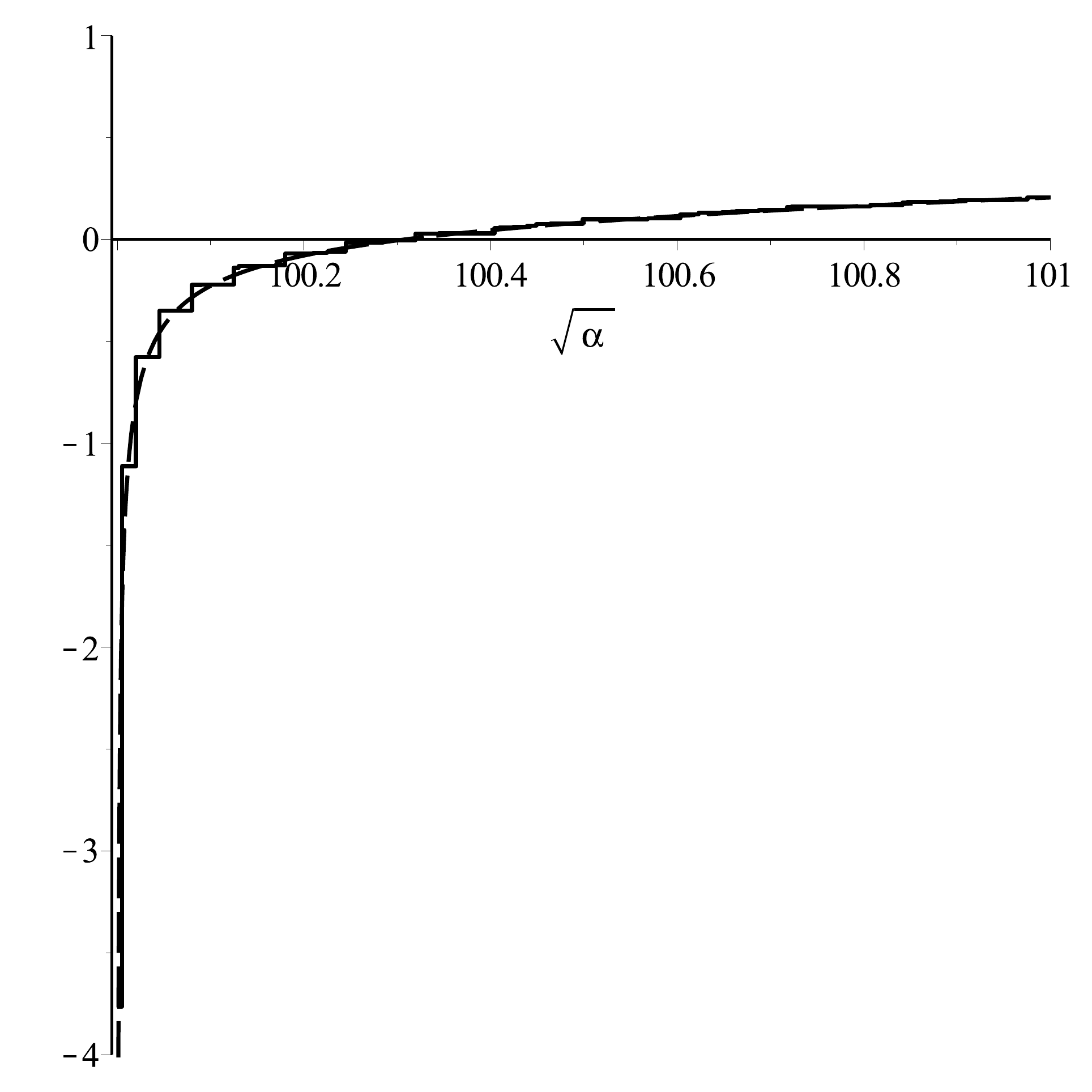} \\
	  (a) $\frac{1}{2}(2v_1+1)^2 \cT^{\prime}(\alpha)$ \\
        \end{minipage}
        \quad
        \begin{minipage}{7cm}
	  \centering
	  \includegraphics[scale=0.35]{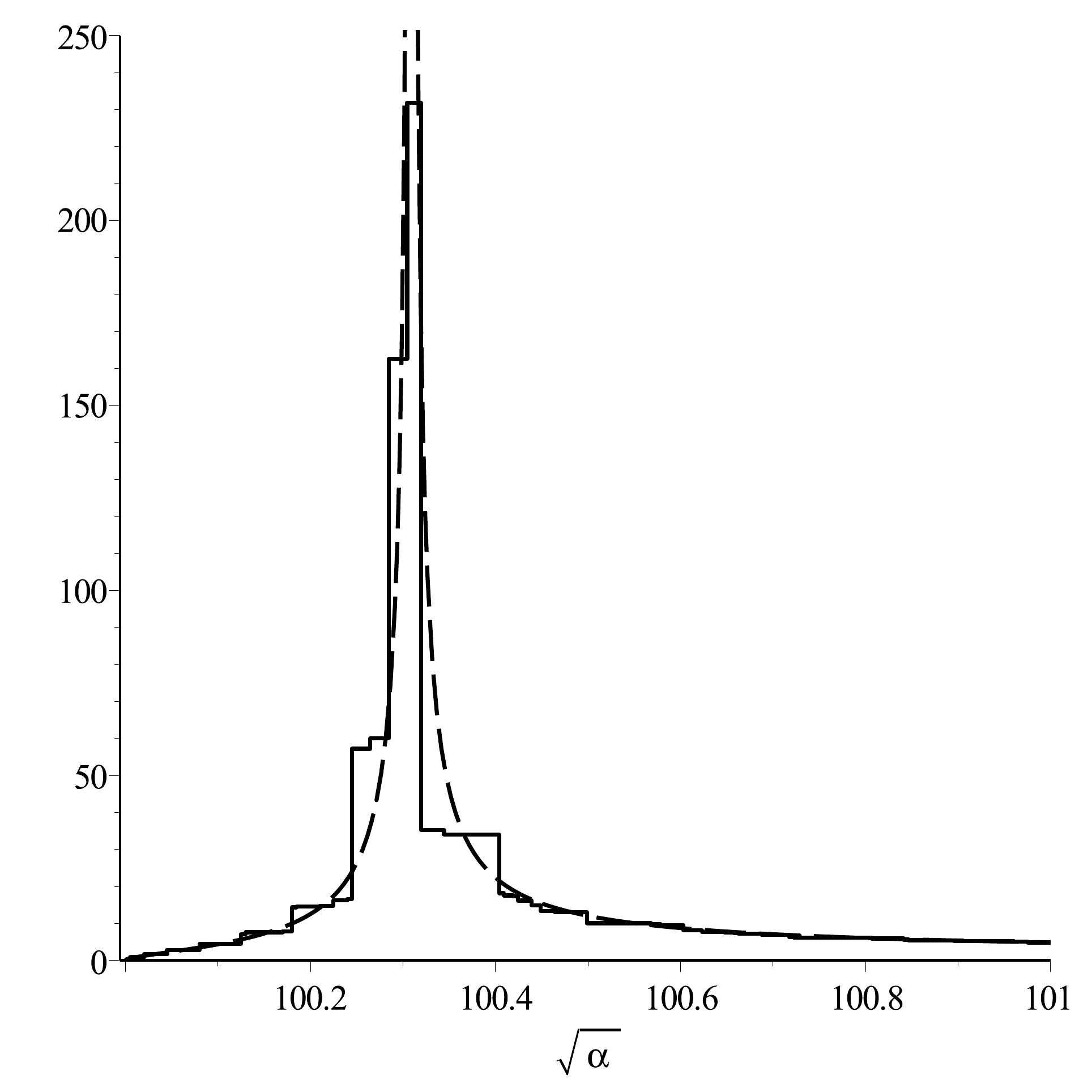} \\
	  (b)\hl{ $|\brho(\alpha)| = 2/(2v_1+1)^2 |\cT^{\prime}(\alpha)|$} \\
	\end{minipage}
        \caption{A plot of (a) $\frac{1}{2}(2v_1+1)^2 \cT^{\prime}(\alpha)$ and (b)  $|\brho(\alpha)|$ (solid lines)
        against $\sqrt{\alpha}$ for $\sqrt{\alpha}\in[100,101)$, i.e., for $v_k=100$ and $b\in[0,1)$. 
        The dashed lines show the approximate limiting functions given above.}
        \label{fig:rho_bar}
\end{figure}

Thus we see that $\brho(\alpha(v_k,b))$ typically grows slowly, like $\sqrt{v_k}$, 
except near $b=(-1+\sqrt{5})/4\approx0.3$, 
where our approximate form for the derivative $\cT^{\prime}$ is zero.
We can exploit this behaviour to examine orbits of $\Phi$ in domains $\Xe$ where the twist $K(e)$ of $\Omega^e$ is 
large ($b=0$---see figure \ref{fig:resonance}(a)), 
moderate ($b=0.8$---see figure \ref{fig:resonance}(b)) 
or almost zero ($b=0.3$---see figure \ref{fig:PrimaryResonances}).

What we see in the case of vanishing twist ($b\neq 0$) is a sea of discrete resonance structures:
the discrete analogue of the island chains of Hamiltonian perturbation theory.
The global structure of the twist map recedes to infinity,
and the dynamics are dominated by the local rotation number.
Although orbits may wander over a significant range in the $\rho$-direction,
the variation of the rotation number $\nu(z)$ along orbits is vanishingly small
(see table \ref{table:nu_range}(b)).


\begin{figure}[h!] 
        \centering
        \includegraphics[scale=0.38]{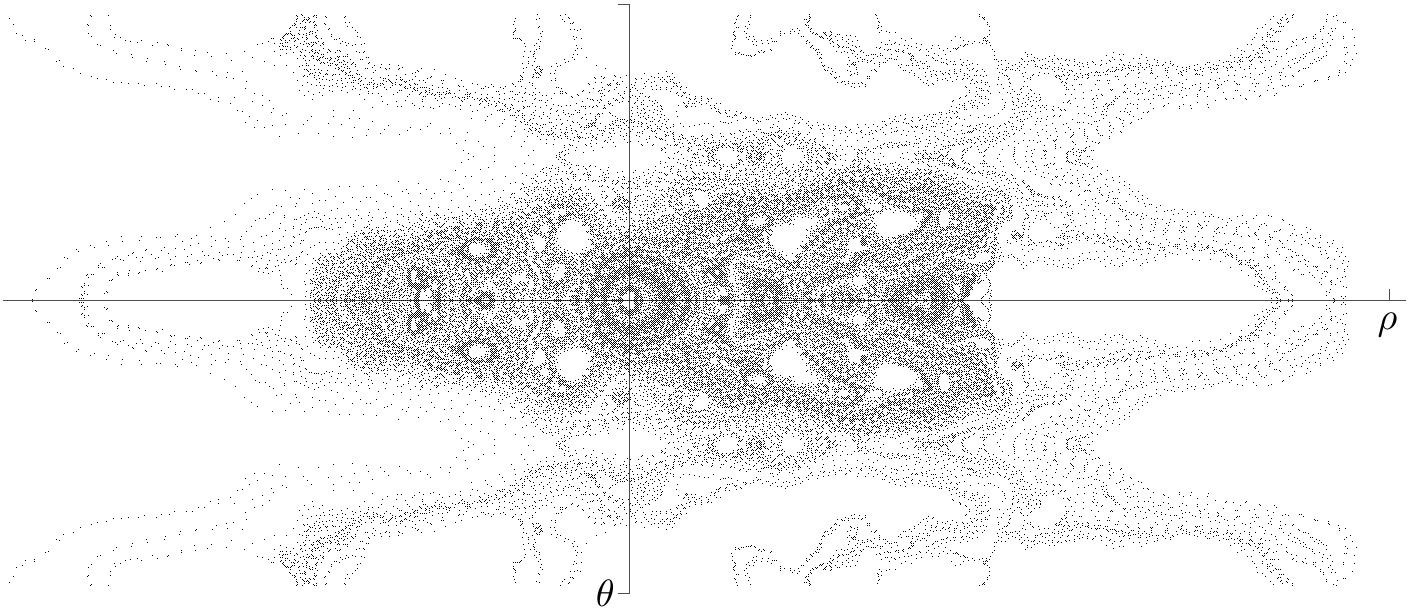} \\
        \caption{\hl{A pixel plot of a primary resonance for 
        $e=40925\approx202.3^2$ and $\lambda\approx 2\times 10^{-8}$. 
        The plot shows a large number of symmetric orbits of $\Phi$ in the cylindrical coordinates $(\theta,\rho)\in\bbS^1 \times \R$.
        The resolution of the plot is such that the (unit) width of the cylinder consists of approximately $280$ lattice sites.
        For this value of $e$, the natural lengthscale $\brho$ of the twist dynamics in the $\rho$-direction is large ($\brho\approx -150$).}  }
        \label{fig:PrimaryResonances}
\end{figure}


\begin{table}[h!]
        \centering
        \small
        \begin{tabular}{ c c c c | c|c | c|c |}
                \cline{5-8}
                 & & & & \multicolumn{2}{ |c| }{$\Delta\rho$} & \multicolumn{2}{ |c| }{$\Delta\nu$} \\
                \hline
                \multicolumn{1}{ |c| }{$e$} & \multicolumn{1}{ |c| }{$v_k$} & \multicolumn{1}{ |c| }{$\brho$} & \multicolumn{1}{ |c| }{$\trho$} 
			      & Median & Maximum & Median & Maximum \\
                \hline
                \multicolumn{1}{|c|}{$10\,000$} & \multicolumn{1}{|c|}{$100$} & \multicolumn{1}{|c|}{$0.266$} & \multicolumn{1}{|c|}{$9.5\times 10^{71}$} 
			    & $0.18$ & $0.65$ & $0.69$ & $2.5$ \\
                \multicolumn{1}{ |c| }{$40\,000$} & \multicolumn{1}{ |c| }{$200$} & \multicolumn{1}{ |c| }{$0.259$} & \multicolumn{1}{ |c| }{$5.1\times 10^{147}$} 
			    & $0.16$ & $0.55$ & $0.63$ & $2.1$ \\
                \multicolumn{1}{ |c| }{$160\,000$} & \multicolumn{1}{ |c| }{$400$} & \multicolumn{1}{ |c| }{$0.257$} & \multicolumn{1}{ |c| }{$2.0\times 10^{297}$} 
			    & $0.14$ & $0.48$ & $0.54$ & $1.9$\\
                \multicolumn{1}{ |c| }{$640\,000$} & \multicolumn{1}{ |c| }{$800$} & \multicolumn{1}{ |c| }{$0.255$} & \multicolumn{1}{ |c| }{$4.3\times 10^{605}$} 
			    & $0.13$ & $0.49$ & $0.51$ & $1.9$ \\
                \hline
        \end{tabular}\\[0.2cm]
        (a) $b=0$ \\[0.5cm]

        \begin{tabular}{ c c c c | c|c | c|c |}
                \cline{5-8}
                 & & & & \multicolumn{2}{ |c| }{$\Delta\rho$} & \multicolumn{2}{ |c| }{$\Delta\nu$} \\
                \hline
                \multicolumn{1}{ |c| }{$e$} & \multicolumn{1}{ |c| }{$v_k$} & \multicolumn{1}{ |c| }{$\brho$} & \multicolumn{1}{ |c| }{$\trho$} 
			      & Median & Maximum & Median & Maximum \\
                \hline
                \multicolumn{1}{ |c| }{$10\,057$} & \multicolumn{1}{|c|}{$100$} & \multicolumn{1}{ |c| }{$163$} & \multicolumn{1}{ |c| }{$1.4\times 10^{77}$} 
			    & $0.14$ & $2.9$ & $8.7\times 10^{-4}$ & $1.8\times 10^{-2}$ \\
                \multicolumn{1}{ |c| }{$40\,113$} & \multicolumn{1}{|c|}{$200$} & \multicolumn{1}{ |c| }{$106$} & \multicolumn{1}{ |c| }{$3.4\times 10^{153}$} 
			    & $0.25$ & $2.0$ & $2.3\times 10^{-3}$ & $1.9\times 10^{-2}$ \\
                \multicolumn{1}{ |c| }{$160\,234$} & \multicolumn{1}{|c|}{$400$} & \multicolumn{1}{ |c| }{$4105$} & \multicolumn{1}{ |c| }{$1.8\times 10^{285}$} 
			    & $4.8$ & $8.5$ & $1.2\times 10^{-3}$ & $2.1\times 10^{-3}$ \\
                \hline
        \end{tabular}\\[0.2cm]
        (b) $b=0.3$
        \caption{A table showing the values of $\brho$ and $\trho$ for various values of $e$, 
        and the typical range $\Delta\rho$ and $\Delta\nu$ of $\rho(z)$ and $\nu(z)$, respectively, along orbits. 
        The distribution of the range was calculated according to the fraction of points sampled whose orbit has the given range.}
        \label{table:nu_range}
\end{table}

In the $b=0$ case the global structure remains.
Within each fundamental domain there is little scope for resonance to develop
and phase portraits are largely featureless.
The typical variation in the rotation number along orbits is $1/2$ (see table \ref{table:nu_range}(a)), 
leading to orbits which typically do not cluster in the $\theta$-direction.
In this case it makes sense to consider the statistical properties of orbits of $\Phi$,
and in the next section we consider the period distribution function.


\section{The period distribution function} \label{sec:pdf}

The lattice structure described in section \ref{sec:MainTheorems}, 
whereby the dynamics of $\Phi$ on each of the domains $\Xe$ is equivariant under
the group of lattice translations generated by $\lambda\Le$, gives a natural
finite structure on which to define the period distribution function of $\Phi$.

For $e\in\cE$ and $z\in\Z^2$, we write $[z]$ to denote the equivalence class of $z$ modulo $\Le$:
 $$ [z] = z+ \Le. $$
The set of all equivalence classes is denoted $\Z^2/\,\Le$.
For sufficiently small $\lambda$, and for all equivalence classes $[z]\in\Z^2/\,\Le$, 
$[z]$ has a representative in $\Xe$ whose image under $\Phi$ also lies in $\Xe$:
 $$ \exists \, w\in[z]: \; \lambda w, \Phi(\lambda w)\in\Xe. $$
Thus we can let $\Phi$ act on $\Z^2/\,\Le$ by defining
 $$ \Phi([z]) = [\Phi(\lambda w)/\lambda]. $$
By the equivariance of $\Phi$ described in theorem \ref{thm:Phi_equivariance} (page \pageref{thm:Phi_equivariance}),
this action is well defined.
Similarly the reversing symmetry $G^e$ of proposition \ref{prop:Ge} 
can be defined over $\Z^2/\,\Le$.

The set $\Z^2/\,\Le$ is finite, with size $N$ given by
 $$ N = \#\left(\Z^2/\,\Le\right) = q(e) = (2v_1+1)^2\trho $$
(see equation (\ref{eq:theta_e})), and hence all orbits of $\Phi$ are periodic.
We define the period $T$ of $\Phi$ over $\Z^2/\,\Le$ in the usual fashion:
 $$ T([z]) = \min\{k\in\N \,: \; \Phi^k([z])=[z] \}. $$
Then the period distribution function $\cD^e = \cD^e(\lambda)$ of $\Phi$ is given as follows (cf.~equation (\ref{def:pdf})): 
\begin{displaymath}
 \cD^e(x) = \frac{1}{N} \, \# \{ [z]\in\Z^2/\,\Le : \; T([z])\leq \kappa x \},
\end{displaymath}
where $\kappa=\kappa(e,\lambda)$ is the scaling constant given by (cf.~theorem \ref{thm:GammaDistribution}, page \pageref{thm:GammaDistribution})
\begin{equation*} 
 \kappa = \frac{2N}{g+h} \hskip 40pt g=\#\Fix{G^e} \hskip 40pt h=\#\Fix{(\Phi\circ G^e)} .
\end{equation*}

For $e=v_k^2$, i.e., $b=0$, we wish to investigate whether the dynamics of $\Phi$ are sufficiently 
disordered that its period statistics are consistent with the limiting distribution $\cR$ of 
theorem \ref{thm:GammaDistribution}, which corresponds to random dynamics.
In principle, it is possible to calculate $\cD^e$ exactly;
however, the factorially diverging size of $N$ (and hence of $\trho$---see table \ref{table:nu_range}) makes this unfeasible.
In practice, we find we can approximate $\cD^e$ by calculating a sequence of local distributions.

\FloatBarrier

\begin{figure}[t]
        \centering
        \includegraphics[scale=1.1]{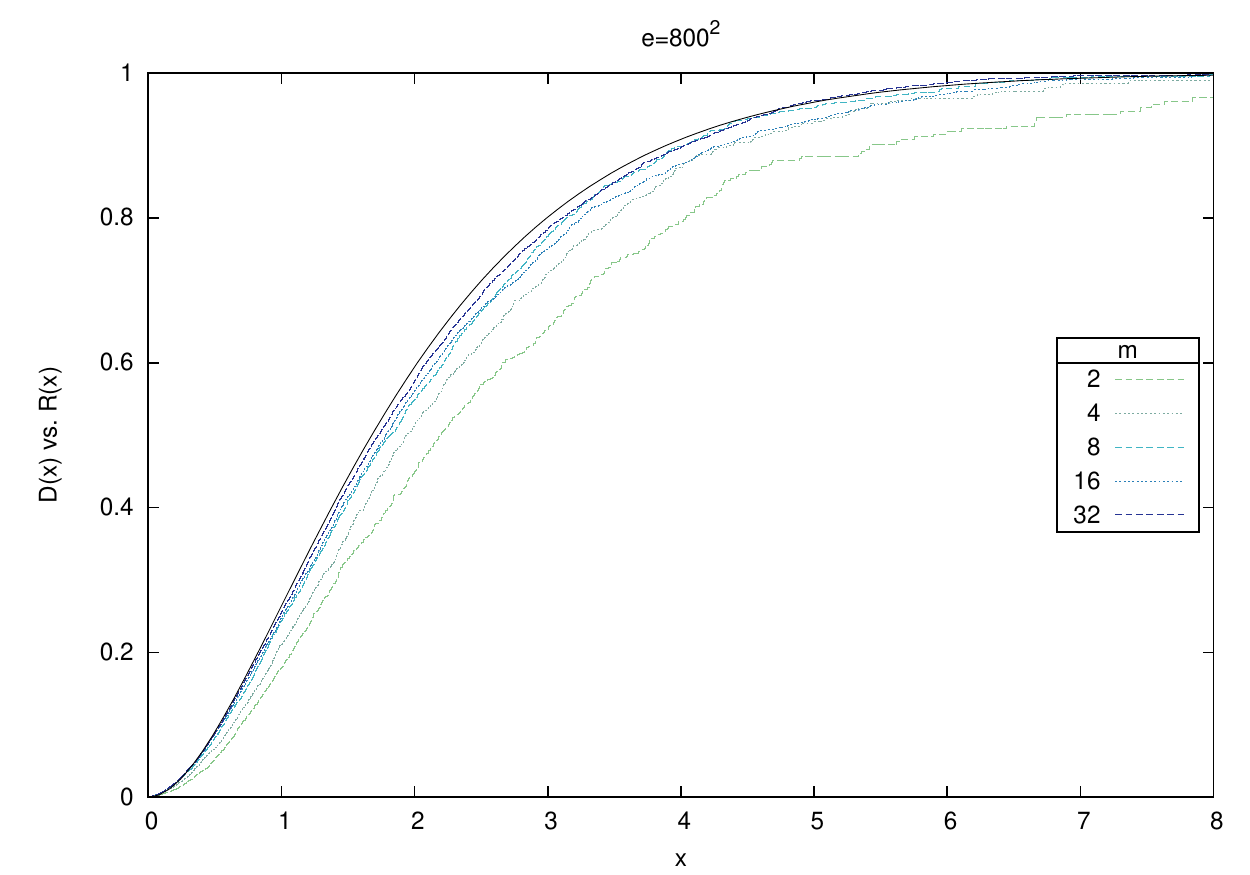} \\
        \caption{A number of the distributions $\cD(x)$ calculated for $v_k=800$ (green-blue dashed lines). 
        The number $m$ refers to the multiple of the characteristic length scale $\brho$ sampled. 
        The solid black line is the limiting distribution $\cR(x)$.}
        \label{fig:Distribution}
\end{figure}


\begin{table}[b]
        \centering
        \begin{tabular}{ |c|c|c|c|c|}
                \hline
                $v_k$ & $m$ & Sample size & $\int (\cR-\cD) dx$ & Approx. time \\
                \hline
                $100$ & $32$ & $350\,000$ & $0.03$ & 15sec \\
                $200$ & $32$ & $1\,360\,000$ & $0.04$ & 2min \\
                $400$ & $32$ & $5\,370\,000$ & $0.06$ & 15min\\
                $800$ & $32$ & $21\,200\,000$ & $0.06$ & 2hr \\
                $1600$ & $16$ & $42\,900\,000$ & $0.11$ & 11hr \\
                \hline
        \end{tabular}
        \caption{Some data relating to the calculation of distributions $\cD$ for various different values of $v_k$, 
        and their corresponding maximum values of $m$. 
        For a given value of $m$, nine distinct distributions were calculated: 
        for three different values of $\lambda$, each with three different values of $z_0$ (recall $z_0=(\eta^e)^{-1}(0,0))$.
        Neither the value of $\lambda$ nor $z_0$ was found to have any discernible effect on the distribution.
        The sample size is indicative of the number of points sampled during the calculation of each distribution, 
        similarly for the approximate calculation time. 
        The integral $\int(\cR-\cD)dx$ was calculated over the interval $[0,16]$ (every distribution calculated satisfies $\cD(16)=1$),
        and has been averaged over the nine individual distributions calculated.}
        \label{table:D_table}
\end{table}

\vfill

\FloatBarrier

Experimental observations show that the reduction of the dynamics modulo $\lambda\Le$ is unnecessary:
not only are all the orbits we computed already periodic, as conjecture \ref{conj:Periodicity} suggests,
but the number of equivalence classes in $\Z^2/\,\Le$ grows much faster than the range of any orbit
(compare the characteristic length scale $\trho$ to the typical range $\Delta\rho$ of an orbit
as given in table \ref{table:nu_range}, page \pageref{table:nu_range}).
We observe that the reduction has no effect on the dynamics of $\Phi$ or its period function,
so that the period distribution $\cD^e$ is also representative of the dynamics of $\Phi$ on its original domain $\Xe$.
In what follows, it will always be the case that the period $T$ of $\Phi$ on $\Z^2/\,\Le$ is
equal to the period $\tau$ of $\Phi$ on $\Xe$:
 $$ T([z]) = \tau(z) \hskip 40pt z\in\Xe, $$
and in our discussion we assume that this holds for all orbits.

Furthermore, the dynamics of $\Phi$ are sufficiently uniform from one fundamental domain to the next,
that we can achieve good approximations to $\cD^e$ by sampling the periods in just a small number such domains,
i.e., over vanishing subsets of $\Z^2/\,\Le$. 
It is this fact that allows us to estimate $\cD^e$ with numerical calculations, which we describe in the next section.

\subsection*{Computational investigation} 

For $e=v_k^2\in\cE$, we wish to calculate a sequence of period distribution functions, 
which will serve as approximations to $\cD^e$ as $v_k\to\infty$.
The factorially diverging number of equivalence classes of $\Z^2/\,\Le$
dictates that we must approximate $\cD^e$ by sampling a much smaller subset of the phase space.
However, discounting the lattice structure, there are no natural $\Phi$-invariant 
subsets of $\Xe$. Below, we construct a sequence of $\Phi$-invariant sets which mimic
the fundamental domains defined in (\ref{eq:fundamental_domain})---the natural invariant structures of the twist dynamics.

We consider subsets $A=A(v_k,m,\lambda)$ of $\Xe$ of the form
\begin{equation} \label{eq:A_m}
 A(v_k,m,\lambda) = \{ z\in\Xe \,: \; \nu(z) \in [-1/2,m-1/2] \} \hskip 40pt m\in\N. 
\end{equation}
For sufficiently small $\lambda$, the counterpart of $A$ on the cylinder
covers $m$ copies of the fundamental domain of the twist dynamics, so that as $v_k\to\infty$:
 $$ \# A \sim 2m(2v_1+1)^2|\brho| \to \infty $$
(cf.~equation (\ref{eq:pts_per_fundamental_domain})).
Furthermore, since the length $\brho$ is small relative to the length $\trho$ of the lattice $\Le$,
the set $A$ represents a vanishing subset of the equivalence classes of $\lambda\Le$:
 $$ \frac{\# A}{N} \sim \frac{2m|\brho|}{\trho} \to 0. $$

The set $A$ is not invariant under the perturbed dynamics $\Phi$. 
Hence we define $\bA$ to be the smallest invariant set which contains $A$:
\begin{equation*} 
 \bA(v_k,m,\lambda) = \bigcup_{n\in\Z} \Phi^n(A(v_k,m,\lambda)).
\end{equation*}
In what follows, it is assumed that there is a critical parameter value $\lambda_c(v_k,m)$ 
such that $\bA\subset\Xe$ for all $\lambda<\lambda_c$.

We observe that the overspill from $A$ under the map $\Phi$, 
i.e., the set $\bA\setminus A$, 
is small relative to $A$ as $m\to\infty$ (see figure \ref{fig:NoOfPts}(a)).
\begin{observation} \label{prop:A_m}
Let $v_k\in\N$. Then for $m\in\N$ and $(\lambda(m))_{m\in\Z}$ satisfying $\lambda(m)<\lambda_c(v_k,m)$, we have:
 $$ \frac{\#\bA(v_k,m,\lambda(m))}{\# A(v_k,m,\lambda(m))} \to 1 $$
as $m\to\infty$.
\end{observation}

Then we measure the period distribution function $\cD=\cD(v_k,m,\lambda)$ of $\Phi$ over $\bA$:
\begin{displaymath}
 \cD(x) = \frac{\# \{ z\in \bA \,: \; \tau(z)\leq \kappa x \}}{\# \bA},
\end{displaymath}
where the scaling constant $\kappa=\kappa(v_k,m,\lambda)$ is given by
\begin{equation} \label{eq:Phikappa_approx}
 \kappa = \frac{2\#\bA}{g+h} \hskip 20pt g=\#\left(\Fix{G^e}\cap\bA\right) \hskip 20pt h=\#\left(\Fix{(\Phi\circ G^e)}\cap\bA\right).
\end{equation}

For any $e=v_k^2\in\cE$, $m\in\N$ and $\lambda(m)<\lambda_c(v_k,m)$, we have:
 $$ (\cD(v_k,m,\lambda(m)) - \cD^e) \to 0 \hskip 40pt m\to\infty, $$
where $\cD^e$ is the period distribution function of $\Phi$ over $\Z^2/\,\Le$.
To study the behaviour of $\cD^e$ as $v_k\to\infty$,
we need to let both $m$ and $v_k$ go to infinity simultaneously.
We do not have sufficient numerical evidence to specify a scheme $m(v_k)$ for
which the convergence
 $$ (\cD(v_k,m(v_k),\lambda(v_k,m)) - \cD^e) \to 0 \hskip 40pt v_k\to\infty $$
holds. However, we do note that small values of $m$
were sufficient in all numerical experiments (see table \ref{table:D_table}),
which suggests that a scheme of the form
 $$ m = C(2v_1+1) \hskip 40pt C>0 $$
may be sufficient.

\medskip

Since $\Fix{G^e}$ is the pair of lines $x=y$ and $x-y=-\lambda(2v_1+1)$, 
the corresponding set on the cylinder is given by (cf.~(\ref{eq:rho_theta_lattice}))
 $$ \eta^e(\Fix{G^e}) = \left\{ \frac{1}{2(2v_1+1)} \, (i,i+2j)\,: \; i\in\{-(2v_1+1),0\}, \; j\in\Z \right\}. $$
Intersecting this with $A$ restricts the index $j$ according to
 $$ \frac{i+2j}{2\brho(2v_1+1)} \in \left[ -\frac{1}{2}, m-\frac{1}{2} \right) $$
(see equation (\ref{eq:A_m})).
Thus, equating $A$ with $\bA$ in the limit, we have
 $$ g \sim \#\left(\Fix{G^e}\cap A\right) \sim 2m\brho(2v_1+1) \to \infty $$
as $m,v_k\to\infty$. 
The fixed space $\Fix{(\Phi\circ G^e)}$ is the lattice equivalent of the line $\Fix{\cH^e}$ of equation (\ref{eq:Fix(cH)}). We have the following experimental observation for the size of $h$ (see figure \ref{fig:NoOfPts}(b)).

\begin{observation} \label{obs:g_h}
 Let $v_k,m\in\N$, $\lambda(v_k,m)<\lambda_c(v_k,m)$,
 and $g$, $h$ be as in equation (\ref{eq:Phikappa_approx}).
 Then as $m,v_k\to\infty$:
  $$ g \sim \sqrt{2} h. $$
\end{observation}

\begin{figure}[t]
        \centering
        \begin{minipage}{7cm}
	  \centering
	  \includegraphics[scale=0.55]{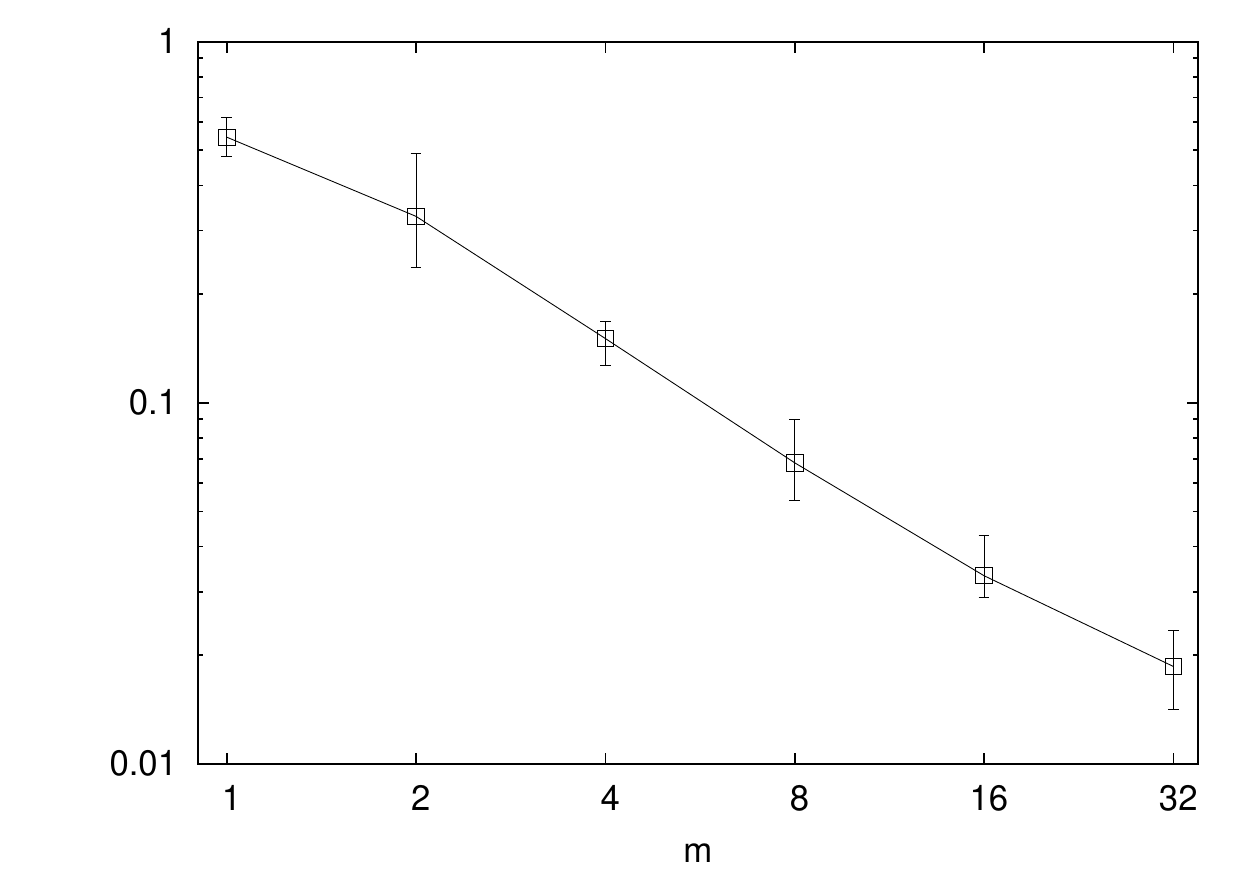} \\
	  (a) $\# \bA/\# A-1$ \\
	\end{minipage}
        \quad
        \begin{minipage}{7cm}
          \centering
	  \includegraphics[scale=0.55]{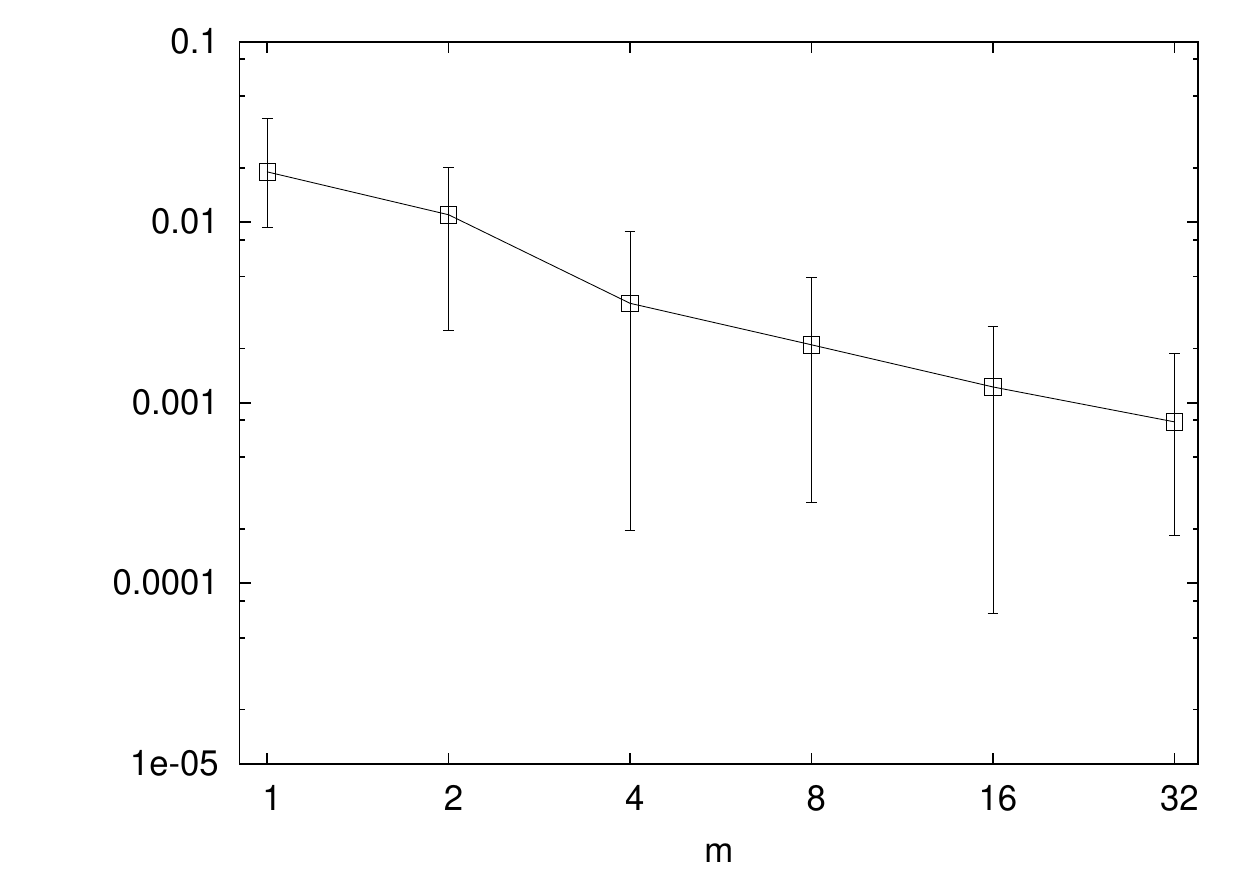} \\
	  (b) $h/g - 1/\sqrt{2}$ \\
        \end{minipage}
        \caption{The convergence (a) of the ratio $\# \bA/\# A$ to $1$ and (b) of the ratio $h/g$ to $1/\sqrt{2}$ as $m$ becomes large, for $v_k=800$. The line shows the average value of the relevant ratio among all experiments performed: the error bars indicate its minimum and maximum value. All axes are displayed with a logarithmic scale.}
        \label{fig:NoOfPts}
\end{figure}


From this observation, it follows that 
 $$ \frac{g+h}{\#\bA} \sim \frac{(2+\sqrt{2})}{2(2v_1+1)} \to 0 $$
as $m,v_k\to\infty$, and hence that the quantities $g$ and $h$ satisfy the conditions 
(\ref{eq:g_h_conds}) of theorem \ref{thm:GammaDistribution}.
Indeed, we observe that the universal distribution $\cR(x)$ 
is the limiting distribution for $\cD$ in the limits $m,v_k\to\infty$ 
(see figure \ref{fig:Distribution}).

\begin{observation} \label{obs:De}
Let $v_k,m\in\N$ and $\lambda(v_k,m)<\lambda_c(v_k,m)$.
Then as $m,v_k\to\infty$:
  $$ \cD(v_k,m,\lambda(v_k,m)) \to \cR, $$
where $\cR$ is the universal distribution of equation \eqref{def:R(x)}.
\end{observation}

Finally we note that, as in theorem \ref{thm:GammaDistribution}, 
the symmetric orbits of $\Phi$ have full density (see figure \ref{fig:symm_points}).

\begin{observation} \label{obs:symm_points}
Let $v_k,m\in\N$ and $\lambda(v_k,m)<\lambda_c(v_k,m)$. 
Furthermore, let $S=S(v_k,m,\lambda)$ be the set of points in $\bA$ whose orbit under $\Phi$ is symmetric:
 $$ S = \{ z\in \bA \,: \; \cO(z)=G^e(\cO(z)) \}. $$
Then $S$ has full density in $\bA(m)$ as $m,v_k\to\infty$:
  $$ \frac{\#S(v_k,m,\lambda(v_k,m))}{\#\bA(v_k,m,\lambda(v_k,m))}\to 1. $$
\end{observation}

\begin{figure}[!h]
        \centering
        \includegraphics[scale=0.55]{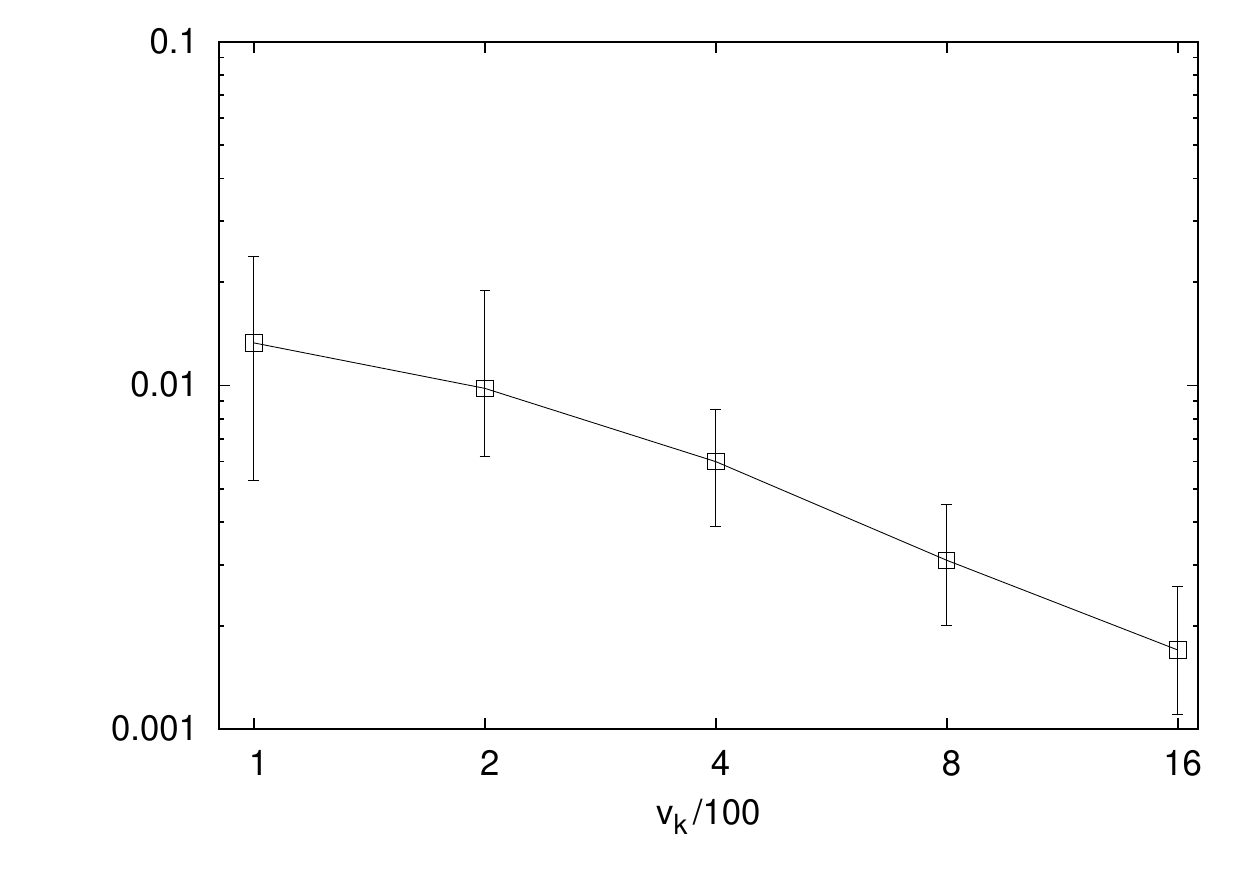} 
        \caption{The quantity $1-\#S/\#\bA$ as $v_k$ becomes large. 
        The line shows the average value of the relevant ratio over all experiments performed 
        (including over $m=1,2,4,8,16,32$---unlike the distribution $\cD$ of figure \ref{fig:Distribution}, this ratio does not vary significantly with $m$): the error bars indicate its minimum and maximum value.
        The axes are displayed with a logarithmic scale.}
        \label{fig:symm_points}
\end{figure}


%% file: Conclusion.tex
\chapter*{Concluding remarks}

In this thesis we investigated the dynamics of the discretised rotation (\ref{def:F})
in a new parameter regime: the limit $\lambda\to 0$.
A natural embedding of the lattice $\Z^2$ into the plane
transformed the discretised rotation into a perturbation $\F$ of an integrable,
piecewise-affine Hamiltonian system,
which was found to be nonlinear.
Thus we were lead to consider $\F$ as a discrete near-integrable system.

In this setting, the perturbation mechanism was no longer that of round-off,
but of linked strip maps: 
in each of the polygonal annuli defined by the polygon classes,
indexed by the sums of squares $e\in\cE$,
the dynamics of $\F$ are similar to those of a polygonal outer billiard.
This structure introduced a non-Archimedean character to the behaviour of $\F$.
We defined a symbolic coding associated with the strip map,
built out of a sequence of congruences modulo two-dimensional lattices,
which, for sufficiently small $\lambda$, induces a lattice structure on the return map $\Phi$.

This lattice structure removes $\lambda$ from its role as the perturbation parameter.
Instead, a change of coordinates allowed us to consider $\Phi$
as a sequence of discretised twist maps on the cylinder:
one for each polygon class.
In this setting, the limit of vanishing discretisation, 
and hence of vanishing perturbation, corresponds to the limit $e\to\infty$.

The twist $K(e)$ also varies between polygon classes.
In the case where the twist vanishes in the limit, i.e., $K(e)\to 0$, we found discrete resonances,
whose behaviour depends on the local rotation number.
By contrast, for the sequence of perfect squares, where $K(e)\to 4$, we found that the limiting period statistics 
coincide with those of a random reversible map on a discrete phase space.

Finally, we discuss open questions and avenues for further investigation.

\medskip

In the introduction to this work, we outlined the difficulty in reproducing 
the features of Hamiltonian perturbation theory in a discrete phase space.
At the outset, figure \ref{fig:PolygonalOrbits} suggested that such features
could be found for the map $F$ in the limit $\lambda\to 0$,
when considered relative to the correct `unperturbed' dynamics.
This proposal was later reinforced by phase plots of the return map, 
such as figures \ref{fig:resonance}(b) and \ref{fig:PrimaryResonances}.

We identified the minimal orbits, which close after just one revolution around the origin,
as the analogue of KAM curves: the minimal orbits are the simplest orbits,
which retain the natural recurrence time of the underlying dynamics (rather than some larger multiple thereof),
and are confined to convex invariant polygons, each of which is a small perturbation of an invariant curve of the integrable system.
However, like all orbits of $F$ encountered in this study, the minimal orbits are periodic,
and do not disconnect the space like their quasi-periodic counterparts on the continuum.

The apparent island chains we observe are more complex.
Although orbits cluster in the $\theta$-direction according to the local rotation number,
preliminary numerical experiments suggest that the organisation within each island
does not conform to the phenomenology of smooth Hamiltonian perturbation theory.
In particular, islands are not necessarily invariant: 
orbits can wander between one island and the next---see figure \ref{fig:ResonancesCloseUp}.

\begin{figure}[h] 
        \centering
        \includegraphics[scale=0.45]{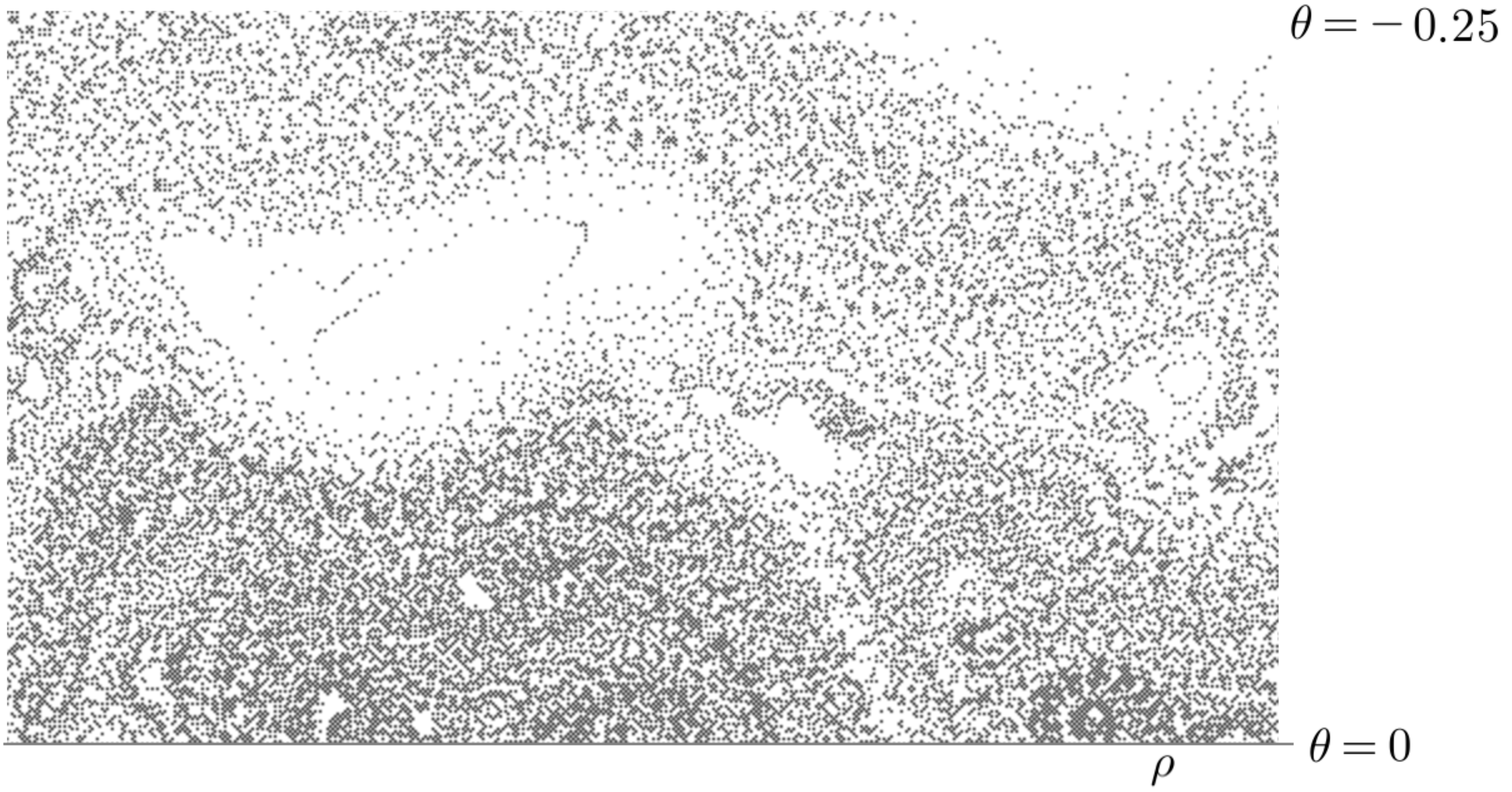} \\
        \caption{\hl{A close-up of a primary resonance for $e=160234\approx400.3^2$ and $\lambda\approx 3.5\times 10^{-9}$. 
        The plot shows part of a large number of symmetric orbits of $\Phi$ in the cylindrical coordinates $(\theta,\rho)\in\bbS^1 \times \R$.
        There are approximately $560$ lattice sites per unit length in each of the coordinate directions.
        We see that there are parts of the cylinder which are filled with symmetric orbits, 
        and others which are devoid of them, but no sharp boundary between the two. } }
        \label{fig:ResonancesCloseUp}
\end{figure}


To explore this new phenomenon, further extensive numerical investigation is required,
and a simplification of the model may prove necessary.
The character of the perturbation which distinguishes the return map $\Phi$ 
from the unperturbed dynamics is still unclear: the perturbation could be probabilistic in nature,
and hence best modelled by a random reversible perturbation;
could have a complicated but deterministic structure; or could be a mixture of the two.

One of the few similar systems found in the literature is 
a toy model of a discrete twist map, investigated numerically in \cite{ZhangVivaldi}.
In that case, a suitably chosen one-dimensional surface of section revealed
an interval exchange transformation on an infinite sequence of intervals.
We do not expect our dynamics to be so simple.
However, a first step in studying the perturbation could be to define a 
suitable (quasi-one-dimensional) surface of section for the $\Phi$ dynamics.

%% file: Appendix.tex
\chapter{Tedious proofs} \label{chap:Appendix}

\begin{prop_nonumber}[Proposition \ref{prop:octagon_orbits}, page \pageref{prop:octagon_orbits}]
For all $\lambda>0$ and $x\in\N$ in the range
\begin{equation} \label{eq:xrange1_II}
 \frac{1}{2\lambda}+2 \leq x\leq \frac{1}{\lambda}-1,
\end{equation}
the orbit of $z=(x,x)$ under $F$ is symmetric and minimal if and only if
 $$ 2x + \Bceil{\frac{1}{\lambda}} - 2\left\fl{ \frac{1}{\lambda}\right} \equiv 2 \mod{3}. $$
\end{prop_nonumber}

\begin{proof}
As in the proof of proposition \ref{prop:square_orbits}, page \pageref{prop:square_orbits}, 
we begin by considering the fourth iterates of $F$.
From equation (\ref{eq:F4}), we have that
\begin{equation} \label{eq:box1}
  F^4(x,y) = (x+1,y-1) 
  			\hskip 40pt 
  			0\leq x \leq \frac{1}{\lambda}-1, \quad 
  			1\leq y \leq \frac{1}{\lambda}.
\end{equation}
A similar calculation reveals another set of lattice points 
on which the fourth iterates of $F$ produce a uniform translation:
\begin{equation} \label{eq:box2}
  F^4(x,y) = (x+1,y-3) 
  			\hskip 40pt 
  			\frac{1}{\lambda} \leq x < \frac{2}{\lambda}-1, \quad 
  			3\leq y \leq \frac{1}{\lambda}+1.
\end{equation}
We use these two regimes of uniform behaviour to trace symmetric orbits from $\Fix{G}$ to $\Fix{H}$,
taking care at the boundaries between regimes.
  
We consider the orbit of $(x,x)$ with $x$ in the range (\ref{eq:xrange1_II}).
For any natural number $m$ satisfying
 $$ x + (m-1) \leq \frac{1}{\lambda}-1, \hskip 40pt x - (m-1) \geq 1, $$
the behaviour (\ref{eq:box1}) gives us that
 $$ F^{4m}(x,x) = (x+m,x-m). $$
Hence we let $m=\fl{1/\lambda}-x$, so that $F^{4m}(x,x)$ is given by
 $$ F^{4m}(x,x) = \left( \Bfl{\frac{1}{\lambda}}, 2x - \Bfl{\frac{1}{\lambda}} \right), $$
and by the range (\ref{eq:xrange1_II}) of $x$:
 $$ 4 \leq 2x - \Bfl{\frac{1}{\lambda}} \leq \frac{1}{\lambda}-1. $$
  
There are now two cases to consider.
If $\fl{1/\lambda}=1/\lambda$, i.e., if $1/\lambda\in\N$, then $F^{4m}(x,x)$ belongs to the set described by (\ref{eq:box2}),
in which case
 $$ F^{4(m+1)}(x,x) = \left(\Bfl{\frac{1}{\lambda}}+1,2x - \Bfl{\frac{1}{\lambda}}-3 \right). $$
If not, then this point lies on the boundary between the regimes (\ref{eq:box1}) and (\ref{eq:box2}),
so we must calculate the behaviour of $F^4$ explicitly.
We find
 $$ F^{4(m+1)}(x,x) = \left(\Bfl{\frac{1}{\lambda}}+1,2x - \Bfl{\frac{1}{\lambda}}-2 \right). $$
We summarise these two cases by writing
 $$ F^{4(m+1)}(x,x) = \left(\Bfl{\frac{1}{\lambda}}+1,2x-3 +\Bceil{\frac{1}{\lambda}} - 2\Bfl{\frac{1}{\lambda}} \right). $$
  
Now the point $F^{4(m+1)}(x,x)$ is described by (\ref{eq:box2}), and for any natural number $n$ satisfying
 $$ \Bfl{\frac{1}{\lambda}}+1 + (n-1) < \frac{2}{\lambda}-1, 
 			\hskip 40pt 
 			2x-3 +\Bceil{\frac{1}{\lambda}} - 2\Bfl{\frac{1}{\lambda}} -3(n-1) \geq 3, $$
we have
 $$ F^{4(m+n+1)}(x,x) = \left( \Bfl{\frac{1}{\lambda}}+n+1,
	2x +\Bceil{\frac{1}{\lambda}} - 2\Bfl{\frac{1}{\lambda}}-3(n+1) \right). $$
Hence we take
 $$ n = \left\lfloor \frac{1}{3} \left( 2x -3 +\Bceil{\frac{1}{\lambda}} - 2\Bfl{\frac{1}{\lambda}} \right) \right\rfloor, $$
so that the $y$-coordinate of $F^{4(m+n+1)}(x,x)$ is given by
 $$ 2x +\Bceil{\frac{1}{\lambda}} - 2\Bfl{\frac{1}{\lambda}}-3(n+1) = \delta, $$
where $\delta\in\{0,1,2\}$ is the residue of $2x +\ceil{1/\lambda} - 2\fl{1/\lambda}$ modulo $3$.
 
The point $F^{4(m+n+1)}(x,x)$ lies just above the positive $x$-axis. We apply $F^3$, to move the orbit close to the negative $y$-axis:
 $$ F^{4(m+n+1)+3}(x,x) = \left( \delta - 3,
	\fl{\lambda(1-\delta)} -\left( \Bfl{\frac{1}{\lambda}}+n+1 \right) \right). $$
The orbit is symmetric and minimal if and only if this point lies in $\Fix{H}$.
In this case, the relevant segment of $\Fix{H}$ is given by
 $$ \left\{(x,y)\in\Z^2\,: \; x=-1, \; \fl{\lambda y}=-2 \right\}, $$
so the orbit is symmetric and minimal if and only if $\delta-3=-1$, or
 $$ 2x +\Bceil{\frac{1}{\lambda}} - 2\Bfl{\frac{1}{\lambda}} \equiv 2 \mod{3}. $$
\end{proof}

\medskip

\begin{lemma_nonumber}[Lemma \ref{lemma:cP_variation}, page \pageref{lemma:cP_variation}]
Let $w\in\R^2$ and let $z=R_{\lambda}(w)$ be the lattice point in $\lZ$ associated with $w$. 
Then as $\lambda\to 0$:
 $$ \forall \xi\in \Ot(z): \hskip 20pt |\cP(\xi)-\cP(w)| = O(\lambda). $$
\end{lemma_nonumber}

\begin{proof}
Let $r>0$ and $A(r,\lambda)$ be as in equation (\ref{eq:A}).
We begin by bounding the change in $\cP$ under $\F^4$ in the set $A(r,\lambda)$.
By lemma \ref{lemma:Lambda} (page \pageref{lemma:Lambda}), we have that for sufficiently small $\lambda$, 
all non-zero $z\in A(r,\lambda)\setminus\Lambda$ satisfy $\F^4(z) = z + \lambda\bfw(z)$.
For such $z$, there is no change in $\cP$ under $\F^4$:
 $$ \cP(\F^4(z)) - \cP(z) = 0 \hskip 40pt z\in A(r,\lambda)\setminus\Lambda. $$
If $z\in A(r,\lambda)\cap\Lambda$, then $\F^4(z) = z + \bfv(z)$,
where an explicit expression for $\bfv$ is given in equation (\ref{eq:v_abcd}), page \pageref{eq:v_abcd}.

For any $z,v\in\R^2$ we have
\begin{align}
 \left| \cP(z+v)- \cP(z) \right| &= \left| \; \int_{[z,z+v]} \nabla\cP(\xi)\cdot d\mathbf{\xi} \; \right| \nonumber \\
    &\leq \max_{\xi\in[z,z+v]} \left( \|\nabla\cP(\xi)\| \right) \, \|v\|, \label{eq:Delta_cP_bound}
\end{align}
where $[z,z+v]$ denotes the line segment joining the points $z$ and $z+v$, $d\mathbf{\xi}$ is 
the line element tangent to this segment, and $\nabla\cP$ is the gradient of $\cP$, given by
 $$ \nabla\cP(x,y) = (2\fl{x} +1,2\fl{y} +1) \hskip 40pt 
(x,y)\in \R^2\setminus \Delta. $$
If $z=\lambda(x,y)$ and $v=\bfv(z)$ is the discrete vector field, then for sufficiently small $\lambda$, 
equations (\ref{eq:v_abcd}) and (\ref{eq:bd_ac_sets}) can be combined to give
\begin{align*}
 \| \bfv(z) \| 
 &\leq \lambda \sqrt{(|2\fl{\lambda y}+1| +2)^2 + (|2\fl{\lambda x}+1| +1)^2} \\
 &\leq \lambda \sqrt{(2|\fl{\lambda y}|+3)^2 + (2|\fl{\lambda x}|+2)^2} \\
 &< \lambda \sqrt{(2|\lambda y|+5)^2 + (2|\lambda x|+4)^2} \\
 &\leq \lambda\sqrt{2} (2\|z\|_{\infty} + 5).
\end{align*}
This inequality ensures that the length of the line segment $[z,z+v]$ goes to zero with $\lambda$, 
so that for sufficiently small $\lambda$, the piecewise-constant form of the gradient $\nabla\cP$ gives
\begin{align*}
 \max_{\xi\in[z,z+v]} (\|\nabla\cP(\xi)\|) 
 &\leq \sqrt{(|2\fl{\lambda x}+1| +2)^2 + (|2\fl{\lambda y}+1| +2)^2} \\
 &\leq \sqrt{(2|\fl{\lambda x}|+3)^2 + (2|\fl{\lambda y}|+3)^2} \\
 &\leq \sqrt{(2|\lambda x|+5)^2 + (2|\lambda y|+5)^2} \\
 &\leq \sqrt{2} (2\|z\|_{\infty} + 5).
\end{align*}
Substituting these into the inequality (\ref{eq:Delta_cP_bound}), we have that for sufficiently small 
$\lambda$:
 $$ \left| \cP(\F^4(z))- \cP(z) \right| = \left|\cP(z+\bfv(z))- \cP(z) \right| 
 \leq 2\lambda (2\|z\|_{\infty} + 5)^2. $$

Similarly we consider the change in $\cP$ under $\F$.
If $z=\lambda(x,y)$, then by the same sort of analysis, we have that for sufficiently small $\lambda$:
\begin{align*}
 \left| \cP(\F(z))- \cP(z) \right|  &= \left| P(\lambda(\fl{\lambda x} -y)) - P(\lambda y) \right|\\
&= \left| P(\lambda(y-\fl{\lambda x})) - P(\lambda y) \right| \\
&\leq \lambda |\fl{\lambda x}| \, (|P^{\prime}(\lambda y)|+2) \\
&\leq \lambda |\fl{\lambda x}| \, (2|\fl{\lambda y}| +3) \\
&\leq \lambda (|\lambda x| +1) \, (2|\lambda y| +5) \\
&\leq 2\lambda (\|z\|_{\infty}+3)^2,
\end{align*}
where $P$ is the piecewise-affine function defined in equation (\ref{def:P}). 
(We refer the reader to page \pageref{thm:Polygons} for the proof that $P$ is even.)

It follows that for any orbit contained in $A(r,\lambda)$, if $k\in\Z_{\geq0}$ and $0\leq l<4$, then
\begin{equation} \label{eq:Delta_cP_bound_II}
 \left| \cP(\F^{4k+l}(z)) - \cP(z) \right| \leq 2\lambda(m+l) (2\|z\|_{\infty} + 5)^2,
\end{equation}
where $m$ is the number of transition points in the orbit of $z$ under $\F^4$:
 $$ m = \# \left( \{ z, \F^4(z), \dots, \F^{4k}(z) \} \cap \Lambda \right). $$
Similar expressions hold in backwards time, for iterates of $\F^{-4}$ and $\F^{-1}$.
For fixed $\lambda$, this estimate bounds the perturbed orbit of a point $z\in\lZ$ to a polygonal annulus 
around the polygon $\Pi(z)$, which grows in thickness as the number of transition points in the orbit increases. 

By construction, the return orbit of $z$ under $\F^4$ contains exactly one 
transition point for every time the orbit passes from one of the boxes $B_{m,n}$ to another.
Furthermore, the fourth iterates of $\F$ move parallel to the flow within each box,
so that, per revolution, there is one transition point per box that the return orbit intersects.
This number is (essentially) equal to the number of sides of $\Pi(w)$, and does not scale with $\lambda$.
Hence, we have
 $$ \left| \cP(\xi) - \cP(z) \right| =O(\lambda) $$
for all $\xi\in\Ot(z)$.
\end{proof}

\medskip

\begin{lemma_nonumber}[Lemma \ref{lemma:epsilon_bounds}, page \pageref{lemma:epsilon_bounds}]
For $b\in[0,1)$ and $v_k\in\N$, let $v_1=\fl{(v_k+b)/\sqrt{2}}$.
Then the following limit exists
\begin{equation} \label{eq:epsilon(b)_II}
 \epsilon(b) = \lim_{v_k\rightarrow\infty} \left( v_k^{3/2} \; \sum_{n=v_1+1}^{v_k-1} 
      \int_{n-1/2}^{n+1/2} \frac{\sqrt{(v_k+b)^2 - n^2}}{n^2} - \frac{\sqrt{(v_k+b)^2 - x^2}}{x^2} \, dx \right),
\end{equation}
and satisfies
\begin{equation} \label{eq:epsilon_bounds_II}
 \frac{1}{36} \, \frac{1}{\sqrt{3(b+1)}} \leq \epsilon(b) \leq \frac{1}{12} \, \frac{1}{\sqrt{b+1}} \, \frac{2b+3}{2b+2}.
\end{equation}
\end{lemma_nonumber}

\begin{proof}
For $n$ in the range $v_1+1 \leq n \leq v_k-1$, let
\begin{equation*} 
 I_n(v_k,b) = \int_{n-1/2}^{n+1/2} \frac{\sqrt{(v_k+b)^2-n^2}}{n^2} - \frac{\sqrt{(v_k+b)^2-x^2}}{x^2} \, dx.
\end{equation*}
Using the substitution $y=x-n$, we can write $I_n$ as
 $$ I_n = \frac{\sqrt{(v_k+b)^2-n^2}}{n^2} \, \int_{-1/2}^{1/2} 1 - \left( 1 + \frac{y}{n} \right)^{-2}
      \sqrt{1 - \frac{2ny+y^2}{(v_k+b)^2-n^2}} \, dy. $$
To simplify notation, we define the sequence
 $$ A_n(v_k,b) = \frac{n}{(v_k+b)^2-n^2}  \hskip 40pt v_1+1 \leq n \leq v_k-1, $$
which is increasing in $n$ and bounded according to
\begin{equation} \label{eq:An_range}
 \frac{\sqrt{2}}{v_k+b} <  A_n < \frac{1}{2}.
\end{equation}
Then $I_n$ becomes
\begin{equation} \label{eq:I_n_A_n}
I_n = \frac{1}{n^{3/2}\sqrt{A_n}} \, 
 \int_{-1/2}^{1/2} 1 - \left( 1 + \frac{y}{n} \right)^{-2} \sqrt{1 - 2A_ny -\frac{A_n y^2}{n}} \, dy.
\end{equation}

We expand the integrand of $I_n$ in powers of $1/n$, 
retaining any terms which are order $1/n$ or larger.
Firstly, expanding the inverse power, we have
\begin{equation} \label{eq:I_n_expansion_I}
 \left( 1 + \frac{y}{n} \right)^{-2} = 1 - \frac{2y}{n} + O\left(\frac{1}{n^2}\right) \hskip 40pt y\in[-1/2,1/2]
\end{equation}
as $n\rightarrow\infty$.
Then we tackle the square root by writing
 $$ \sqrt{1 - 2A_ny -\frac{A_n y^2}{n}} = \sqrt{1 - 2A_n y} \, \sqrt{1 - \frac{A_n y^2}{n(1 - 2A_n y)}}. $$
The second of these factors can be expanded as follows:
\begin{equation} \label{eq:I_n_expansion_III}
 \sqrt{1 - \frac{A_n y^2}{n(1 - 2A_n y)}} = 1 - \frac{A_n y^2}{2n(1 - 2A_n y)} + O\left(\frac{1}{n^2}\right) \hskip 40pt y\in[-1/2,1/2].
\end{equation}
The first factor, however, cannot be expanded in powers of $1/n$.
Instead, we use Taylor's Theorem (see, for example, \cite[Theorem 4.82]{Burkill}), applied to $f(x)=\sqrt{1+x}$ at $x=0$, 
to obtain an explicit remainder term. This gives
\begin{equation} \label{eq:I_n_expansion_II}
 \sqrt{1 - 2A_n y}  = 1 - A_n y - R_2(y) \hskip 40pt y\in[-1/2,1/2],
\end{equation}
where $R_2$ is given by
\begin{equation} \label{eq:R2}
 R_2(y) = \frac{A_n^2 y^2}{2(1 - 2\theta(y) A_n y)^{3/2}} \hskip 40pt \theta(y)\in(0,1).
\end{equation}
Thus, combining the expansions (\ref{eq:I_n_expansion_I}), (\ref{eq:I_n_expansion_III}), (\ref{eq:I_n_expansion_II}) and simplifying,
the integrand of $I_n$ is given by
\begin{align}
 &1 - \left( 1 + \frac{y}{n} \right)^{-2} \sqrt{1 - 2A_ny -\frac{A_n y^2}{n}} \nonumber \\
 &= A_n y + \frac{2y}{n} - \frac{3A_n y^2}{2n} + \frac{A_n^2 y^3}{2n(1-2A_n y)} 
 		+ R_2(y) \left( 1 + O\left(\frac{1}{n}\right)\right) + O\left(\frac{1}{n^2}\right). \label{eq:I_n_integrand}
\end{align}

Now we integrate the expression (\ref{eq:I_n_integrand}) over $y$.
The terms which are linear in $y$ integrate to zero:
 $$ \int_{-1/2}^{1/2} A_n y + \frac{2y}{n} \, dy = 0, $$
whereas the quadratic term integrates to give
 $$ \int_{-1/2}^{1/2} - \frac{3A_n y^2}{2n} \, dy = -\frac{A_n}{8n}. $$
Using the definition (\ref{eq:R2}) of $R_2(y)$, the remaining terms in (\ref{eq:I_n_integrand}) can be regrouped to give
 $$\int_{-1/2}^{1/2} \frac{A_n^2 y^3}{2n(1-2A_n y)} + R_2(y) \left( 1 + O\left(\frac{1}{n}\right)\right) \, dy 
  = \left( \int_{-1/2}^{1/2} R_2(y) \, dy \right) \left( 1 + O\left(\frac{1}{n}\right)\right). $$
Thus, by (\ref{eq:I_n_A_n}), $I_n$ is given by
\begin{equation} \label{eq:I_n_R}
 I_n = -\frac{\sqrt{A_n}}{8n^{5/2}} 
	  +\frac{1}{n^{3/2}\sqrt{A_n}} \left( \int_{-1/2}^{1/2} R_2(y) \, dy \right) \left( 1 + O\left(\frac{1}{n}\right)\right) +
 		O\left(\frac{1}{n^3}\right)
\end{equation}
as $n\rightarrow\infty$. 
(In the final error term, we have used the fact that $(n^{3/2}\sqrt{A_n})^{-1}=O(1/n)$---cf. equation (\ref{eq:An_range}).)

We consider the behaviour of each term in $I_n$ as we sum over $n$.
We have already seen in the proof of proposition \ref{prop:Tprime_asymptotics}, equation (\ref{eq:Tprime_bound2}), 
that the sum over the first term in (\ref{eq:I_n_R}) behaves like
\begin{equation*} 
 \sum_{n=v_1+1}^{v_k-1} \, \frac{\sqrt{A_n}}{n^{5/2}} = \sum_{n=v_1+1}^{v_k-1} \, \frac{1}{n^2\sqrt{(v_k+b)^2-n^2}}
 		= O \left(\frac{1}{v_k^2}\right)
\end{equation*}
as $v_k\to\infty$.
Thus this term does not contribute:
noting that the sum is over order $n$ (i.e., order $v_k$) terms, equation (\ref{eq:I_n_R}) gives us that
\begin{equation} \label{eq:I_n_R_II}
 \sum_{n=v_1+1}^{v_k-1} I_n 
 = \sum_{n=v_1+1}^{v_k-1} \left[ \frac{1}{n^{3/2}\sqrt{A_n}}
 		 \left( \int_{-1/2}^{1/2} R_2(y) \, dy \right)\right] \left( 1 + O\left(\frac{1}{v_k}\right)\right) 
 		 + O \left(\frac{1}{v_k^2}\right),
\end{equation}
so that the only relevant contribution comes from the $R_2$ term.

\medskip

We bound the following integral over $y$:
 $$ \frac{\sqrt{2}}{18\sqrt{3}} = \int_{-1/2}^{1/2} \frac{y^2}{(3/2)^{3/2}} \, dy
  < \int_{-1/2}^{1/2} \frac{y^2}{(1 - 2\theta(y) A_n y)^{3/2}} \, dy 
  < \int_{-1/2}^{1/2} \frac{y^2}{(1/2)^{3/2}} \, dy = \frac{\sqrt{2}}{6}, $$
so that by the definition (\ref{eq:R2}) of $R_2$:
\begin{equation} \label{eq:R2_sum_bounds}
 \frac{\sqrt{2}}{36\sqrt{3}} \, \left(\frac{A_n}{n}\right)^{3/2} 
	< \frac{1}{n^{3/2}\sqrt{A_n}} \left( \int_{-1/2}^{1/2} R_2(y) \, dy \right)
 		 < \frac{\sqrt{2}}{12} \, \left(\frac{A_n}{n}\right)^{3/2}.
\end{equation}
Now we consider the sum
\begin{align*}
 \sum_{n=v_1+1}^{v_k-1} \, \left(\frac{A_n}{n}\right)^{3/2}
  &= \sum_{n=v_1+1}^{v_k-1} \, \left(\frac{1}{(v_k+b)^2-n^2}\right)^{3/2}.
\end{align*}
The summand is increasing in $n$, so we can bound the sum according to
\begin{align*}
 \sum_{n=v_1+1}^{v_k-1} \, \left(\frac{A_n}{n}\right)^{3/2} 
 &\geq \int_{v_1}^{v_k-1} \, \left(\frac{1}{(v_k+b)^2-x^2}\right)^{3/2} \, dx \\
 &= \frac{1}{(v_k+b)^2} \left[ \frac{x}{\sqrt{(v_k+b)^2-x^2}} \right]_{v_1}^{v_k-1} \\
 &= \frac{1}{(v_k+b)^2} \left( \frac{v_k-1}{\sqrt{(v_k+b)^2-(v_k-1)^2}} - \frac{v_1}{\sqrt{(v_k+b)^2-v_1^2}}\right)\\
  &= \frac{1}{(v_k+b)^2} \left( \frac{v_k-1}{\sqrt{2v_k(b+1)}} + O\left(1\right) \right)\\
  &= \frac{1}{v_k^{3/2}} \, \frac{1}{\sqrt{2(b+1)}} + O\left(\frac{1}{v_k^2}\right).
\end{align*}
Combining this with (\ref{eq:R2_sum_bounds}), we have that
 $$ \liminf_{v_k\to\infty} 
	\left[ \sum_{n=v_1+1}^{v_k-1} \, \frac{v_k^{3/2}}{n^{3/2}\sqrt{A_n}} \left( \int_{-1/2}^{1/2} R_2(y) \, dy \right) \right]
	\geq \frac{1}{36\sqrt{3}} \, \frac{1}{\sqrt{b+1}}. $$
Similarly
\begin{align*}
 \sum_{n=v_1+1}^{v_k-1} \, \left(\frac{A_n}{n}\right)^{3/2} 
 & \leq \int_{v_1+1}^{v_k-1} \, \left(\frac{1}{(v_k+b)^2-x^2}\right)^{3/2} \, dx +\left(\frac{A_{v_k-1}}{v_k-1}\right)^{3/2}  \\
 & = \frac{1}{v_k^{3/2}} \left( \frac{1}{\sqrt{2(b+1)}} + \frac{1}{(2(b+1))^{3/2}} \right) + O\left(\frac{1}{v_k^2}\right) \\
 & = \frac{1}{v_k^{3/2}} \, \frac{1}{\sqrt{2(b+1)}} \, \frac{2b+3}{2b+2} + O\left(\frac{1}{v_k^2}\right),
\end{align*}
which combines with (\ref{eq:R2_sum_bounds}) to give
\begin{equation} \label{eq:limsup}
 \limsup_{v_k\to\infty} 
	\left[ \sum_{n=v_1+1}^{v_k-1} \, \frac{v_k^{3/2}}{n^{3/2}\sqrt{A_n}} \left( \int_{-1/2}^{1/2} R_2(y) \, dy \right) \right]
	\leq \frac{1}{12} \, \frac{1}{\sqrt{b+1}} \, \frac{2b+3}{2b+2}.
\end{equation}
Equation (\ref{eq:I_n_R_II}) gives us that the same limit inferior and limit superior apply to the sum over $v_k^{3/2} I_n$:
thus if the limit $\epsilon(b)$ exists, then it must satisfy (\ref{eq:epsilon_bounds_II}).

\medskip
 
It remains to show the convergence of the sum over the remainder term $R_2$.
This is not straightforward since both the bounds of the sum and the terms themselves vary as $v_k\to\infty$:
although all terms are positive and the number of terms increases with $v_k$, the size of each term also varies with $v_k$.

To get an explicit expression for $R_2$, we use the full Taylor's series representation \cite[Theorem 5.8]{Burkill}, whereby
\begin{equation} \label{eq:R2_II}
 R_2(y) = \sum_{j=2}^{\infty} \binom{1/2}{j} (-2A_n y)^j,
\end{equation}
and the binomial coefficients are defined as follows:
 $$ \binom{1/2}{j} = \prod_{k=1}^j \frac{1/2-(j-k)}{k}. $$
Now the sum under consideration is given by
\begin{align*}
& \sum_{n=v_1+1}^{v_k-1} \left[ \frac{v_k^{3/2}}{n^{3/2}\sqrt{A_n}}
 	\left( \int_{-1/2}^{1/2} R_2(y) \, dy \right)\right] \\
 &= \sum_{n=v_1+1}^{v_k-1} \left[ \frac{v_k^{3/2}}{n^{3/2}\sqrt{A_n}} \left( \int_{-1/2}^{1/2} \,
	\sum_{j=2}^{\infty} \left[\binom{1/2}{j} (-2A_n y)^j \right] dy \right) \right] \\
 &= \sum_{n=v_1+1}^{v_k-1} \left[ \frac{v_k^{3/2}}{n^{3/2}\sqrt{A_n}} \, \sum_{j=1}^{\infty} 
	\left[ \frac{1}{2j+1} \binom{1/2}{2j} A_n^{2j} \right] \right]  \\
 &= \sum_{n=v_1+1}^{v_k-1} \left[ \frac{v_k^{3/2}}{n^{3/2}} \, \sum_{j=1}^{\infty} 
	\left[ \frac{1}{2j+1} \binom{1/2}{2j} A_n^{2j-1/2}  \right] \right]. 
\end{align*}
Note that all terms are positive, so the series in $j$ converges absolutely.
Furthermore, the sum over $n$ is finite. Thus we may exchange the order of summation to obtain
\begin{equation} \label{eq:I_n_IV}
 \sum_{j=1}^{\infty} \left[ \frac{1}{2j+1} \binom{1/2}{2j} \sum_{n=v_1+1}^{v_k-1} \left[ \left(\frac{v_k}{n}\right)^{3/2}  
	 A_n^{2j-1/2} \right]\right].
\end{equation}

To prove that this sum converges, we let
 $$ S_j(v_k) = \sum_{n=v_1+1}^{v_k-1} \left(\frac{v_k}{n}\right)^{3/2} A_n^{2j-1/2} \hskip 40pt j\in\N, \; v_k\in\N, $$
and show first that $S_j(v_k)$ converges as $v_k\to\infty$ for all values of $j$.
We do this by showing that the sequence is Cauchy, i.e., that for all $\delta>0$ there exists $N\in\N$ such that
 $$ v_k>N, \; l\in\N \quad \Rightarrow \quad |S_j(v_k+l)-S_j(v_k)|<\delta. $$
We begin by replacing the index $n$ by $m=v_k-n$, which gives
\begin{align*}
 S_j(v_k) &= \sum_{n=v_1+1}^{v_k-1} \left(\frac{v_k}{n}\right)^{3/2} \left(\frac{n}{(v_k+b)^2-n^2}\right)^{2j-1/2} \\
 &= \sum_{m=1}^{v_k-v_1-1} \, \left(\frac{v_k}{v_k-m}\right)^{3/2} \, \left(\frac{v_k-m}{(v_k+b)^2 - (v_k-m)^2}\right)^{2j-1/2} \\
 &= \sum_{m=1}^{v_k-v_1-1} \, \left(1 + \frac{m}{v_k-m}\right)^{3/2} \, \left(\frac{v_k-m}{(b+m)(2v_k+b-m)}\right)^{2j-1/2}.
\end{align*}
Recall the definition $v_1=\fl{(v_k+b)/\sqrt{2}}$ of $v_1$. For $l\in\N$, we write
 $$ v^{\prime}_1 = \Bfl{\frac{v_k+l+b}{\sqrt{2}}}. $$
Then the difference between terms in the sequence $S_j(v_k)$ behaves like
\begin{align*}
 & | S_j(v_k+l) - S_j(v_k) | \\
  & \hskip 20pt = \left| \sum_{m=1}^{v_k+l-v^{\prime}_1-1} \, \left(1 + \frac{m}{v_k+l-m}\right)^{3/2} \, 
	    \left(\frac{v_k+l-m}{(b+m)(2v_k+2l+b-m)}\right)^{2j-1/2} \right. \\
    & \hskip 90pt \left. - \sum_{m=1}^{v_k-v_1-1} \, 
	    \left(1 + \frac{m}{v_k-m}\right)^{3/2} \, \left(\frac{v_k-m}{(b+m)(2v_k+b-m)}\right)^{2j-1/2} \right| \\
    & \hskip 20pt = \left| \sum_{m=1}^{v_k-v_1-1} \, \left(1 + \frac{m}{v_k-m}\right)^{3/2} \, \left(\frac{v_k-m}{(b+m)(2v_k+b-m)}\right)^{2j-1/2} \right.\\
    & \hskip 45pt \times \left[ \left(1 - \frac{m}{v_k}\left(\frac{l}{v_k+l-m}\right)\right)^{3/2} \, \left(\frac{1+l/(v_k-m)}{1+2l/(2v_k+b-m)}\right)^{2j-1/2} - 1 \right] \\
    & \hskip 70pt \left. + \sum_{m=v_k-v_1}^{v_k+l-v^{\prime}_1-1} \, 
	    \left(1 + \frac{m}{v_k+l-m}\right)^{3/2} \, \left(\frac{v_k+l-m}{(b+m)(2v_k+2l+b-m)}\right)^{2j-1/2} \right| \\
    & \hskip 20pt = S_j(v_k) \, O\left(\frac{1}{v_k}\right) + O\left(\frac{1}{v_k^{2j-1/2}}\right)
\end{align*}
as $v_k\rightarrow\infty$.
We know that $S_j(v_k)$ is bounded, since by (\ref{eq:limsup}) and the above exchange of summation:
 $$ \limsup_{v_k\to\infty} 
	\left[ \sum_{j=1}^{\infty} \left[ \frac{1}{2j+1} \binom{1/2}{2j} \, S_j(v_k) \right] \right]
	\leq \frac{1}{12} \, \frac{1}{\sqrt{b+1}} \, \frac{2b+3}{2b+2}.$$
Thus the distance $| S_j(v_k+l) - S_j(v_k) |$ can be made arbitrarily small for sufficiently large $v_k$,
and the sequence $S_j(v_k)$ is Cauchy.

Furthermore, when we substitute this bound into (\ref{eq:I_n_IV}), we have
\begin{align*}
 & v_k^{3/2} \, \left| \, \sum_{n=v^{\prime}_1+1}^{v_k+l-1} \, I_n(v_k+l,b) - \sum_{n=v_1+1}^{v_k-1} \, I_n(v_k,b) \, \right| \\
 & = \sum_{j=2}^{\infty} \left[ \frac{-1}{2j+1} \binom{1/2}{2j} | S_j(v_k+l) - S_j(v_k) |\right] + O\left(\frac{1}{\alpha^{1/4}}\right) \\
  & = \sum_{j=2}^{\infty} \left[ \frac{-1}{2j+1} \binom{1/2}{2j} \left( S_j(v_k) \, O\left(\frac{1}{v_k}\right) + O\left(\frac{1}{v_k^{2j-1/2}}\right) \right) \right] + O\left(\frac{1}{\sqrt{v_k}}\right) \\
    & = v_k^{3/2} \, \sum_{n=v^{\prime}_1+1}^{v_k+l-1} \, I_n(v_k,b) \, O\left(\frac{1}{v_k}\right)
    + \sum_{j=2}^{\infty} \left[ \frac{-1}{2j+1} \binom{1/2}{2j} O\left(\frac{1}{v_k^{2j-1/2}}\right) \right] + O\left(\frac{1}{\sqrt{v_k}}\right) \\
    & = v_k^{3/2} \, \sum_{n=v^{\prime}_1+1}^{v_k+l-1} \, I_n(v_k,b) \, O\left(\frac{1}{v_k}\right)
     + O\left(\frac{1}{\sqrt{v_k}}\right) \rightarrow 0
\end{align*}
as $v_k\rightarrow0$. Again we know that the sum over $I_n$ is bounded, and the convergence of the limit $\epsilon(b)$ follows.
\end{proof}

Note that the bound can be chosen to be uniform in $b$.